%% file: main2.tex
\documentclass[10pt, oneside]{book}

\usepackage{amssymb}
\usepackage{latexsym}
\usepackage{amsfonts}
\usepackage{amsmath}
\usepackage{amssymb}

\usepackage{subfiles}

\usepackage{geometry}

\newenvironment{proof}[1][Proof.]{\begin{trivlist}
\item[\hskip \labelsep {\bfseries #1}]}{$\blacksquare$ \end{trivlist}}

  \newtheorem{theorem}{Theorem}[section]

\newtheorem{lemma}[theorem]{Lemma}
\newtheorem{corollary}[theorem]{Corollary}

\newtheorem{definition}[theorem]{Definition}

\newcommand{\Aut}{\mathrm{\mathop{Aut}}}

\def\la{{\langle}}
\def\ra{{\rangle}}
\def\char{{\rm char \,}}

\def\a{\alpha}
\def\b{\beta}
\def\e{\varepsilon}

\renewcommand{\leq}{\leqslant}
\renewcommand{\geq}{\geqslant}

\newcommand{\F}{\mathbb F}

\newcommand{\res}[1]{#1^{[p]}}

\newcommand{\ad}{{\sf ad}\,}

\begin{document}

\title{Classification of 5-dimensional restricted Lie algebras over perfect fields, I}

\author{\sc Iren Darijani \ and \ Hamid Usefi
	\thanks{The research was supported in part  by NSERC of Canada, grant \# RGPIN 418201.}
  \\ \\
{Department of Mathematics and Statistics}\\
{ Memorial University of Newfoundland} \\
{ St. John's, NL,
Canada, A1C 5S7
} \\\\
{\it id6010\@@mun.ca \quad\quad usefi\@@mun.ca}
}

\date{\today}

  \maketitle

\input{chapter0}

\input{preliminaries}

\input{abe}

\input{5,2.tex}

\input{5,3.tex}

\input{1-dim.tex}

\input{5,8.tex}

\input{5,9.tex}



\end{document}

%% file: chapter0.tex
%
%

\chapter{Introduction}
Let $L$ be a Lie algebra over a field $\mathbb F$ of positive characteristic $p$. Recall that $L$ is called \textit{restricted} if $L$ affords a $p$-map that satisfies the following conditions for all $x, y\in L $ and $\lambda\in\mathbb F$ 
\begin{enumerate}
	\item $(\lambda x)^{[p]}=\lambda^p x^{[p]}$;
	\item $(\ad x)^p=\ad (x^{[p]})$;
	\item $(x+y)^{[p]}=x^{[p]}+y^{[p]}+\sum_{j=1}^{p-1} s_j(x,y),$\\
	where $js_j(x,y)$ is the coefficient of $t^{j-1}$ in $(\ad (t x+y))^{p-1}(x)$, $t$ an indeterminate. 
\end{enumerate}
Recall that a restricted Lie algebra $L$ is called $p$-nilpotent, if there exists an integer $n$ such that
$x^{[p]^n}=0$, for all $x\in L$. The purpose of this paper is to classify all $p$-nilpotent restricted Lie algebras of dimension 5 over perfect fields of characteristic $p\geq 5$. It follows from the Engel Theorem that if $L$ is finite-dimensional and $p$-nilpotent then $L$ is nilpotent. Our work builds upon the recent work of Schneider and  Usefi \cite{SU}  on the  classification of $p$-nilpotent restricted Lie algebras of dimension up to 4 over perfect fields of characteristic $p$. Our method is different than what is used in \cite{SU} as we describe below.
The analogous classification for small dimensional nilpotent Lie algebras has a long history.
The classification of all nilpotent Lie algebras of  dimension up to five over any  field  has been known for a long time. However, in dimension 6, the characterization depends on the underlying field. In 1958  Morozov \cite{morozov} gave a classification of nilpotent Lie algebras of dimension 6 over a filed of characteristic zero, see also   \cite{BK, nielsen, gong}  for a classification over other fields. 
These classifications, however, differ and it was not easy to compare them until recently that de Graaf \cite{dg} gave a complete classification over any field of characteristic other than 2.
de Graaf's approach  can be  verified computationally and was later revised and extended to characteristic 2
in \cite{CGS}. The classification in dimensions more than 6 is still in progress, see for example \cite{See, Ro}.

We now describe the method used in \cite{SU}  and explain why this method is not applicable in dimension 5. Note that in order to define a $p$-map on $L$, it is enough to define it on a basis of $L$ and then extend it linearly using property (3).
Let $\varphi_1,\ \varphi_2:L\rightarrow  L$ be two $p$-maps on $L$.
Then the restricted Lie algebras $(L,\varphi_1)$ and 
$(L,\varphi_2)$ are isomorphic if and only if there exists
$A\in\Aut(L)$ such that 
$$
A(\varphi_1(x))=\varphi_2( A(x))\quad\mbox{holds for all}\quad x\in L.
$$
Hence, $\varphi_1$ and $\varphi_2$ define isomorphic restricted Lie algebras 
if and only if  there exists 
$A\in\Aut(L)$ such that 
$A\varphi_1A^{-1}=\varphi_2$; that is, they are conjugate under
the automorphism group of $L$.
In this case we say that the $p$-maps $\varphi_1$ and $\varphi_2$ are
{\em equivalent}.
This defines a left action of $\Aut(L)$ on the set of $p$-maps and the
isomorphism classes of restricted Lie algebras correspond to 
the $\Aut(L)$-orbits under this action. The main task  using this approach would be then to find the 
$\Aut(L)$-orbits. This is exactly what the authors did in \cite{SU} to determine 
all  $p$-nilpotent restricted Lie algebras of dimension up to 4 over perfect fields.
However, this task  becomes computationally infeasible to carry out in dimension 5.

The  method we use to classify $p$-nilpotent restricted Lie algebras of dimension 5 
is the analogue of Skjelbred-Sund method \cite{SS} for classifying nilpotent Lie algebras.  
We describe this method below.
Let  $L$ be a restricted Lie algebra and  $M$ a vector space. We view $M$  as a trivial $L$-module and  define  $Z^2(L,M)$, the space of 2-cocycles, and the second cohomology group $H^2(L,M)$ of $L$ with coefficients in $M$ in  Section  \ref{def}.
Let   $\theta=(\phi , \omega) \in Z^2(L,M)$.  We  set  $L_{\theta}=L\oplus M$ as a vector space and  define the Lie bracket and $p$-map  on $L_{\theta}$ by: 
\begin{align*}
[(x_1+m_1),(x_2+m_2)]=[x_1,x_2]+\phi (x_1,x_2), \quad
\res{(x+m)}=\res x+\omega (x).
\end{align*} 
We prove in Theorem \ref{Ltheta} that $L_{\theta}$ with the given bracket and $p$-map is a restricted Lie algebra. 
Now let $K$ be a $p$-nilpotent restricted Lie algebra. Then its center $Z(K)$ is nonzero and there exists $x\in Z(K)$ such that $\res x= 0$. Let $M$ be the one dimensional restricted ideal of $K$ spanned by $x$, and set $L=K/M$. Let $\pi :K\rightarrow L$ be the projection map. We have the exact sequence  of restricted Lie algebras:
$$ 0\rightarrow M\rightarrow K\rightarrow L\rightarrow 0.$$
Choose an injective linear map $\sigma :L\rightarrow K$ such that $\pi \sigma =1_L$. 
Define $\phi : L\times L\rightarrow M$ by $\phi (x_1,x_2)=[\sigma (x_1),\sigma (x_2)]-\sigma ([x_1,x_2])$ and $\omega: L\rightarrow M$ by $\omega (x)=\res {\sigma (x)} -\sigma (\res x)$. We show in Theorem \ref{K=L-theta} that $\theta=(\phi,\omega)\in Z^2(L,M)$ and $K\cong L_{\theta}$.
Therefore, any $p$-nilpotent restricted Lie algebra $K$ of dimension $n$ can be  constructed as a central extensions of  a $p$-nilpotent restricted Lie algebras of dimension $n-1$. 

Now, the group  $\Aut_p(L)$ of restricted automorphisms of $L$ acts on $H^2(L,M)$ in a natural way. 
Let $A \in \Aut_p(L)$ and $\theta = (\phi ,\omega) \in Z^2(L,M)$.  We define $A\theta =(A\phi ,A\omega)$, where $A\phi (x,y)=\phi (A(x),A(y))$ and $A\omega (x)=\omega (A(x))$. 
Let  $\theta_1, \theta_2 \in Z^2(L,M)$.
It follows from Theorem \ref{Aut-orbits}   that  if $\theta_1$ and $\theta_2$ are in the same $\Aut_p(L)$-orbit then exists an isomorphism $f: L_{\theta_1}\to  L_{\theta_2}$ such that $f(M)=M$.
Therefore, we use the action of $\Aut_p(L)$  to reduce the number of isomorphic restricted Lie algebras.

There are nine nilpotent Lie algebras of dimension 5 listed in Theorem \ref{5-dim}. Let $K$ be a nilpotent Lie algebra of dimension 5. Since the $p$-maps are $p$-nilpotent, there exists a central  element $x\in K$ such that $x^{[p]}=0$. Then we let $L=K/\la x\ra$ and find all 1-dimensional central  extensions of $L$ that lead to $K$. That is, we choose those $\theta = (\phi ,\omega) \in Z^2(L,\F  )$ such that $L_{\theta}$ is isomorphic as a Lie algebra to $K$. Then we list all possible $p$-maps that are obtained via different choices of $\theta$ and $x$. We still need to detect and remove the isomorphic algebras from this list. Finally, we shall prove that the remaining algebras in the list are pairwise non-isomorphic.

If $p\geq 5$, then every $p$-map on $K$ is semilinear and it is easier to construct a basis for $H^2(L,\F)$ as described in Section \ref{basis for the second cohomology}. However, for the cases that $p=2$ or $p=3$, the $p$-maps are no longer semilinear and we have to come up with a different way of finding a basis for $H^2(L,\F)$. This difference in turn changes the computations and results in possibly different  restricted Lie algebra  structures on $K$. We shall treat the cases of $p=2, 3$ in a subsequent paper.

A compact version of this work  is available at \cite{DU}. We thank Csaba Schneider and Behrang Noohi for valuable discussions and comments.

\tableofcontents

%% file: preliminaries.tex
\chapter{Preliminaries}	
\section {Restricted Lie algebras}

Throughout this paper $\F$ denotes  a perfect filed of characteristic $p\geq 5$ although most of the material presented in this section is valid over any $p>0$. For basic notation we refer to \cite{SF}.

\subsection{Restricted Lie algebras}\label{def}

\begin{definition}
	\label{def:maindef}
	A restricted Lie algebra of characteristic $p>0$ is a Lie algebra $L
	$ of characteristic $p$ together with a map $L
	\to L
	$, denoted by $x\mapsto x^{[p]}$, that satisfies
	\begin{itemize}
		\item  $(\lambda x)^{[p]}=\lambda^px^{[p]}$, 
		\item $(x+y)^{[p]}=x^{[p]}+y^{[p]}+\sum_{j=1}^{p-1} s_j(x,y)$\\
		where $js_j(x,y)$ is the coefficient of $t^{j-1}$ in $(\ad (t x+y))^{p-1}(x)$, $t$ an indeterminate,
		\item $[x, y^{[p]}]=[x,\underbrace{y,\dots,y}_{p}],$
	\end{itemize}
	for all $x,y\in L
	$ and all $\lambda\in\F$. 
\end{definition}
The map $x\mapsto x^{[p]}$ is referred to as the $p$-map.  
We remark that the second property is equivalent to$$ (x+y)^{[p]}=x^{[p]}+y^{[p]}+\displaystyle \sum_{\stackrel{x_j=x\ {\rm or}\  y}{\scriptscriptstyle x_1=x,x_2=y}} \frac{1}{\#x}[x_1,x_2,\dots,x_p],$$
where $\#x$ denotes the number of $x$'s among the $x_j$. Note that long commutators are left-tapped, that is
$$[x_1,\ldots,x_k,x_{k+1}]=[[x_1,\ldots,x_k],x_{k+1}].$$

For a subset $S$ of $L$, we denote by $\la S\ra_p$ the restricted subalgebra of $L$ generated by $S$ and by $\la S\ra_{\F}$ the subspace spanned by $S$.   Also, we denote by $S^{[p]^n}$ the restricted  subalgebra generated by all $x^{[p]^n}$, where $x\in S$. 
Recall that $S$ is called \emph{$p$-nilpotent}  if there exists an integer $n$ such that $S^{[p]^n}=0$.

If $L$ is a Lie algebra over $\mathbb F$ and $M$ is vector space, a  $q$-dimensional cochain of $L
$ with coefficients in $M$ is a skew-symmetric, $q$-linear map on $L
$ taking values in $M$. We denote the space of $q$-dimensional cochains of a  Lie algebra $L$ with coefficients in $M$ by $C_{\text{cl}}^q(L,M)$. So, we have
\[ 
C_{\text{cl}}^q(L,M)=Hom_{\mathbb F}(\Lambda^q L,M). 
\] 
The coboundary map $\delta^q:C_{\text{cl}}^q(L,M)\to C_{\text{cl}}^{q+1}(L,M)$ is defined by
\begin{eqnarray*}
	(\delta^q\phi)(\ell_1,\dots,\ell_{q+1})&=&\sum_{1\le s<t\le q+1}(-1)^{s+t-1}\phi([\ell_s,\ell_t],\ell_1,\dots,\widehat {\ell_s},\dots,\widehat {\ell_t},\dots,\ell_{q+1}), 
\end{eqnarray*}
where the symbol $\widehat {\ell_s}$ indicates that this term is to be omitted.

\begin{definition}\label{star-prop}
	Let $L$ be a restricted Lie algebra over $\F$ and $M$ a vector space. If $\phi \in C_{\text{cl}}^2(L,M)$ and $\omega : L \rightarrow M$ a function, we say $\omega$ has the $\star$-property with respect to $\phi $ if for every  $x,y \in L$ and $\lambda \in \F$, we have
	\begin{enumerate}
		\item $\omega (\lambda x)=\lambda ^p \omega (x)$
		\item 
		$\omega (x+y) = \omega (x) + \omega (y) +\displaystyle \sum_{\stackrel{x_j=x\ {\rm or}\  y}{\scriptscriptstyle x_1=x,x_2=y}} \frac{1}{\#x}\phi([x_1,x_2,\dots,x_{p-1}],x_p)$,  where  $\# x$  is the number of $x$.
	\end{enumerate}
\end{definition} 

Now, we define the space of 2-dimensional \textit{cochains} of a restricted Lie algebra $L$ with coefficients in $M$ as the subspace spanned by all $(\phi,\omega)$ such that $\phi$ is skew-symmetric and $\omega:L\to M$ has the $\star$-property with respect to $\phi$. We denote this vector space by $C^2(L,M)$.
Evidently if $\omega$ and $\omega^\prime$ have the $\star$-property with respect to  $\phi$ and $\phi^\prime$ respectively, then $\omega+\omega^\prime$ has the $\star$-property with respect to  $\phi+\phi^\prime$, and hence $C^2(L,M)$ is a vector space over $\mathbb F$ by point wise addition in each coordinate.  We have adopted  Definition  \ref{star-prop} from \cite{E}. However, definition of $\star$-property  in the whole generality given in \cite{E} is ambiguous.

\begin{lemma}
	Let $M$ be a vector space and $\psi: L \rightarrow M$ a linear map. Then $\tilde{\psi}: L \rightarrow M$ defined by  $\tilde{\psi }(x)=\psi (\res x )$ has the $\star$-property with respect to $\delta^1\psi$.
\end{lemma}

\begin{proof} Since $\psi[x_1,\dots,x_p]=(\delta^1\psi)([x_1,\dots,x_{p-1}],x_p)$,
we have
\begin{align*}
&\tilde\psi(x+y)=\psi(\res{(x+y)})=\psi(\res x)+\psi(\res y)+\psi(\displaystyle \sum_{\stackrel{x_j=x\ {\rm or}\  y}{\scriptscriptstyle x_1=x,x_2=y}} \frac{1}{\#x}[x_1,x_2,\dots,x_p])\\
&=\psi(x^{[p]})+\psi(y^{[p]}) 
 + \sum_{\stackrel{x_j=x\ {\rm or}\ y}{\scriptscriptstyle x_1=x,x_2=y}} \frac{1}{\#(x)}(\delta^1\psi)([x_1,\dots,x_{p-1}],x_{p}),  \nonumber
\end{align*} for every $x,y \in L$.
\end{proof}

Let $Z^2(L,M)$ be the set  consisting of all $(\phi, \omega) \in C^2(L,M)$ such that 
$\delta^2 \phi=0$ and  
$$
\phi (x,\res y)=\phi([x,\underbrace{y,\ldots,y}_{p-1}],y),
$$
for all $x, y\in L$. The elements of $Z^2(L,M)$ are called \textit{cocycles}. 
Also, let $B^2(L,M)$ be the set consisting of all $(\phi, \omega) \in C^2(L,M)$ such that 
there exists  $\psi \in C^1_{cl}(L,M)$ satisfying  $\delta^1\psi=\phi$ and  $\tilde \psi=\omega$.
The elements of $B^2(L,M)$ are called \textit{coboundaries}. 
It is easy to see that $Z^2(L,M)$ and $B^2(L,M)$ are subspaces of $C^2(L,M)$.

\begin{theorem}
 $B^2(L,M)\subseteq Z^2(L,M),$ so that, the quotient $$H^2(L,M)=Z^2(L,M) /B^2(L,M)$$ is well-defined. 
\end{theorem}
\begin{proof} Let $(\delta^1\psi,\tilde \psi)\in B^2(L,M)$. First, We claim that
$\delta^2\delta^1\psi=0$. Indeed, for all $x,y,z\in L,$ we have
\begin{align*}
\delta^2\delta^1\psi(x,y,z)=&\delta^1\psi([x,y],z)+\delta^1\psi([y,z],x)+
\delta^1\psi([z,x],y)\\
=&\psi([[x,y],z]])+\psi([[y,z],x]])+\psi([[z,x],y]])\\
=&\psi([[x,y],z]]+[[y,z],x]]+[[z,x],y]]),\\
\end{align*}
which is equal to zero by jaccobi identity. Next, we claim that $$(\delta^1\psi)(x,\res y)=(\delta^1\psi)([x,\underbrace {y,...,y}_{p-1}],y),$$ for all $x,y\in L$. Indeed, for all $x,y \in L$, we have
\begin{align*}
(\delta^1\psi)(x,\res y)-(\delta^1\psi)([x,\underbrace {y,\ldots,y}_{p-1}],y)
=&\psi[x,\res y]-\psi ([x,\underbrace {y,\ldots,y}_{p}])\\
=&\psi([x,\res y]-[x,\underbrace {y,\ldots,y}_{p}])\\
=&0.
\end{align*}
The proof is complete.
\end{proof}

We call $H^2(L,M)$ the \textit{second cohomology group} of $L$ with coefficients in $M$. Let $\theta =(\phi, \omega) \in Z^2(L,M)$. Then we denote by
$[\theta]$ the image of $\theta$ in $H^2(L,M)$.
\begin{definition}
	A restricted Lie algebra $M$ is called strongly abelian if $[M,M]=0$ and $\res M=0$.
\end{definition}

\section{Constructing $p$-nilpotent restricted Lie algebras}
Let  $L$ be a restricted Lie algebra, $M$ a vector space and $\theta=(\phi , \omega)\in Z^2(L,M)$. We construct a restricted extension of $L$ by $M$ as follows.
\begin{lemma}\label{Ltheta} Let $L_{\theta}=L\oplus M$ as a vector space and  define the Lie bracket and $p$-map  on $L_{\theta}$ by: 
\begin{align*}
[(x_1+m_1),(x_2+m_2)]=[x_1,x_2]+\phi (x_1,x_2), \quad
\res{(x+m)}=\res x+\omega (x).
\end{align*} 
Then $L_{\theta}$ with the given bracket and $p$-map is a resticted Lie algebra.
\end{lemma}
\begin{proof} 
The bracket is clearly bilinear and skew symmetric and it is well known that the jaccobi identity is equivalent to $\delta^2\phi=0$. We claim that $L$ is restricted with the given $p$-map. Let $x_1,\ldots, x_{k+1}\in L$, $m_1,\ldots m_{k+1}\in M$. Note that by induction we have 
\begin{align}\label{comm}
[x_1+m_1,x_2+m_2,\ldots,x_{k+1}+m_{k+1}]=[x_1,\ldots, x_{k+1}]+\phi([x_1,\ldots,x_k],x_{k+1}).
\end{align}
Now, we have
\begin{align*}
[x_1+m_1,\res {(x_2+m_2)}]=&[x_1+m_1,\res x_2+\omega (x)]\\
=&[x_1,\res x_2]+\phi(x_1,\res x_2).
\end{align*} 
On the other hand,
\begin{align*}
[x_1+m_1,\underbrace{x_2+m_2,\ldots,x_2+m_2}_{p}]=[x_1,\underbrace
{x_2,\ldots,x_2}_{p}]+\phi([x_1,\underbrace{x_2,\ldots,x_2}_{p-1}],x_2),
\end{align*} by equation \eqref{comm}.
We have
$$\phi(x_1,\res x_2)=\phi([x_1,\underbrace{x_2,\ldots,x_2}_{p-1}],x_2),$$
also we have
$$[x_1,\res x_2]=[x_1,\underbrace{x_2,\ldots,x_2}_{p}].$$
Therefore,
$$[x_1+m_1,\res {(x_2+m_2)}]=[x_1+m_1,\underbrace{x_2+m_2,\ldots,x_2+m_2}_{p}].$$
Next, we have
\begin{align*}
\res{(\lambda(x+m))}=\res{(\lambda x+\lambda m)}=\res{(\lambda x)}+\omega(\lambda x)=&\lambda^p\res x+\lambda^p\omega(x)\\
=&\lambda^p(\res x+\omega(x))\\
=&\lambda^p\res{(x+m)}.
\end{align*}
Finally, we have
\begin{align*}
\res{((x_1+m_1)+(x_2+m_2))}=&\res{((x_1+x_2)+(m_1+m_2))}\\
=&\res{(x_1+x_2)}+\omega(x_1+x_2)\\
=&x_1^{[p]}+x_2^{[p]}+\displaystyle \sum_{\stackrel{x_{l_j}=x_1\ {\rm or}\  x_2}{\scriptscriptstyle x_{l_1}=x_1,x_{l_2}=x_2}}
\frac{1}{\#x_1}[x_{l_1},x_{l_2},\dots,x_{l_{p}}]\\
&+\omega (x_1) + \omega (x_2) +\displaystyle \sum_{\stackrel{x_{l_j}=x_1\ {\rm or}\  x_2}{\scriptscriptstyle x_{l_1}=x_1,x_{l_2}=x_2}}\frac{1}{\# x_1} \phi ([x_{l_1},x_{l_2},\ldots,x_{l_{p-1}}],x_{l_{p}}).  
\end{align*}
On the other hand,
\begin{align*}
&\res{(x_1+m_1)}+\res{(x_2+m_2)}+\displaystyle \sum_{\stackrel {l_j=1\ {\rm or}\  2}{\scriptscriptstyle l_1=1, l_2=2}}
\frac{1}{\#(x_1+m_1)}[x_{l_1}+m_{l_1},x_{l_2}+m_{l_2},\dots,x_{l_p}+m_{l_p}]\\
&=\res x_1+\omega(x_1)+\res x_2+\omega(x_2)+\displaystyle \sum_{\stackrel{x_{l_j}=x_1\ {\rm or}\  x_2}{\scriptscriptstyle x_{l_1}=x_1,x_{l_2}=x_2}}
\frac{1}{\#x_1}[x_{l_1},x_{l_2},\dots,x_{l_{p}}]\\
&+\displaystyle \sum_{\stackrel{x_{l_j}=x_1\ {\rm or}\  x_2}{\scriptscriptstyle x_{l_1}=x_1,x_{l_2}=x_2}}\frac{1}{\# x_1} \phi ([x_{l_1},x_{l_2},\ldots,x_{l_{p-1}}],x_{l_{p}}),\\
\end{align*} by equation \eqref{comm}.
Therefore,
\begin{align*}
 &\res{((x_1+m_1)+(x_2+m_2))}=\res{(x_1+m_1)}+\res{(x_2+m_2)}\\
 &+\displaystyle \sum_{\stackrel {l_j=1\ {\rm or}\  2}{\scriptscriptstyle l_1=1, l_2=2}}
\frac{1}{\#x_1+m_1}[x_{l_1}+m_{l_1},x_{l_2}+m_{l_2},\dots,x_{l_p}+m_{l_p}].
\end{align*}
The proof is complete.
\end{proof}
Now let $K$ be a restricted  Lie algebra, and suppose that its center $Z(K),$ is nonzero. Let $0\neq M\subseteq Z(K)$ such that $\res M=0$, and set $L=K/M$. Let $\pi :K\rightarrow L$ be the projection map. We have the exact sequence 
$$ 0\rightarrow M\rightarrow K\rightarrow L\rightarrow 0.$$
 Choose an injective linear map $\sigma :L\rightarrow K$ such that $\pi \sigma =1_L$. Note that we can easily show that $\pi([\sigma (x_i),\sigma (x_j)]-\sigma ([x_i,x_j]))=0,$ for every $x_i,x_j \in L$ and $\pi(\res {\sigma (x)} -\sigma (\res x))=0$, for every $x\in L$. Therefore, 
 \[ [\sigma (x_i),\sigma (x_j)]-\sigma ([x_i,x_j]), \res {\sigma (x)} -\sigma (\res x) \in M, 
 \]
 for every $x_i,x_j,x\in L$.
  Now, we define $\phi : L\times L\rightarrow M$ by $\phi (x_i,x_j)=[\sigma (x_i),\sigma (x_j)]-\sigma ([x_i,x_j])$ and $\omega: L\rightarrow M$ by $\omega (x)=\res {\sigma (x)} -\sigma (\res x)$.
 With these notation, we have:
 \begin{lemma}\label{K=L-theta}
 Let $\theta=(\phi,\omega)$. Then $\theta \in Z^2(L,M)$ and $K\cong L_{\theta}$.
 \end{lemma}
  \begin{proof} It is easy to see that $\phi$ is a bilinear and skew-symmetric form on $L$. We claim that $\omega$ has the $\star$-property with respect to $\phi$. Indeed:
\begin{align*}
\omega(x_1+x_2)=&\res{(\sigma(x_1+x_2))}-\sigma(\res{(x_1+x_2)})\\
=&\res{(\sigma(x_1)+\sigma(x_2))}-\sigma(\res x_1+\res x_2+\displaystyle \sum_{\stackrel{x_{l_j}=x_1\ {\rm or}\  x_2}{\scriptscriptstyle x_{l_1}=x_1,x_{l_2}=x_2}}
\frac{1}{\#x_1}[x_{l_1},x_{l_2},\dots,x_{l_{p}}])\\
=&\res{\sigma(x_1)}+\res{\sigma(x_2)}+\displaystyle \sum_{\stackrel{l_j=1\ {\rm or}\  2}{\scriptscriptstyle l_1=1,l_2=2}}
\frac{1}{\#\sigma(x_1)}[\sigma(x_{l_1}),\sigma(x_{l_2}),\dots,\sigma(x_{l_{p})}])\\
-&\sigma(\res x_1)-\sigma(\res x_2)-\displaystyle \sum_{\stackrel{x_{l_j}=x_1\ {\rm or}\  x_2}{\scriptscriptstyle x_{l_1}=x_1,x_{l_2}=x_2}}
\frac{1}{\#x_1}\sigma([x_{l_1},x_{l_2},\dots,x_{l_{p}}])\\
=&\omega(x_1)+\omega(x_2)+\displaystyle \sum_{\stackrel{l_j=1\ {\rm or}\  2}{\scriptscriptstyle l_1=1,l_2=2}}
\frac{1}{\#x_1}[\sigma(x_{l_1}),\sigma(x_{l_2}),\dots,\sigma(x_{l_{p})}])\\
-&\displaystyle \sum_{\stackrel{x_{l_j}=x_1\ {\rm or}\  x_2}{\scriptscriptstyle x_{l_1}=x_1,x_{l_2}=x_2}}
\frac{1}{\#x_1}\sigma([x_{l_1},x_{l_2},\dots,x_{l_{p}}]).\\
\end{align*}
Since $\phi(x,y)\in M \subseteq Z(L)$, we have
\begin{align*}
&\sigma([x_{l_1},x_{l_2},\dots,x_{l_{p}}])\\
&=[\sigma([x_{l_1},x_{l_2},\dots,x_{l_{p-1}}]),\sigma(x_{l_p})]-\phi([x_{l_1},x_{l_2},\dots,x_{l_{p-1}}],x_{l_p})\\
&=[[\sigma([x_{l_1},x_{l_2},\dots,x_{l_{p-2}}]),\sigma(x_{l_{p-1}})]-\phi([x_{l_1},x_{l_2},\dots,x_{l_{p-2}}],x_{l_{p-1}}),\sigma(x_{l_p})]\\
&-\phi([x_{l_1},x_{l_2},\dots,x_{l_{p-1}}],x_{l_p})\\
&=[[\sigma([x_{l_1},x_{l_2},\dots,x_{l_{p-2}}]),\sigma(x_{l_{p-1}})],\sigma(x_{l_p})]-\phi([x_{l_1},x_{l_2},\dots,x_{l_{p-1}}],x_{l_p}).
\end{align*}
If we repeat this procedure, we obtain that 
\begin{align*}
\sigma([x_{l_1},x_{l_2},\dots,x_{l_{p}}])=[\sigma(x_{l_1}),\sigma(x_{l_2}),\dots,\sigma(x_{l_{p})}])-\phi([x_{l_1},x_{l_2},\dots,x_{l_{p-1}}],x_{l_p}).
\end{align*}
Therefore, we have
\begin{align*}
\omega(x_1+x_2)=\omega(x_1)+\omega(x_2)+\displaystyle \sum_{\stackrel{x_{l_j}=x_1\ {\rm or}\  x_2}{\scriptscriptstyle x_{l_1}=x_1,x_{l_2}=x_2}}\frac{1}{\# x_1} \phi ([x_{l_1},x_{l_2},\ldots,x_{l_{p-1}}],x_{l_{p}}).
\end{align*}
Next, we claim that $\delta^2\phi=0$. Indeed, for all $x,y,z\in L$, we have
\begin{align*}
(\delta^2\phi)(x,y,z)=&\phi([x,y],z)+\phi([y,z],x)+\phi([z,x],y)\\
=&[\sigma([x,y]),\sigma(z)]-\sigma([[x,y],z])+[\sigma([y,z]),\sigma(x)]-\sigma([[y,z],x])\\
+&[\sigma([z,x]),\sigma(y)]-\sigma([[z,x],y])\\
=&[\sigma([x,y]),\sigma(z)]+[\sigma([y,z]),\sigma(x)]+[\sigma([z,x]),\sigma(y)]\\
-&\sigma([[x,y],z]+[[y,z],x]+[[z,x],y]).
\end{align*}
By jaccobi identity, $[x,y,z]+[y,z,x]+[z,x,y]=0$. Therefore,
\begin{align*}
&\delta^2\phi(x,y,z)\\
&=[\sigma([x,y]),\sigma(z)]+[\sigma([y,z]),\sigma(x)]+[\sigma([z,x]),\sigma(y)]\\
&=[[\sigma(x),\sigma(y)]-\phi(x,y),\sigma(z)]+[[\sigma(y),\sigma(z)]-\phi(y,z),\sigma(x)]\\
&+[[\sigma(z),\sigma(x)]-\phi(z,x),\sigma(y)]\\
&=[[\sigma(x),\sigma(y)],\sigma(z)]-[\phi(x,y),\sigma(z)]+[[\sigma(y),\sigma(z)],\sigma(x)]-[\phi(y,z),\sigma(x)]\\
&+[[\sigma(z),\sigma(x)],\sigma(y)]-[\phi(z,x),\sigma(y)]\\
&=[[\sigma(x),\sigma(y)],\sigma(z)]+[[\sigma(y),\sigma(z)],\sigma(x)]+[[\sigma(z),\sigma(x)],\sigma(y)],
\end{align*}
which is equal to zero by jaccobi identity.

Finally, we claim that $\phi(x,\res y)=\phi([x,\underbrace{y,\ldots,y}_{p-1}],y)$, for all $x,y\in L$. Indeed:
\begin{align*}
&\phi(x,\res y)-\phi([x,\underbrace{y,\ldots,y}_{p-1}],y)\\
&=[\sigma(x),\sigma(\res y)]-\sigma([x,\res y])-[\sigma([x,\underbrace{y,\ldots,y}_{p-1}]),\sigma(y)]+\sigma([x,\underbrace{y,\ldots,y}_{p}])\\
&=[\sigma(x),\sigma(\res y)]-[\sigma([x,\underbrace{y,\ldots,y}_{p-1}]),\sigma(y)]+\sigma([x,\underbrace{y,\ldots,y}_{p}]-[x,\res y])\\
&=[\sigma(x),\sigma(\res y)]-[\sigma([x,\underbrace{y,\ldots,y}_{p-1}]),\sigma(y)].
\end{align*}
We have
\begin{align*}
[\sigma([x,\underbrace{y,\ldots,y}_{p-1}]),\sigma(y)]=&[[\sigma([x,\underbrace{y,\ldots,y}_{p-2}]),\sigma(y)]-\phi([x,\underbrace{y,\ldots,y}_{p-2}],y),\sigma(y)]\\
=&[[\sigma([x,\underbrace{y,\ldots,y}_{p-2}]),\sigma(y)],\sigma(y)].
\end{align*}
If we repeat this procedure, we obtain that
\begin{align*}
[\sigma([x,\underbrace{y,\ldots,y}_{p-1}]),\sigma(y)]&=[\sigma(x),\underbrace{\sigma(y),\ldots,\sigma(y)}_{p}]\\
&=[\sigma(x),\res{(\sigma(y))}].
\end{align*}
Therefore,
\begin{align*}
&\phi(x,\res y)-\phi([x,\underbrace{y,\ldots,y}_{p-1}],y)\\
&=[\sigma(x),\sigma(\res y)]-[\sigma(x),\res{(\sigma(y))}]\\
&=[\sigma(x),\res{\sigma(y)}-\omega(x)]-[\sigma(x),\res{(\sigma(y))}]\\
&=[\sigma(x),\res{\sigma(y)}]-[\sigma(x),\res{(\sigma(y))}]\\
&=0.
\end{align*}

 Finally, we show that $K\cong L_{\theta}$. 
 Let $x\in K$. Then there exist unique $y\in L$ and $z\in M$ such that $x=\sigma (y)+z$. Indeed, we can take $z=x-\sigma(\pi(x))$ and $y=\pi(x)$. Since, $\pi(z)=0$, we have $z\in M$. Define $f:K\rightarrow L_{\theta}$ by $f(x)=y+z$. Then $f$ is an isomorphism. Indeed:

$$f(\res x)=f(\res{(\sigma (y)+z)})=f\bigg(\res{\sigma (y)} +\res z+\sum_{\substack{g_j=\sigma (y) \text{ or } z \\ g_1 =\sigma (y),\,  g_2=z}} \frac {1}{\#(\sigma (y))}[g_1,\ldots,g_p]\bigg).$$
Since $z\in M$, we have 
\begin{align*}
\res z=0, \text{ and } \sum_{\substack{g_j=\sigma (y) \text{ or } z \\ g_1 =\sigma (y),\, g_2=z}} \frac {1}{\#(\sigma (y))}[g_1,\ldots, g_p]=0.
\end{align*} 
Thus,
\begin{align*}
f(\res x)=f(\res{\sigma (y)})=f(\sigma (\res y)+\omega (y))=\res y +\omega (y)=\res y+\res{\sigma (y)}-\sigma (\res y).
\end{align*} 
On the other hand, we have
 \begin{align*}
 \res {f(x)}=\res {(y+z)}=\res y+\omega (y)=\res y+\res {\sigma (y)}-\sigma (\res y).
 \end{align*} 
Thus, $f(\res x)=\res {f(x)}$.

Similarly we can show that $f([x,y])=[f(x),f(y)]$.
\end{proof}

We conclude that any $p$-nilpotent restricted Lie algebra $K$ of dimension $n$  can be constructed from a restricted Lie algebra of lower dimension.

\begin{lemma}
Let $\theta=(\phi,\omega) \in Z^2(L,M)$ and $\eta=(\nu,\xi) \in B^2(L,M)$. Then we have $L_{\theta}\cong L_{\theta+\eta}$.
\end{lemma}
\begin{proof}
We have $\eta=(\nu,\xi) \in B^2(L,M)$. Therefore, there exits $\psi \in C^1_{\text{cl}}(L,M)$ such that $\delta^1(\psi)=\nu$ and 
$\tilde{\psi} =\xi$. Note that $f: L_{\theta}\to L_{\theta+\eta}$ such that $f(x)=x+\psi(x)$ for $x\in L$ and $f(m)=m$ for $m\in M$
is an isomorphism. Indeed, let $x_1,x_2\in L$ and $m_1,m_2 \in M$. Then we have
\begin{align*}
f([x_1+m_1,x_2+m_2])=f([x_1,x_2]+\phi(x_1,x_2))=[x_1,x_2]+\psi([x_1,x_2])+\phi(x_1,x_2).
\end{align*}
On the other hand, we have 
\begin{align*}
&[f(x_1+m_1),f(x_2+m_2)]=[x_1+\psi(x_1)+m_1,x_2+\psi(x_2)+m_2]=[x_1,x_2]+(\phi+\nu)(x_1,x_2)=\\
&[x_1,x_2]+\psi([x_1,x_2])+\phi(x_1,x_2).
\end{align*}
Therefore, $f([x_1+m_1,x_2+m_2])=[f(x_1+m_1),f(x_2+m_2)]$. Also, we have
\begin{align*}
f(\res {(x_1+m_1)})=f(\res x_1+\omega(x_1))=\res x_1+\psi(\res x_1)+\omega(x_1).
\end{align*}
On the other hane, we have
\begin{align*}
\res {f(x_1+m_1)}=\res {(x_1+\psi(x_1)+m_1)}=\res x_1+(\omega+\xi)(x_1)=\res x_1+\psi(\res x_1)+\omega(x_1).
\end{align*}
Therefore, $f(\res {(x_1+m_1)})=\res {f(x_1+m_1)}$.
\end{proof}

\section{Strategy}

In this section we describe our strategy to find a (possibly redundant but complete) list of 
all possible $p$-nilpotent $p$-maps on a given nilpotent Lie algebra $K$. Since the $p$-maps are 
$p$-nilpotent, there exists $z\in Z(K)$ such that $z^{[p]}=0$. We let $M=\la z\ra_p$ and  $L=K/M$. We know from Theorem \ref{K=L-theta} that $K\cong L_{\theta}$, for some 
$\theta\in Z^2(L, M)$. Hence, we can construct every restricted Lie algebra of dimension $n$ from a  restricted Lie algebra of dimension $n-1$ via a cocycle. 

For our purpose,   we need the list of all 5-dimensional nilpotent Lie algebras and the list of all $p$-nilpotent restricted Lie algebras of dimension 4. For    reader's convenience, we include these lists below.

\begin{theorem}[\cite{dg}]\label{5-dim}
	The isomorphism class of all nilpotent Lie algebras of dimension 5 over an arbitrary field is as follows:
	\begin{itemize}
		\item $L_{5,1}=\text{abelian}$;
		\item $L_{5,2}=\langle x_1,\ldots,x_5\mid [x_1,x_2]=x_3 \rangle$;
		\item $L_{5,3}=\langle x_1,\ldots,x_5\mid [x_1,x_2]=x_3, [x_1,x_3]=x_4\rangle$;
		\item $L_{5,4}=\langle x_1,\ldots,x_5\mid [x_1,x_2]=x_5, [x_3,x_4]=x_5 \rangle$;
		\item $L_{5,5}=\langle x_1,\ldots,x_5\mid [x_1,x_2]=x_3, [x_1,x_3]=x_5, [x_2,x_4]=x_5\rangle$;
		\item $L_{5,6}=\langle x_1,\ldots,x_5\mid [x_1,x_2]=x_3, [x_1,x_3]=x_4, [x_1,x_4]=x_5, [x_2,x_3]=x_5\rangle$;
		\item $L_{5,7} = \langle x_1, \ldots, x_5\mid  [x_1,x_2]=x_3 , [x_1,x_3]=x_4 , [x_1,x_4]=x_5\rangle$;
		\item $L_{5,8}=\langle x_1,\ldots,x_5\mid [x_1,x_2]=x_4, [x_1,x_3]=x_5\rangle$;
		\item $L_{5,9} = \langle x_1, \ldots, x_5\mid  [x_1,x_2]=x_3 , [x_1,x_3]=x_4 , [x_2,x_3]=x_5\rangle$.
	\end{itemize}
\end{theorem}

\begin{theorem}[\cite{SU}]\label{4-dim}
	Let $L$ be a nilpotent Lie algebra of dimension 
	$4$ over a perfect field $\F$ of characteristic $p\geq 5$. Then the equivalence classes of 
	the $[p]$-maps on $L$ are as follows.
	\begin{itemize}		
		\item[(a)] If $L=\left<x_1,x_2,x_3,x_4\right>$, then
		\begin{enumerate}
			\item Trivial $p$-map;
			\item $\res x_1= x_2$;
			\item $\res x_1=x_2$, $\res x_3=x_4$;
			\item $\res x_1=x_2$, $\res x_2=x_3$;
			\item $\res x_1=x_2$, $\res x_2=x_3$, $\res x_3=x_4$.
		\end{enumerate}
		\item[(b)] If $L=\left<x_1,x_2,x_3,x_4\mid [x_1,x_2]=x_3\right>$, then
		\begin{enumerate}
			\item Trivial $p$-map;
			\item $\res x_1=x_3$;
			\item $\res x_1=x_4$;
			\item $\res x_1=x_3$, $\res x_2=x_4$;
			\item $\res x_3=x_4$;
			\item $\res x_3=x_4$, $\res x_2=x_3$;
			\item $\res x_4=x_3$;
			\item $\res x_4=x_3$, $\res x_2=x_4$.
		\end{enumerate}
		\item[(c)] If  $L=\left<x_1,x_2,x_3,x_4\mid [x_1,x_2]=x_3,	[x_1,x_3]=x_4\right>$, then	
		\begin{enumerate}
			\item Trivial $p$-map;
			\item $\res x_1=x_4$;
			\item $\res x_3=x_4$;
			\item $\res x_2=\xi x_4$,
			where  $\xi \in\F^*$ and  $\xi_1$ and $\xi_2$ represent isomorphic algebras if and only
			if $\xi_1\xi_2^{-1}$ is a square in $\F$.
		\end{enumerate}
	\end{itemize}
\end{theorem}

Next, we need to determine when $L_{\theta_1}\cong L_{\theta_2}$, for  $\theta_1, \theta_2\in Z^2(L, M)$. 
We will do this in Theorem   \ref{Aut-orbits} which involves an action of $\Aut_p(L)$ on $ H^2(L, M)$ that we need to address first. Here, $\Aut_p(L)$ denotes the group of restricted automorphisms of $L$. The group of automorphisms of $L$ where $L$ is just considered as a Lie algebra is denoted by $\Aut(L)$.

So, let $L$ is a restricted Lie algebra and $M$  a strongly abelian restricted Lie algebra which is considered as a trivial L-module.  Let $A \in \Aut_p(L)$ and $\theta = (\phi ,\omega) \in Z^2(L,M)$.  We define $A\theta =(A\phi ,A\omega)$, where $A\phi (x,y)=\phi (A(x),A(y))$ and $A\omega (x)=\omega (A(x))$. With these definitions we have:

\begin{lemma}\label{Aut-action}
Let $A\in \Aut_p(L)$ and $\theta\in Z^2(L,M)$. Then $A\theta\in Z^2(L,M)$. Furthermore, if
$\theta\in B^2(L,M)$ then  $A\theta\in B^2(L,M)$.
\end{lemma}
\begin{proof}
 It is easy to see that $A\phi$ is skew symmetric and bilinear map. Since $\delta^2\phi=0$, we have
$$
\delta^2A\phi (x,y)=\delta^2\phi (A(x),A(y))=0,   
$$
which implies that $\delta^2A\phi =0$. Also, $A\omega$ satisfies the $\star$-property with respect to $A\phi$.  To see this, let  $x,x_1,x_2 \in L$ and $\lambda \in \F$. Then we have,  

$$A\omega (\lambda x)=\omega (A(\lambda x))=\omega (\lambda A(x))=\lambda^p \omega (A(x))=\lambda^pA\omega (x),$$
and
\begin{align*}
A\omega (x_1 +x_2)&=\omega (A(x_1 +x_2))=\omega (Ax_1+Ax_2) \\
&=\omega (Ax_1)+\omega (Ax_2)+ \sum_ {Ax_j =Ax_1 or Ax_2 } \frac {1}{\# (Ax_1)}  \phi ([Ax_1,\ldots,Ax_{p-1}],Ax_p)\\
&=\omega (Ax_1)+\omega (Ax_2) +\sum_{x_j=x_1 or x_2} \frac{1}{\# (x_1)} \phi (A[x_1,\ldots,x_{p-1}],Ax_p)\\
&=\omega (Ax_1)+\omega (Ax_2) +\sum_{x_j=x_1 or x_2} \frac{1}{\#(x_1)} A\phi ([x_1,\ldots,x_{p-1}],x_p)\\
&=A\omega (x_1)+A\omega (x_2)+\sum_{x_j=x_1 or x_2} \frac{1}{\#(x_1)} A\phi ([x_1,\ldots,x_{p-1}],x_p).\\
\end{align*}

We also have
\begin{align*}
 &A\phi (x,\res y)-A\phi ([x,\underbrace{y,\ldots,y}_{p-1} ],y)\\
&=\phi (A(x),A(\res y))- \phi (A([x,\underbrace{y,\ldots,y}_{p-1} ]),A(y))\\
&=\phi (A(x),\res {A(y)})- \phi (([A(x),\underbrace{A(y),\ldots,A(y)}_{p-1} ]),A(y))\\
&=0.
\end{align*}
Now, we show that $\Aut_p(L)$ preserves $B^2(L,M)$. Let $(\delta \psi, \tilde \psi)\in B^2(L,M)$, where $\psi:L\to M$ is linear. We have $A(\delta \psi, \tilde \psi)=(A\delta \psi,A\tilde \psi)$. Then
\begin{align*}
&A\delta \psi(x,y)=\delta\psi(Ax,Ay)=\psi([Ax,Ay])=\psi(A[x,y])=A\psi([x,y])=\delta(A\psi)(x,y), \text{ and } \\
&A\tilde \psi(x)=\tilde\psi (Ax)=\psi(\res {A(x)})=\psi(A(\res x))=A\psi(\res x)=\overset{\sim}{(A\psi)}(x).
\end{align*}So $A(\delta \psi, \tilde \psi)=(\delta(A\psi),\overset{\sim}{(A\psi)})$. 
\end{proof}

It follows from Lemma \ref{Aut-action} that $\Aut_p(L)$ acts on $H^2(L,M)$.
Now, let $e_1,\ldots ,e_s$ be a basis of $M$ and set  $\theta=(\phi ,\omega )\in Z^2(L,M)$. Then 
\begin{align*}
\phi (x,y)=\sum_{i=1}^{s} \phi_i (x,y) e_i,\text{ and }
\omega(x)=\sum_{i=1}^{s} \omega_i(x) e_i.
\end{align*}
\begin{lemma} For every $i$, we have $(\phi_i, \omega_i)\in Z^2(L,\F)$.
\end{lemma}
\begin{proof} Since $\phi$ is skew symmetric and bilinear form, so is every $\phi_i$. We claim that $\delta^2 \phi_i =0$. Indeed, 
$0=\delta^2\phi(x,y) = \sum_{i=1}^{s}\delta^2\phi_i(x,y)e_i$, implying that  $\delta^2\phi_i(x,y)=0$.

 Now we show that $\omega_i$ has $\star$-property with respect $\phi_i$. First note that 
$\omega (\lambda x)=\lambda ^p\omega (x)$. Thus,   $\sum_{i=1}^{s} \omega_i (\lambda x) e_i =\lambda^p \sum_{i=1}^{s} \omega_i(x)e_i$ and so  $ \omega_i(\lambda x)=\lambda^p \omega_i(x)$. Furthermore, we have
\begin{align*}
\omega (x+y) &= \omega (x)+\omega (y)+ \sum_{\substack{x_{j}=x or y\\ x_1=x , x_2=y}} \frac {1}{\# (x)} \phi ([x_1,\ldots,x_{p-1}],x_p).
\end{align*}
Therefore,
\begin{align*}
&\sum_{i=1}^{s} \omega_i(x+y)e_i\\
&=\sum_{i=1}^{s} \omega_i(x)e_i +\sum_{i=1}^{s}\omega_i(y)e_i  +\sum_{\substack{x_j =x or y \\x_1=x , x_2=y}} \frac {1}{\# (x)} \sum_{i=1}^s \phi_i ([x_1,\ldots,x_{p-1}],x_p)e_i,
\end{align*}
which implies that
$$ \omega_i(x+y)=\omega_i(x)+\omega_i(y)+\sum_{\substack {{x_j}=x or y \\ x_1=x , x_2=y }}\frac {1}{\#(x)} \phi_i([x_1,\ldots,x_{p-1}],x_p).
$$
Finally, we have
\begin{align*}
0&=\phi (x,\res y)-\phi ([x,\underbrace{y,\ldots,y}_\text{p-1}],y)\\
&=\sum_{i=1}^s \phi_i(x,\res y)e_i- \sum_{i=1}^s\phi_i([x,\underbrace{y,\ldots,y}_\text{p-1}],y)e_i.
\end{align*}
Therefore 
$$
\phi_i(x,\res y)- \phi_i([x,\underbrace{y,...,y}_\text{p-1}],y)=0.
$$
The proof is complete.
\end{proof}

 Let $\theta_1 =(\phi,\omega), \theta_2=(\psi,\eta)\in Z^2(L,M)$ where 
 $$\phi(x,y)=\sum_{i=1}^s \phi_{i}(x,y) e_i,$$

  $$\psi(x,y)=\sum_{i=1}^s \psi_i(x,y) e_i,$$ and $\omega (x)=\sum_{i=1}^s \omega_{i}(x)e_i$ and $\eta (x)=\sum_{i=1}^s \eta_i(x)e_i$. With these notation we have:

\begin{lemma}\label{Aut-orbits}
There exists an  isomorphism $f: L_{\theta_1}\to  L_{\theta_2}$ with  $f(M)=M$  if and only if there is an $A\in Aut_p(L)$ such that the images of $(A\psi_i,A\eta_i)$'s in $H^2(L,\F)$ span the same subspace of $H^2(L,\F)$ as the images of $(\phi_{i},\omega_{i})$'s.
\end{lemma}
 \begin{proof}
Let $f : L_{\theta_1} \rightarrow L_{\theta_2}$ be an isomorphism such that  $f (M)=M$. Then  $f$ induces an isomorphism of $L_{\theta_1}/M$ to $L_{\theta_2}/M$, that is an automorphism of $L$. Denote this automorphism by $A$. Let $L$ be spanned by $x_1,\ldots, x_n$. Then  $f(x_i)=A(x_i)+v_i$, for some $v_i \in M$. Furthermore,  $f(e_i)=\sum_{j=1}^s a_{ji} e_{j}$, and $v_i=\sum_{l=1}^s \beta_{i\ell}e_{\ell}$. Also write $[x_i,x_j]_L=\sum_{k=1}^n c_{ij}^k x_k$, and $\res x_i=\sum_{k=1}^n b_{ik}x_k$.  
We claim that 
\begin{align}\label{claim1}
\psi_{l}(A(x_i),A(x_j))=\sum_{k=1}^s a_{\ell k}\phi_{k}(x_i,x_j) + \sum_{k=1}^n c_{ij}^k \beta _{k\ell}.
 \end{align}
 To prove the claim, we note that  $f ([x_i,x_j]_{L_{\theta_1}} )=[f(x_i),f(x_j)]_{L_{\theta_2}}$. Then 
 \begin{align*}
 f([x_i,x_j]_{L_{\theta_1}} )=f([x_i,x_j]_L+\phi(x_i,x_j))=f([x_i,x_j]_L)+f(\phi (x_i,x_j)),
 \end{align*}
 where
 \begin{align*}
f([x_i,x_j]_L)=f(\sum_{k=1}^n c_{ij}^k x_k)= \sum_{k=1}^n c_{ij}^kf(x_k)&=\sum_{k=1}^n c_{ij}^k(A(x_k)+v_k)\\
&=\sum_{k=1}^n c_{ij}^kA(x_k)+\sum_{k=1}^nc_{ij}^k v_k,
\end{align*}
and
\begin{align*}
 f(\phi (x_i,x_j))=f(\sum_{k=1}^s \phi_{k}(x_i,x_j) e_k)&=\sum_{k=1}^s \phi_{k}(x_i,x_j)f(e_k)\\
 &=\sum_{k=1}^s \phi_{k}(x_i,x_j)\sum_{l=1}^sa_{\ell k}e_{\ell}\\
 &=\sum_{l=1}^s\sum_{k=1}^s a_{\ell k}\phi_{k}(x_i,x_j)e_{\ell}.
 \end{align*}
 Therefore, we have
 \begin{align*}
  f([x_i,x_j]_{L_{\theta_1}} )&=\sum_{k=1}^n c_{ij}^kA(x_k)+\sum_{k=1}^nc_{ij}^k v_k+\sum_{l=1}^s\sum_{k=1}^s a_{\ell k}\phi_{k}(x_i,x_j)e_{\ell}\\
  &=\sum_{k=1}^n c_{ij}^kA(x_k)+\sum_{l=1}^s\sum_{k=1}^nc_{ij}^k\beta_{k \ell}e_{\ell}+\sum_{l=1}^s\sum_{k=1}^s a_{\ell k}\phi_{k}(x_i,x_j)e_{\ell}.
  \end{align*}
 On the other hand, we have
 \begin{align*}
 [f(x_i),f(x_j)]_{L_{\theta_2}}&=[A(x_i)+v_i,A(x_j)+v_j]_{L_{\theta_2}}\\
 &=[A(x_i),A(x_j)]_L+\psi(A(x_i),A(x_j)),
 \end{align*}
 where
 \begin{align*}
[A(x_i),A(x_j)]_L=A([x_i,x_j]_L)=A(\sum_{k=1}^nc_{ij}^kx_k)= \sum_{k=1}^nc_{ij}^kA(x_k),
\end{align*}
and
\begin{align*}
\psi(A(x_i),A(x_j))=\sum_{l=1}^s\psi_{\ell}(A(x_i),A(x_j))e_{\ell}.
\end{align*}
 Therefore, we have
 \begin{align*}
 [f(x_i),f(x_j)]_{L_{\theta_2}}=\sum_{k=1}^nc_{ij}^kA(x_k)+\sum_{l=1}^s\psi_{\ell}(A(x_i),A(x_j))e_{\ell}.
 \end{align*} 
 As a result, we have
 \begin{align*}
 \sum_{l=1}^s\psi_{\ell}(A(x_i),A(x_j))e_{\ell}=\sum_{l=1}^s\sum_{k=1}^nc_{ij}^k\beta_{k \ell}e_{\ell}+\sum_{l=1}^s\sum_{k=1}^s a_{\ell k}\phi_{k}(x_i,x_j)e_{\ell},
 \end{align*}
which implies that
\begin{align*}
\psi_{\ell}(A(x_i),A(x_j))=\sum_{k=1}^nc_{ij}^k\beta_{k \ell}+\sum_{k=1}^s a_{\ell k}\phi_{k}(x_i,x_j),
\end{align*}
 for $1\leq \ell \leq s$. This proves the claim. Next we claim that
 \begin{align}\label{claim2}
\eta_{\ell}(A(x_i))=\sum_{k=1}^n b_{ik}\beta_{k \ell}+  \sum_{k=1}^s a_{\ell k} \omega_{k}(x_i),
\end{align}
 for every $1\leq \ell \leq s$. To prove the  claim, we note that  $f(\res x_i)=\res {f(x_i)}$. So, 
\begin{align*}
f(\res x_i)=f(\omega(x_i) + \res x_i)=f(\omega(x_i))+f(\res x_i).
\end{align*}
Now we have
\begin{align*}
f(\omega (x_i))=f(\sum_{k=1}^s \omega_{k}(x_i)e_k)=\sum_{k=1}^s \omega_{k}(x_i)f(e_k)&=\sum_{k=1}^s \omega_{k}(x_i) \sum_{\ell=1}^s a_{\ell k}e_{\ell}\\
&=\sum_{\ell =1}^s \sum_{k=1}^s a_{\ell k} \omega_{k}(x_i)e_{\ell}, 
\end{align*}
and
\begin{align*}
f(\res x_i)=f(\sum_{k=1}^n b_{ik} x_k)=\sum_{k=1}^n b_{ik}f(x_k)=\sum_{k=1}^n b_{ik}(A(x_k)+v_k)= \sum_{k=1}^n b_{ik} A(x_k) + \sum_{k=1}^n b_{ik}v_k.
\end{align*}
Therefore, we have
\begin{align*}
f(\res x_i)&=\sum_{k=1}^n b_{ik}A(x_k) +\sum_{k=1}^n b_{ik}v_k + \sum _{\ell=1}^s \sum_{k=1}^s a_{\ell k} \omega_{k}(x_i)e_{\ell}\\
&=\sum_{k=1}^n b_{ik}A(x_k) +\sum_{\ell=1}^s\sum_{k=1}^n b_{ik}\beta_{k \ell}e_{\ell} + \sum _{\ell=1}^s \sum_{k=1}^s a_{\ell k} \omega_{k}(x_i)e_{\ell}.
\end{align*}
On the other hand, we have
\begin{align*}
\res {f(x_i)}=\res {(A(x_i)+v_i)}=\eta (A(x_i))+\res {A(x_i)},
\end{align*}
where
\begin{align*}
&\eta(A(x_i))=\sum_{\ell=1}^s \eta_{\ell}(A(x_i))e_{\ell}, \text{ and }\\
&\res {A(x_i)}=A(\res x_i)=A(\sum_{k=1}^n b_{ik}x_k)=\sum_{k=1}^n b_{ik} A(x_k).
\end{align*}
Therefore, we have
\begin{align*}
\res {f(x_i)}=\sum_{\ell=1}^s \eta_{\ell}(A(x_i))e_{\ell}+\sum_{k=1}^n b_{ik} A(x_k).
\end{align*}
As a result,
\begin{align*}
\sum_{\ell=1}^s \eta_{\ell}(A(x_i))e_{\ell}=\sum_{\ell=1}^s\sum_{k=1}^n b_{ik}\beta_{k \ell}e_{\ell} + \sum _{\ell=1}^s \sum_{k=1}^s a_{\ell k} \omega_{k}(x_i)e_{\ell},
\end{align*}
which implies that
\begin{align*}
\eta_{\ell}(A(x_i))=\sum_{k=1}^n b_{ik}\beta_{k \ell}+  \sum_{k=1}^s a_{\ell k} \omega_{k}(x_i),
\end{align*}
 for $1\leq \ell \leq s$. This proves the second claim.
 
Now define the linear function $f_{\ell} :L\rightarrow \F$ by $f_{\ell}(x_k)=\beta_{k \ell}$. Then 
\begin{align*}
 &(\delta^1 f_{\ell})(x_i,x_j)=f_{\ell}([x_i,x_j])=\sum_{k=1 }^n c_{ij}^k \beta_{k \ell}, \text{ and }\\
 &\tilde {f_{\ell}} (x_i)=f_{\ell}(\res x_i)=f_{\ell}(\sum_{k=1 }^n b_{ik}x_k)=\sum_{k=1 }^n b_{ik}\beta_{k \ell}.
 \end{align*}
Hence, $(\delta^1 f_{\ell}, \tilde {f_{\ell}} )\in B^2(L,\F)$. We deduce, by Equations \eqref{claim1} and \eqref{claim2}, that the subspace spanned by all the  $(A\psi_i, A\eta_i)$'s is the same as  the subspace spanned by all the $(\phi_{i},\omega_{i})$'s  modulo $B^2(L,\F)$.

Conversely, suppose that the image of $(A\psi_i,A\eta_i)$'s span the same subspace of $H^2(L,\F)$ as the the image of $(\phi_{i},\omega_{i})$'s. Then there are linear functions $f_{\ell} :L\rightarrow \F$ and $a_{\ell k} \in \F$ so that 
\begin{align*}
&A\psi_{\ell}(x_i,x_j)=\sum_{k=1}^s a_{\ell k}\phi_{k}(x_i,x_j) + f_{\ell}([x_i,x_j]), \text{ and }\\
&A\eta_{\ell}(x_i)=\sum_{k=1}^s a_{\ell k} \omega_{k}(x_i) +\tilde {f_{\ell}}(x_i).
\end{align*}
If we set $\beta_{k \ell}=f_{\ell}(x_k)$, then we see that Equations \eqref{claim1} and \eqref{claim2} hold. This means that, if we define  $f : L_{\theta_1} \rightarrow L_{\theta_2}$ by $f(x_i)=A(x_i) +\sum_{\ell} \beta _{il}e_{\ell}$ and  $f(e_i) =\sum_j a_{ji}e_j$, then $f$ is an isomorphism.
\end{proof}

Even though Lemma \ref{Aut-orbits} provides us with a criteria to know when 
$L_{\theta_1}\cong L_{\theta_2}$, it is hard to apply Lemma \ref{Aut-orbits}  in practice. Instead, we consider the following convention:

\begin{definition}
Let $\theta \in Z^2(L,M)$. We define  $\bar \theta=\{A\theta \mid A\in \Aut_p(L)\}$. We call $\bar \theta$ an $\Aut_p(L)$-orbit and $\theta$ an $\Aut_p(L)$-orbit representative. We say $\theta_1$ and $\theta_2$ are in the same $\Aut_p(L)$-orbit if there exists $A\in \Aut_p(L)$ such that $A\theta_1=\theta_2$.
\end{definition}

Clearly, if $\theta_1$ and $\theta_2$ are in the same $\Aut_p(L)$-orbit then by Theorem 
\ref{Aut-orbits}, $L_{\theta_1}\cong L_{\theta_2}$. However, the converse is not true, see for instance Lemma \ref{first-example-orbit}. 
Nevertheless, it is practical to find all $\Aut_p(L)$-orbit representatives and this way we get a list  of all possible extensions of $L$ by $M$.  We shall then  eliminate all redundancies. 

\section{Finding a basis for $H^2(L,\F)$}\label{basis for the second cohomology}

In this section, we describe how to find a basis for $H^2(L,\F)$, where 
$L$ is  a $p$-nilpotent 4-dimensional restricted Lie algebra and $p\geq 5$. Let   $x_1,\ldots,x_4$ be a basis of $L$ and $(\phi,\omega)\in C^2(L,\F)$.  We have $\phi=\sum_{1\leq i<j\leq 4} c_{ij}\Delta _{ij}$, where $\Delta_{ij}$ is the skew symmetric matrix with $(i,j)$-entry equal to $1$, $(j,i)$-entry equal to $-1$ and all  other entries equal to  $0$. Also, $\omega$ has the $\star$-property with respect to $\phi$:
\begin{align}\label{con-2}
\omega (x+y) = \omega (x) + \omega (y) +\displaystyle \sum_{\substack{x_j=x \text{ or } y} } \frac{1}{\# x} \phi ([x,y, x_1,\ldots,x_{p-3}],x_{p-2}),
\end{align}
for all $x,y\in L$. Since $p\geq 5$ and $L$ is  nilpotent of class at most 3, Equation \eqref{con-2} reduces to
$\omega (x+y) = \omega (x) + \omega (y)$, for all $x,y\in L$. This means that   $\omega$ is semilinear.

Since  $\omega$ is determined by its evaluation on $x_1,\ldots, x_4$, we can write   $\omega=\alpha f_1+\beta f_2+\gamma f_3 + \delta f_4$, where $f_i(x_j)=\delta_{i,j}$, and $\alpha, \beta, \gamma, \delta\in  \F$.
We deduce that the set 
\begin{align}\label{basis-C2}
\{(\Delta_{ij},0),(0,f_k)\mid 1\leq i<j\leq 4, 1\leq k\leq 4\}
\end{align}
forms a basis for $C^2(L,\F)$. Now suppose that $(\phi,\omega)\in Z^2(L,\F)$.  Then we must have 
$\delta^2\phi=0$ and  

\begin{align}
\phi (x,\res y)=\phi([x,\underbrace{y,\ldots,y}_{p-1}],y), \label{con-1}
\end{align}
for all $x,y\in L$. Since $p\geq 5$ and  $L$ is nilpotent of class at most 3, Equation \eqref{con-1} reduces to  $\phi (x,\res y)=0$.
The two conditions $\delta^2\phi=0$ and $\phi (x,\res y)=0$ impose certain   constraints on the  $c_{ij}$'s.
This way, we    obtain a subset $S$ of  the set given in \eqref{basis-C2} that serves as a basis  for $Z^2(L,\F)$.

Furthermore, if  $(\phi,\omega)\in B^2(L,\F)$ then we must  have  $\phi=\delta^1\psi$
and $\omega=\tilde \psi$, for some linear map  $\psi:L\to \F$. These latter conditions impose further restrictions on   the  $c_{ij}$'s and $\alpha_k$'s. This way, we obtain a basis for $B^2(L,\F)$ in which every basis element is expressed as a linear combination of elements of $S$. 
So every element of this basis of $B^2(L,\F)$ serves as a  dependence relation between the elements of $S$ in $Z^2(L,\F)/B^2(L,\F)$. As a basis for $H^2(L,\F)$ we take the elements $[s]$, with $s\in S$,  modulo these dependence relations.

\begin{lemma}\label{invariants}
Let $f: L \to H$ be an isomorphism of restricted Lie algebras. Then  $f$ induces an isomorphism between $L/\la L^{[p]}\ra_p$ and $H/\la H^{[p]}\ra_p$.
\end{lemma}
\begin{proof}
 Let $\phi : H\to H/\la H^{[p]}\ra_p$ be the canonical homomorphism. Then $\psi=\phi f :L\to  H/\la H^{[p]}\ra_p$ is a surjective homomorphism of restricted Lie algeras with $\ker \psi=\la L^{[p]}\ra_p$. Therefore,  $L/\la L^{[p]}\ra_p$ and $H/\la H^{[p]}\ra_p$ are isomorphic.
\end{proof}

%% file: abe.tex
\chapter{Restriction maps on the abelian Lie algebra}

Let $x_1, \ldots, x_5$ be a basis of the abelian Lie algebra  $K_1=L_{5,1}$ of dimension 5.  Since we want to determine $p$-nilpotent
maps on $L_{5,1}$, without loss of generality, we assume that $\res x_5=0$.
Let $L=L_{5,1}/\la x_5\ra$. Since $\dim L=4$, by Theorem \ref{4-dim}, there are five restricted Lie algebra structures on $L$ given by the following $p$-maps:   
\begin{enumerate}
\item[I.1] Trivial $p$-map;
\item[I.2]$\res x_1=x_2$;
\item[I.3]$\res x_1 =x_2 , \res x_3=x_4$;
\item[I.4]$\res x_1 =x_2 , \res x_2=x_3$;
\item[I.5]$\res x_1 =x_2 , \res x_2=x_3, \res x_3=x_4$.
\end{enumerate}

\section{Extensions of ($L_{5,1}/\la x_5\ra$, trivial $p$-map)}

We have $\res x_i\in \la x_5\ra$, for all $1\leq i\leq 4$. Hence $\res x_i=\alpha_i x_5$,  for some $\alpha_i\in \F$ and for all $1\leq i\leq 4$. If $\alpha_i \neq 0$, then without loss of generality, we assume $\alpha_1\neq 0$. Now, we rescale $x_1$ such that $\res x_1=x_5$. Now, if $\alpha_j \neq 0$ then we rescale $x_j$ such that $\res x_j=x_5$. Finally, if $\res x_j=x_5$, then we replace $x_j$ with $x_j-x_1$ such that $\res x_j=0$.
Therefore, the possible $p$-maps are as follows:
\begin{align*}
&K_1^1=\la x_1, \ldots, x_5\ra;\\
&K_1^2=\la x_1,\ldots, x_5\mid \res x_1=x_5\ra.
\end{align*}

\section{Extensions of ($L_{5,1}/\la x_5\ra, \res x_1=x_2 $)}

In this case we have $\res x_1-x_2 \in \la x_5\ra$ and $\res x_2, \res x_3,\res x_4 \in \la x_5\ra$. Hence, $\res x_1=x_2+\alpha_1 x_5$, $\res x_2=\alpha_2 x_5$, $\res x_3=\alpha_3 x_5$, and $\res x_4=\alpha_4 x_5$, for some $\alpha_1, \alpha_2, \alpha_3, \alpha_4 \in \F$. We replace $x_2$ with $x_2+\alpha_1 x_5$ to obtain that $\res x_1=x_2$. We consider two cases:

\textbf {Case 1.} $\alpha_2=0$: First, if $\alpha_3=\alpha_4= 0$ then we have the following $p$-map:
$$K_1^3=\la x_1,\ldots, x_5\mid \res x_1=x_2\ra.$$
Next, if one of $\alpha_3, \alpha_4$ is zero and one is nonzero, without loss of generality, we assume that $\alpha_4=0$. Now, we rescale $x_3$ so that $\res x_3=x_5$. Therefore, we have the following $p$-map:
$$K_1^4=\la x_1,\ldots,x_5\mid \res x_1=x_2, \res x_3=x_5\ra.$$
Finally, if both $\alpha_3$ and $\alpha_4$ are non-zero then we rescale $x_3$ and $x_4$ so that $\res x_3=x_5$ and $\res x_4=x_5$. Now, we replace $x_4$ with $x_4-x_3$ so that $\res x_4=0$. Therefore, we obtain the same $p$-map as the previous one.

\textbf {Case 2.} $\alpha_2\neq 0$: First, if $\alpha_3=\alpha_4= 0$ then in $L_{5,1}$ we replace $x_5$ with $\alpha_2x_5$ so that $\res x_2=x_5$. Therefore, we have the following $p$-map: 
$$K_1^5=\la x_1,\ldots,x_5\mid \res x_1=x_2, \res x_2=x_5\ra.$$
Next, if one of $\alpha_3, \alpha_4$ is zero and one is nonzero, without loss of generality, we assume that $\alpha_4=0$. Now, in $L_{5,1}$ we replace $x_3$ with $(\alpha_2/\alpha_3)^{1/p}x_3$ and $x_5$ with $\alpha_2 x_5$ so that $\res x_2=x_5$ and $\res x_3=x_5$. Then, we replace $x_3$ with $x_3-x_2$ so that $\res x_3=0$. Therefore, we obtain the same $p$-map as $K_1^5$.

Finally, if both $\alpha_3$ and $\alpha_4$ are non-zero then in $L_{5,1}$ we replace  $x_3$ with $(\alpha_2/\alpha_3)^{1/p}x_3$,
$x_4$ with $(\alpha_2/\alpha_4)^{1/p}x_4$ and $x_5$ with $\alpha_2 x_5$ so that $\res x_2=x_5$, $\res x_3=x_5$ and $\res x_4=x_5$. Now, we replace $x_3$ with $x_3-x_2$ and $x_4$ with $x_4-x_2$ so that $\res x_3=0$ and $\res x_4=0$.
Therefore, we obtain the same $p$-map as $K_1^5$.

\section{Extensions of ($L_{5,1}/\la x_5\ra, \res x_1 =x_2 , \res x_3=x_4$)}

In this case we have $\res x_1-x_2 \in \la x_5\ra$, $\res x_3-x_4 \in \la x_5\ra$, and $\res x_2,\res x_4 \in \la x_5\ra$.
Hence, $\res x_1=x_2+\alpha_1 x_5$, $\res x_3=x_4+\alpha_3 x_5$, $\res x_2=\alpha_2 x_5$, and $\res x_4=\alpha_4 x_5$ 
for some $\alpha_1, \alpha_2, \alpha_3, \alpha_4 \in \F$. We replace $x_2$ with $x_2+\alpha_1 x_5$ so that $\res x_1=x_2$ and $x_4$ with $x_4+\alpha_3 x_5$ so that $\res x_3=x_4$.

 First, if $\alpha_2= \alpha_4= 0$ then we have the following $p$-map: 
$$K_1^6=\la x_1,\ldots,x_5\mid \res x_1=x_2, \res x_3=x_4\ra.$$
Next, if $\alpha_2\neq 0$ and $\alpha_4=0$ then in $L_{5,1}$ we replace $x_5$ with $\alpha_2 x_5$ so that $\res x_2=x_5$.
Therefore, we have the following $p$-map:
$$K_1^7=\la x_1,\ldots,x_5\mid \res x_1=x_2, \res x_3=x_4, \res x_2=x_5\ra.$$
Next, if $\alpha_2= 0$ and $\alpha_4\neq 0$ then in $L_{5,1}$ we replace $x_5$ with $\alpha_4 x_5$ so that $\res x_4=x_5$.
Therefore, we have the following $p$-map:
$$K_1^8=\la x_1,\ldots,x_5\mid \res x_1=x_2, \res x_3=x_4, \res x_4=x_5\ra.$$
Finally, if $\alpha_2\neq 0$ and $\alpha_4\neq 0$ then in $L_{5,1}$ we replace $x_3$ with $((\alpha_2/\alpha_4)^{1/p})^{1/p}x_3$, $x_4$ with $(\alpha_2/\alpha_4)^{1/p}x_4$ and $x_5$ with $\alpha_2x_5$ so that $\res x_2=x_5$ and $\res x_4=x_5$. Therefore, we have the following $p$-map:
$$K_1^9=\la x_1,\ldots,x_5\mid \res x_1=x_2, \res x_2=x_5,\res x_3=x_4, \res x_4=x_5\ra.$$

\section{Extensions of ($L_{5,1}/\la x_5\ra, \res x_1 =x_2 , \res x_2=x_3$)}
In this case we have $\res x_1-x_2 \in \la x_5\ra$, $\res x_2-x_3 \in \la x_5\ra$, and $\res x_3,\res x_4 \in \la x_5\ra$.
Hence, $\res x_1=x_2+\alpha_1 x_5$, $\res x_2=x_3+\alpha_2 x_5$, $\res x_3=\alpha_3 x_5$, and $\res x_4=\alpha_4 x_5$ 
for some $\alpha_1, \alpha_2, \alpha_3, \alpha_4 \in \F$. We replace $x_2$ with $x_2+\alpha_1 x_5$ so that $\res x_1=x_2$ and $x_3$ with $x_3+\alpha_2 x_5$ so that $\res x_2=x_3$.

 First, if $\alpha_3=\alpha_4= 0$ then we have the following $p$-map: 
$$K_1^{10}=\la x_1,\ldots,x_5\mid \res x_1=x_2, \res x_2=x_3\ra.$$
Next, if $\alpha_3\neq 0$ and $\alpha_4=0$ then in $L_{5,1}$ we replace $x_5$ with $\alpha_3 x_5$ so that $\res x_3=x_5$.
Therefore, we have the following $p$-map:
$$K_1^{11}=\la x_1,\ldots,x_5\mid\res x_1=x_2, \res x_2=x_3, \res x_3=x_5\ra.$$
Next, if $\alpha_3= 0$ and $\alpha_4\neq 0$ then in $L_{5,1}$ we replace $x_5$ with $\alpha_4 x_5$ so that $\res x_4=x_5$.
Therefore, we have the following $p$-map:
$$K_1^{12}=\la x_1,\ldots,x_5\mid \res x_1=x_2, \res x_2=x_3, \res x_4=x_5\ra.$$
Finally, if $\alpha_3\neq 0$ and $\alpha_4\neq 0$ then in $L_{5,1}$ we replace $x_4$ with $(\alpha_3/\alpha_4)^{1/p}x_4$ and $x_5$ with $\alpha_3x_5$ so that $\res x_3=x_5$ and $\res x_4=x_5$. Now, we replace $x_4$ with $x_4-x_3$ so that $\res x_4$=0. Therefore, we have the following $p$-map:
$$K_1^{13}=\la x_1,\ldots,x_5 \mid \res x_1=x_2, \res x_2=x_3,\res x_3=x_5\ra.$$

\section{Extensions of ($L_{5,1}/\la x_5\ra, \res x_1 =x_2 , \res x_2=x_3, \res x_3=x_4$)}
In this case we have $\res x_1-x_2 \in \la x_5\ra$, $\res x_2-x_3 \in \la x_5\ra$, $\res x_3-x_4 \in \la x_5\ra$ and $\res x_4 \in \la x_5\ra$.
Hence, $\res x_1=x_2+\alpha_1 x_5$, $\res x_2=x_3+\alpha_2 x_5$, $\res x_3=x_4+\alpha_3 x_5$, and $\res x_4=\alpha_4 x_5$ 
for some $\alpha_1, \alpha_2, \alpha_3, \alpha_4 \in \F$. We replace $x_2$ with $x_2+\alpha_1 x_5$ so that $\res x_1=x_2$, $x_3$ with $x_3+\alpha_2 x_5$ so that $\res x_2=x_3$ and $x_4$ with $x_4+\alpha_3x_5$ so that $\res x_3=x_4$.
First, if $\alpha_4=0$ then we have the following $p$-map:
$$K_1^{14}=\la x_1,\ldots,x_5 \mid \res x_1=x_2, \res x_2=x_3,\res x_3=x_4\ra.$$
Next, if $\alpha_4 \neq 0$  then in $L_{5,1}$ we replace $x_5$ with $\alpha_4 x_5$ so that $\res x_4=x_5$.
Therefore, we have the following $p$-map:
$$K_1^{15}=\la x_1,\ldots,x_5\mid \res x_1=x_2, \res x_2=x_3, \res x_3=x_4, \res x_4=x_5\ra.$$

Therefore, the list of all restricted Lie algebra structures on $L_{5,1}$ are as follows:
\begin{align*}
&K_1^1=\la x_1,\ldots,x_5\ra;\\
&K_1^2=\la x_1,\ldots,x_5\mid \res x_1=x_5\ra;\\
&K_1^3=\la x_1,\ldots,x_5 \mid \res x_1=x_2\ra;\\
&K_1^4=\la x_1,\ldots,x_5 \mid \res x_1=x_2, \res x_3=x_5\ra;\\
&K_1^5=\la x_1,\ldots,x_5 \mid \res x_1=x_2, \res x_2=x_5\ra;\\
&K_1^6=\la x_1,\ldots,x_5 \mid \res x_1=x_2, \res x_3=x_4, \res x_2=x_5\ra;\\
&K_1^7=\la x_1,\ldots,x_5 \mid \res x_1=x_2, \res x_3=x_4\ra;\\
&K_1^8=\la x_1,\ldots,x_5 \mid \res x_1=x_2, \res x_3=x_4, \res x_4=x_5\ra;\\
&K_1^9=\la x_1,\ldots,x_5 \mid \res x_1=x_2, \res x_2=x_5,\res x_3=x_4, \res x_4=x_5\ra;\\
&K_1^{10}=\la x_1,\ldots,x_5 \mid \res x_1=x_2, \res x_2=x_3\ra;\\
&K_1^{11}=\la x_1,\ldots,x_5 \mid \res x_1=x_2, \res x_2=x_3, \res x_3=x_5\ra;\\
&K_1^{12}=\la x_1,\ldots,x_5 \mid \res x_1=x_2, \res x_2=x_3, \res x_4=x_5\ra;\\
&K_1^{13}=\la x_1,\ldots,x_5 \mid \res x_1=x_2, \res x_2=x_3,\res x_3=x_5\ra;\\
&K_1^{14}=\la x_1,\ldots,x_5 \mid \res x_1=x_2, \res x_2=x_3,\res x_3=x_4\ra;\\
&K_1^{15}=\la x_1,\ldots,x_5 \mid \res x_1=x_2, \res x_2=x_3, \res x_3=x_4, \res x_4=x_5\ra.
\end{align*}
Note that $K_1^2\cong K_1^3$, $K_1^7\cong K_1^4$, $K_1^{10}\cong K_1^{5}$, $K_1^{14}\cong K_1^{11}$, and $K_1^6\cong K_1^8\cong K_1^{12}$. Furthermore, $K_1^{13}$ is identical to $K_1^{11}$. 

\begin{theorem}\label{L51}
 The list of all  restricted Lie algebra structures on $L_{5,1}$, up to isomorphism, is as follows:
\begin{align*}
& L_{5,1}^1=\la x_1, \dots, x_5\ra\\
&L_{5,1}^2=\la x_1, \dots ,x_5 \mid \res x_1=x_5\ra\\
&L_{5,1}^3=\la x_1, \dots, x_5 \mid \res x_1=x_2, \res x_3 =x_5\ra\\
&L_{5,1}^4=\la x_1, \dots, x_5 \mid \res x_1 =x_2, \res x_2 =x_5\ra\\
&L_{5,1}^5=\la x_1 ,\dots, x_5 \mid \res x_1=x_2, \res x_3=x_4, \res x_2=x_5\ra\\
&L_{5,1}^6=\la x_1 ,\dots, x_5 \mid \res x_1=x_2, \res x_2 =x_5, \res x_3=x_4, \res x_4=x_5\ra\\
&L_{5,1}^7=\la x_1, \dots, x_5 \mid \res x_1=x_2, \res x_2=x_3, \res x_3=x_5\ra\\
&L_{5,1}^8=\la x_1, \dots, x_5 \mid \res x_1=x_2, \res x_2 =x_3, \res x_3=x_4, \res x_4=x_5\ra
\end{align*}
\end{theorem}

It is clear from  the table below that the restricted Lie algebras given in Theorem  \ref{L51} are pairwise non-isomorphic.

\begin{table}[htbp]
\begin{tabular}{l | l | l | l}\hline 
L &$\dim L^{[p]}$ & $\dim L^{[p]^2}$ & $\dim L^{[p]^3}$ \\ \hline
$L_{5,1}^1=\la x_1, \dots, x_5\ra$ &0&0&0\\
$L_{5,1}^2=\la x_1, \dots ,x_5 \mid \res x_1=x_5\ra$ &1&0&0\\
$L_{5,1}^3=\la x_1, \dots, x_5 \mid \res x_1=x_2, \res x_3 =x_5\ra$ &2&0&0\\
$L_{5,1}^4=\la x_1, \dots, x_5 \mid \res x_1 =x_2, \res x_2 =x_5\ra$ &2&1&0\\
$L_{5,1}^5=\la x_1 ,\dots, x_5 \mid \res x_1=x_2, \res x_3=x_4, \res x_2=x_5\ra$ &3&1&0\\
$L_{5,1}^6=\la x_1 ,\dots, x_5 \mid \res x_1=x_2, \res x_2 =x_5, \res x_3=x_4, \res x_4=x_5\ra$ &3&2&0\\
$L_{5,1}^7=\la x_1, \dots, x_5 \mid \res x_1=x_2, \res x_2=x_3, \res x_3=x_5\ra$ &3&2&1\\
$L_{5,1}^8=\la x_1, \dots, x_5 \mid \res x_1=x_2, \res x_2 =x_3, \res x_3=x_4, \res x_4=x_5\ra$ &4&3&2
\end{tabular}
\end{table}

%% file: 5,2.tex
\chapter{Restriction maps on $L_{5,2}$}
Let 
$$
K_2=L_{5,2}=\langle x_1,\ldots,x_5\mid [x_1,x_2]=x_3 \rangle_{\F}.
$$  
Note that   $Z(L_{5,2})=\la x_3,x_4,x_5\ra_{\F}$ and the group $\Aut(L_{5,2})$ consists of invertible matrices of the form 
\[\begin{pmatrix} a_{11} & a_{12} & 0 & 0 & 0 \\
a_{21} & a_{22} & 0& 0 & 0\\
a_{31} & a_{32} &r & a_{34} & a_{35} \\
a_{41} & a_{42} & 0& a_{44}
& a_{45} \\
a_{51} & a_{52} & 0& a_{54}
& a_{55}
\end{pmatrix},\]
where $r=a_{11}a_{22}-a_{12}a_{21}$. 
There exists an element $\alpha x_3+\beta x_4+\gamma x_5 \in Z(L_{5,2})$ such that 
 $\res {(\alpha x_3+\beta x_4+\gamma x_5)}=0,$ for some $\alpha ,\beta, \gamma \in \F$. If $\gamma \neq 0$, then consider 
 $$ H_1=\langle x_1',\ldots,x_5'\mid [x_1',x_2']=x_3'\rangle,$$
 where $x_1'=x_1, x_2'=x_2, x_3'=x_3, x_4'=x_4, x_5'=\alpha x_3+\beta x_4+\gamma x_5$. 
 Let $\phi :K_2\to H_1$ such that $x_1\mapsto x_1'$,  $x_2\mapsto x_2'$,  $x_3\mapsto x_3'$,  $x_4\mapsto x_4'$, and  $x_5\mapsto x_5'$. It is easy to see that $\phi$ is an isomorphism. Therefore, in this case we can suppose that $\res x_5=0$. Next, if $\gamma=0$ and $\beta \neq 0$  
 then consider 
 $$ H_2=\langle x_1',\ldots,x_5'\mid [x_1',x_2']=x_3'\rangle,$$
 where $x_1'=x_1, x_2'=x_2, x_3'=x_3, x_4'=\alpha x_3+\beta x_4, x_5'= x_5$.  Let $\phi :K_2\to H_2$ such that $x_1\mapsto x_1'$,  $x_2\mapsto x_2'$,  $x_3\mapsto x_3'$,  $x_4\mapsto x_4'$, and  $x_5\mapsto x_5'$. It is easy to see that $\phi$ is an isomorphism. Therefore, in this case we can suppose that $\res x_4=0$. Finally, if $\gamma=\beta=0$ and $\alpha \neq 0$, then consider
$$ H_3=\langle x_1',\ldots,x_5'\mid [x_1',x_2']=x_3'\rangle,$$
where $x_1'=x_1, x_2'=x_2, x_3'=\alpha x_3, x_4'=x_4, x_5'= x_5$. Let $\phi :K_2\to H_3$ such that $x_1\mapsto x_1'$,  $x_2\mapsto x_2'$,  $x_3\mapsto x_3'$,  $x_4\mapsto x_4'$, and  $x_5\mapsto x_5'$. It is easy to see that $\phi$ is an isomorphism. Therefore, in this case we can suppose that $\res x_3=0$.
 Hence we have three cases:
 \begin{enumerate}
\item[I.]  $\res x_3=0$; 
\item[II.] $\res x_5=0$;\
\item[III.] $\res x_4=0$.
\end{enumerate}

Note that cases II and III yield the same restricted Lie algebras because given a restricted Lie algebra structure on $L_{5,2}$ for which $\res x_5=0$, the automorphism of $L_{5,2}$ obtained by  switching $x_4$ and $x_5$ gives rise to 
a restricted Lie algebra structure on $L_{5,2}$ for which $\res x_4=0$.

\section{Extensions of  $L=\frac {L_{5,2}}{\la x_3\ra}$}\label{L52-x3}
In this section we find all non-isomorphic   $p$-maps   on $L_{5,2}$ such that $\res x_3=0$.  We let
$$L=\frac {L_{5,2}}{\la x_3\ra}\cong L_{4,1},$$
where $L_{4,1}=\la x_1,x_2,x_4,x_5\ra$. Note that we denote the image of $x_i$ in $L$ by $x_i$ again.  We rename the 
$x_i$'s so that  $y_1=x_1$, $y_2=x_2$, $y_3=x_4$,  $y_4=x_5$, and $y_5=x_3$ and at the end we will switch them. Therefore, we have $L=L_{4,1}=\la y_1,y_2,y_3,y_4\ra$.  The group $\Aut(L)$ consists of invertible matrices of the form 
\[\begin{pmatrix} a_{11}  & a_{12} & a_{13}&a_{14} \\
a_{21} & a_{22} & a_{23}&a_{24}\\
a_{31} & a_{32} & a_{33} & a_{34}\\
a_{41} & a_{42} & a_{43} & a_{44}
\end{pmatrix}.\]

\begin{lemma}\label{LemmaL52}
Let $K=L_{5,2}$ and $[p]:K\to K$ be a $p$-map on $K$ such that $\res x_3=0$ and let $L=\frac{K}{M}$ where $M=\la x_3\ra_{\F}$. Then $K\cong L_{\theta}$ where $\theta=(\Delta_{12},\omega)\in Z^2(L,\F)$.
\end{lemma}
\begin{proof}
 Let $\pi :K\rightarrow L$ be the projection map. We have the exact sequence 
$$ 0\rightarrow M\rightarrow K\rightarrow L\rightarrow 0.$$
Let $\sigma :L\rightarrow K$ be a linear map given by $x_i\mapsto x_i$, for all $i$. Then $\sigma$ is an injective linear map and $\pi \sigma =1_L$. Now, we define $\phi : L\times L\rightarrow M$ by $\phi (x_i,x_j)=[\sigma (x_i),\sigma (x_j)]-\sigma ([x_i,x_j])$, for all $i,j$ and $\omega: L\rightarrow M$ by $\omega (x)=\res {\sigma (x)} -\sigma (\res x)$. Note that
\begin{align*}
&\phi(x_1,x_2)=[\sigma(x_1),\sigma(x_2)]-\sigma([x_1,x_2])=[x_1,x_2]=x_3;\\
&\phi(x_1,x_3)=[\sigma(x_1),\sigma(x_3)]-\sigma([x_1,x_3])=[x_1,x_3]=0.
\end{align*}
Similarly, we can show that $\phi(x_1,x_4)=\phi(x_2,x_3)=\phi(x_2,x_4)=\phi(x_3,x_4)=0$. Therefore, $\phi=\Delta_{12}$.
Now, by Lemma \ref{K=L-theta}, we have $\theta=(\Delta_{12},\omega)\in Z^2(L,\F)$ and $K\cong L_{\theta}$.
\end{proof}

Note that by Theorem \ref{4-dim}, there are five non-isomorphic restricted Lie algebra structures on $L$ given by the following $p$-maps:
\begin{enumerate}
\item[I.1] Trivial $p$-map;
\item[I.2]$\res y_1=y_2$;
\item[I.3]$\res y_1 =y_2 , \res y_3=y_4$;
\item[I.4]$\res y_1 =y_2 , \res y_2=y_3$;
\item[I.5]$\res y_1 =y_2 , \res y_2=y_3, \res y_3=y_4$.
\end{enumerate}

We make $L$ into a restricted Lie algebra by equipping it with each of the above $p$-maps. Then,  in  each  case,  we  find all possible orbit representatives of the form $(\Delta_{12},\omega)$ under the action of $\Aut_p(L)$ on $H^2(L,\F)$. By  Lemma \ref{LemmaL52}, we do get all possible $p$-maps on $K_2$ with the property that $x_3^{[p]}=0$.

 Let us consider the case I.2 where $L$ is a restricted Lie algebra with the $p$-map $\res y_1 = y_2$.
Let $[(\phi,\omega)]\in H^2(L,\F)$. Then we must have $\phi (x,\res y)=0$, for all $x,y\in L$,  where $\phi =a\Delta_{12}+b\Delta_{13}+c\Delta_{14}+d\Delta_{23}+e\Delta_{24}+f\Delta_{34}$, for some $a,b,c,d,e,f \in \F$. Since $\res L=\la y_2\ra$, we get $\phi (x,y_2)=0$ for all $x\in L$. Therefore, $\phi (y_1,y_2)=0$ which implies that $a=0$.  Since $\phi =\Delta_{12}$ gives us $L_{5,2}$, we deduce by Lemma \ref{LemmaL52} that $L_{5,2}$ cannot be constructed in this case. Similarly, we can show that in cases I.3, I.4, or I.5 we also get $a=0$. This means that if we equip $L$ with any of the $p$-maps in the cases I.2, I.3, I.4, or I.5, then 
$L_{5,2}\ncong L_{\theta}$, for every  $\theta \in H^2(L,\F)$. 
In the following subsections,  we consider the remaining cases.

\subsection{Extensions of ($L$, trivial $p$-map)}\label{Lx3}
We have 
$$
C^2(L,\F)=\la (\Delta_{12},0),(\Delta_{13},0),(\Delta_{14},0),(\Delta_{23},0),(\Delta_{24},0),(\Delta_{34},0),(0,f_1),
(0,f_2),(0,f_3),(0,f_4)\ra_{\F}.
$$
 First, we find a basis for $Z^2(L,\F)$. Let $(\phi,\omega)\in Z^2(L,\F)$. Then we must have $\delta^2\phi =0$ and $\phi(x,\res y)=0$, for all $x,y \in L$. Note that since $L$ is an abelian Lie algebra and the $p$-map is trivial, $\delta^2\phi =0$ and $\phi(x,\res y)=0$, for all $x,y \in L$. Therefore, a basis for $Z^2(L,\F)$ is as follows
$$
 (\Delta_{12},0),(\Delta_{13},0),(\Delta_{14},0),(\Delta_{23},0),(\Delta_{24},0),(\Delta_{34},0),(0,f_1),
(0,f_2),(0,f_3),(0,f_4).
$$
Next, we find a basis for $B^2(L,\F)$. Let $(\phi,\omega)=(a\Delta _{12}+b\Delta_{13}+c\Delta_{14}+d\Delta_{23}+e\Delta_{24}+f\Delta_{34},\alpha f_1+\beta f_2+ \gamma f_3+\delta f_4)\in B^2(L,\F)$. Then there exists a linear map $\psi:L\to \F$ such that $\delta^1\psi(x,y)=\phi(x,y)$ and $\tilde \psi(x)=\omega(x)$, for all $x,y \in L$. So, we have
\begin{align*}
\phi(y_1,y_2)=\delta^1\psi(y_1,y_2)=\psi([y_1,y_2])=0.
\end{align*}
But $\phi(y_1,y_2)=a$. So we deduce that $a=0$.
Similarly, we can show that $b=c=d=e=f=0$. Also, we have
\begin{align*}
\omega(y_1)&=\tilde \psi(y_1)=\psi(\res y_1)=0.
\end{align*}
Hence, $\alpha=0$.
Similarly, we can show that $\beta=\gamma=\delta=0$. Therefore, $(\phi,\omega)=0$ and hence $B^2(L,\F)=0$. We deduce that  a basis for $H^2(L,\F)$ is as follows
$$
[(\Delta_{12},0)],[(\Delta_{13},0)],[(\Delta_{14},0)],[(\Delta_{23},0)],[(\Delta_{24},0)],[(\Delta_{34},0)],[(0,f_1)],
[(0,f_2)],[(0,f_3)],[(0,f_4)].
$$

Let $[(\phi,\omega)] \in H^2(L,\F)$. Then, we have $\phi = a\Delta_{12} + b\Delta_{13}+c\Delta_{14}+d\Delta_{23}+e\Delta_{24}+f\Delta_{34}$, for some $a,b,c,d,e,f \in \F$. Suppose that $A\phi = a'\Delta_{12}+b'\Delta_{13}+c'\Delta_{14}+ d'\Delta_{23}+e'\Delta_{24}+f'\Delta_{34}$, for some $a',b',c',d',e',f' \in \F$. Then
\begin{align*}
A\phi(y_1,y_2)&=\phi(Ay_1,Ay_2)\\
&=\phi(a_{11}y_1+a_{21}y_2+a_{31}y_3+a_{41}y_4,a_{12}y_1+a_{22}y_2+a_{32}y_3+a_{42}y_4)\\
&=(a_{11}a_{22}-a_{12}a_{21})a+( a_{11}a_{32}-a_{31}a_{12})b+(a_{11}a_{42}-a_{41}a_{12})c\\
&+(a_{21}a_{32}-a_{31}a_{22})d+(a_{21}a_{42}-a_{41}a_{22})e+(a_{31}a_{42}-a_{41}a_{32})f, \text{ and }\\
\end{align*}
\begin{align*}
A\phi(y_1,y_3)&=\phi(Ay_1,Ay_3)\\
&=\phi(a_{11}y_1+a_{21}y_2+a_{31}y_3+a_{41}y_4,a_{13}y_1+a_{23}y_2+a_{33}y_3+a_{43}y_4)\\
&=(a_{11}a_{23}-a_{21}a_{13})a+(a_{11}a_{33}-a_{31}a_{13})b+(a_{11}a_{43}-a_{13}a_{41})c\\
&+( a_{21}a_{33}-a_{31}a_{23})d+(a_{21}a_{43}-a_{23}a_{41})e+(a_{31}a_{43}-a_{33}a_{41})f, \text{ and }
\end{align*}
\begin{align*}
A\phi(y_1,y_4)&=\phi(Ay_1,Ay_4)\\
&=\phi(a_{11}y_1+a_{21}y_2+a_{31}y_3+a_{41}y_4,a_{14}y_1+a_{24}y_2+a_{34}y_3+a_{44}y_4)\\
&=(a_{11}a_{24}-a_{14}a_{21})a+(a_{11}a_{34}-a_{14}a_{31})b+(a_{11}a_{44}-a_{14}a_{41})c\\
&+(a_{21}a_{34}-a_{24}a_{31})d+(a_{21}a_{44}-a_{24}a_{41})e+(a_{31}a_{44}-a_{34}a_{41})f, \text{ and }
\end{align*}
\begin{align*}
A\phi(y_2,y_3)&=\phi(Ay_2,Ay_3)\\
&=\phi(a_{12}y_1+a_{22}y_2+a_{32}y_3+a_{42}y_4,a_{13}y_1+a_{23}y_2+a_{33}y_3+a_{43}y_4)\\
&=(a_{12}a_{23}-a_{22}a_{13})a+(a_{12}a_{33}-a_{32}a_{13})b+(a_{12}a_{43}-a_{13}a_{42})c\\
&+(a_{22}a_{33}-a_{32}a_{23})d+(a_{22}a_{43}-a_{23}a_{42})e+(a_{32}a_{43}-a_{33}a_{42})f, \text{ and }
\end{align*}
\begin{align*}
A\phi(y_2,y_4)&=\phi(Ay_2,Ay_4)\\
&=\phi(a_{12}y_1+a_{22}y_2+a_{32}y_3+a_{42}y_4,a_{14}y_1+a_{24}y_2+a_{34}y_3+a_{44}y_4)\\
&=(a_{12}a_{24}-a_{14}a_{22})a+(a_{12}a_{34}-a_{14}a_{32})b+(a_{12}a_{44}-a_{14}a_{42})c\\
&+(a_{22}a_{34}-a_{24}a_{32})d+(a_{22}a_{44}-a_{24}a_{42})e+(a_{32}a_{44}-a_{34}a_{42})f, \text{ and }
\end{align*}
\begin{align*}
A\phi(y_3,y_4)&=\phi(Ay_3,Ay_4)\\
&=\phi(a_{13}y_1+a_{23}y_2+a_{33}y_3+a_{43}y_4,a_{14}y_1+a_{24}y_2+a_{34}y_3+a_{44}y_4)\\
&=(a_{13}a_{24}-a_{14}a_{23})a+(a_{13}a_{34}-a_{14}a_{33})b+(a_{13}a_{44}-a_{14}a_{43})c\\
&+(a_{23}a_{34}-a_{24}a_{33})d+(a_{23}a_{44}-a_{24}a_{43})e+(a_{33}a_{44}-a_{34}a_{43})f.
\end{align*}
 Therefore, the action of $\Aut_p(L)$ on the set of $\phi$'s in the matrix form is as follows:

{\tiny
\begin{align*}
& \begin{pmatrix}
a' \\ b'\\c'\\d'\\e'\\f'
\end{pmatrix}{=}\\
& \begin{pmatrix}
a_{11}a_{22}-a_{12}a_{21} & a_{11}a_{32}-a_{31}a_{12}&a_{11}a_{42}-a_{41}a_{12}&a_{21}a_{32}-a_{31}a_{22}&a_{21}a_{42}-a_{41}a_{22}&a_{31}a_{42}-a_{41}a_{32} \\
a_{11}a_{23}-a_{21}a_{13}&a_{11}a_{33}-a_{31}a_{13}&a_{11}a_{43}-a_{13}a_{41}& a_{21}a_{33}-a_{31}a_{23}&a_{21}a_{43}-a_{23}a_{41}&a_{31}a_{43}-a_{33}a_{41}\\
a_{11}a_{24}-a_{14}a_{21}&a_{11}a_{34}-a_{14}a_{31}&a_{11}a_{44}-a_{14}a_{41}&a_{21}a_{34}-a_{24}a_{31}&a_{21}a_{44}-a_{24}a_{41}&a_{31}a_{44}-a_{34}a_{41}\\
a_{12}a_{23}-a_{22}a_{13}&a_{12}a_{33}-a_{32}a_{13}&a_{12}a_{43}-a_{13}a_{42}&a_{22}a_{33}-a_{32}a_{23}&a_{22}a_{43}-a_{23}a_{42}&a_{32}a_{43}-a_{33}a_{42}\\
a_{12}a_{24}-a_{14}a_{22}&a_{12}a_{34}-a_{14}a_{32}&a_{12}a_{44}-a_{14}a_{42}&a_{22}a_{34}-a_{24}a_{32}&a_{22}a_{44}-a_{24}a_{42}&a_{32}a_{44}-a_{34}a_{42}\\
a_{13}a_{24}-a_{14}a_{23}&a_{13}a_{34}-a_{14}a_{33}&a_{13}a_{44}-a_{14}a_{43}&a_{23}a_{34}-a_{24}a_{33}&a_{23}a_{44}-a_{24}a_{43}&a_{33}a_{44}-a_{34}a_{43}
\end{pmatrix}
\begin{pmatrix}
a \\ b\\c\\d\\e\\f
\end{pmatrix}.
\end{align*}
}
The orbit with representative 
$\begin{pmatrix}
1\\ 0\\0\\0\\0\\0
\end{pmatrix}$ of this action  gives us $L_{5,2}$.

Also, we have $\omega=\alpha f_1+\beta f_2+\gamma f_3+\delta f_4$,  for some $\alpha ,\beta ,\gamma ,\delta \in \F$. Suppose that $A\omega = \alpha' f_1+\beta' f_2+\gamma' f_3+\delta' f_4$,  for some $\alpha' ,\beta' ,\gamma' ,\delta' \in \F$. Then we can verify that the action of $\Aut_p(L)$ on the set of $\omega$'s in the matrix form is as follows:
\begin{align*}
\begin{pmatrix}
\alpha' \\ 
\beta'\\
\gamma'\\
\delta'
\end{pmatrix}=
\begin{pmatrix}
a_{11}^p& a_{21}^p&a_{31}^p&a_{41}^p\\
a_{12}^p& a_{22}^p& a_{32}^p&a_{42}^p\\
a_{13}^p&a_{23}^p&a_{33}^p&a_{43}^p\\
a_{14}^p&a_{24}^p&a_{34}^p&a_{44}^p
\end{pmatrix}
\begin{pmatrix}
\alpha\\ 
\beta \\
\gamma \\
\delta
\end{pmatrix}.
\end{align*}
Now we find the representatives of the orbits of the action of $\Aut (L)$ on the set of $\omega$'s such that 
the orbit represented by
 $\begin{pmatrix}
1\\ 0\\0\\0\\0\\0
\end{pmatrix}$
 is preserved under the action of $\Aut (L)$ on the set of $\phi$'s.

Let $\nu = \begin{pmatrix}
\alpha\\ 
\beta \\
\gamma \\
\delta
\end{pmatrix} \in \F^4$. If $\nu= \begin{pmatrix}
0\\ 
0 \\
0 \\
0
\end{pmatrix}$, then $\{\nu \}$ is clearly an $\Aut_p(L)$-orbit. Let $\nu \neq 0$ and suppose that $\delta \neq 0$. Then
{\footnotesize
\begin{align*}
&\bigg[\begin{pmatrix}
1&0& (-\beta/\delta)^{1/p}&0&(\alpha/\delta)^{1/p}&0\\
0& 1&(-\gamma /\delta)^{1/p}&0&0&(\alpha/\delta )^{1/p}\\
0&0&1&0&0&0\\
0&0&0&1&(-\gamma/\delta)^{1/p}&(\beta/\delta)^{1/p}\\
0&0&0&0&1&0\\
0&0&0&0&0&1
\end{pmatrix},
\begin{pmatrix}
1& 0&0&-\alpha/\delta\\
0& 1&0&-\beta/\delta\\
0&0&1&-\gamma/\delta\\
0&0&0&1
\end{pmatrix}\bigg]
\bigg[\begin{pmatrix}
1 \\ 0\\0\\0\\0\\0
\end{pmatrix},
\begin{pmatrix}
\alpha\\ 
\beta \\
\gamma \\
\delta
\end{pmatrix}\bigg]=
\bigg[\begin{pmatrix}
1\\0\\0\\0\\0\\0
\end{pmatrix},
\begin{pmatrix}
0 \\ 
0\\
0\\
\delta
\end{pmatrix}\bigg],
\end{align*}
\begin{align*}
&\bigg[\begin{pmatrix}
1&0& 0&0&0&0\\
0& 1&0&0&0&0\\
0&0&(1/\delta)^{1/p}&0&0&0\\
0&0&0&1&0&0\\
0&0&0&0&(1/\delta)^{1/p}&0\\
0&0&0&0&0&(1/\delta)^{1/p}
\end{pmatrix},
\begin{pmatrix}
1& 0&0&0\\
0& 1&0&0\\
0&0&1&0\\
0&0&0&1/\delta
\end{pmatrix}\bigg]
\bigg[\begin{pmatrix}
1 \\ 0\\0\\0\\0\\0
\end{pmatrix},
\begin{pmatrix}
0\\ 
0 \\
0\\
\delta
\end{pmatrix}\bigg]=
\bigg[\begin{pmatrix}
1\\0\\0\\0\\0\\0
\end{pmatrix},
\begin{pmatrix}
0 \\ 
0\\
0\\
1
\end{pmatrix}\bigg].
\end{align*}
}
Next, if $\delta= 0$ , but $\gamma \neq 0$, then
{\small
\begin{align*}
& \bigg[\begin{pmatrix}
1&(-\beta/\gamma)^{1/p}&0&(\alpha/\gamma)^{1/p}&0&0\\
0& 1&0&0&0&0\\
0&0&1&0&0&(-\alpha/\gamma)^{1/p}\\
0&0&0&1&0&0\\
0&0&0&0&1&(-\beta/\gamma)^{1/p}\\
0&0&0&0&0&1
\end{pmatrix},
\begin{pmatrix}
1& 0&-\frac{\alpha}{\gamma}&0\\
0& 1&-\frac{\beta}{\gamma}&0\\
0&0&1&0\\
0&0&0&1
\end{pmatrix}\bigg]
\bigg[\begin{pmatrix}
1 \\ 0\\0\\0\\0\\0
\end{pmatrix},
\begin{pmatrix}
\alpha\\ 
\beta \\
\gamma \\
0
\end{pmatrix}\bigg]
{=}
\bigg[\begin{pmatrix}
1\\0\\0\\0\\0\\0
\end{pmatrix},
\begin{pmatrix}
0 \\ 
0\\
\gamma\\
0
\end{pmatrix}\bigg],\\
& \bigg[\begin{pmatrix}
1&0& 0&0&0&0\\
0& (1/\gamma)^{1/p}&0&0&0&0\\
0&0&1&0&0&0\\
0&0&0&(1/\gamma)^{1/p}&0&0\\
0&0&0&0&1&0\\
0&0&0&0&0&(1/\gamma)^{1/p}
\end{pmatrix},
\begin{pmatrix}
1& 0&0&0\\
0& 1&0&0\\
0&0&1/\gamma&0\\
0&0&0&1
\end{pmatrix}\bigg]
\bigg[\begin{pmatrix}
1 \\ 0\\0\\0\\0\\0
\end{pmatrix},
\begin{pmatrix}
0\\ 
0 \\
\gamma\\
0
\end{pmatrix}\bigg]=
\bigg[\begin{pmatrix}
1\\0\\0\\0\\0\\0
\end{pmatrix},
\begin{pmatrix}
0 \\ 
0\\
1\\
0
\end{pmatrix}\bigg].
\end{align*}
}
Next, if $\delta =\gamma =0$, but $\beta \neq 0$, then
{\small
\begin{align*}
& \bigg[\begin{pmatrix}
1&0& 0&0&0&0\\
0& 1&0&(-\alpha/\beta)^{1/p}&0&0\\
0&0&1&0&(-\alpha/\beta)^{1/p}&0\\
0&0&0&1&0&0\\
0&0&0&0&1&0\\
0&0&0&0&0&1
\end{pmatrix},
\begin{pmatrix}
1& -\alpha/\beta&0&0\\
0& 1&0&0\\
0&0&1&0\\
0&0&0&1
\end{pmatrix}\bigg]
\bigg[\begin{pmatrix}
1 \\ 0\\0\\0\\0\\0
\end{pmatrix},
\begin{pmatrix}
\alpha\\ 
\beta\\
0\\
0
\end{pmatrix}\bigg]
{=}
\bigg[\begin{pmatrix}
1\\0\\0\\0\\0\\0
\end{pmatrix},
\begin{pmatrix}
0 \\ 
\beta\\
0\\
0
\end{pmatrix}\bigg],\\
& \bigg[\begin{pmatrix}
1&0& 0&0&0&0\\
0& \beta^{1/p}&0&0&0&0\\
0&0&\beta^{1/p}&0&0&0\\
0&0&0&(1/\beta)^{1/p}&0&0\\
0&0&0&0&(1/\beta)^{1/p}&0\\
0&0&0&0&0&1
\end{pmatrix},
\begin{pmatrix}
\beta& 0&0&0\\
0& 1/\beta&0&0\\
0&0&1&0\\
0&0&0&1
\end{pmatrix}\bigg]
\bigg[\begin{pmatrix}
1 \\ 0\\0\\0\\0\\0
\end{pmatrix},
\begin{pmatrix}
0\\ 
\beta \\
0\\
0
\end{pmatrix}\bigg]=
\bigg[\begin{pmatrix}
1\\0\\0\\0\\0\\0
\end{pmatrix},
\begin{pmatrix}
0 \\ 
1\\
0\\
0
\end{pmatrix}\bigg].
\end{align*}
}
Finally, if $\delta=\gamma =\beta =0$, but $\alpha \neq 0$, then
\begin{align*}
\bigg[\begin{pmatrix}
1&0& 0&0&0&0\\
0& \alpha^{-1/p}&0&0&0&0\\
0&0&\alpha^{-1/p}&0&0&0\\
0&0&0&\alpha^{1/p}&0&0\\
0&0&0&0&\alpha^{1/p}&0\\
0&0&0&0&0&1
\end{pmatrix},
\begin{pmatrix}
1/\alpha& 0&0&0\\
0& \alpha&0&0\\
0&0&1&0\\
0&0&0&1
\end{pmatrix}\bigg]
\bigg[\begin{pmatrix}
1 \\ 0\\0\\0\\0\\0
\end{pmatrix},
\begin{pmatrix}
\alpha\\ 
0 \\
0\\
0
\end{pmatrix}\bigg]=
\bigg[\begin{pmatrix}
1\\0\\0\\0\\0\\0
\end{pmatrix},
\begin{pmatrix}
1 \\ 
0\\
0\\
0
\end{pmatrix}\bigg].
\end{align*}
Thus the following elements are $\Aut (L)$-orbit representatives:
\begin{align*}
\begin{pmatrix}
0\\ 0 \\ 0\\0
\end{pmatrix},
\begin{pmatrix}
1\\ 0 \\ 0\\0
\end{pmatrix},
\begin{pmatrix}
0\\ 1 \\ 0\\0
\end{pmatrix},
\begin{pmatrix}
0\\ 0 \\ 1\\0
\end{pmatrix},
\begin{pmatrix}
0\\ 0 \\ 0\\1
\end{pmatrix}.
\end{align*}
Now, we find the restricted Lie algebra structure corresponding to the orbit representative 
$\begin{pmatrix}
1\\ 0 \\ 0\\0
\end{pmatrix}$. We have 
$\begin{pmatrix}
1\\ 0 \\ 0\\0
\end{pmatrix}=f_1$ and hence $\omega=f_1y_5$. First, by Lemma \ref{Ltheta} we get
\begin{align*}
&\res y_1=\res y_1+\omega(y_1)=y_5,\\
&\res y_2=\res y_2+\omega(y_2)=0,\\
&\res y_3=\res y_3+\omega(y_3)=0,\\
&\res y_4=\res y_4+\omega(y_4)=0.
\end{align*}
Therefore, the corresponding restricted Lie algebra structure is as follows:
$$K_2^2=\la y_1,\ldots,y_5 \mid [y_1,y_2]=y_5, \res y_1=y_5\ra.$$ Next, we use the substitutions 
 $y_1=x_1$, $y_2=x_2$, $y_3=x_4$,  $y_4=x_5$, and $y_5=x_3$.  Hence, we have 
 $$
 K_2^{2}=\langle x_1,\ldots,x_5\mid [x_1,x_2]=x_3,
\res x_1= x_3 \rangle.
$$

Similarly, we  obtain the restricted Lie algebra structures corresponding to the other orbits.
Therefore, the corresponding restricted Lie algebra structures are as follows:
\begin{align*}
&K_2^{1}=\langle x_1,\ldots,x_5\mid [x_1,x_2]=x_3 \rangle;\\
&K_2^{2}=\langle x_1,\ldots,x_5\mid [x_1,x_2]=x_3,
\res x_1= x_3 \rangle;\\
&K_2^{3}=\langle x_1,\ldots,x_5\mid [x_1,x_2]=x_3,
\res x_2= x_3 \rangle;\\
&K_2^{4}=\langle x_1,\ldots,x_5\mid [x_1,x_2]=x_3,
\res x_4= x_3 \rangle;\\
&K_2^{5}=\langle x_1,\ldots,x_5\mid [x_1,x_2]=x_3,
\res x_5= x_3 \rangle.
\end{align*}

\section{Extensions of  $L=\frac {L_{5,2}}{\la x_5\ra}$}\label{L_{5,2} mod x5} 
In this section we find all non-isomorphic   $p$-maps   on $L_{5,2}$ such that $\res x_5=0$.  We let
$$
L=\frac {L_{5,2}}{\la x_5\ra}\cong  L_{4,2},
$$ where $L_{4,2}=\la x_1,x_2,x_3,x_4 \mid [x_1,x_2]=x_3\ra$.
The group $\Aut(L)$ consists of invertible matrices of the form 
$$\begin{pmatrix}
 a_{11}  & a_{12} & 0 & 0 \\
a_{21} & a_{22} & 0 & 0\\
a_{31} & a_{32} & r & a_{34} \\
a_{41} & a_{42} & 0& a_{44}
\end{pmatrix},$$
where $r = a_{11}a_{22}-a_{12}a_{21}\neq 0$.
\begin{lemma}\label{LemmaL522}
Let $K=L_{5,2}$ and $[p]:K\to K$ be a $p$-map on $K$ such that $\res x_5=0$ and let $L=\frac{K}{M}$ where $M=\la x_5\ra_{\F}$. Then $K\cong L_{\theta}$ where $\theta=(0,\omega)\in Z^2(L,\F)$.
\end{lemma}
\begin{proof}
 Let $\pi :K\rightarrow L$ be the projection map. We have the exact sequence 
$$ 0\rightarrow M\rightarrow K\rightarrow L\rightarrow 0.$$
Let $\sigma :L\rightarrow K$ such that $x_i\mapsto x_i$, $1\leq i\leq 4$. Then $\sigma$ is an injective linear map and $\pi \sigma =1_L$. Now, we define $\phi : L\times L\rightarrow M$ by $\phi (x_i,x_j)=[\sigma (x_i),\sigma (x_j)]-\sigma ([x_i,x_j])$, $1\leq i,j\leq 4$. and $\omega: L\rightarrow M$ by $\omega (x)=\res {\sigma (x)} -\sigma (\res x)$. Note that
\begin{align*}
&\phi(x_1,x_2)=[\sigma(x_1),\sigma(x_2)]-\sigma([x_1,x_2])=[x_1,x_2]-\sigma(x_3)=0;\\
&\phi(x_1,x_3)=[\sigma(x_1),\sigma(x_3)]-\sigma([x_1,x_3])=0.
\end{align*}
Similarly, we can show that $\phi(x_1,x_4)=\phi(x_2,x_3)=\phi(x_2,x_4)=\phi(x_3,x_4)=0$. Therefore, $\phi=0$.
Now, by Lemma \ref{K=L-theta}, we have $\theta=(0,\omega)\in Z^2(L,\F)$ and $K\cong L_{\theta}$.
\end{proof}

Note that by Theorem \ref{4-dim}, there are eight non-isomorphic restricted Lie algebra structures on $L$ given by the following $p$-maps:
\begin{enumerate}
\item[II.1]Trivial $p$-map;
\item[II.2]$\res x_1=x_3;$
\item[II.3]$\res x_1 =x_4;$
\item[II.4]$ \res x_1=x_3, \res x_2=x_4;$
\item[II.5]$\res x_3=x_4;$
\item[II.6]$\res x_3=x_4, \res x_2=x_3;$
\item[II.7]$\res x_4=x_3;$
\item[II.8]$\res x_4=x_3, \res x_2=x_4.$
\end{enumerate}

In the following subsections, we make $L$ into a restricted Lie algebra by equipping it with each of the above $p$-maps. Then,  in  each  case,  we  find all possible orbit representatives of the form $(0,\omega)$ under the action of $\Aut_p(L)$ on $H^2(L,\F)$. By  Lemma \ref{LemmaL522}, we do get all possible $p$-maps on $K_2$ with the property that $x_5^{[p]}=0$.

\subsection{Extensions of ($L$, trivial $p$-map)}\label{L_{5,2} mod x5 trivial pmap} 

 First, we find a basis for $Z^2(L,\F)$. Let $(\phi,\omega)=(a\Delta _{12}+b\Delta_{13}+c\Delta_{14}+d\Delta_{23}+e\Delta_{24}+f\Delta_{34},\alpha f_1+\beta f_2+ \gamma f_3+\delta f_4)\in Z^2(L,\F)$. Then we must have $\delta^2\phi(x,y,z) =0$ and $\phi(x,\res y)=0$, for all $x,y,z \in L$. Therefore,
\begin{align*}
0=(\delta^2\phi)(x_1,x_2,x_4)&=\phi([x_1,x_2],x_4)+\phi([x_2,x_4],x_1)+\phi([x_4,x_1],x_2)=\phi(x_3,x_4).
\end{align*}
Thus, we get $f=0$.
Since the $p$-map is trivial, $\phi(x,\res y)=\phi(x,0)=0$, for all $x,y\in L$. Therefore, a basis for $Z^2(L,\F)$ is as follows:
$$
 (\Delta_{12},0),(\Delta_{13},0),(\Delta_{14},0),(\Delta_{23},0),(\Delta_{24},0),(0,f_1),
(0,f_2),(0,f_3),(0,f_4).
$$
Next, we find a basis for $B^2(L,\F)$. Let $(\phi,\omega)\in B^2(L,\F)$. Since $B^2(L,\F)\subseteq Z^2(L,\F)$, we have $(\phi,\omega)=(a\Delta _{12}+b\Delta_{13}+c\Delta_{14}+d\Delta_{23}+e\Delta_{24},\alpha f_1+\beta f_2+ \gamma f_3+\delta f_4)$. So, there exists a linear  map $\psi:L\to \F$ such that $\delta^1\psi(x,y)=\phi(x,y)$ and $\tilde \psi(x)=\omega(x)$, for all $x,y \in L$. So, we have
\begin{align*}
b=\phi(x_1,x_3)=\delta^1\psi(x_1,x_3)=\psi([x_1,x_3])=0.
\end{align*}
Similarly, we can show that $c=d=e=0$. Also, we have
\begin{align*}
\alpha=\omega(x_1)=\tilde \psi(x_1)=\psi(\res x_1)=0.
\end{align*}
Similarly, we can show that $\beta=\gamma=\delta=0$. Therefore, $(\phi,\omega)=(a\Delta_{12},0)$ and hence
 $B^2(L,\F)=\la(\Delta_{12},0)\ra_{\F}$. We deduce that a basis for $H^2(L,\F)$ is as follows:
$$
[(\Delta_{13},0)],[(\Delta_{14},0)],[(\Delta_{23},0)],[(\Delta_{24},0)],[(0,f_1)],
[(0,f_2)],[(0,f_3)],[(0,f_4)].
$$

Let $[\theta]=[(\phi,\omega)]\in H^2(L,\F)$. Since we want $L_{\theta}$ and  $L_{5,2}$ to be isomorphic as  Lie algebras, we should have $\phi=0$. Since $0$ is preserved under $\Aut_p(L)$, it is enough to find $\Aut_p(L)$-representatives of the $\omega$'s.

Let $\omega=\alpha f_1+\beta f_2+\gamma f_3+\delta f_4$,  for some $\alpha ,\beta ,\gamma ,\delta \in \F$. 
Then, $A\omega = \alpha' f_1+\beta' f_2+\gamma' f_3+\delta' f_4$, for some $\alpha',\beta',\gamma',\delta' \in \F$. We can verify that the action of $\Aut_p(L)$ on the set of $\omega$'s in the matrix form is as follows: 
\begin{equation*}
\begin{pmatrix}
\alpha' \\ 
\beta'\\
\gamma'\\
\delta'
\end{pmatrix}=
\begin{pmatrix}
a_{11}^p& a_{21}^p&a_{31}^p&a_{41}^p\\
a_{12}^p& a_{22}^p& a_{32}^p&a_{42}^p\\
0&0&r ^p&0\\
0&0&a_{34}^p&a_{44}^p
\end{pmatrix}
\begin{pmatrix}
\alpha\\ 
\beta \\
\gamma \\
\delta
\end{pmatrix}.
\end{equation*}


Now we find the representatives of the orbits of the action of $\Aut (L)$ on the set of $\omega$'s.
Let $\nu = \begin{pmatrix}
\alpha\\ 
\beta \\
\gamma \\
\delta
\end{pmatrix} \in \F^4$. If $\nu= \begin{pmatrix}
0\\ 
0 \\
0 \\
0
\end{pmatrix}$, then $\{\nu \}$ is clearly an  $\Aut_p(L)$-orbit.  Let $\nu \neq 0$. Suppose that $\gamma \neq 0$. Then 
\begin{align*}
&\begin{pmatrix}
1& 0&-\alpha /\gamma&0\\
0& 1& -\beta /\gamma&0\\
0&0&1&0\\
0&0&-\delta/\gamma&1
\end{pmatrix}
\begin{pmatrix}
\alpha\\ 
\beta \\
\gamma \\
\delta
\end{pmatrix}=
\begin{pmatrix}
0 \\ 
0\\
\gamma\\
0
\end{pmatrix}, \text{ and }\\
&\begin{pmatrix}
1/\gamma& 0&0&0\\
0& 1& 0&0\\
0&0&1/\gamma&0\\
0&0&0&1
\end{pmatrix}
\begin{pmatrix}
0\\ 
0 \\
\gamma \\
0
\end{pmatrix}=
\begin{pmatrix}
0 \\ 
0\\
1\\
0
\end{pmatrix}.
\end{align*}
Next, if $\gamma =0$, but $\delta\neq 0$, then
\begin{align*}
&\begin{pmatrix}
1& 0&0&-\alpha/\delta\\
0& 1& 0&-\beta/\delta\\
0&0&1&0\\
0&0&0&1
\end{pmatrix}
\begin{pmatrix}
\alpha\\ 
\beta \\
0\\
\delta
\end{pmatrix}=
\begin{pmatrix}
0 \\ 
0\\
0\\
\delta
\end{pmatrix}, \text{ and }\\
&\begin{pmatrix}
1& 0&0&0\\
0& 1& 0&0\\
0&0&1&0\\
0&0&0&1/\delta
\end{pmatrix}
\begin{pmatrix}
0\\ 
0 \\
0 \\
\delta
\end{pmatrix}=
\begin{pmatrix}
0 \\ 
0\\
0\\
1
\end{pmatrix}.
\end{align*}
Next, if $\gamma =\delta = 0$, but $\beta \neq 0$, then
\begin{align*}
&\begin{pmatrix}
1& -\alpha/\beta&0&0\\
0& 1& 0&0\\
0&0&1&0\\
0&0&0&1
\end{pmatrix}
\begin{pmatrix}
\alpha\\ 
\beta \\
0 \\
0
\end{pmatrix}=
\begin{pmatrix}
0 \\ 
\beta\\
0\\
0
\end{pmatrix}, \text{ and }\\
&\begin{pmatrix}
1& 0&0&0\\
0& 1/\beta& 0&0\\
0&0&1/\beta&0\\
0&0&0&1
\end{pmatrix}
\begin{pmatrix}
0\\ 
\beta \\
0 \\
0
\end{pmatrix}=
\begin{pmatrix}
0 \\ 
1\\
0\\
0
\end{pmatrix}.
\end{align*}
Finally, if $\gamma =\delta =\beta =0$, but $\alpha \neq 0$, then
\begin{align*}
\begin{pmatrix}
1/\alpha& 0&0&0\\
0& 1& 0&0\\
0&0&1/\alpha&0\\
0&0&0&1
\end{pmatrix}
\begin{pmatrix}
\alpha\\ 
0 \\
0 \\
0
\end{pmatrix}=
\begin{pmatrix}
1 \\ 
0\\
0\\
0
\end{pmatrix}.
\end{align*}
Thus the following elements are $\Aut (L)$-orbit representatives:
\begin{align*}
\begin{pmatrix}
0\\ 
0\\
0\\
0
\end{pmatrix},
\begin{pmatrix}
1\\ 
0\\
0\\
0
\end{pmatrix},
\begin{pmatrix}
0 \\ 
1\\
0\\
0
\end{pmatrix},
\begin{pmatrix}
0 \\ 
0\\
1\\
0
\end{pmatrix},
\begin{pmatrix}
0 \\ 
0\\
0\\
1
\end{pmatrix}.
\end{align*}
Therefore, the corresponding restricted Lie algebra structures are as follows:
\begin{align*}
&K_2^{6}=\langle x_1,\ldots,x_5\mid [x_1,x_2]=x_3 \rangle;\\
&K_2^{7}=\langle x_1,\ldots,x_5\mid [x_1,x_2]=x_3,
\res x_1= x_5 \rangle;\\
&K_2^{8}=\langle x_1,\ldots,x_5\mid [x_1,x_2]=x_3,
\res x_2= x_5 \rangle;\\
&K_2^{9}=\langle x_1,\ldots,x_5\mid [x_1,x_2]=x_3,
\res x_3= x_5 \rangle;\\
&K_2^{10}=\langle x_1,\ldots,x_5\mid [x_1,x_2]=x_3,
\res x_4= x_5 \rangle.
\end{align*}

\subsection{Extensions of $(L, \res x_1 = x_3)$}\label{Lx5}
 First, we find a basis for $Z^2(L,\F)$. Let $(\phi,\omega)=(a\Delta _{12}+b\Delta_{13}+c\Delta_{14}+d\Delta_{23}+e\Delta_{24}+f\Delta_{34},\alpha f_1+\beta f_2+ \gamma f_3+\delta f_4)\in Z^2(L,\F)$. Then we must have $\delta^2\phi(x,y,z) =0$ and $\phi(x,\res y)=0$, for all $x,y,z \in L$. Therefore, we have
\begin{align*}
0=(\delta^2\phi)(x_1,x_2,x_4)&=\phi([x_1,x_2],x_4)+\phi([x_2,x_4],x_1)+\phi([x_4,x_1],x_2)=\phi(x_3,x_4)=f.
\end{align*}
Also, we have $\phi(x,\res y)=0$. Therefore, $\phi(x,x_3)=0$, for all $x\in L$ and hence $\phi(x_1,x_3)=\phi(x_2,x_3)=\phi(x_4,x_3)=0$ which implies that $b=d=f=0$. Therefore, $Z^2(L,\F)$
has a basis consisting of:
$$
 (\Delta_{12},0),(\Delta_{14},0),(\Delta_{24},0),(0,f_1),
(0,f_2),(0,f_3),(0,f_4).
$$
Next, we find a basis for $B^2(L,\F)$. Let $(\phi,\omega)\in B^2(L,\F)$. Since $B^2(L,\F)\subseteq Z^2(L,\F)$, we have $(\phi,\omega)=(a\Delta _{12}+c\Delta_{14}+e\Delta_{24},\alpha f_1+\beta f_2+ \gamma f_3+\delta f_4)$. Note that there exists a linear  map $\psi:L\to \F$ such that $\delta^1\psi(x,y)=\phi(x,y)$ and $\tilde \psi(x)=\omega(x)$, for all $x,y \in L$. So, we have
\begin{align*}
a=\phi(x_1,x_2)&=\delta^1\psi(x_1,x_2)=\psi([x_1,x_2])=\phi(x_1,x_2)=\psi(x_3), \text{ and }\\
c=\phi(x_1,x_4)&=\delta^1\psi(x_1,x_4)=\psi([x_1,x_4])=0.
\end{align*}
Similarly, we can show that $e=0$. Also, we have
\begin{align*}
\alpha=\omega(x_1)&=\tilde \psi(x_1)=\psi(\res x_1)=\psi(x_3), \text{ and }\\
\beta=\omega(x_2)&=\tilde \psi(x_2)=\psi(\res x_2)=0.
\end{align*}
Similarly, we can show that $\gamma=\delta=0$. Note that $\psi(x_3)=a=\alpha$.  Therefore, $(\phi,\omega)=(a\Delta_{12},af_1)$ and hence $B^2(L,\F)=\la(\Delta_{12},f_1)\ra_{\F}$. Note that
$$
\bigg\{[(\Delta_{12},0)],[(\Delta_{14},0)],[(\Delta_{24},0)],[(0,f_1)],
[(0,f_2)],[(0,f_3)],[(0,f_4)]\bigg\}
$$
spans $H^2(L,\F)$.
Since $[(\Delta_{12},0)]+[(0,f_1)]=[(\Delta_{12},f_1)]=[0]$, then $[(0,f_1)]$ is an scalar multiple of $[(\Delta_{12},0)]$
in $H^2(L,\F)$. Note that $\dim H^2=\dim Z^2-\dim B^2=6$. Therefore, 
$$
\bigg \{ [(\Delta_{12},0)],[(\Delta_{14},0)],[(\Delta_{24},0)],[(0,f_2)],[(0,f_3)],[(0,f_4)]\bigg\}
$$
 forms a basis for $H^2(L,\F)$.
 
 Note that the group $\Aut_p(L)$ consists of invertible matrices of the form 
$$\begin{pmatrix}
 a_{11}  & a_{12} & 0 & 0 \\
a_{21} & a_{22} & 0 & 0\\
a_{31} & a_{32} & r & a_{34} \\
a_{41} & a_{42} & 0& a_{44}
\end{pmatrix},$$
where $r = a_{11}a_{22}-a_{12}a_{21}\neq 0$, $r=a_{11}^p$ and $a_{12}=0$.

Let $[\theta]=[(\phi,\omega)]\in H^2(L,\F)$. As in Section \ref{L_{5,2} mod x5 trivial pmap},  it is enough to find $\Aut_p(L)$-representatives of the $\omega$'s.
We can verify that the action of $\Aut_p(L)$ on the $\omega$'s  in the matrix form is as follows:
\begin{align*}
\begin{pmatrix}
\beta'\\
\gamma'\\
\delta'
\end{pmatrix}=
\begin{pmatrix}
a_{22}^p& a_{32}^p&a_{42}^p\\
0&r ^p&0\\
0&a_{34}^p&a_{44}^p
\end{pmatrix}
\begin{pmatrix}
\beta \\
\gamma \\
\delta
\end{pmatrix}.
\end{align*}
Note that $\omega=\beta f_2+\gamma f_3+\delta f_4$. Since $\omega(x_1)=0$, we have $A\omega(x_1)=0$ which implies that
$a_{21}^p\beta+a_{31}^p\gamma+a_{41}^p\delta=0$.

Now we find the representatives of the orbits of this action. Note that we need to have $a_{21}^p\beta+a_{31}^p\gamma+a_{41}^p\delta=0$.
Let $\nu = \begin{pmatrix} 
\beta \\
\gamma \\
\delta
\end{pmatrix} \in \F^3$. If $\nu= \begin{pmatrix}
0 \\
0 \\
0
\end{pmatrix}$, then $\{\nu \}$ is clearly an  $\Aut_p(L)$-orbit.  Let $\nu \neq 0$. Suppose that $\gamma \neq 0$. Then 
\begin{align*}
&\begin{pmatrix}
1& -\beta /\gamma&0\\
0&1&0\\
0&-\delta/\gamma&1
\end{pmatrix}
\begin{pmatrix}
\beta \\
\gamma \\
\delta
\end{pmatrix}=
\begin{pmatrix}
0\\
\gamma\\
0
\end{pmatrix}, \text{ and }\\
&\begin{pmatrix}
 (1/\gamma)^\frac {p-1}{p}& 0&0\\
0&1/\gamma&0\\
0&0&1
\end{pmatrix}
\begin{pmatrix} 
0 \\
\gamma \\
0
\end{pmatrix}=
\begin{pmatrix}
0\\
1\\
0
\end{pmatrix}.
\end{align*}
Next, if $\gamma =0$, but $\delta\neq 0$, then
\begin{align*}
&\begin{pmatrix}
1& 0&-\beta/\delta\\
0&1&0\\
0&0&1
\end{pmatrix}
\begin{pmatrix}
\beta \\
0\\
\delta
\end{pmatrix}=
\begin{pmatrix}
0\\
0\\
\delta
\end{pmatrix}, \text{ and }\\
&\begin{pmatrix}
1& 0&0\\
0&1&0\\
0&0&1/\delta
\end{pmatrix}
\begin{pmatrix}
0 \\
0 \\
\delta
\end{pmatrix}=
\begin{pmatrix}
0\\
0\\
1
\end{pmatrix}.
\end{align*}
Next, if $\gamma =\delta = 0$, but $\beta \neq 0$, then we have
$
\begin{pmatrix}
\beta \\
0 \\
0
\end{pmatrix}$.
Thus the following elements are $\Aut_p(L)$-orbit representatives:
\begin{align*}
\begin{pmatrix}
0 \\ 
0\\
0
\end{pmatrix},
\begin{pmatrix}
\beta\\ 
0\\
0
\end{pmatrix},
\begin{pmatrix}
0\\ 
1\\
0
\end{pmatrix},
\begin{pmatrix}
0\\ 
0\\
1
\end{pmatrix}.
\end{align*}

Therefore, the corresponding restricted Lie algebra structures are as follows:
\begin{align*}
&K_2^{11}=\langle x_1,\ldots,x_5\mid [x_1,x_2]=x_3,
\res x_1= x_3 \rangle;\\
&K_2^{12}(\beta)=\langle x_1,\ldots,x_5\mid [x_1,x_2]=x_3,
\res x_1= x_3, \res x_2=\beta x_5 \rangle;\\
&K_2^{13}=\langle x_1,\ldots,x_5\mid [x_1,x_2]=x_3,
\res x_1= x_3, \res x_3=x_5 \rangle;\\
&K_2^{14}=\langle x_1,\ldots,x_5\mid [x_1,x_2]=x_3,
\res x_1= x_3, \res x_4=x_5 \rangle
\end{align*} where $\beta \in \F^*$.

\begin{lemma}\label{first-example-orbit}
We have $K_2^{12}(\beta_1)\cong K_2^{12}(1)$, for every $\beta\in \F^*$.
\end{lemma}
\begin{proof}
The following linear map yields the desired automorphism:
\begin{align*}
\begin{pmatrix}
\beta^{1/p} & 0 & 0 & 0& 0\\
 0 & \beta^{(p-1)/p} & 0 & 0& 0\\
 0 & 0 & \beta &  0& 0\\
 0 & 0 & 0& 1 & 0\\
 0 & 0 & 0& 0 &\beta^{p-2} 
\end{pmatrix}
\end{align*}

\end{proof}

\subsection{Extensions of $(L, \res x_1 = x_4)$}
 First, we find a basis for $Z^2(L,\F)$. Let $(\phi,\omega)=(a\Delta _{12}+b\Delta_{13}+c\Delta_{14}+d\Delta_{23}+e\Delta_{24}+f\Delta_{34},\alpha f_1+\beta f_2+ \gamma f_3+\delta f_4)\in Z^2(L,\F)$. Then we must have $\delta^2\phi(x,y,z) =0$ and $\phi(x,\res y)=0$, for all $x,y,z \in L$. Therefore, we have
\begin{align*}
0=(\delta^2\phi)(x_1,x_2,x_4)&=\phi([x_1,x_2],x_4)+\phi([x_2,x_4],x_1)+\phi([x_4,x_1],x_2)=\phi(x_3,x_4)=f.
\end{align*}
Also, we have $\phi(x,\res y)=0$. Therefore, $\phi(x,x_4)=0$, for all $x\in L$ and hence $\phi(x_1,x_4)=\phi(x_2,x_4)=\phi(x_3,x_4)=0$ which implies that $c=e=f=0$. Therefore, $Z^2(L,\F)$ has a basis consisting of:
$$ (\Delta_{12},0),(\Delta_{13},0),(\Delta_{23},0),(0,f_1),
(0,f_2),(0,f_3),(0,f_4).$$
Next, we find a basis for $B^2(L,\F)$. Let $(\phi,\omega)\in B^2(L,\F)$. Since $B^2(L,\F)\subseteq Z^2(L,\F)$, we have $(\phi,\omega)=(a\Delta _{12}+b\Delta_{13}+d\Delta_{23},\alpha f_1+\beta f_2+ \gamma f_3+\delta f_4)$. Note that there exists a linear map $\psi:L\to \F$  such that $\delta^1\psi(x,y)=\phi(x,y)$ and $\tilde \psi(x)=\omega(x)$, for all $x,y \in L$. So, we have
\begin{align*}
&a=\phi(x_1,x_2)=\delta^1\psi(x_1,x_2)=\psi([x_1,x_2])=\psi(x_3), \text{ and }\\
&b=\phi(x_1,x_3)=\delta^1\psi(x_1,x_3)=\psi([x_1,x_3])=0.
\end{align*}
Similarly, we can show that $d=0$. Also, we have
\begin{align*}
&\alpha=\omega(x_1)=\tilde \psi(x_1)=\psi(\res x_1)=\psi(x_4), \text{ and }\\
&\beta=\omega(x_2)=\tilde \psi(x_2)=\psi(\res x_2)=0.
\end{align*}
Similarly, we can show that $\gamma=\delta=0$. Therefore, $(\phi,\omega)=(a\Delta_{12},\alpha f_1)$ and hence $B^2(L,\F)=\la(\Delta_{12},0),(0,f_1)\ra_{\F}$. We deduce that a basis for $H^2(L,\F)$ is as follows:
$$[(\Delta_{13},0)],[(\Delta_{23},0)],[(0,f_2)],[(0,f_3)],[(0,f_4)].$$
 Note that the group $\Aut_p(L)$ in this case consists of invertible matrices of the form 
$$\begin{pmatrix}
 a_{11}  & a_{12} & 0 & 0 \\
a_{21} & a_{22} & 0 & 0\\
a_{31} & a_{32} & r & a_{34} \\
a_{41} & a_{42} & 0& a_{44}
\end{pmatrix},$$
where $r = a_{11}a_{22}-a_{12}a_{21}\neq 0$, $a_{44}=a_{11}^p$ and $a_{12}=a_{34}=0$.

Let $[\theta]=[(\phi,\omega)]\in H^2(L,\F)$. As in Section  \ref{L_{5,2} mod x5 trivial pmap},  it is enough to find $\Aut_p(L)$-representatives of the $\omega$'s. We can verify the action of $\Aut_p(L)$ on the $\omega$'s in the matrix form is as follows:
\begin{align*}
\begin{pmatrix}
\beta'\\
\gamma'\\
\delta'
\end{pmatrix}=
\begin{pmatrix}
a_{22}^p& a_{32}^p&a_{42}^p\\
0&r ^p&0\\
0&0&a_{11}^{p^2}
\end{pmatrix}
\begin{pmatrix}
\beta \\
\gamma \\
\delta
\end{pmatrix}.
\end{align*}

Now we find the representatives of the orbits of this action. 
Let $\nu = \begin{pmatrix} 
\beta \\
\gamma \\
\delta
\end{pmatrix} \in \F^3$. If $\nu= \begin{pmatrix}
0 \\
0 \\
0
\end{pmatrix}$, then $\{\nu \}$ is clearly an  $\Aut_p(L)$-orbit.  Let $\nu \neq 0$. Suppose that $\delta \neq 0$. Then 
\begin{align*}
&\begin{pmatrix}
1&0& -\beta /\delta\\
0&1&0\\
0&0&1
\end{pmatrix}
\begin{pmatrix}
\beta \\
\gamma \\
\delta
\end{pmatrix}=
\begin{pmatrix}
0\\
\gamma\\
\delta
\end{pmatrix}, \text{ and }\\
&\begin{pmatrix}
 (1/\delta)^\frac{-1}{p^2}& 0&0\\
0&1&0\\
0&0&1/\delta
\end{pmatrix}
\begin{pmatrix} 
0 \\
\gamma \\
\delta
\end{pmatrix}=
\begin{pmatrix}
0\\
\gamma\\
1
\end{pmatrix}.
\end{align*}
If $\gamma \neq 0$, then
\begin{align*}
&\begin{pmatrix}
1/\gamma& 0&0\\
0&1/\gamma&0\\
0&0&1
\end{pmatrix}
\begin{pmatrix}
0 \\
\gamma\\
1
\end{pmatrix}=
\begin{pmatrix}
0\\
1\\
1
\end{pmatrix}.
\end{align*}
If $\gamma=0$ then we have
$\begin{pmatrix}
0\\
0\\
1
\end{pmatrix}.$
Next, if $\delta = 0$, but $\gamma \neq 0$, then
\begin{align*}
\begin{pmatrix}
1& -\beta/\delta&0\\
0&1&0\\
0&0&1
\end{pmatrix}
\begin{pmatrix}
\beta \\
\gamma\\
0
\end{pmatrix}=
\begin{pmatrix}
0\\
\gamma\\
0
\end{pmatrix}, \text { and }\\
\begin{pmatrix}
1/\gamma&0&0\\
0&1/\gamma&0\\
0&0&1
\end{pmatrix}
\begin{pmatrix}
0 \\
\gamma\\
0
\end{pmatrix}=
\begin{pmatrix}
0\\
1\\
0
\end{pmatrix}.
\end{align*}
Finally, if $\delta=\gamma=0$ then we have
\begin{align*}
\begin{pmatrix}
1/\beta&0&0\\
0&1/\beta&0\\
0&0&1
\end{pmatrix}
\begin{pmatrix}
\beta\\
0\\
0
\end{pmatrix}=
\begin{pmatrix}
1\\
0\\
0
\end{pmatrix}.
\end{align*}

Thus the following elements are $\Aut_p(L)$-orbit representatives:
\begin{align*}
\begin{pmatrix}
0 \\ 
0\\
0
\end{pmatrix},
\begin{pmatrix}
1\\ 
0\\
0
\end{pmatrix},
\begin{pmatrix}
0\\ 
1\\
0
\end{pmatrix},
\begin{pmatrix}
0\\ 
0\\
1
\end{pmatrix},
\begin{pmatrix}
0\\ 
1\\
1
\end{pmatrix}.
\end{align*}
Therefore, the corresponding restricted Lie algebra structures are as follows:
\begin{align*}
&K_2^{15}=\langle x_1,\ldots,x_5\mid [x_1,x_2]=x_3,
\res x_1= x_4 \rangle;\\
&K_2^{16}=\langle x_1,\ldots,x_5\mid [x_1,x_2]=x_3,
\res x_1= x_4, \res x_2=x_5 \rangle;\\
&K_2^{17}=\langle x_1,\ldots,x_5\mid [x_1,x_2]=x_3,
\res x_1= x_4, \res x_3=x_5 \rangle;\\
&K_2^{18}=\langle x_1,\ldots,x_5\mid [x_1,x_2]=x_3,
\res x_1= x_4, \res x_4=x_5 \rangle;\\
&K_2^{19}=\langle x_1,\ldots,x_5\mid [x_1,x_2]=x_3,
\res x_1= x_4,\res x_3=x_5, \res x_4=x_5 \rangle.
\end{align*}

\subsection{Extensions of $(L,  \res x_1 = x_3, \res x_2 =x_4)$}
 First, we find a basis for $Z^2(L,\F)$. Let $(\phi,\omega)=(a\Delta _{12}+b\Delta_{13}+c\Delta_{14}+d\Delta_{23}+e\Delta_{24}+f\Delta_{34},\alpha f_1+\beta f_2+ \gamma f_3+\delta f_4)\in Z^2(L,\F)$. Then we must have $\delta^2\phi(x,y,z) =0$ and $\phi(x,\res y)=0$, for all $x,y,z \in L$. Therefore, we have
\begin{align*}
0=(\delta^2\phi)(x_1,x_2,x_4)&=\phi([x_1,x_2],x_4)+\phi([x_2,x_4],x_1)+\phi([x_4,x_1],x_2)=\phi(x_3,x_4)=f.
\end{align*}
Also, we have $\phi(x,\res y)=0$. Therefore, $\phi(x,x_3)=0$ and $\phi(x,x_4)=0$, for all $x\in L$ and hence $\phi(x_1,x_3)=\phi(x_2,x_3)=\phi(x_1,x_4)=\phi(x_2,x_4)=\phi(x_3,x_4)=0$ which implies that $b=c=d=e=f=0$. Therefore, $Z^2(L,\F)$ has a basis consisting of:
$$(\Delta_{12},0),(0,f_1),
(0,f_2),(0,f_3),(0,f_4).$$
Next, we find a basis for $B^2(L,\F)$. Let $(\phi,\omega)\in B^2(L,\F)$. Since $B^2(L,\F)\subseteq Z^2(L,\F)$, we have $(\phi,\omega)=(a\Delta _{12},\alpha f_1+\beta f_2+ \gamma f_3+\delta f_4)$. Note that there exists a linear map $\psi:L\to \F$ such that $\delta^1\psi(x,y)=\phi(x,y)$ and $\tilde \psi(x)=\omega(x)$, for all $x,y \in L$. So, we have
\begin{align*}
&a=\phi(x_1,x_2)=\delta^1\psi(x_1,x_2)=\psi([x_1,x_2])=\psi(x_3).
\end{align*}
Also, we have
\begin{align*}
&\alpha=\omega(x_1)=\tilde \psi(x_1)=\psi(\res x_1)=\psi(x_3), \text{ and }\\
&\beta=\omega(x_2)=\tilde \psi(x_2)=\psi(\res x_2)=\psi(x_4), \text{ and }\\
&\gamma=\omega(x_3)=\tilde \psi(x_3)=\psi(\res x_3)=0.
\end{align*}
Similarly, we can show that $\delta=0$. Note that $\psi(x_3)=a=\alpha$.  Therefore, $(\phi,\omega)=(a\Delta_{12},af_1+\beta f_2)$ and hence $B^2(L,\F)=\la(\Delta_{12},f_1),(0,f_2)\ra_{\F}$. Note that  
$$\bigg\{[(\Delta_{12},0)],[(0,f_1)],
[(0,f_2)],[(0,f_3)],[(0,f_4)]\bigg\}$$ spans $H^2(L,\F)$.
Since $[(\Delta_{12},0)]+[(0,f_1)]=[(\Delta_{12},f_1)]=[0]$, then $[(0,f_1)]$ is an scalar multiple of $[(\Delta_{12},0)]$
in $H^2(L,\F)$. Note that $\dim H^2=\dim Z^2-\dim B^2=3$. Therefore, $$\bigg \{ [(\Delta_{12},0)],[(0,f_3)],[(0,f_4)]\bigg\}$$ forms a basis for $H^2(L,\F)$.

 Note that the group $\Aut_p(L)$ in this case consists of invertible matrices of the form 
$$\begin{pmatrix}
 a_{11}  & a_{12} & 0 & 0 \\
a_{21} & a_{22} & 0 & 0\\
a_{31} & a_{32} & r & a_{34} \\
a_{41} & a_{42} & 0& a_{44}
\end{pmatrix},$$
where $r = a_{11}a_{22}-a_{12}a_{21}=a_{11}^p$ and $a_{21}=0$, $a_{34}=a_{12}^p$, and $a_{44}=a_{22}^p$.

Let $[\theta]=[(\phi,\omega)]\in H^2(L,\F)$. As in Section  \ref{L_{5,2} mod x5 trivial pmap},  it is enough to find $\Aut_p(L)$-representatives of the $\omega$'s. We can verify that the action of $\Aut_p(L)$ on the $\omega$'s in the matrix form is as follows:
\begin{align*}
\begin{pmatrix}
\gamma'\\
\delta'
\end{pmatrix}=
\begin{pmatrix}
a_{11}^{p^2}&0\\
a_{12}^{p^2}&a_{22}^{p^2}
\end{pmatrix}
\begin{pmatrix}
\gamma \\
\delta
\end{pmatrix}.
\end{align*}

Now we find the representatives of the orbits of this action. Note that we need to have $A\omega(x_1)=0$  which implies that
$a_{31}^p\gamma+a_{41}^p\delta=0$.
Let $\nu = \begin{pmatrix} 
\gamma \\
\delta
\end{pmatrix} \in \F^2$. If $\nu= \begin{pmatrix}
0 \\
0
\end{pmatrix}$, then $\{\nu \}$ is clearly an  $\Aut_p(L)$-orbit.  Let $\nu \neq 0$. Suppose that $\gamma \neq 0$. Then 
\begin{align*}
&\begin{pmatrix}
1&0\\
-\delta/\gamma&1
\end{pmatrix}
\begin{pmatrix}
\gamma \\
\delta
\end{pmatrix}=
\begin{pmatrix}
\gamma\\
0
\end{pmatrix}, \text{ and }\\
&\begin{pmatrix}
1/\gamma&0\\
0&(1/\gamma)^{p-1}
\end{pmatrix}
\begin{pmatrix} 
\gamma \\
0
\end{pmatrix}=
\begin{pmatrix}
1\\
0
\end{pmatrix}.
\end{align*}
Next, if $\gamma =0$, but $\delta\neq 0$, then we have
$\begin{pmatrix}
0 \\
\delta
\end{pmatrix}.$
Thus the following elements are $\Aut_p(L)$-orbit representatives:
\begin{align*}
\begin{pmatrix}
0 \\ 
0
\end{pmatrix},
\begin{pmatrix}
1\\ 
0
\end{pmatrix},
\begin{pmatrix}
0\\ 
\delta
\end{pmatrix}.
\end{align*}
Therefore, the corresponding restricted Lie algebra structures are as follows:
\begin{align*}
&K_2^{20}=\langle x_1,\ldots,x_5\mid [x_1,x_2]=x_3,
\res x_1= x_3, \res x_2=x_4 \rangle;\\
&K_2^{21}=\langle x_1,\ldots,x_5\mid [x_1,x_2]=x_3,
\res x_1= x_3, \res x_2=x_4,  \res x_3=x_5 \rangle;\\
&K_2^{22}(\delta)=\langle x_1,\ldots,x_5\mid [x_1,x_2]=x_3,
\res x_1= x_3, \res x_2=x_4,  \res x_4=\delta x_5 \rangle
\end{align*}
where $\delta \in \F^*$.

\begin{lemma}
We have $K_2^{22}(\delta)\cong K_2^{22}(1)$, for every $\delta\in \F^*$.
\end{lemma}
\begin{proof}
The following linear map yields the desired automorphism:
\begin{align*}
\begin{pmatrix}
\delta^{1/p^2} & 0 & 0 & 0& 0\\
 0 & \delta^{(p-1)/p^2} & 0 & 0& 0\\
 0 & 0 & \delta^{1/p}  &  0& 0\\
 0 & 0 & 0& \delta^{(p-1)/p}  & 0\\
 0 & 0 & 0& 0 &\delta^{p-2} 
\end{pmatrix}
\end{align*}

\end{proof}

\subsection{Extensions of $(L, \res x_3 = x_4)$}
 First, we find a basis for $Z^2(L,\F)$. Let $(\phi,\omega)=(a\Delta _{12}+b\Delta_{13}+c\Delta_{14}+d\Delta_{23}+e\Delta_{24}+f\Delta_{34},\alpha f_1+\beta f_2+ \gamma f_3+\delta f_4)\in Z^2(L,\F)$. Then we must have $\delta^2\phi(x,y,z) =0$ and $\phi(x,\res y)=0$, for all $x,y,z \in L$. Therefore, we have
\begin{align*}
0=(\delta^2\phi)(x_1,x_2,x_4)&=\phi([x_1,x_2],x_4)+\phi([x_2,x_4],x_1)+\phi([x_4,x_1],x_2)=\phi(x_3,x_4)=f.
\end{align*}
Also, we have $\phi(x,\res y)=0$. Therefore, $\phi(x,x_4)=0$, for all $x\in L$ and hence $\phi(x_1,x_4)=\phi(x_2,x_4)=\phi(x_3,x_4)=0$ which implies that $c=e=f=0$. Therefore, $Z^2(L,\F)$ has a basis consisting of:
$$(\Delta_{12},0),(\Delta_{13},0),(\Delta_{23},0),(0,f_1),
(0,f_2),(0,f_3),(0,f_4).$$
Next, we find a basis for $B^2(L,\F)$. Let $(\phi,\omega)\in B^2(L,\F)$. Since $B^2(L,\F)\subseteq Z^2(L,\F)$, we have $(\phi,\omega)=(a\Delta _{12}+b\Delta_{13}+d\Delta_{23},\alpha f_1+\beta f_2+ \gamma f_3+\delta f_4)$. Note that there exists a linear map $\psi:L\to \F$ such that $\delta^1\psi(x,y)=\phi(x,y)$ and $\tilde \psi(x)=\omega(x)$, for all $x,y \in L$. So, we have
\begin{align*}
&a=\phi(x_1,x_2)=\delta^1\psi(x_1,x_2)=\psi([x_1,x_2])=\psi(x_3), \text{ and }\\
&b=\phi(x_1,x_3)=\delta^1\psi(x_1,x_3)=\psi([x_1,x_3])=0.\\
\end{align*}
Similarly, we can show that $d=0$. Also, we have
\begin{align*}
&\gamma=\omega(x_3)=\tilde \psi(x_3)=\psi(\res x_3)=\psi(x_4), \text{ and }\\
&\alpha=\omega(x_1)=\tilde \psi(x_1)=\psi(\res x_1)=0.
\end{align*}
Similarly, we can show that $\beta=\delta=0$. Therefore, $(\phi,\omega)=(a\Delta_{12},\gamma f_3)$ and hence $B^2(L,\F)=\la(\Delta_{12},0),(0,f_3)\ra_{\F}$. We deduce that a basis for $H^2(L,\F)$ is as follows:
$$[(\Delta_{13},0)],[(\Delta_{23},0)],[(0,f_1)],[(0,f_2)],[(0,f_4)].$$
Note that the group $\Aut_p(L)$ in this case consists of invertible matrices of the form 
$$\begin{pmatrix}
 a_{11}  & a_{12} & 0 & 0 \\
a_{21} & a_{22} & 0 & 0\\
a_{31} & a_{32} & r & a_{34} \\
a_{41} & a_{42} & 0& a_{44}
\end{pmatrix},$$
where $r = a_{11}a_{22}-a_{12}a_{21}$ and $a_{31}=a_{32}=a_{34}=0$, and $a_{44}=r^p$.

Let $[\theta]=[(\phi,\omega)]\in H^2(L,\F)$. As in Section  \ref{L_{5,2} mod x5 trivial pmap},  it is enough to find $\Aut_p(L)$-representatives of the $\omega$'s. We can verify that the action of $\Aut_p(L)$ on the $\omega$'s in the matrix form is as follows:
\begin{align*}
\begin{pmatrix}
\alpha' \\ 
\beta'\\
\delta'
\end{pmatrix}=
\begin{pmatrix}
a_{11}^p& a_{21}^p&a_{41}^p\\
a_{12}^p& a_{22}^p&a_{42}^p\\
0&0&r^{p^2}
\end{pmatrix}
\begin{pmatrix}
\alpha\\ 
\beta \\
\delta
\end{pmatrix}.
\end{align*}

Now we find the representatives of the orbits of this action.
Let $\nu = \begin{pmatrix} 
\alpha \\
\beta \\
\delta
\end{pmatrix} \in \F^3$. If $\nu= \begin{pmatrix}
0 \\
0 \\
0
\end{pmatrix}$, then $\{\nu \}$ is clearly an  $\Aut_p(L)$-orbit.  Let $\nu \neq 0$. Suppose that $\delta \neq 0$. Then 
\begin{align*}
&\begin{pmatrix}
1& 0&-\alpha/\delta\\
0&1&-\beta/\delta\\
0&0&1
\end{pmatrix}
\begin{pmatrix}
\alpha \\
\beta \\
\delta
\end{pmatrix}=
\begin{pmatrix}
0\\
0\\
\delta
\end{pmatrix}, \text{ and }\\
&\begin{pmatrix}
 (1/\delta)^{1/p}& 0&0\\
0&1&0\\
0&0&1/\delta
\end{pmatrix}
\begin{pmatrix} 
0 \\
0\\
\delta
\end{pmatrix}=
\begin{pmatrix}
0\\
0\\
1
\end{pmatrix}.
\end{align*}
Next, if $\delta =0$, but $\beta\neq 0$, then
\begin{align*}
&\begin{pmatrix}
1& -\alpha/\beta&0\\
0&1&0\\
0&0&1
\end{pmatrix}
\begin{pmatrix}
\alpha \\
\beta\\
0
\end{pmatrix}=
\begin{pmatrix}
0\\
\beta\\
0
\end{pmatrix}, \text{ and }\\
&\begin{pmatrix}
1& 0&0\\
0&1/\beta&0\\
0&0&(1/\beta)^p
\end{pmatrix}
\begin{pmatrix}
0 \\
\beta \\
0
\end{pmatrix}=
\begin{pmatrix}
0\\
1\\
0
\end{pmatrix}.
\end{align*}
Next, if $\delta =\beta= 0$, but $\alpha \neq 0$, then
\begin{align*}
\begin{pmatrix}
1/\alpha& 0&0\\
0&1&0\\
0&0&(1/\alpha)^p
\end{pmatrix}
\begin{pmatrix}
\alpha \\
0 \\
0
\end{pmatrix}=
\begin{pmatrix}
1\\
0\\
0
\end{pmatrix}.
\end{align*}
Thus the following elements are $\Aut_p(L)$-orbit representatives:
\begin{align*}
\begin{pmatrix}
0 \\ 
0\\
0
\end{pmatrix},
\begin{pmatrix}
1\\ 
0\\
0
\end{pmatrix},
\begin{pmatrix}
0\\ 
1\\
0
\end{pmatrix},
\begin{pmatrix}
0\\ 
0\\
1
\end{pmatrix}.
\end{align*}
Therefore, the corresponding restricted Lie algebra structures are as follows:
\begin{align*}
&K_2^{23}=\langle x_1,\ldots,x_5\mid [x_1,x_2]=x_3,
\res x_3= x_4 \rangle;\\
&K_2^{24}=\langle x_1,\ldots,x_5\mid [x_1,x_2]=x_3,
\res x_1= x_5, \res x_3=x_4 \rangle;\\
&K_2^{25}=\langle x_1,\ldots,x_5\mid [x_1,x_2]=x_3,
\res x_2= x_5, \res x_3=x_4 \rangle;\\
&K_2^{26}=\langle x_1,\ldots,x_5\mid [x_1,x_2]=x_3,
\res x_3= x_4, \res x_4=x_5 \rangle.
\end{align*}
\subsection{Extensions of $(L, \res x_2 = x_3, \res x_3 =x_4)$}
 First, we find a basis for $Z^2(L,\F)$. Let $(\phi,\omega)=(a\Delta _{12}+b\Delta_{13}+c\Delta_{14}+d\Delta_{23}+e\Delta_{24}+f\Delta_{34},\alpha f_1+\beta f_2+ \gamma f_3+\delta f_4)\in Z^2(L,\F)$. Then we must have $\delta^2\phi(x,y,z) =0$ and $\phi(x,\res y)=0$, for all $x,y,z \in L$. Therefore, we have
\begin{align*}
0=(\delta^2\phi)(x_1,x_2,x_4)&=\phi([x_1,x_2],x_4)+\phi([x_2,x_4],x_1)+\phi([x_4,x_1],x_2)=\phi(x_3,x_4)=f.
\end{align*}
Also, we have $\phi(x,\res y)=0$. Therefore, $\phi(x,x_3)=0$ and $\phi(x,x_4)=0$, for all $x\in L$ and hence $\phi(x_1,x_3)=\phi(x_2,x_3)=\phi(x_1,x_4)=\phi(x_2,x_4)=\phi(x_3,x_4)=0$ which implies that $b=c=d=e=f=0$. Therefore, $Z^2(L,\F)$ has a basis consisting of:
$$(\Delta_{12},0),(0,f_1),
(0,f_2),(0,f_3),(0,f_4).$$
Next, we find a basis for $B^2(L,\F)$. Let $(\phi,\omega)\in B^2(L,\F)$. Since $B^2(L,\F)\subseteq Z^2(L,\F)$, we have $(\phi,\omega)=(a\Delta _{12},\alpha f_1+\beta f_2+ \gamma f_3+\delta f_4)$. Note that there exists a linear map $\psi:L\to \F$ such that $\delta^1\psi(x,y)=\phi(x,y)$ and $\tilde \psi(x)=\omega(x)$, for all $x,y \in L$. So, we have
\begin{align*}
a=\phi(x_1,x_2)=\delta^1\psi(x_1,x_2)=\psi([x_1,x_2])=\psi(x_3).
\end{align*}
Also, we have
\begin{align*}
&\alpha=\omega(x_1)=\tilde \psi(x_1)=\psi(\res x_1)=0, \text{ and }\\
&\beta=\omega(x_2)=\tilde \psi(x_2)=\psi(\res x_2)=\psi(x_3), \text{ and }\\
&\gamma=\omega(x_3)=\tilde \psi(x_3)=\psi(\res x_3)=\psi(x_4), \text{ and }\\
&\delta=\omega(x_4)=\tilde \psi(x_4)=\psi(\res x_4)=0.
\end{align*}
Note that $\psi(x_3)=a=\beta$.  Hence, $(\phi,\omega)=(a\Delta_{12},af_2+\gamma f_3)$ and  $B^2(L,\F)=\la(\Delta_{12},f_2),(0,f_3)\ra_{\F}$. Note that  
$$\bigg\{[(\Delta_{12},0)],[(0,f_1)],
[(0,f_2)],[(0,f_4)]\bigg\}$$ spans $H^2(L,\F)$.
Since $[(\Delta_{12},0)]+[(0,f_2)]=[(\Delta_{12},f_2)]=[0]$, then $[(0,f_2)]$ is an scalar multiple of $[(\Delta_{12},0)]$
in $H^2(L,\F)$. Note that $\dim H^2=\dim Z^2-\dim B^2=3$. Therefore, $$\bigg \{ [(\Delta_{12},0)],[(0,f_1)],[(0,f_4)]\bigg\}$$ forms a basis for $H^2(L,\F)$.

Note that the group $\Aut_p(L)$ in this case consists of invertible matrices of the form 
$$\begin{pmatrix}
 a_{11}  & a_{12} & 0 & 0 \\
a_{21} & a_{22} & 0 & 0\\
a_{31} & a_{32} & r & a_{34} \\
a_{41} & a_{42} & 0& a_{44}
\end{pmatrix},$$
where $r = a_{11}a_{22}-a_{12}a_{21}=a_{22}^p$ and $a_{21}=a_{31}=a_{32}=a_{34}=0$, $a_{44}=r^p$.

Let $[\theta]=[(\phi,\omega)]\in H^2(L,\F)$. As in Section  \ref{L_{5,2} mod x5 trivial pmap},  it is enough to find $\Aut_p(L)$-representatives of the $\omega$'s. We can verify that the action of $\Aut_p(L)$ on the $\omega$'s in the matrix form is as follows:
\begin{align*}
\begin{pmatrix}
\alpha' \\ 
\delta'
\end{pmatrix}=
\begin{pmatrix}
a_{11}^p&a_{41}^p\\
0&r^{p^2}
\end{pmatrix}
\begin{pmatrix}
\alpha\\ 
\delta
\end{pmatrix}.
\end{align*}

Now we find the representatives of the orbits of this action. Note that we need to have $A\omega(x_2)=0$ which implies that $a_{21}^p\alpha+a_{42}^p\delta=0$.

Let $\nu = \begin{pmatrix} 
\alpha \\
\delta
\end{pmatrix} \in \F^2$. If $\nu= \begin{pmatrix}
0 \\
0
\end{pmatrix}$, then $\{\nu \}$ is clearly an  $\Aut_p(L)$-orbit.  Let $\nu \neq 0$. Suppose that $\delta \neq 0$. Then 
\begin{align*}
&\begin{pmatrix}
1&-\alpha/\delta\\
0&1
\end{pmatrix}
\begin{pmatrix}
\alpha \\
\delta
\end{pmatrix}=
\begin{pmatrix}
0\\
\delta
\end{pmatrix}, \text{ and }\\
&\begin{pmatrix}
(1/\delta)^\frac{p-1}{p^2}&0\\
0&1/\delta
\end{pmatrix}
\begin{pmatrix} 
0\\
\delta
\end{pmatrix}=
\begin{pmatrix}
0\\
1
\end{pmatrix}.
\end{align*}
Next, if $\delta =0$, but $\alpha\neq 0$, then we have
$\begin{pmatrix}
\alpha \\
0
\end{pmatrix}.$
Thus the following elements are $\Aut_p(L)$-orbit representatives:
\begin{align*}
\begin{pmatrix}
0 \\ 
0
\end{pmatrix},
\begin{pmatrix}
\alpha\\ 
0
\end{pmatrix},
\begin{pmatrix}
0\\ 
1
\end{pmatrix}.
\end{align*}
Therefore, the corresponding restricted Lie algebra structures are as follows:
\begin{align*}
&K_2^{27}=\langle x_1,\ldots,x_5\mid [x_1,x_2]=x_3,
\res x_2= x_3, \res x_3=x_4 \rangle;\\
&K_2^{28}(\alpha)=\langle x_1,\ldots,x_5\mid [x_1,x_2]=x_3,
\res x_1=\alpha x_5, \res x_2=x_3,  \res x_3=x_4 \rangle;\\
&K_2^{29}=\langle x_1,\ldots,x_5\mid [x_1,x_2]=x_3,
\res x_2= x_3, \res x_3=x_4,  \res x_4=x_5 \rangle
\end{align*}
where $\alpha \in \F^*$.

\begin{lemma}
We have $K_2^{28}(\alpha)\cong K_2^{28}(1)$, for every $\alpha\in \F^*$.
\end{lemma}
\begin{proof}
The following linear map yields the desired automorphism:
\begin{align*}
\begin{pmatrix}
\alpha^{(p-1)/p} & 0 & 0 & 0& 0\\
 0 & \alpha^{1/p}& 0 & 0& 0\\
 0 & 0 & \alpha &  0& 0\\
 0 & 0 & 0& \alpha^{p} & 0\\
 0 & 0 & 0& 0 &\alpha^{p-2} 
\end{pmatrix}
\end{align*}

\end{proof}

\subsection{Extensions of  $(L,  \res x_4 = x_3)$}
 First, we find a basis for $Z^2(L,\F)$. Let $(\phi,\omega)=(a\Delta _{12}+b\Delta_{13}+c\Delta_{14}+d\Delta_{23}+e\Delta_{24}+f\Delta_{34},\alpha f_1+\beta f_2+ \gamma f_3+\delta f_4)\in Z^2(L,\F)$. Then we must have $\delta^2\phi(x,y,z) =0$ and $\phi(x,\res y)=0$, for all $x,y,z \in L$. Therefore, we have
\begin{align*}
0=(\delta^2\phi)(x_1,x_2,x_4)&=\phi([x_1,x_2],x_4)+\phi([x_2,x_4],x_1)+\phi([x_4,x_1],x_2)=\phi(x_3,x_4)=f.
\end{align*}
Also, we have $\phi(x,\res y)=0$. Therefore, $\phi(x,x_3)=0$, for all $x\in L$ and hence $\phi(x_1,x_3)=\phi(x_2,x_3)=\phi(x_4,x_3)=0$ which implies that $b=d=f=0$. Therefore, $Z^2(L,\F)$ has a basis consisting of:
$$(\Delta_{12},0),(\Delta_{14},0),(\Delta_{24},0),(0,f_1),
(0,f_2),(0,f_3),(0,f_4).$$
Next, we find a basis for $B^2(L,\F)$. Let $(\phi,\omega)\in B^2(L,\F)$. Since $B^2(L,\F)\subseteq Z^2(L,\F)$, we have $(\phi,\omega)=(a\Delta _{12}+c\Delta_{14}+e\Delta_{24},\alpha f_1+\beta f_2+ \gamma f_3+\delta f_4)$. Note that there exists a linear map $\psi:L\to \F$ such that $\delta^1\psi(x,y)=\phi(x,y)$ and $\tilde \psi(x)=\omega(x)$, for all $x,y \in L$. So, we have
\begin{align*}
&a=\phi(x_1,x_2)=\delta^1\psi(x_1,x_2)=\psi([x_1,x_2])=\psi(x_3), \text{ and }\\
&c=\phi(x_1,x_4)=\delta^1\psi(x_1,x_4)=\psi([x_1,x_4])=0.
\end{align*}
Similarly, we can show that $e=0$. Also, we have
\begin{align*}
&\delta=\omega(x_4)=\tilde \psi(x_4)=\psi(\res x_4)=\psi(x_3), \text{ and }\\
&\alpha=\omega(x_1)=\tilde \psi(x_1)=\psi(\res x_1)=0.
\end{align*}
Similarly, we can show that $\beta=\gamma=0$. Note that $\psi(x_3)=a=\delta$.  Therefore, $(\phi,\omega)=(a\Delta_{12},af_4)$ and hence $B^2(L,\F)=\la(\Delta_{12},f_4)\ra_{\F}$. Note that
$$\bigg\{[(\Delta_{12},0)],[(\Delta_{14},0)],[(\Delta_{24},0)],[(0,f_1)],
[(0,f_2)],[(0,f_3)],[(0,f_4)]\bigg\}$$ spans $H^2(L,\F)$. Since $[(\Delta_{12},0)]+[(0,f_4)]=[(\Delta_{12},f_4)]=[0]$, then $[(0,f_4)]$ is an scalar multiple of $[(\Delta_{12},0)]$
in $H^2(L,\F)$. Note that $\dim H^2=\dim Z^2-\dim B^2=6$. Therefore, $$\bigg \{ [(\Delta_{12},0)],[(\Delta_{14},0)],[(\Delta_{24},0)],[(0,f_1)],[(0,f_2)],[(0,f_3)]\bigg\}$$ forms a basis for $H^2(L,\F)$.

Note that the group $\Aut_p(L)$ in this case consists of invertible matrices of the form 
$$\begin{pmatrix}
 a_{11}  & a_{12} & 0 & 0 \\
a_{21} & a_{22} & 0 & 0\\
a_{31} & a_{32} & r & a_{34} \\
a_{41} & a_{42} & 0& a_{44}
\end{pmatrix},$$
where $r = a_{11}a_{22}-a_{12}a_{21}=a_{44}^p$ and $a_{41}=a_{42}=0$.

Let $[\theta]=[(\phi,\omega)]\in H^2(L,\F)$. As in Section  \ref{L_{5,2} mod x5 trivial pmap},  it is enough to find $\Aut_p(L)$-representatives of the $\omega$'s. We can verify that the action of $\Aut_p(L)$ on the $\omega$'s in the matrix form is as follows:
\begin{align*}
\begin{pmatrix}
\alpha' \\ 
\beta'\\
\gamma'
\end{pmatrix}=
\begin{pmatrix}
a_{11}^p& a_{21}^p&a_{31}^p\\
a_{12}^p& a_{22}^p& a_{32}^p\\
0&0&r ^p\\
\end{pmatrix}
\begin{pmatrix}
\alpha\\ 
\beta \\
\gamma 
\end{pmatrix}.
\end{align*}

Now we find the representatives of the orbits of this action. Note that we need to have $A\omega(x_4)=0$ which implies that $a_{34}^p\gamma=0$.

Let $\nu = \begin{pmatrix}
\alpha\\ 
\beta \\
\gamma 
\end{pmatrix} \in \F^3$. If $\nu= \begin{pmatrix}
0\\ 
0 \\
0
\end{pmatrix}$, then $\{\nu \}$ is clearly an  $\Aut_p(L)$-orbit.  Let $\nu \neq 0$. Suppose that $\gamma \neq 0$. Then 
\begin{align*}
&\begin{pmatrix}
1& 0&-\alpha /\gamma\\
0& 1& -\beta /\gamma\\
0&0&1\\
\end{pmatrix}
\begin{pmatrix}
\alpha\\ 
\beta \\
\gamma 
\end{pmatrix}=
\begin{pmatrix}
0 \\ 
0\\
\gamma
\end{pmatrix}, \text{ and }\\
&\begin{pmatrix}
1/\gamma& 0&0\\
0& 1& 0\\
0&0&1/\gamma
\end{pmatrix}
\begin{pmatrix}
0\\ 
0 \\
\gamma
\end{pmatrix}=
\begin{pmatrix}
0 \\ 
0\\
1
\end{pmatrix}.
\end{align*}
Next, if $\gamma = 0$, but $\beta \neq 0$, then
\begin{align*}
&\begin{pmatrix}
1& -\alpha/\beta&0\\
0& 1& 0\\
0&0&1
\end{pmatrix}
\begin{pmatrix}
\alpha\\ 
\beta \\
0 
\end{pmatrix}=
\begin{pmatrix}
0 \\ 
\beta\\
0
\end{pmatrix}, \text{ and }\\
&\begin{pmatrix}
1& 0&0\\
0& 1/\beta& 0\\
0&0&1/\beta
\end{pmatrix}
\begin{pmatrix}
0\\ 
\beta \\
0 
\end{pmatrix}=
\begin{pmatrix}
0 \\ 
1\\
0
\end{pmatrix}.
\end{align*}
Finally, if $\gamma=\beta =0$, but $\alpha \neq 0$, then
\begin{align*}
\begin{pmatrix}
1/\alpha& 0&0\\
0& 1& 0\\
0&0&1/\alpha
\end{pmatrix}
\begin{pmatrix}
\alpha\\ 
0 \\
0 
\end{pmatrix}=
\begin{pmatrix}
1 \\ 
0\\
0
\end{pmatrix}.
\end{align*}
Thus the following elements are $\Aut (L)$-orbit representatives:
\begin{align*}
\begin{pmatrix}
0 \\ 
0\\
0
\end{pmatrix},
\begin{pmatrix}
1\\ 
0\\
0
\end{pmatrix},
\begin{pmatrix}
0\\ 
1\\
0
\end{pmatrix},
\begin{pmatrix}
0\\ 
0\\
1
\end{pmatrix}.
\end{align*}
Therefore, the corresponding restricted Lie algebra structures are as follows:
\begin{align*}
&K_2^{30}=\langle x_1,\ldots,x_5\mid [x_1,x_2]=x_3,
\res x_4= x_3 \rangle;\\
&K_2^{31}=\langle x_1,\ldots,x_5\mid [x_1,x_2]=x_3,
\res x_1= x_5, \res x_4=x_3 \rangle;\\
&K_2^{32}=\langle x_1,\ldots,x_5\mid [x_1,x_2]=x_3,
\res x_2= x_5, \res x_4=x_3 \rangle;\\
&K_2^{33}=\langle x_1,\ldots,x_5\mid [x_1,x_2]=x_3,
\res x_3= x_5, \res x_4=x_3 \rangle.
\end{align*}

\subsection{Extensions of  $(L, \res x_2 = x_4, \res x_4 =x_3)$}
 First, we find a basis for $Z^2(L,\F)$. Let $(\phi,\omega)=(a\Delta _{12}+b\Delta_{13}+c\Delta_{14}+d\Delta_{23}+e\Delta_{24}+f\Delta_{34},\alpha f_1+\beta f_2+ \gamma f_3+\delta f_4)\in Z^2(L,\F)$. Then we must have $\delta^2\phi(x,y,z) =0$ and $\phi(x,\res y)=0$, for all $x,y,z \in L$. Therefore, we have
\begin{align*}
0=(\delta^2\phi)(x_1,x_2,x_4)&=\phi([x_1,x_2],x_4)+\phi([x_2,x_4],x_1)+\phi([x_4,x_1],x_2)=\phi(x_3,x_4)=f.
\end{align*}
Also, we have $\phi(x,\res y)=0$. Therefore, $\phi(x,x_3)=0$ and $\phi(x,x_4)=0$, for all $x\in L$ and hence $\phi(x_1,x_3)=\phi(x_2,x_3)=\phi(x_1,x_4)=\phi(x_2,x_4)=\phi(x_3,x_4)=0$ which implies that $b=c=d=e=f=0$. Therefore,
$Z^2(L,\F)$ has a basis consisting of:
$$(\Delta_{12},0),(0,f_1),
(0,f_2),(0,f_3),(0,f_4).$$
Next, we find a basis for $B^2(L,\F)$. Let $(\phi,\omega)\in B^2(L,\F)$. Since $B^2(L,\F)\subseteq Z^2(L,\F)$, we have $(\phi,\omega)=(a\Delta _{12},\alpha f_1+\beta f_2+ \gamma f_3+\delta f_4)$. Note that there exists a linear map $\psi:L\to \F$ such that $\delta^1\psi(x,y)=\phi(x,y)$ and $\tilde \psi(x)=\omega(x)$, for all $x,y \in L$. So, we have
\begin{align*}
a=\phi(x_1,x_2)=\delta^1\psi(x_1,x_2)=\psi([x_1,x_2])=\psi(x_3).
\end{align*}
Also, we have
\begin{align*}
&\alpha=\omega(x_1)=\tilde \psi(x_1)=\psi(\res x_1)=0, \text{ and }\\
&\beta=\omega(x_2)=\tilde \psi(x_2)=\psi(\res x_2)=\psi(x_4), \text{ and }\\
&\gamma=\omega(x_3)=\tilde \psi(x_3)=\psi(\res x_3)=0, \text{ and }\\
&\delta=\omega(x_4)=\tilde \psi(x_4)=\psi(\res x_4)=\psi(x_3).
\end{align*}
Note that $\psi(x_3)=a=\delta$.  Hence, $(\phi,\omega)=(a\Delta_{12},\beta f_2+a f_4)$ and  $B^2(L,\F)=\la(\Delta_{12},f_4),(0,f_2)\ra_{\F}$. Note that  
$$
\bigg\{[(\Delta_{12},0)],[(0,f_1)],[(0,f_3)],[(0,f_4)]\bigg\}
$$ 
spans $H^2(L,\F)$.
Since $[(\Delta_{12},0)]+[(0,f_4)]=[(\Delta_{12},f_4)]=[0]$, then $[(0,f_4)]$ is an scalar multiple of $[(\Delta_{12},0)]$
in $H^2(L,\F)$. Note that $\dim H^2=\dim Z^2-\dim B^2=3$. Therefore, $$\bigg \{ [(\Delta_{12},0)],[(0,f_1)],[(0,f_3)]\bigg\}$$ forms a basis for $H^2(L,\F)$.

Note that the group $\Aut_p(L)$ in this case consists of invertible matrices of the form 
$$\begin{pmatrix}
 a_{11}  & a_{12} & 0 & 0 \\
a_{21} & a_{22} & 0 & 0\\
a_{31} & a_{32} & r & a_{34} \\
a_{41} & a_{42} & 0& a_{44}
\end{pmatrix},$$
where $r = a_{11}a_{22}-a_{12}a_{21}=a_{44}^p$ and $a_{21}=a_{41}=0$, $a_{34}=a_{42}^p$, and $a_{44}=a_{22}^p$.

Let $[\theta]=[(\phi,\omega)]\in H^2(L,\F)$. As in Section  \ref{L_{5,2} mod x5 trivial pmap},  it is enough to find $\Aut_p(L)$-representatives of the $\omega$'s. We can verify that the action of $\Aut_p(L)$ on the $\omega$'s in the matrix form is as follows:
\begin{align*}
\begin{pmatrix}
\alpha' \\ 
\gamma'
\end{pmatrix}=
\begin{pmatrix}
a_{11}^p&a_{31}^p\\
0&r ^p\\
\end{pmatrix}
\begin{pmatrix}
\alpha\\ 
\gamma 
\end{pmatrix}.
\end{align*}

Now we find the representatives of the orbits of this action. Note that we need to have $A\omega(x_4)=0$ which implies that $a_{34}^p\gamma=0$.

Let $\nu = \begin{pmatrix} 
\alpha \\
\gamma
\end{pmatrix} \in \F^2$. If $\nu= \begin{pmatrix}
0 \\
0
\end{pmatrix}$, then $\{\nu \}$ is clearly an  $\Aut_p(L)$-orbit.  Let $\nu \neq 0$. Suppose that $\gamma \neq 0$. Then 
\begin{align*}
&\begin{pmatrix}
1&-\alpha/\gamma\\
0&1
\end{pmatrix}
\begin{pmatrix}
\alpha \\
\gamma
\end{pmatrix}=
\begin{pmatrix}
0\\
\gamma
\end{pmatrix}, \text{ and }\\
&\begin{pmatrix}
(1/\gamma)^\frac{p^2-1}{p^2}&0\\
0&1/\gamma
\end{pmatrix}
\begin{pmatrix} 
0\\
\gamma
\end{pmatrix}=
\begin{pmatrix}
0\\
1
\end{pmatrix}.
\end{align*}
Next, if $\gamma=0$, but $\alpha\neq 0$, then we have 
$
\begin{pmatrix}
\alpha \\
0
\end{pmatrix}$.
Thus the following elements are $\Aut_p(L)$-orbit representatives:
\begin{align*}
\begin{pmatrix}
0 \\ 
0
\end{pmatrix},
\begin{pmatrix}
\alpha\\ 
0
\end{pmatrix},
\begin{pmatrix}
0\\ 
1
\end{pmatrix}.
\end{align*}
Therefore, the corresponding restricted Lie algebra structures are as follows:
\begin{align*}
&K_2^{34}=\langle x_1,\ldots,x_5\mid [x_1,x_2]=x_3,
\res x_2= x_4, \res x_4=x_3 \rangle;\\
&K_2^{35}(\alpha)=\langle x_1,\ldots,x_5\mid [x_1,x_2]=x_3,
\res x_1=\alpha x_5, \res x_2=x_4,  \res x_4=x_3 \rangle;\\
&K_2^{36}=\langle x_1,\ldots,x_5\mid [x_1,x_2]=x_3,
\res x_2= x_4, \res x_3=x_5,  \res x_4=x_3 \rangle
\end{align*}
where $\alpha \in \F^*$.

\begin{lemma}
We have $K_2^{35}(\alpha)\cong K_2^{35}(1)$, for every $\alpha\in \F^*$.
\end{lemma}
\begin{proof}
The following linear map yields the desired automorphism:
\begin{align*}
\begin{pmatrix}
\alpha^{(p^2-1)/p} & 0 & 0 & 0& 0\\
 0 & \alpha^{1/p}& 0 & 0& 0\\
 0 & 0 & \alpha^p &  0& 0\\
 0 & 0 & 0& \alpha & 0\\
 0 & 0 & 0& 0 &\alpha^{p^2-2} 
\end{pmatrix}
\end{align*}

\end{proof}

\section{Detecting Isomorphisms}\label{iso-L52}
We can easily see that some of the algebras given above are identical.  The following is the list of all irredundant restricted Lie algebra structures on $L_{5,2}$ and yet, as we shall see below, we prove that some of them are isomorphic.
\begin{align*}
&K_2^{1}=\langle x_1,\ldots,x_5\mid [x_1,x_2]=x_3 \rangle;\\
&K_2^{2}=\langle x_1,\ldots,x_5\mid [x_1,x_2]=x_3,
\res x_1= x_3 \rangle;\\
&K_2^{3}=\langle x_1,\ldots,x_5\mid [x_1,x_2]=x_3,
\res x_2= x_3 \rangle;\\
&K_2^{4}=\langle x_1,\ldots,x_5\mid [x_1,x_2]=x_3,
\res x_4= x_3 \rangle;\\
&K_2^{5}=\langle x_1,\ldots,x_5\mid [x_1,x_2]=x_3,
\res x_5= x_3 \rangle;\\
&K_2^{7}=\langle x_1,\ldots,x_5\mid [x_1,x_2]=x_3,
\res x_1= x_5 \rangle;\\
&K_2^{8}=\langle x_1,\ldots,x_5\mid [x_1,x_2]=x_3,
\res x_2= x_5 \rangle;\\
&K_2^{9}=\langle x_1,\ldots,x_5\mid [x_1,x_2]=x_3,
\res x_3= x_5 \rangle;\\
&K_2^{10}=\langle x_1,\ldots,x_5\mid [x_1,x_2]=x_3,
\res x_4= x_5 \rangle;\\
&K_2^{12}=\langle x_1,\ldots,x_5\mid [x_1,x_2]=x_3,
\res x_1= x_3, \res x_2= x_5 \rangle;\\
&K_2^{13}=\langle x_1,\ldots,x_5\mid [x_1,x_2]=x_3,
\res x_1= x_3, \res x_3=x_5 \rangle;\\
&K_2^{14}=\langle x_1,\ldots,x_5\mid [x_1,x_2]=x_3,
\res x_1= x_3, \res x_4=x_5 \rangle;\\
&K_2^{15}=\langle x_1,\ldots,x_5\mid [x_1,x_2]=x_3,
\res x_1= x_4 \rangle;\\
&K_2^{16}=\langle x_1,\ldots,x_5\mid [x_1,x_2]=x_3,
\res x_1= x_4, \res x_2=x_5 \rangle;\\
&K_2^{17}=\langle x_1,\ldots,x_5\mid [x_1,x_2]=x_3,
\res x_1= x_4, \res x_3=x_5 \rangle;\\
&K_2^{18}=\langle x_1,\ldots,x_5\mid [x_1,x_2]=x_3,
\res x_1= x_4, \res x_4=x_5 \rangle;\\
&K_2^{19}=\langle x_1,\ldots,x_5\mid [x_1,x_2]=x_3,
\res x_1= x_4,\res x_3=x_5, \res x_4=x_5 \rangle;\\
&K_2^{20}=\langle x_1,\ldots,x_5\mid [x_1,x_2]=x_3,
\res x_1= x_3, \res x_2=x_4 \rangle;\\
&K_2^{21}=\langle x_1,\ldots,x_5\mid [x_1,x_2]=x_3,
\res x_1= x_3, \res x_2=x_4,  \res x_3=x_5 \rangle;\\
&K_2^{22}=\langle x_1,\ldots,x_5\mid [x_1,x_2]=x_3,
\res x_1= x_3, \res x_2=x_4,  \res x_4= x_5 \rangle;\\
&K_2^{23}=\langle x_1,\ldots,x_5\mid [x_1,x_2]=x_3,
\res x_3= x_4 \rangle;\\
&K_2^{24}=\langle x_1,\ldots,x_5\mid [x_1,x_2]=x_3,
\res x_1= x_5, \res x_3=x_4 \rangle;\\
&K_2^{25}=\langle x_1,\ldots,x_5\mid [x_1,x_2]=x_3,
\res x_2= x_5, \res x_3=x_4 \rangle;\\
&K_2^{26}=\langle x_1,\ldots,x_5\mid [x_1,x_2]=x_3,
\res x_3= x_4, \res x_4=x_5 \rangle;\\
&K_2^{27}=\langle x_1,\ldots,x_5\mid [x_1,x_2]=x_3,
\res x_2= x_3, \res x_3=x_4 \rangle;\\
&K_2^{28}=\langle x_1,\ldots,x_5\mid [x_1,x_2]=x_3,
\res x_1= x_5, \res x_2=x_3,  \res x_3=x_4 \rangle;\\
&K_2^{29}=\langle x_1,\ldots,x_5\mid [x_1,x_2]=x_3,
\res x_2= x_3, \res x_3=x_4,  \res x_4=x_5 \rangle;\\
&K_2^{31}=\langle x_1,\ldots,x_5\mid [x_1,x_2]=x_3,
\res x_1= x_5, \res x_4=x_3 \rangle;\\
&K_2^{32}=\langle x_1,\ldots,x_5\mid [x_1,x_2]=x_3,
\res x_2= x_5, \res x_4=x_3 \rangle;\\
&K_2^{33}=\langle x_1,\ldots,x_5\mid [x_1,x_2]=x_3,
\res x_3= x_5, \res x_4=x_3 \rangle;\\
&K_2^{34}=\langle x_1,\ldots,x_5\mid [x_1,x_2]=x_3,
\res x_2= x_4, \res x_4=x_3 \rangle;\\
&K_2^{35}=\langle x_1,\ldots,x_5\mid [x_1,x_2]=x_3,
\res x_1= x_5, \res x_2=x_4,  \res x_4=x_3 \rangle;\\
&K_2^{36}=\langle x_1,\ldots,x_5\mid [x_1,x_2]=x_3,
\res x_2= x_4, \res x_3=x_5,  \res x_4=x_3 \rangle.
\end{align*}

Recall that he group $\Aut(L_{5,2})$ consists of invertible matrices of the form 
\[\begin{pmatrix} a_{11} & a_{12} & 0 & 0 & 0 \\
a_{21} & a_{22} & 0& 0 & 0\\
a_{31} & a_{32} & a_{11}a_{22}-a_{12}a_{21} & a_{34} & a_{35} \\
a_{41} & a_{42} & 0& a_{44}
& a_{45} \\
a_{51} & a_{52} & 0& a_{54}
& a_{55}
\end{pmatrix},\]
where $a_{11}a_{22}-a_{12}a_{21} \neq0$.

We have $K_2^{25} \cong K_2^{17}, \quad K_2^{27} \cong K_2^{13},\quad  K_2^{28} \cong K_2^{21}$ by the following automorphism of $L_{5,2}$:
$$\begin{pmatrix} 0& -1 & 0 & 0 & 0 \\
-1 &0& 0& 0&0\\
0 & 0& -1 & 0 & 0 \\
0& 0 & 0& 0
&-1 \\
0 & 0 & 0 & -1
& 0
\end{pmatrix}
$$

Consider the following automorphism of $L_{5,2}$:
$$\begin{pmatrix} 0& -1 & 0 & 0 & 0 \\
-1 &0& 0& 0&0\\
0 & 0& -1 & 0 & 0 \\
0& 0 & 0& -1
&0 \\
0 & 0 & 0 & 0
& -1
\end{pmatrix}.$$
Therefore,
$$K_2^{3} \cong K_2^{2}, \quad K_2^{8} \cong K_2^{7}, \quad K_2^{32} \cong K_2^{31}.$$
Finally, using  the following automorphism of $L_{5,2}$ that only switches $x_4$ and $x_5$

$$
\begin{pmatrix}
 1& 0 & 0 & 0 & 0 \\
0&1& 0& 0&0\\
0 & 0& 1 & 0 & 0 \\
0& 0 & 0& 0
&1 \\
0 & 0 & 0 & 1
& 0
\end{pmatrix},
$$
we get 
$$K_2^{5} \cong K_2^{4}, \quad K_2^{15} \cong K_2^{7}, \quad K_2^{20} \cong K_2^{12}, \quad K_2^{24} \cong K_2^{17}.$$

\begin{theorem}\label{thm L52}
The list of all restricted Lie algebra structures on $L_{5,2}$, up to isomorphism, is as follows:
\begin{align*}
&L_{5,2}^{1}=\langle x_1,\ldots,x_5\mid [x_1,x_2]=x_3 \rangle;\\
&L_{5,2}^{2}=\langle x_1,\ldots,x_5\mid [x_1,x_2]=x_3,
\res x_1= x_3 \rangle;\\
&L_{5,2}^{3}=\langle x_1,\ldots,x_5\mid [x_1,x_2]=x_3,
\res x_4= x_3 \rangle;\\
&L_{5,2}^{4}=\langle x_1,\ldots,x_5\mid [x_1,x_2]=x_3,
\res x_1= x_5 \rangle;\\
&L_{5,2}^{5}=\langle x_1,\ldots,x_5\mid [x_1,x_2]=x_3,
\res x_3= x_5 \rangle;\\
&L_{5,2}^{6}=\langle x_1,\ldots,x_5\mid [x_1,x_2]=x_3,
\res x_4= x_5 \rangle;\\
&L_{5,2}^{7}=\langle x_1,\ldots,x_5\mid [x_1,x_2]=x_3,
\res x_1= x_3, \res x_2= x_5 \rangle;\\
&L_{5,2}^{8}=\langle x_1,\ldots,x_5\mid [x_1,x_2]=x_3,
\res x_1= x_3, \res x_3=x_5 \rangle;\\
&L_{5,2}^{9}=\langle x_1,\ldots,x_5\mid [x_1,x_2]=x_3,
\res x_1= x_3, \res x_4=x_5 \rangle;\\
&L_{5,2}^{10}=\langle x_1,\ldots,x_5\mid [x_1,x_2]=x_3,
\res x_1= x_4, \res x_2=x_5 \rangle;\\
&L_{5,2}^{11}=\langle x_1,\ldots,x_5\mid [x_1,x_2]=x_3,
\res x_1= x_4, \res x_3=x_5 \rangle;\\
&L_{5,2}^{12}=\langle x_1,\ldots,x_5\mid [x_1,x_2]=x_3,
\res x_1= x_4, \res x_4=x_5 \rangle;\\
&L_{5,2}^{13}=\langle x_1,\ldots,x_5\mid [x_1,x_2]=x_3,
\res x_1=  x_5, \res x_2=x_3,  \res x_3=x_4 \rangle;\\
&L_{5,2}^{14}=\langle x_1,\ldots,x_5\mid [x_1,x_2]=x_3,
\res x_1= x_3, \res x_2=x_4,  \res x_4= x_5 \rangle;\\
&L_{5,2}^{15}=\langle x_1,\ldots,x_5\mid [x_1,x_2]=x_3,
\res x_3= x_4, \res x_4=x_5 \rangle;\\
&L_{5,2}^{16}=\langle x_1,\ldots,x_5\mid [x_1,x_2]=x_3,
\res x_2= x_3, \res x_3=x_4,  \res x_4=x_5 \rangle;\\
&L_{5,2}^{17}=\langle x_1,\ldots,x_5\mid [x_1,x_2]=x_3,
\res x_1= x_5, \res x_4=x_3 \rangle;\\
&L_{5,2}^{18}=\langle x_1,\ldots,x_5\mid [x_1,x_2]=x_3,
\res x_3= x_5, \res x_4=x_3 \rangle;\\
&L_{5,2}^{19}=\langle x_1,\ldots,x_5\mid [x_1,x_2]=x_3,
\res x_2= x_4, \res x_4=x_3 \rangle;\\
&L_{5,2}^{20}=\langle x_1,\ldots,x_5\mid [x_1,x_2]=x_3,
\res x_1=  x_5, \res x_2=x_4,  \res x_4=x_3 \rangle;\\
&L_{5,2}^{21}=\langle x_1,\ldots,x_5\mid [x_1,x_2]=x_3,
\res x_2= x_4, \res x_3=x_5,  \res x_4=x_3 \rangle;\\
&L_{5,2}^{22}=\langle x_1,\ldots,x_5\mid [x_1,x_2]=x_3,
\res x_1= x_4,\res x_3=x_5, \res x_4=x_5 \rangle.
\end{align*}
\end{theorem}
In the remaining of this section we establish that the algebras given in Theorem \ref{thm L52} are pairwise non-isomorphic, thereby
completing the proof of Theorem \ref{thm L52}.

It is clear that $L_{5,2}^1$ is not isomorphic to the other restricted Lie algebras.

We claim that $L_{5,2}^2$ and $L_{5,2}^3$ are not isomorphic. Suppose to the contrary that there exists an isomorphism $A: L_{5,2}^3\to L_{5,2}^2$. Then 
\begin{align*}
&A(\res x_4)=\res {A(x_4)}\\
&A(x_3)=\res {(a_{34}x_3+a_{44}x_4+a_{54}x_5)}\\
&(a_{11}a_{22}-a_{12}a_{21})x_3=0.
\end{align*}
 Therefore, $a_{11}a_{22}-a_{12}a_{21}=0$ which is a contradiction.
Note  that $L_{5,2}^2\ncong L_{5,2}^4$ because 
\begin{align}\label{L/Lp}
L_{5,2}^2/(L_{5,2}^2){^{[p]}}\ncong L_{5,2}^4/(L_{5,2}^4){^{[p]}}.
\end{align}
 Similarly, 
 $L_{5,2}^2\ncong L_{5,2}^5$ and $L_{5,2}^2\ncong L_{5,2}^6$. It is clear that $L_{5,2}^2$ is not isomorphic to the other restricted Lie algebras.

Similar argument as in \eqref{L/Lp} shows that $L_{5,2}^3\ncong L_{5,2}^4$, 
$L_{5,2}^3\ncong L_{5,2}^5$, and $L_{5,2}^3\ncong L_{5,2}^6$. It is clear that $L_{5,2}^3$ is not isomorphic to the other restricted Lie algebras.

Next, we claim that $L_{5,2}^4$ and $L_{5,2}^5$ are not isomorphic. Suppose to the contrary that there exists an isomorphism $A: L_{5,2}^4\to L_{5,2}^5$. Then 
\begin{align*}
&A(\res x_3)=\res {A(x_3)}\\
&0=\res {((a_{11}a_{22}-a_{12}a_{21})x_3)}\\
&0=(a_{11}a_{22}-a_{12}a_{21})^px_5.
\end{align*}
Therefore, $a_{11}a_{22}-a_{12}a_{21}=0$ which is a contradiction.

Next, we claim that $L_{5,2}^4$ and $L_{5,2}^6$ are not isomorphic. Suppose to the contrary that there exists an isomorphism $A: L_{5,2}^6\to L_{5,2}^4$. Then we have
\begin{align*}
&A(\res x_1)=\res {A(x_1)}\\
&0=\res {(a_{11}x_1+a_{21}x_2+a_{31}x_3+a_{41}x_4+a_{51}x_5)}\\
&0=a_{11}^px_5. 
\end{align*}
Therefore, $a_{11}=0$.
Also we have
\begin{align*}
&A(\res x_2)=\res {A(x_2)}\\
&0=\res {(a_{12}x_1+a_{22}x_2+a_{32}x_3+a_{42}x_4+a_{52}x_5)}\\
&0=a_{12}^px_5. 
\end{align*}
Therefore, $a_{12}=0$. Hence, $a_{11}a_{22}-a_{12}a_{21}=0$ which is a contradiction.

It is clear that $L_{5,2}^4$ is not isomorphic to the other restricted Lie algebras.

Next, we claim that $L_{5,2}^5$ and $L_{5,2}^6$ are not isomorphic. Suppose to the contrary that there exists an isomorphism $A: L_{5,2}^6\to L_{5,2}^5$. Then 
\begin{align*}
&A(\res x_3)=\res {A(x_3)}\\
&0=\res {((a_{11}a_{22}-a_{12}a_{21})x_3)}\\
&0=(a_{11}a_{22}-a_{12}a_{21})^px_5.
\end{align*}
Therefore, $a_{11}a_{22}-a_{12}a_{21}=0$ which is a contradiction.

It is clear that $L_{5,2}^5$ and $L_{5,2}^6$ are not isomorphic to the other restricted Lie algebras.

Note that $L_{5,2}^7$ and $L_{5,2}^8$ are not isomorphic because 
$(L_{5,2}^7)^{[p]^2}=0$ but $(L_{5,2}^8)^{[p]^2}\neq 0$. Also, $L_{5,2}^7$ and $L_{5,2}^9$ are not isomorphic. Suppose to the contrary that there exists an isomorphism $A: L_{5,2}^9\to L_{5,2}^7$. Then 
\begin{align*}
&A(\res x_2)=\res {A(x_2)}\\
&0=\res {(a_{12}x_1+a_{22}x_2+a_{32}x_3+a_{42}x_4+a_{52}x_5)}\\
&0=a_{12}^px_3+a_{22}^p x_5. 
\end{align*}
Therefore, $a_{12}=0$ and $a_{22}=0$. Hence, $a_{11}a_{22}-a_{12}a_{21}=0$ which is a contradiction.

Similar argument as in \eqref{L/Lp} shows that $L_{5,2}^7\ncong L_{5,2}^{10}$, 
$L_{5,2}^7\ncong L_{5,2}^{11}$,  $L_{5,2}^7\ncong L_{5,2}^{12}$, and 
$L_{5,2}^7  \ncong L_{5,2}^{15}$.
Next, we claim that $L_{5,2}^7  $ and $L_{5,2}^{17}$ are not isomorphic. Suppose to the contrary that there exists an isomorphism $A: L_{5,2}^{17}\to L_{5,2}^7  $. Then 
\begin{align*}
&A(\res x_4)=\res {A(x_4)}\\
&A(x_3)=\res {(a_{34}x_3+a_{44}x_4+a_{54}x_5)}\\
&(a_{11}a_{22}-a_{12}a_{21})x_3=0.
\end{align*}
Therefore, $a_{11}a_{22}-a_{12}a_{21}=0$ which is a contradiction.

Next, we claim that $L_{5,2}^7  $ and $L_{5,2}^{18}$ are not isomorphic. Suppose to the contrary that there exists an isomorphism $A: L_{5,2}^7  \to L_{5,2}^{18}$. Then 
\begin{align*}
&A(\res x_3)=\res {A(x_3)}\\
&0=\res {((a_{11}a_{22}-a_{12}a_{21})x_3)}\\
&0=(a_{11}a_{22}-a_{12}a_{21})^px_5.
\end{align*}
Therefore, $a_{11}a_{22}-a_{12}a_{21}=0$ which is a contradiction.

Note  that $L_{5,2}^7  $ is not isomorphic to any of  $L_{5,2}^{19}$ nor $L_{5,2}^{22}$. Because 
$(L_{5,2}^7  )^{[p]^2}=0$ but this is not the case for  $L_{5,2}^{19}$ and $L_{5,2}^{22}$.

It is clear that $L_{5,2}^7  $ is not isomorphic to the remaining  restricted Lie algebras.

Note  that $L_{5,2}^8$ is not isomorphic to any of  $L_{5,2}^{9}$, $L_{5,2}^{10}$,  $L_{5,2}^{11}$, 
$L_{5,2}^{17}$ , $L_{5,2}^{18}$
 because 
$(L_{5,2}^9)^{[p]^2}=(L_{5,2}^{10})^{[p]^2}=
(L_{5,2}^{11})^{[p]^2}=(L_{5,2}^{17})^{[p]^2}=(L_{5,2}^{18})^{[p]^2}=0$ but 
$(L_{5,2}^8)^{[p]^2}\neq 0$.

Next, we claim that $L_{5,2}^8$ and $L_{5,2}^{19}$ are not isomorphic. Suppose to the contrary that there exists an isomorphism $A: L_{5,2}^{19}\to L_{5,2}^8$. Then 
\begin{align*}
&A(\res x_4)=\res {A(x_4)}\\
&A(x_3)=\res {(a_{34}x_3+a_{44}x_4+a_{54}x_5)}\\
&(a_{11}a_{22}-a_{12}a_{21})x_3=a_{34}^px_5.
\end{align*}
Therefore, $a_{11}a_{22}-a_{12}a_{21}=0$ which is a contradiction.

Next, we claim that $L_{5,2}^8$ and $L_{5,2}^{22}$ are not isomorphic. Suppose to the contrary that there exists an isomorphism $A: L_{5,2}^8\to L_{5,2}^{22}$. Then we have
\begin{align*}
&A(\res x_1)=\res {A(x_1)}\\
&A(x_3)=\res {(a_{11}x_1+a_{21}x_2+a_{31}x_3+a_{41}x_4+a_{51}x_5)}\\
&(a_{11}a_{22}-a_{12}a_{21})x_3=a_{11}^px_4+a_{31}^px_5+a_{41}^px_5. 
\end{align*}
Therefore, $a_{11}a_{22}-a_{12}a_{21}=0$ which is a contradiction.

It is clear that $L_{5,2}^8$ is not isomorphic to the other restricted Lie algebras.

Similar argument as in \eqref{L/Lp} shows that  $L_{5,2}^9$  is not isomorphic to any of 
$L_{5,2}^{10}$, $L_{5,2}^{11}$, $L_{5,2}^{12}$, or $L_{5,2}^{15}$.

Next, we claim that $L_{5,2}^9$ and $L_{5,2}^{17}$ are not isomorphic. Suppose to the contrary that there exists an isomorphism $A: L_{5,2}^{17}\to L_{5,2}^9$. Then 
\begin{align*}
&A(\res x_4)=\res {A(x_4)}\\
&A(x_3)=\res {(a_{34}x_3+a_{44}x_4+a_{54}x_5)}\\
&(a_{11}a_{22}-a_{12}a_{21})x_3=a_{44}^px_5.
\end{align*}
Therefore, $a_{11}a_{22}-a_{12}a_{21}=0$ which is a contradiction.

Note that  $L_{5,2}^9$ is not isomorphic to any of  $L_{5,2}^{18}$, $L_{5,2}^{19}$, and  $L_{5,2}^{22}$. Because  $(L_{5,2}^9)^{[p]^2}=0$ but this is not the case for $L_{5,2}^{18}$, $L_{5,2}^{19}$, and  $L_{5,2}^{22}$.
It is clear that $L_{5,2}^9$ is not isomorphic to the other restricted Lie algebras.

Next, we claim that $L_{5,2}^{10}$ and $L_{5,2}^{11}$ are not isomorphic. Suppose to the contrary that there exists an isomorphism $A: L_{5,2}^{10}\to L_{5,2}^{11}$. Then 
\begin{align*}
&A(\res x_3)=\res {A(x_3)}\\
&0=\res {((a_{11}a_{22}-a_{12}a_{21})x_3)}\\
&0=(a_{11}a_{22}-a_{12}a_{21})^px_5.
\end{align*}
Therefore, $a_{11}a_{22}-a_{12}a_{21}=0$ which is a contradiction.
Note that  $L_{5,2}^{10}$ is not isomorphic to $L_{5,2}^{12}$, $L_{5,2}^{15}$, $L_{5,2}^{18}$, $L_{5,2}^{19}$ nor $L_{5,2}^{22}$ because  $(L_{5,2}^{10})^{[p]^2}=0$ but $(L_{5,2}^{12})^{[p]^2}\neq 0$, $(L_{5,2}^{15})^{[p]^2}\neq 0$,
$(L_{5,2}^{18})^{[p]^2}\neq 0$, $(L_{5,2}^{19})^{[p]^2}\neq 0$ and $(L_{5,2}^{22})^{[p]^2}\neq 0$.

Next, we claim that $L_{5,2}^{10}$ and $L_{5,2}^{17}$ are not isomorphic. Suppose to the contrary that there exists an isomorphism $A: L_{5,2}^{17}\to L_{5,2}^{10}$. Then 
\begin{align*}
&A(\res x_4)=\res {A(x_4)}\\
&A(x_3)=\res {(a_{34}x_3+a_{44}x_4+a_{54}x_5)}\\
&(a_{11}a_{22}-a_{12}a_{21})x_3=0.
\end{align*}
Therefore, $a_{11}a_{22}-a_{12}a_{21}=0$ which is a contradiction.
It is clear that $L_{5,2}^{10}$ is not isomorphic to the other restricted Lie algebras.

Note that  $L_{5,2}^{11}$ is not isomorphic to any of  $L_{5,2}^{12}$, $L_{5,2}^{15}$, $L_{5,2}^{18}$,  $L_{5,2}^{19}$  nor $L_{5,2}^{22}$ because  $(L_{5,2}^{11})^{[p]^2}=0$ but $(L_{5,2}^{12})^{[p]^2}\neq 0$, $(L_{5,2}^{15})^{[p]^2}\neq 0$,
$(L_{5,2}^{18})^{[p]^2}\neq 0$, $(L_{5,2}^{19})^{[p]^2}\neq 0$  and $(L_{5,2}^{22})^{[p]^2}\neq 0$. 

Next, we claim that $L_{5,2}^{11}$ and $L_{5,2}^{17}$ are not isomorphic. Suppose to the contrary that there exists an isomorphism $A: L_{5,2}^{17}\to L_{5,2}^{11}$. Then 
\begin{align*}
&A(\res x_4)=\res {A(x_4)}\\
&A(x_3)=\res {(a_{34}x_3+a_{44}x_4+a_{54}x_5)}\\
&(a_{11}a_{22}-a_{12}a_{21})x_3=a_{34}^px_5.
\end{align*}
Therefore, $a_{11}a_{22}-a_{12}a_{21}=0$ which is a contradiction.

It is clear that $L_{5,2}^{11}$ is not isomorphic to the other restricted Lie algebras.

Next, we claim that $L_{5,2}^{12}$ and $L_{5,2}^{15}$ are not isomorphic. Suppose to the contrary that there exists an isomorphism $A: L_{5,2}^{12}\to L_{5,2}^{15}$. Then 
\begin{align*}
&A(\res x_3)=\res {A(x_3)}\\
&0=\res {((a_{11}a_{22}-a_{12}a_{21})x_3)}\\
&0=(a_{11}a_{22}-a_{12}a_{21})^px_4.
\end{align*}
Therefore, $a_{11}a_{22}-a_{12}a_{21}=0$ which is a contradiction.

Similar argument as in \eqref{L/Lp} shows that  $L_{5,2}^{12}$  is not isomorphic to any of 
to any of   $L_{5,2}^{17}$, $L_{5,2}^{18}$, or  $L_{5,2}^{19}$.

Note that $L_{5,2}^{12}$ is not isomorphic to $L_{5,2}^{22}$. Because,  $(L_{5,2}^{12})'{^{[p]}}=0$, but $(L_{5,2}^{22})'{^{[p]}}\neq 0$.

It is clear that $L_{5,2}^{12}$ is not isomorphic to the other restricted Lie algebras.

Next, we claim that $L_{5,2}^{13}$ and $L_{5,2}^{14}  $ are not isomorphic. Suppose to the contrary that there exists an isomorphism $A: L_{5,2}^{14}  \to L_{5,2}^{13}  $. Then 
\begin{align*}
&A(\res x_3)=\res {A(x_3)}\\
&0=\res {((a_{11}a_{22}-a_{12}a_{21})x_3)}\\
&0=(a_{11}a_{22}-a_{12}a_{21})^px_4.
\end{align*}
Therefore, $a_{11}a_{22}-a_{12}a_{21}=0$ which is a contradiction.

Note that  $L_{5,2}^{13}  $ is not isomorphic to any of    $L_{5,2}^{16}$ or  $L_{5,2}^{21}$  because  $(L_{5,2}^{13}  )^{[p]^3}=0$ but $(L_{5,2}^{16})^{[p]^3}\neq 0$  and $(L_{5,2}^{21})^{[p]^3}\neq 0$. 

Next, we claim that $L_{5,2}^{13}  $ and $L_{5,2}^{20}  $ are not isomorphic. Suppose to the contrary that there exists an isomorphism $A: L_{5,2}^{13}  \to L_{5,2}^{20}  $. Then 
\begin{align*}
&A(\res x_3)=\res {A(x_3)}\\
&A(x_4)=\res {((a_{11}a_{22}-a_{12}a_{21})x_3)}\\
&a_{34}x_3+a_{44}x_4+a_{54}x_5=0.
\end{align*}
Therefore, $a_{44}=0$ which is a contradiction.

It is clear that $L_{5,2}^{13}  $ is not isomorphic to the other restricted Lie algebras.

Note that  $L_{5,2}^{14}  $ is not isomorphic to any of    $L_{5,2}^{16}$ or  $L_{5,2}^{21}$  because  $(L_{5,2}^{14}  )^{[p]^3}=0$ but $(L_{5,2}^{16})^{[p]^3}\neq 0$  and $(L_{5,2}^{21})^{[p]^3}\neq 0$. 

Next, we claim that $L_{5,2}^{14}  $ and $L_{5,2}^{20}  $ are not isomorphic. Suppose to the contrary that there exists an isomorphism $A: L_{5,2}^{14}  \to L_{5,2}^{20}  $. Then 
\begin{align*}
&A(\res x_1)=\res {A(x_1)}\\
&A(x_3)=\res {(a_{11}x_1+a_{21}x_2+a_{31}x_3+a_{41}x_4+a_{51}x_5)}\\
&(a_{11}a_{22}-a_{12}a_{21})x_3=a_{11}^p\alpha x_5+a_{21}^px_4+a_{41}^px_3, 
\end{align*}
which implies that $a_{11}=0$ and $a_{21}=0$. Therefore, $a_{11}a_{22}-a_{12}a_{21}=0$ which is a contradiction.

It is clear that $L_{5,2}^{14}  $ is not isomorphic to the other restricted Lie algebras.

Similar argument as in \eqref{L/Lp} shows that  $L_{5,2}^{15}$  is not isomorphic to any of 
to any of   $L_{5,2}^{17}$, $L_{5,2}^{18}$, or  $L_{5,2}^{19}$.

Next, we claim that $L_{5,2}^{15}$ and $L_{5,2}^{22}$ are not isomorphic. Suppose to the contrary that there exists an isomorphism $A: L_{5,2}^{15}\to L_{5,2}^{22}$. Then 
\begin{align*}
&A(\res x_1)=\res {A(x_1)}\\
&0=\res {(a_{11}x_1+a_{21}x_2+a_{31}x_3+a_{41}x_4+a_{51}x_5)}\\
&0=a_{11}^p x_4+a_{31}^px_5+a_{41}^px_5, 
\end{align*}
which implies that $a_{11}=0$.

Also, we have 
\begin{align*}
&A(\res x_2)=\res {A(x_2)}\\
&0=\res {(a_{12}x_1+a_{22}x_2+a_{32}x_3+a_{42}x_4+a_{52}x_5)}\\
&0=a_{12}^p x_4+a_{32}^px_5+a_{42}^px_5, 
\end{align*}
which implies that $a_{12}=0$. Therefore, $a_{11}a_{22}-a_{12}a_{21}=0$ which is a contradiction.

It is clear that $L_{5,2}^{15}$ is not isomorphic to the other restricted Lie algebras.

Note  that $L_{5,2}^{16}$ and $L_{5,2}^{20}  $ are not isomorphic 
because $(L_{5,2}^{16})^{[p]^3}\neq 0$ but $(L_{5,2}^{20}  )^{[p]^3}= 0$.

Next, we claim that $L_{5,2}^{16}$ and $L_{5,2}^{21}$ are not isomorphic. Suppose to the contrary that there exists an isomorphism $A: L_{5,2}^{21}\to L_{5,2}^{16}$. Then 
\begin{align*}
&A(\res x_4)=\res {A(x_4)}\\
&A(x_3)=\res {(a_{34}x_3+a_{44}x_4+a_{54}x_5)}\\
&(a_{11}a_{22}-a_{12}a_{21})x_3=a_{34}^px_4+a_{44}^px_5.
\end{align*}
Therefore, $a_{11}a_{22}-a_{12}a_{21}=0$ which is a contradiction.
It is clear that $L_{5,2}^{16}$ is not isomorphic to the other restricted Lie algebras.

Note that  $L_{5,2}^{17}$ is not isomorphic to any of  $L_{5,2}^{18}$,  $L_{5,2}^{19}$ nor $L_{5,2}^{22}$ because  $(L_{5,2}^{17})^{[p]^2}=0$ but $(L_{5,2}^{18})^{[p]^2}\neq 0$, $(L_{5,2}^{19})^{[p]^2}\neq 0$, and  $(L_{5,2}^{22})^{[p]^2}\neq 0$. 

It is clear that $L_{5,2}^{17}$ is not isomorphic to the other restricted Lie algebras.

Next, we claim that $L_{5,2}^{18}$ and $L_{5,2}^{19}$ are not isomorphic. Suppose to the contrary that there exists an isomorphism $A: L_{5,2}^{18}\to L_{5,2}^{19}$. Then 
\begin{align*}
&A(\res x_1)=\res {A(x_1)}\\
&0=\res {(a_{11}x_1+a_{21}x_2+a_{31}x_3+a_{41}x_4+a_{51}x_5)}\\
&0=a_{21}^px_4+a_{41}^px_3, 
\end{align*}
which implies that,  $a_{21}=0$.
Also we have
\begin{align*}
&A(\res x_2)=\res {A(x_2)}\\
&0=\res {(a_{12}x_1+a_{22}x_2+a_{32}x_3+a_{42}x_4+a_{52}x_5)}\\
&0=a_{22}^px_4+a_{42}^px_3, 
\end{align*}
which implies that $a_{22}=0$. Therefore, $a_{11}a_{22}-a_{12}a_{21}=0$ which is a contradiction.

Next, we claim that $L_{5,2}^{18}$ and $L_{5,2}^{22}$ are not isomorphic. Suppose to the contrary that there exists an isomorphism $A: L_{5,2}^{18}\to L_{5,2}^{22}$. Then 
\begin{align*}
&A(\res x_4)=\res {A(x_4)}\\
&A(x_3)=\res {(a_{34}x_3+a_{44}x_4+a_{54}x_5)}\\
&(a_{11}a_{22}-a_{12}a_{21})x_3=a_{34}^px_4+a_{44}^px_5.
\end{align*}
Therefore, $a_{11}a_{22}-a_{12}a_{21}=0$ which is a contradiction.

It is clear that $L_{5,2}^{18}$ is not isomorphic to the other restricted Lie algebras.

Note that $L_{5,2}^{19}$ is not isomorphic to $L_{5,2}^{22}$. Because,  $(L_{5,2}^{19})'{^{[p]}}=0$, but $(L_{5,2}^{22})'{^{[p]}}\neq 0$. 
It is clear that $L_{5,2}^{19}$ is not isomorphic to the other restricted Lie algebras.

Finally,  $L_{5,2}^{20}  $ and $L_{5,2}^{21}$ are not isomorphic
because $(L_{5,2}^{20}  )^{[p]^3}= 0$ but $(L_{5,2}^{21 })^{[p]^3}\neq 0$.

%% file: 5,3.tex
\chapter{Restriction maps on $L_{5,3}$}

Let $$
K_3=L_{5,3}=\langle x_1,\ldots,x_5\mid [x_1,x_2]=x_3, [x_1,x_3]=x_4\rangle.$$

Note that $Z(L_{5,3})=\la x_4,x_5\ra_{\F}$ and the group $\Aut(L_{5,3})$ consists of invertible matrices of the form
$$\begin{pmatrix}
 a_{11}  & 0 & 0 & 0&0 \\
a_{21} & a_{22} & 0 & 0&0\\
a_{31} & a_{32} & a_{11}a_{22}& 0&0 \\
a_{41} & a_{42} & a_{11}a_{32}& a_{11}^2a_{22}&a_{45}\\
a_{51}&a_{52}&0&0&a_{55}
\end{pmatrix}.$$

 There exists an element $\alpha x_4+\beta x_5 \in Z(L_{5,3})$ such that 
 $\res {(\alpha x_4+\beta x_5)}=0,$ for some $\alpha ,\beta \in \F$. If $\beta \neq 0$ then consider 
 $$ K=\langle x_1',\ldots,x_5'\mid [x_1',x_2']=x_3', [x_1',x_3']=x_4'\rangle,$$
 where $x_1'=x_1, x_2'=x_2, x_3'=x_3, x_4'=x_4, x_5'=\alpha x_4+\beta x_5$. Let $\phi :K_3\to K$ given by  $x_i\mapsto x_i'$, for $1\leq i\leq 5$. It is easy to see that $\phi$ is an isomorphism. Therefore, in this case we can suppose that $\res x_5=0$.
If $\beta=0$ then $\alpha \neq 0$ and we rescale $x_4$ so that $\res x_4=0$.
Hence we have two cases:
 \begin{enumerate}
\item[I.]  $\res x_4=0$; 
\item[II.] $\res x_5=0$. 
\end{enumerate}

\section{Extensions of  $L=\frac {L_{5,3}}{\la x_4\ra}$}
In this section we find all non-isomorphic p-maps on $L_{5,3}$ such that $\res x_4=0$.  We let
$$L=\frac {L_{5,3}}{\la x_4\ra}\cong L_{4,2},$$
where $L_{4,2}=\la x_1,x_2,x_3,x_4 \mid [x_1,x_2]=x_3\ra $. Note that we denote the image of $x_i$ in $L$ by $x_i$ again. We rename $x_4$ with $x_5$ and $x_5$ with $x_4$ and at the end we will switch them.
The group $\Aut_p(L)$ consists of invertible matrices of the form
$$\begin{pmatrix}
 a_{11}  & a_{12} & 0 & 0 \\
a_{21} & a_{22} & 0 & 0\\
a_{31} & a_{32} & r & a_{34} \\
a_{41} & a_{42} & 0& a_{44}
\end{pmatrix},$$
where $r = a_{11}a_{22}-a_{12}a_{21}\neq 0$. 
\begin{lemma}\label{L53-x4}
Let $K=L_{5,3}$ and $[p]:K\to K$ be a $p$-map on $K$ such that $\res x_4=0$ and let $L=\frac{K}{M}$, where $M=\la x_4\ra_{\F}$.  Then $K\cong L_{\theta}$ where $\theta=(\Delta_{13},\omega)\in Z^2(L,\F)$.
\end{lemma}
\begin{proof}
 Let $\pi :K\rightarrow L$ be the projection map. We have the exact sequence 
$$ 0\rightarrow M\rightarrow K\rightarrow L\rightarrow 0.$$
Let $\sigma :L\rightarrow K$ such that $x_i\mapsto x_i$, $1\leq i\leq 4$. Then $\sigma$ is an injective linear map and $\pi \sigma =1_L$. Now, we define $\phi : L\times L\rightarrow M$ by $\phi (x_i,x_j)=[\sigma (x_i),\sigma (x_j)]-\sigma ([x_i,x_j])$, $1\leq i,j\leq 4$ and $\omega: L\rightarrow M$ by $\omega (x)=\res {\sigma (x)} -\sigma (\res x)$. Note that
\begin{align*}
&\phi(x_1,x_3)=[\sigma(x_1),\sigma(x_3)]-\sigma([x_1,x_3])=[x_1,x_3]=x_4;\\
&\phi(x_1,x_2)=[\sigma(x_1),\sigma(x_2)]-\sigma([x_1,x_2])=0.
\end{align*}
Similarly, we can show that $\phi(x_1,x_4)=\phi(x_2,x_3)=\phi(x_2,x_4)=\phi(x_3,x_4)=0$. Therefore, $\phi=\Delta_{13}$.
Now, by Lemma \ref{K=L-theta}, we have $\theta=(\Delta_{13},\omega)\in Z^2(L,\F)$ and $K\cong L_{\theta}$.
\end{proof}

Note that by Theorem \ref{4-dim}, there are eight non-isomorphic restricted Lie algebra structures on $L$ given by the following $p$-maps:
\begin{enumerate}
\item[I.1]  Trivial $p$-map;
\item[I.2]$\res x_1=x_3$; 
\item[I.3] $\res x_1 =x_4 $;
\item[I.4] $ \res x_1=x_3, \res x_2=x_4$;
\item[I.5] $\res x_3=x_4$;
\item[I.6] $\res x_3=x_4, \res x_2=x_3$;
\item[I.7] $\res x_4=x_3$;
\item[I.8] $\res x_4=x_3, \res x_2=x_4$.
\end{enumerate}

We make $L$ into a restricted Lie algebra by equipping it with each of the above $p$-maps. Then,  in  each  case,  we  find all possible orbit representatives of the form $(\Delta_{13},\omega)$ under the action of $\Aut_p(L)$ on $H^2(L,\F)$. By  Lemma \ref{L53-x4}, we do get all possible $p$-maps on $K_3$ with the property that $x_4^{[p]}=0$. 

Consider case I.2 where the $p$-map of $L$ is $ \res x_1=x_3$. 
Let $[(\phi,\omega)]\in H^2(L,\F)$. Then we must have $\phi (x,\res y)=0$, for all $x,y\in L$, where $\phi =a\Delta_{13}+b\Delta_{14}+c\Delta_{23}+d\Delta_{24}$, for some $a,b,c,d \in \F$.  Since, $\res L=\la x_3\ra$ we get $\phi (x, x_3)=0$, for all $x\in L$. Therefore, $\phi (x_1,x_3)= 0$ which implies that $a=0$. Since $\phi=\Delta_{13}$ gives us $L_{5,3}$, we deduce by Lemma \ref{L53-x4} that  $L_{5,3}$ cannot be constructed  in this case. Similarly, we can show that in cases I.4, I.6, I.7, and I.8 we also get $a=0$.
In the follwoing subsections, we consider the remaining cases.

\subsection{Extensions of ($L$,  trivial $p$-map)}\label{L_{5,3} mod x4 trivial pmap}
 First, we find a basis for $Z^2(L,\F)$. Let $(\phi,\omega)=(a\Delta _{12}+b\Delta_{13}+c\Delta_{14}+d\Delta_{23}+e\Delta_{24}+f\Delta_{34},\alpha f_1+\beta f_2+ \gamma f_3+\delta f_4)\in Z^2(L,\F)$. Then we must have $\delta^2\phi(x,y,z) =0$ and $\phi(x,\res y)=0$, for all $x,y,z \in L$. Therefore,
\begin{align*}
0=(\delta^2\phi)(x_1,x_2,x_4)&=\phi([x_1,x_2],x_4)+\phi([x_2,x_4],x_1)+\phi([x_4,x_1],x_2)=\phi(x_3,x_4).
\end{align*}
Thus, we get $f=0$.
Since the $p$-map is trivial, $\phi(x,\res y)=\phi(x,0)=0$, for all $x,y\in L$. Therefore, a basis for $Z^2(L,\F)$ is as follows:
$$
 (\Delta_{12},0),(\Delta_{13},0),(\Delta_{14},0),(\Delta_{23},0),(\Delta_{24},0),(0,f_1),
(0,f_2),(0,f_3),(0,f_4).
$$
Next, we find a basis for $B^2(L,\F)$. Let $(\phi,\omega)\in B^2(L,\F)$. Since $B^2(L,\F)\subseteq Z^2(L,\F)$, we have $(\phi,\omega)=(a\Delta _{12}+b\Delta_{13}+c\Delta_{14}+d\Delta_{23}+e\Delta_{24},\alpha f_1+\beta f_2+ \gamma f_3+\delta f_4)$. So, there exists a linear map $\psi:L\to \F$ such that $\delta^1\psi(x,y)=\phi(x,y)$ and $\tilde \psi(x)=\omega(x)$, for all $x,y \in L$. So, we have
\begin{align*}
b=\phi(x_1,x_3)=\delta^1\psi(x_1,x_3)=\psi([x_1,x_3])=0.
\end{align*}
Similarly, we can show that $c=d=e=0$. Also, we have
\begin{align*}
\alpha=\omega(x_1)=\tilde \psi(x_1)=\psi(\res x_1)=0.
\end{align*}
Similarly, we can show that $\beta=\gamma=\delta=0$. Therefore, $(\phi,\omega)=(a\Delta_{12},0)$ and hence
 $B^2(L,\F)=\la(\Delta_{12},0)\ra_{\F}$. We deduce that a basis for $H^2(L,\F)$ is as follows:
$$
[(\Delta_{13},0)],[(\Delta_{14},0)],[(\Delta_{23},0)],[(\Delta_{24},0)],[(0,f_1)],
[(0,f_2)],[(0,f_3)],[(0,f_4)].
$$
Let $[(\phi,\omega)] \in H^2(L,\F)$. Then, we have $\phi = a\Delta_{13} + b\Delta_{14}+c\Delta_{23}+d\Delta_{24}$, for some $a,b,c,d \in \F$. Suppose that $A\phi = a'\Delta_{13}+b'\Delta_{14}+c'\Delta_{23}+d'\Delta_{24}$, for some $a',b',c',d' \in \F$. We can verify that the action of $\Aut_p(L)$ on the set of $\phi$'s in the matrix form is as follows:
\begin{align}\label{ab}
\begin{pmatrix}
a' \\ b'\\c'\\d'
\end{pmatrix}=
\begin{pmatrix}
ra_{11} & 0& r a_{21}&0\\
a_{11}a_{34}& a_{11}a_{44}&a_{21}a_{34}&a_{21}a_{44}\\
r a_{12}&0&r a_{22}&0\\
a_{12}a_{34}&a_{12}a_{44}&a_{22}a_{34}&a_{22}a_{44}
\end{pmatrix}
\begin{pmatrix}
a \\ b\\c\\d
\end{pmatrix}.
\end{align}
 The orbit with representative 
$\begin{pmatrix}
1\\ 0\\0\\0
\end{pmatrix}$ of this action  gives us $L_{5,3}$. 

Also, we have $\omega=\alpha f_1+\beta f_2+\gamma f_3+\delta f_4$, for some $\alpha ,\beta ,\gamma ,\delta \in \F$. Suppose that $A\omega = \alpha' f_1+\beta' f_2+\gamma' f_3+\delta' f_4$, for some $\alpha',\beta',\gamma',\delta' \in \F$. Then we have
\begin{align*}
A\omega (x_1) = &\omega (Ax_1)=\omega (a_{11}x_1+a_{21}x_2+a_{31}x_3+a_{41}x_4)=a_{11}^p \alpha +a_{21}^p\beta + a_{31}^p \gamma+a_{41}^p \delta;\\
A\omega (x_2)=&a_{12}^p \alpha+a_{22}^p\beta +a_{32}^p \gamma + a_{42}^p \delta;\\
A\omega (x_3)= &r ^p \gamma;\\
A\omega(x_4)=&a_{34}^p \gamma +a_{44}^p \delta.
\end{align*}
In the matrix form we can write this as
\begin{align}\label{cd}
\begin{pmatrix}
\alpha' \\ 
\beta'\\
\gamma'\\
\delta'
\end{pmatrix}=
\begin{pmatrix}
a_{11}^p& a_{21}^p&a_{31}^p&a_{41}^p\\
a_{12}^p& a_{22}^p& a_{32}^p&a_{42}^p\\
0&0&r ^p&0\\
0&0&a_{34}^p&a_{44}^p
\end{pmatrix}
\begin{pmatrix}
\alpha\\ 
\beta \\
\gamma \\
\delta
\end{pmatrix}.
\end{align}
Thus, we can write Equations \eqref{ab} and \eqref{cd} together as follows:
\begin{align*}
\small{\bigg[\begin{pmatrix}
 ra_{11} & 0&  ra_{21}&0\\
a_{11}a_{34}& a_{11}a_{44}&a_{21}a_{34}&a_{21}a_{44}\\
r a_{12}&0&r a_{22}&0\\
a_{12}a_{34}&a_{12}a_{44}&a_{22}a_{34}&a_{22}a_{44}
\end{pmatrix},
\begin{pmatrix}
a_{11}^p& a_{21}^p&a_{31}^p&a_{41}^p\\
a_{12}^p& a_{22}^p& a_{32}^p&a_{42}^p\\
0&0&r ^p&0\\
0&0&a_{34}^p&a_{44}^p
\end{pmatrix}\bigg]
\bigg[\begin{pmatrix}
a\\b\\c\\d
\end{pmatrix},
\begin{pmatrix}
\alpha\\ 
\beta \\
\gamma \\
\delta
\end{pmatrix}\bigg]=
\bigg[\begin{pmatrix}
a' \\ b'\\c'\\d'
\end{pmatrix},
\begin{pmatrix}
\alpha' \\ 
\beta'\\
\gamma'\\
\delta'
\end{pmatrix}\bigg]}.
\end{align*}
Now we find the representatives of the orbits of the action of $\Aut (L)$ on the set of $\omega$'s such that 
the orbit represented by
 $\begin{pmatrix}
1\\ 0\\0\\0
\end{pmatrix}$
 is preserved under the action of $\Aut (L)$ on the set of $\phi$'s.

Let $\nu = \begin{pmatrix}
\alpha\\ 
\beta \\
\gamma \\
\delta
\end{pmatrix} \in \F^4$. If $\nu= \begin{pmatrix}
0\\ 
0 \\
0 \\
0
\end{pmatrix}$, then $\{\nu \}$ is clearly an $\Aut_p(L)$-orbit. Let $\nu \neq 0$. Suppose that $\gamma \neq 0$. Then 
\begin{align*}
&\bigg[\begin{pmatrix}
1& 0&0&0\\
0& 1&0&0\\
0&0&1&0\\
0&0&0&1
\end{pmatrix},
\begin{pmatrix}
1& 0&-\alpha /\gamma&0\\
0& 1& -\beta /\gamma&0\\
0&0&1&0\\
0&0&0&1
\end{pmatrix}\bigg]
\bigg[\begin{pmatrix}
1 \\ 0\\0\\0
\end{pmatrix},
\begin{pmatrix}
\alpha\\ 
\beta \\
\gamma \\
\delta
\end{pmatrix}\bigg]=
\bigg[\begin{pmatrix}
1\\0\\0\\0
\end{pmatrix},
\begin{pmatrix}
0 \\ 
0\\
\gamma\\
\delta
\end{pmatrix}\bigg], \text{ and }\\
&\footnotesize{\bigg[\begin{pmatrix}
1& 0&0&0\\
0& \gamma ^{1/p}&0&0\\
0&0&\gamma ^{-3/p}&0\\
0&0&0&\gamma ^{-2/p}
\end{pmatrix},
\begin{pmatrix}
\gamma& 0&0&0\\
0& \gamma ^{-2}& 0&0\\
0&0&1/\gamma&0\\
0&0&0&1
\end{pmatrix}\bigg]
\bigg[\begin{pmatrix}
1\\ 0\\0\\0
\end{pmatrix},
\begin{pmatrix}
0\\ 
0 \\
\gamma \\
\delta
\end{pmatrix}\bigg]=
\bigg[\begin{pmatrix}
1\\0\\0\\0
\end{pmatrix},
\begin{pmatrix}
0 \\ 
0\\
1\\
\delta
\end{pmatrix}\bigg]}.
\end{align*}
Next, if $\delta\neq 0$, then
\begin{align*}
\bigg[\begin{pmatrix}
1& 0&0&0\\
0& (1/\delta )^{1/p}&0&0\\
0&0&1&0\\
0&0&0&(1/\delta )^{1/p}
\end{pmatrix},
\begin{pmatrix}
1& 0&0&0\\
0& 1& 0&0\\
0&0&1&0\\
0&0&0&1/\delta
\end{pmatrix}\bigg]
\bigg[\begin{pmatrix}
1\\ 0\\0\\0
\end{pmatrix},
\begin{pmatrix}
0\\ 
0 \\
1 \\
\delta
\end{pmatrix}\bigg]=
\bigg[\begin{pmatrix}
1\\0\\0\\0
\end{pmatrix},
\begin{pmatrix}
0 \\ 
0\\
1\\
1
\end{pmatrix}\bigg].
\end{align*}
 
If $\delta =0$, then we have 
\bigg[$\begin{pmatrix}
1\\0\\0\\0
\end{pmatrix},
\begin{pmatrix}
0 \\ 
0\\
1\\
0
\end{pmatrix}\bigg].$\\

Next, if $\gamma =0$, but $\delta \neq 0$, then\\
\begin{align*}
&\bigg[\begin{pmatrix}
1& 0&0&0\\
0& 1&0&0\\
0&0&1&0\\
0&0&0&1
\end{pmatrix},
\begin{pmatrix}
1& 0&0&-\alpha /\delta\\
0& 1& 0&-\beta/\delta\\
0&0&1&0\\
0&0&0&1
\end{pmatrix}\bigg]
\bigg[\begin{pmatrix}
1\\ 0\\0\\0
\end{pmatrix},
\begin{pmatrix}
\alpha\\ 
\beta\\
0 \\
\delta
\end{pmatrix}\bigg]=
\bigg[\begin{pmatrix}
1\\0\\0\\0
\end{pmatrix},
\begin{pmatrix}
0 \\ 
0\\
0\\
\delta
\end{pmatrix}\bigg], \text{ and }\\
&\bigg[\begin{pmatrix}
1& 0&0&0\\
0& (1/\delta )^{1/p}&0&0\\
0&0&1&0\\
0&0&0&(1/\delta )^{1/p}
\end{pmatrix},
\begin{pmatrix}
1& 0&0&0\\
0& 1& 0&0\\
0&0&1&0\\
0&0&0&1/\delta
\end{pmatrix}\bigg]
\bigg[\begin{pmatrix}
1\\ 0\\0\\0
\end{pmatrix},
\begin{pmatrix}
0\\ 
0\\
0 \\
\delta
\end{pmatrix}\bigg]=
\bigg[\begin{pmatrix}
1\\0\\0\\0
\end{pmatrix},
\begin{pmatrix}
0 \\ 
0\\
0\\
1
\end{pmatrix}\bigg].
\end{align*}
Next, if $\gamma =\delta =0$, but $\beta \neq 0$, then
\begin{align*}
\footnotesize{\bigg[\begin{pmatrix}
1& 0&(-\alpha /\beta )^{1/p}&0\\
0& 1&0&(-\alpha /\beta )^{1/p}\\
0&0&1&0\\
0&0&0&1
\end{pmatrix},
\begin{pmatrix}
1& -\alpha /\beta&0&0\\
0& 1& 0&0\\
0&0&1&0\\
0&0&0&1
\end{pmatrix}\bigg]
\bigg[\begin{pmatrix}
1\\ 0\\0\\0
\end{pmatrix},
\begin{pmatrix}
\alpha\\ 
\beta\\
0 \\
0
\end{pmatrix}\bigg]=
\bigg[\begin{pmatrix}
1\\0\\0\\0
\end{pmatrix},
\begin{pmatrix}
0 \\ 
\beta\\
0\\
0
\end{pmatrix}\bigg]}.
\end{align*}
Finally, if $\gamma =\delta =\beta =0$, but $\alpha \neq 0$, then
\begin{align*}
\bigg[\begin{pmatrix}
1& 0&0&0\\
0& \alpha ^{-1/p}&0&0\\
0&0&\alpha ^{3/p}&0\\
0&0&0&\alpha ^{2/p}
\end{pmatrix},
\begin{pmatrix}
1/\alpha& 0&0&0\\
0& \alpha ^{2}& 0&0\\
0&0&\alpha&0\\
0&0&0&1
\end{pmatrix}\bigg]
\bigg[\begin{pmatrix}
1\\ 0\\0\\0
\end{pmatrix},
\begin{pmatrix}
\alpha\\ 
0\\
0 \\
0
\end{pmatrix}\bigg]=
\bigg[\begin{pmatrix}
1\\0\\0\\0
\end{pmatrix},
\begin{pmatrix}
1 \\ 
0\\
0\\
0
\end{pmatrix}\bigg].
\end{align*}
Thus the following elements are $\Aut (L)$-orbit representatives:
\begin{align*}
\begin{pmatrix}
0\\ 
0\\
0\\
0
\end{pmatrix},
\begin{pmatrix}
0 \\ 
0\\
1\\
1
\end{pmatrix},
\begin{pmatrix}
0 \\ 
0\\
1\\
0
\end{pmatrix},
\begin{pmatrix}
0\\ 
0\\
0\\
1
\end{pmatrix},
\begin{pmatrix}
0 \\ 
\beta\\
0\\
0
\end{pmatrix},
\begin{pmatrix}
1 \\ 
0\\
0\\
0
\end{pmatrix}.
\end{align*}

Therefore, the corresponding restricted Lie algebra structures are as follows: (Note that we need to switch $x_4$ and $x_5$.)
\begin{align*}
&K_3^{1}=\langle x_1,\ldots,x_5\mid [x_1,x_2]=x_3,[x_1,x_3]=x_4 \rangle;\\
&K_3^{2}=\langle x_1,\ldots,x_5\mid [x_1,x_2]=x_3,[x_1,x_3]=x_4, 
 \res x_3=x_4, \res x_5=x_4 \rangle;\\
&K_3^{3}=\langle x_1,\ldots,x_5\mid [x_1,x_2]=x_3,[x_1,x_3]=x_4, 
 \res x_3=x_4 \rangle;\\
&K_3^{4}(\beta)=\langle x_1,\ldots,x_5\mid [x_1,x_2]=x_3,[x_1,x_3]=x_4,
 \res x_2=\beta x_4 \rangle;\\
&K_3^{5}=\langle x_1,\ldots,x_5\mid [x_1,x_2]=x_3,[x_1,x_3]=x_4, 
 \res x_1=x_4 \rangle;\\
&K_3^{6}=\langle x_1,\ldots,x_5\mid [x_1,x_2]=x_3,[x_1,x_3]=x_4, 
 \res x_5=x_4 \rangle.
\end{align*}

\begin{lemma}\label{lemma-K34-2}
We have $K_3^4(\beta_1) \cong K_3^4(\beta_2)$ if and only if $\beta_1 \beta_2^{-1} \in (\F^*)^2$. 
\end{lemma}
\begin{proof}
Suppose that $f=(a_{ij}): K_3^4(\beta_1) \to   K_3^4(\beta_2)$ is an isomorphism. Then, we have 
$f(\res x_2)=f(x_2)^{[p]}$ which in turn implies that $\beta_1 a_{11}^2a_{22} x_4=\beta_2 a_{22}^p x_4$. Hence, $\beta_1/\beta_2=a_{22}^{p-1}a_{11}^{-2}\in (\F^*)^2$. To prove the converse, suppose that $\beta_1/\beta_2=\epsilon^2\in (\F^*)^2$. Then the following is an isomorphism from $K_3^4(\beta_1)$ to $K_3^4(\beta_2)$:
\begin{align*}
\begin{pmatrix}
\epsilon^{-1/p}&0&0&0 & 0\\
0& \epsilon^{2/p}& 0&0& 0\\
0&0&\epsilon^{1/p} &0& 0\\
0&0&0&1& 0\\
0&0&0&0& 1\\
\end{pmatrix}.
\end{align*}

\end{proof}

\subsection{Extensions of $(L, \res x_1=x_4)$}
 First, we find a basis for $Z^2(L,\F)$. Let $(\phi,\omega)=(a\Delta _{12}+b\Delta_{13}+c\Delta_{14}+d\Delta_{23}+e\Delta_{24}+f\Delta_{34},\alpha f_1+\beta f_2+ \gamma f_3+\delta f_4)\in Z^2(L,\F)$. Then we must have $\delta^2\phi(x,y,z) =0$ and $\phi(x,\res y)=0$, for all $x,y,z \in L$. Therefore, we have
\begin{align*}
0=(\delta^2\phi)(x_1,x_2,x_4)&=\phi([x_1,x_2],x_4)+\phi([x_2,x_4],x_1)+\phi([x_4,x_1],x_2)=\phi(x_3,x_4)=f.
\end{align*}
Also, we have $\phi(x,\res y)=0$. Therefore, $\phi(x,x_4)=0$, for all $x\in L$ and hence $\phi(x_1,x_4)=\phi(x_2,x_4)=\phi(x_3,x_4)=0$ which implies that $c=e=f=0$. Therefore, $Z^2(L,\F)$ has a basis consisting of:
$$ (\Delta_{12},0),(\Delta_{13},0),(\Delta_{23},0),(0,f_1),
(0,f_2),(0,f_3),(0,f_4).$$
Next, we find a basis for $B^2(L,\F)$. Let $(\phi,\omega)\in B^2(L,\F)$. Since $B^2(L,\F)\subseteq Z^2(L,\F)$, we have $(\phi,\omega)=(a\Delta _{12}+b\Delta_{13}+c\Delta_{23},\alpha f_1+\beta f_2+ \gamma f_3+\delta f_4)$. Note that there exists a linear map $\psi:L\to \F$  such that $\delta^1\psi(x,y)=\phi(x,y)$ and $\tilde \psi(x)=\omega(x)$, for all $x,y \in L$. So, we have
\begin{align*}
&a=\phi(x_1,x_2)=\delta^1\psi(x_1,x_2)=\psi([x_1,x_2])=\psi(x_3), \text{ and }\\
&b=\phi(x_1,x_3)=\delta^1\psi(x_1,x_3)=\psi([x_1,x_3])=0.
\end{align*}
Similarly, we can show that $c=0$. Also, we have
\begin{align*}
&\alpha=\omega(x_1)=\tilde \psi(x_1)=\psi(\res x_1)=\psi(x_4), \text{ and }\\
&\beta=\omega(x_2)=\tilde \psi(x_2)=\psi(\res x_2)=0.
\end{align*}
Similarly, we can show that $\gamma=\delta=0$. Therefore, $(\phi,\omega)=(a\Delta_{12},\alpha f_1)$ and hence $B^2(L,\F)=\la(\Delta_{12},0),(0,f_1)\ra_{\F}$. We deduce that a basis for $H^2(L,\F)$ is as follows:
$$[(\Delta_{13},0)],[(\Delta_{23},0)],[(0,f_2)],[(0,f_3)],[(0,f_4)].$$

Note that the group $\Aut_p(L)$ in this case consists of invertible matrices of the form
$$\begin{pmatrix}
 a_{11}  & a_{12} & 0 & 0 \\
a_{21} & a_{22} & 0 & 0\\
a_{31} & a_{32} & r & a_{34} \\
a_{41} & a_{42} & 0& a_{44}
\end{pmatrix},$$
where $r = a_{11}a_{22}-a_{12}a_{21}$, $a_{12}=a_{34}=0$, and $a_{44}=a_{11}^p$.

Let $[(\phi,\omega)] \in H^2(L,\F)$. Then, we have $\phi = a\Delta_{13}+c\Delta_{23}$, for some $a,c \in \F$. Suppose that $A\phi=a'\Delta_{13}+c'\Delta_{23}$, for some $a',c' \in \F$. We determine $a',c'$. Note that
\begin{align*}
A\phi (x_1,x_3)=\phi (Ax_1,Ax_3)&=\phi (a_{11}x_1+a_{21}x_2+a_{31}x_3+a_{41}x_4,r x_3)\\
&=a_{11}r a+a_{21}r c; \text{ and }\\ 
A\phi (x_2,x_3)=\phi (Ax_2,Ax_3)&=\phi (a_{12}x_1+a_{22}x_2+a_{32}x_3+a_{42}x_4,r x_3)\\
&=a_{12}r a+a_{22}r c.
\end{align*}
In the matrix form we can write this as
\begin{align}\label{ef}
\begin{pmatrix}
a'\\c'
\end{pmatrix}=
\begin{pmatrix}
a_{11}r &a_{21}r\\
0&a_{22}r
\end{pmatrix}
\begin{pmatrix}
a\\c
\end{pmatrix}.
\end{align}
The orbit with representative 
$\begin{pmatrix}
1\\ 0
\end{pmatrix}$ of this action  gives us $L_{5,3}$.\\
Note that we need to have $A\phi(x_1,x_4)=A\phi(x_2,x_4)=A\phi(x_3,x_4)=0$ which implies that
\begin{align*}
&a_{11}a_{34}a+a_{21}a_{34}c=0,\\
&a_{12}a_{34}a+a_{22}a_{34}c=0.
\end{align*}
Also, we have $\omega=\beta f_2+\gamma f_3+\delta f_4$, for some $\beta,\gamma,\delta \in \F$. Suppose that $A\omega =\beta' f_2+\gamma' f_3+\delta' f_4$, for some $\beta',\gamma',\delta' \in \F$. We have 
\begin{align*}
A\omega (x_2)=&a_{22}^p\beta +a_{32}^p \gamma + a_{42}^p \delta;\\
A\omega (x_3)=& r^p \gamma;\\ 
A\omega(x_4)=&a_{34}^p \gamma +a_{44}^p \delta.
\end{align*}
In the matrix form we can write this as
\begin{align}\label{gh}
\begin{pmatrix}
\beta'\\
\gamma'\\
\delta'
\end{pmatrix}=
\begin{pmatrix}
a_{22}^p& a_{32}^p&a_{42}^p\\
0&r ^p&0\\
0&0&a_{11}^{p^2}
\end{pmatrix}
\begin{pmatrix}
\beta \\
\gamma \\
\delta
\end{pmatrix}.
\end{align}
Thus, we can write Equations \eqref{ef} and \eqref{gh} together as follows:
\begin{align*}
\bigg[r \begin{pmatrix}
a_{11}&a_{21}\\
0&a_{22}
\end{pmatrix},
\begin{pmatrix}
a_{22}^p& a_{32}^p&a_{42}^p\\
0&r^p&0\\
0&0&a_{11}^{p^2}
\end{pmatrix}\bigg]
\bigg[\begin{pmatrix}
a\\c
\end{pmatrix},
\begin{pmatrix}
\beta \\
\gamma \\
\delta
\end{pmatrix}\bigg]=
\bigg[\begin{pmatrix}
a' \\ c'
\end{pmatrix},
\begin{pmatrix}
\beta'\\
\gamma'\\
\delta'
\end{pmatrix}\bigg].
\end{align*}
Now we find the representatives of the orbits of the action of $\Aut (L)$ on the set of $\omega$'s such that 
the orbit represented by
$\begin{pmatrix}
1\\ 0
\end{pmatrix}$
 is preserved under the action of $\Aut (L)$ on the set of$\omega$'s and 
\begin{align*}
&a_{11}a_{34}a+a_{21}a_{34}c=0,\\
&a_{12}a_{34}a+a_{22}a_{34}c=0.
\end{align*}
Let $\nu = \begin{pmatrix}
\beta \\
\gamma \\
\delta
\end{pmatrix} \in \F^3$. If $\nu= \begin{pmatrix} 
0 \\
0 \\
0
\end{pmatrix}$, then $\{\nu \}$ is clearly an $\Aut_p(L)$-orbit. Suppose that $\delta\neq 0$. Then 
\begin{align*}
&\bigg[\begin{pmatrix}
1& 0\\
0&1
\end{pmatrix},
\begin{pmatrix}
1& 0&-\beta/\delta\\
0&1&0\\
0&0&1
\end{pmatrix}\bigg]
\bigg[\begin{pmatrix}
1 \\ 0
\end{pmatrix},
\begin{pmatrix} 
\beta \\
\gamma \\
\delta
\end{pmatrix}\bigg]=
\bigg[\begin{pmatrix}
1\\0
\end{pmatrix},
\begin{pmatrix}
0\\
\gamma\\
\delta
\end{pmatrix}\bigg], \text{ and }\\
&\bigg[\delta^{1/p^2}\begin{pmatrix}
\delta^{-1/p^2}& 0\\
0&\delta^{2/p^2}
\end{pmatrix},
\begin{pmatrix}
\delta^{2/p}&0&0\\
0&\delta^{1/p}&0\\
0&0&1/\delta
\end{pmatrix}\bigg]
\bigg[\begin{pmatrix}
1\\ 0
\end{pmatrix},
\begin{pmatrix}
0 \\
\gamma\\
\delta
\end{pmatrix}\bigg]=
\bigg[\begin{pmatrix}
1\\0
\end{pmatrix},
\begin{pmatrix}
0\\
\delta^{1/p}\gamma\\
1
\end{pmatrix}\bigg].
\end{align*}
So, if $\delta \neq 0$ and $\gamma \neq 0$, then we have 
$\begin{pmatrix}
0\\
\gamma\\
1
\end{pmatrix}$.

If $\delta \neq 0$ but $\gamma=0$, then we have 
$\bigg[\begin{pmatrix}
1\\0
\end{pmatrix},
\begin{pmatrix}
0\\
0\\
1
\end{pmatrix}\bigg].$
Next, if $\delta=0$ but $\gamma \neq 0$, then 
\begin{align*}
&\bigg[\begin{pmatrix}
1& 0\\
0&1
\end{pmatrix},
\begin{pmatrix}
1& -\beta/\gamma&0\\
0&1&0\\
0&0&1
\end{pmatrix}\bigg]
\bigg[\begin{pmatrix}
1 \\ 0
\end{pmatrix},
\begin{pmatrix} 
\beta \\
\gamma \\
0
\end{pmatrix}\bigg]=
\bigg[\begin{pmatrix}
1\\0
\end{pmatrix},
\begin{pmatrix}
0\\
\gamma\\
0
\end{pmatrix}\bigg], \text{ and }\\
&\bigg[(1/\gamma)^{1/p} \begin{pmatrix}
\gamma^{1/p}& 0\\
0&\gamma^{-2/p}\\
\end{pmatrix},
\begin{pmatrix}
(1/\gamma )^2& 0&0\\
0& 1/\gamma& 0\\
0&0&\gamma^p\\
\end{pmatrix}\bigg]
\bigg[\begin{pmatrix}
1\\ 0
\end{pmatrix},
\begin{pmatrix}
0\\ 
\gamma \\
0\\
\end{pmatrix}\bigg]=
\bigg[\begin{pmatrix}
1\\0
\end{pmatrix},
\begin{pmatrix}
0\\
1\\
0
\end{pmatrix}\bigg].
\end{align*}
Next, if $\delta=\gamma=0$, then we have
$\bigg[\begin{pmatrix}
1\\0
\end{pmatrix},
\begin{pmatrix}
\beta\\
0\\
0
\end{pmatrix}\bigg].$

 Thus the following elements are $\Aut (L)$-orbit representatives:
 \begin{align*}
\begin{pmatrix}
0\\
0\\
0
\end{pmatrix},
\begin{pmatrix}
0\\
1\\
0
\end{pmatrix},
\begin{pmatrix} 
0\\
0\\
1
\end{pmatrix},
\begin{pmatrix}
\beta\\
0\\
0
\end{pmatrix},
\begin{pmatrix} 
0\\
\gamma\\
1
\end{pmatrix}.
\end{align*}

Therefore, the corresponding restricted Lie algebra structures are as follows: (Note that we need to switch $x_4$ and $x_5.$)
\begin{align*}
&K_3^{7}=\langle x_1,\ldots,x_5\mid [x_1,x_2]=x_3,[x_1,x_3]=x_4 ,
\res x_1=x_5 \rangle;\\
&K_3^{8}(\beta)=\langle x_1,\ldots,x_5\mid [x_1,x_2]=x_3,[x_1,x_3]=x_4, 
\res x_1=x_5, \res x_2=\beta x_4 \rangle;\\
&K_3^{9}=\langle x_1,\ldots,x_5\mid [x_1,x_2]=x_3,[x_1,x_3]=x_4, 
\res x_1=x_5, \res x_3=x_4 \rangle;\\
&K_3^{10}=\langle x_1,\ldots,x_5\mid [x_1,x_2]=x_3,[x_1,x_3]=x_4,
\res x_1=x_5, \res x_5=x_4 \rangle;\\
&K_3^{11}(\gamma)=\langle x_1,\ldots,x_5\mid [x_1,x_2]=x_3,[x_1,x_3]=x_4,
\res x_1=x_5, \res x_3=\gamma x_4 ,\res x_5=x_4 \rangle
\end{align*}

\begin{lemma}
The restricted Lie algebras $K_3^{8}(\beta_1)$ and $K_3^{8}(\beta_2)$ are isomorphic if and only if  $\beta_2\beta_1^{-1}\in (\F^*)^2$. 
\end{lemma}
\begin{proof}
First assume that $f=(a_{ij}):  K_3^{8}(\beta_1)\to K_3^{8}(\beta_2)$ is an isomorphism. 
Then $f(\res x_2)=f(x_2)^{[p]}$ which implies that 
$\beta_1 a_{11}^2a_{22} x_4=\beta_2 a_{22}^p x_4$. Hence, $\beta_1/\beta_2=a_{22}^{p-1}a_{11}^{-2}\in (\F^*)^2$. To prove the converse, suppose that $\beta_1/\beta_2=\epsilon^2\in (\F^*)^2$. Then the following is an isomorphism from $K_3^8(\beta_1)$ to $K_3^8(\beta_2)$:
\begin{align*}
\begin{pmatrix}
\epsilon^{-1/p}&0&0&0 & 0\\
0& \epsilon^{2/p}& 0&0& 0\\
0&0&\epsilon^{1/p} &0& 0\\
0&0&0&1& 0\\
0&0&0&0& \epsilon^{-1}\\
\end{pmatrix}.
\end{align*}

\end{proof}

\begin{lemma}\label{lemma-K3-11}
The restricted Lie algebras $K_3^{11}(\gamma_1)$ and $K_3^{11}(\gamma_2)$ are isomorphic if and only if  $\gamma_1\gamma_2^{-1}=\epsilon^{p(p^2-p-1)}$, for some $\epsilon\in \F^*$. 
\end{lemma}
\begin{proof}
First assume that $f=(a_{ij}):  K_3^{11}(\gamma_1)\to K_3^{11}(\gamma_2)$ is an isomorphism. 
It follows that 
\begin{align*}
  \gamma_1/\gamma_2=  a_{11}^{p-2}a_{22}^{p-1}, \quad 
a_{11}^2 a_{22}=a_{55}^p, \quad 
 a_{55}=a_{11}^p.
\end{align*} 
The above  Equations then imply that  
$ \gamma_1\gamma_2^{-1}=a_{11}^{p(p^2-p-1)}$.  To prove the converse, suppose that $\gamma_1\gamma_2^{-1}=\epsilon^{p(p^2-p-1)}$, for some $\epsilon\in \F^*$. Then the following is an isomorphism from $K_3^{11}(\gamma_1)$ to $K_3^{11}(\gamma_2)$:
\begin{align*}
\begin{pmatrix}
\epsilon&0&0&0 & 0\\
0& \epsilon^{p^2-2}& 0&0& 0\\
0&0&\epsilon^{p^2-1} &0& 0\\
0&0&0&\epsilon^{p^2}& 0\\
0&0&0&0& \epsilon^{p}\\
\end{pmatrix}.
\end{align*}

\end{proof}

\subsection{Extensions of  $(L, \res x_3=x_4)$}
 First, we find a basis for $Z^2(L,\F)$. Let $(\phi,\omega)=(a\Delta _{12}+b\Delta_{13}+c\Delta_{14}+d\Delta_{23}+e\Delta_{24}+f\Delta_{34},\alpha f_1+\beta f_2+ \gamma f_3+\delta f_4)\in Z^2(L,\F)$. Then we must have $\delta^2\phi(x,y,z) =0$ and $\phi(x,\res y)=0$, for all $x,y,z \in L$. Therefore, we have
\begin{align*}
0=(\delta^2\phi)(x_1,x_2,x_4)&=\phi([x_1,x_2],x_4)+\phi([x_2,x_4],x_1)+\phi([x_4,x_1],x_2)=\phi(x_3,x_4)=f.
\end{align*}
Also, we have $\phi(x,\res y)=0$. Therefore, $\phi(x,x_4)=0$, for all $x\in L$ and hence $\phi(x_1,x_4)=\phi(x_2,x_4)=\phi(x_3,x_4)=0$ which implies that $c=e=f=0$. Therefore, $Z^2(L,\F)$ has a basis consisting of:
$$ (\Delta_{12},0),(\Delta_{13},0),(\Delta_{23},0),(0,f_1),
(0,f_2),(0,f_3),(0,f_4).$$
Next, we find a basis for $B^2(L,\F)$. Let $(\phi,\omega)\in B^2(L,\F)$. Since $B^2(L,\F)\subseteq Z^2(L,\F)$, we have $(\phi,\omega)=(a\Delta _{12}+b\Delta_{13}+c\Delta_{23},\alpha f_1+\beta f_2+ \gamma f_3+\delta f_4)$. Note that there exists a linear map $\psi:L\to \F$  such that $\delta^1\psi(x,y)=\phi(x,y)$ and $\tilde \psi(x)=\omega(x)$, for all $x,y \in L$. So, we have
\begin{align*}
&a=\phi(x_1,x_2)=\delta^1\psi(x_1,x_2)=\psi([x_1,x_2])=\psi(x_3), \text{ and }\\
&b=\phi(x_1,x_3)=\delta^1\psi(x_1,x_3)=\psi([x_1,x_3])=0.
\end{align*}
Similarly, we can show that $c=0$. Also, we have
\begin{align*}
&\gamma=\omega(x_3)=\tilde \psi(x_3)=\psi(\res x_3)=\psi(x_4), \text{ and }\\
&\alpha=\omega(x_1)=\tilde \psi(x_1)=\psi(\res x_1)=0.
\end{align*}
Similarly, we can show that $\beta=\delta=0$. Therefore, $(\phi,\omega)=(a\Delta_{12},\gamma f_3)$ and hence $B^2(L,\F)=\la(\Delta_{12},0),(0,f_3)\ra_{\F}$. We deduce that a basis for $H^2(L,\F)$ is as follows:
$$[(\Delta_{13},0)],[(\Delta_{23},0)],[(0,f_1)],[(0,f_2)],[(0,f_4)].$$

Note that the group $\Aut_p(L)$ in this case consists of invertible matrices of the form
$$\begin{pmatrix}
 a_{11}  & a_{12} & 0 & 0 \\
a_{21} & a_{22} & 0 & 0\\
a_{31} & a_{32} & r & a_{34} \\
a_{41} & a_{42} & 0& a_{44}
\end{pmatrix},$$
where $r = a_{11}a_{22}-a_{12}a_{21}$, $a_{31}=a_{32}=a_{34}=0$, and $a_{44}=r^p$.

Let $[\theta]=[(\phi,\omega)]\in H^2(L,\F)$. Then $\phi =a\Delta_{13}+c\Delta_{23}$. We can verify that the action of $\Aut_p(L)$ on the $\phi$'s in the matrix form is as follows:
\begin{align}\label{ij}
\begin{pmatrix}
a'\\c'
\end{pmatrix}=
\begin{pmatrix}
a_{11}r &a_{21}r\\
a_{12}r&a_{22}r
\end{pmatrix}
\begin{pmatrix}
a\\c
\end{pmatrix}.
\end{align}
The orbit with representative 
$\begin{pmatrix}
1\\ 0
\end{pmatrix}$ of this action  gives us $L_{5,3}$.
Note that we need to have $A\phi(x_1,x_4)=A\phi(x_2,x_4)=A\phi(x_3,x_4)=0$ which implies that
\begin{align*}
&a_{11}a_{34}a+a_{21}a_{34}c=0,\\
&a_{12}a_{34}a+a_{22}a_{34}c=0.
\end{align*}
Let $\omega=\alpha f_1+\beta f_2+\delta f_4$, for some $\alpha,\beta,\delta \in \F$. Suppose that $A\omega =\alpha'f_1+\beta' f_2+\delta' f_4$,  for some $\alpha',\beta',\delta' \in \F.$ We have
\begin{align*}
&A\omega (x_1)=a_{11}^p\alpha +a_{21}^p \beta + a_{41}^p \delta;\\
&A\omega (x_2)= a_{12}^p\alpha +a_{22}^p\beta +a_{42}^p\delta;\\
&A\omega(x_4)=a_{44}^p \delta.
\end{align*}
In the matrix form we can write this as
\begin{align}\label{kl}
\begin{pmatrix}
\alpha'\\
\beta'\\
\delta'
\end{pmatrix}=
\begin{pmatrix}
a_{11}^p& a_{21}^p&a_{41}^p\\
a_{12}^p&a_{22}^p&a_{42}^p\\
0&0&r^{p^2}
\end{pmatrix}
\begin{pmatrix}
\alpha \\
\beta \\
\delta
\end{pmatrix}.
\end{align}

Thus, we can write Equations \eqref{ij} and \eqref{kl} together as follows:
\begin{align*}
\bigg[r \begin{pmatrix}
a_{11}&a_{21}\\
a_{12}&a_{22}
\end{pmatrix},
\begin{pmatrix}
a_{11}^p& a_{21}^p&a_{41}^p\\
a_{12}^p&a_{22}^p&a_{42}^p\\
0&0&r^{p^2}
\end{pmatrix}\bigg]
\bigg[\begin{pmatrix}
a\\c
\end{pmatrix}
\begin{pmatrix}
\alpha \\
\beta\\
\delta
\end{pmatrix}\bigg]=
\bigg[\begin{pmatrix}
a' \\ c'
\end{pmatrix},
\begin{pmatrix}
\alpha'\\
\beta'\\
\delta'
\end{pmatrix}\bigg].
\end{align*}
Now we find the representatives of the orbits of the action of $\Aut_p(L)$ on the set of $\omega$'s such that 
the orbit represented by 
$\begin{pmatrix}
1\\ 0
\end{pmatrix}$ is preserved under the action of $\Aut_p(L)$ on the set of $\phi$'s.
Note that we take $a_{34}=0$. Then we have 
\begin{align*}
&a_{11}a_{34}a+a_{21}a_{34}c=0,\\
&a_{12}a_{34}a+a_{22}a_{34}c=0.
\end{align*}
Let $\nu = \begin{pmatrix}
\alpha\\
\beta \\
\delta
\end{pmatrix} \in \F^3$. If $\nu= \begin{pmatrix} 
0 \\
0 \\
0
\end{pmatrix}$, then $\{\nu \}$ is clearly an $\Aut_p(L)$-orbit. Suppose that $\delta \neq 0$. Then 
\begin{align*}
&\bigg[\begin{pmatrix}
1& 0\\
0&1
\end{pmatrix},
\begin{pmatrix}
1& 0&-\alpha/\delta\\
0&1&-\beta/\delta\\
0&0&1
\end{pmatrix}\bigg]
\bigg[\begin{pmatrix}
1 \\ 0
\end{pmatrix},
\begin{pmatrix} 
\alpha \\
\beta \\
\delta
\end{pmatrix}\bigg]=
\bigg[\begin{pmatrix}
1\\0
\end{pmatrix},
\begin{pmatrix}
0\\
0\\
\delta
\end{pmatrix}\bigg], \text{ and }\\
&\bigg[\delta^{-1/p^2}\begin{pmatrix}
\delta^{1/p^2}& 0\\
0&\delta^{-2/p^2}
\end{pmatrix},
\begin{pmatrix}
\delta^{1/p}&0&0\\
0&\delta^{-2/p}&0\\
0&0&1/\delta
\end{pmatrix}\bigg]
\bigg[\begin{pmatrix}
1\\ 0
\end{pmatrix},
\begin{pmatrix}
0 \\
0 \\
\delta
\end{pmatrix}\bigg]=
\bigg[\begin{pmatrix}
1\\0
\end{pmatrix},
\begin{pmatrix}
0\\
0\\
1
\end{pmatrix}\bigg].
\end{align*}

Next, if $\delta =0$, but $\beta \neq 0$, then
\begin{align*}
\bigg[\begin{pmatrix}
1& (-\alpha/\beta)^{1/p}\\
0&1
\end{pmatrix},
\begin{pmatrix}
1&-\alpha/\beta&0\\
0&1&0\\
0&0&1
\end{pmatrix}\bigg]
\bigg[\begin{pmatrix}
1\\ 0
\end{pmatrix},
\begin{pmatrix}
\alpha\\
\beta\\
0
\end{pmatrix}\bigg]=
\bigg[\begin{pmatrix}
1\\0
\end{pmatrix},
\begin{pmatrix}
0\\
\beta\\
0
\end{pmatrix}\bigg].
\end{align*}

Finally, if $\delta =\beta =0$, but $\alpha \neq 0$, then
\begin{align*}
\bigg[\alpha^{1/p}\begin{pmatrix}
\alpha^{-1/p}& 0\\
0&\alpha^{2/p}
\end{pmatrix},
\begin{pmatrix}
1/\alpha&0&0\\
0&\alpha^{2}&0\\
0&0&\alpha^p
\end{pmatrix}\bigg]
\bigg[\begin{pmatrix}
1\\ 0
\end{pmatrix},
\begin{pmatrix}
\alpha \\0\\
0
\end{pmatrix}\bigg]=
\bigg[\begin{pmatrix}
1\\0
\end{pmatrix},
\begin{pmatrix}
1\\
0\\
0
\end{pmatrix}\bigg].
\end{align*}
Thus the following elements are $\Aut_p(L)$-orbit representatives: 
\begin{align*}
\begin{pmatrix}
0\\
0\\
0
\end{pmatrix}, 
\begin{pmatrix}
1\\
0\\
0
\end{pmatrix},
\begin{pmatrix} 
0\\
0\\
1
\end{pmatrix},
\begin{pmatrix}
0\\
\beta\\
0
\end{pmatrix}.
\end{align*}

%

Therefore, the corresponding restricted Lie algebra structures are as follows: (Note that we need to switch $x_4$ and $x_5$.)\\
\begin{align*}
&K_3^{12}=\langle x_1,\ldots,x_5\mid [x_1,x_2]=x_3,[x_1,x_3]=x_4 ,
\res x_3=x_5 \rangle;\\
&K_3^{13}=\langle x_1,\ldots,x_5\mid [x_1,x_2]=x_3,[x_1,x_3]=x_4, 
\res x_1=x_4, \res x_3=x_5 \rangle;\\
&K_3^{14}(\beta)=\langle x_1,\ldots,x_5\mid [x_1,x_2]=x_3,[x_1,x_3]=x_4, 
\res x_2=\beta x_4, \res x_3=x_5 \rangle;\\
&K_3^{15}=\langle x_1,\ldots,x_5\mid [x_1,x_2]=x_3,[x_1,x_3]=x_4,
\res x_3=x_5, \res x_5=x_4 \rangle
\end{align*}
where $\beta \in \F^*$.  

\begin{lemma}\label{lemma-K314}
The restricted Lie algebras $K_3^{14}(\beta_1)$ and $K_3^{14}(\beta_2)$ are isomorphic if and only if  $\beta_2\beta_1^{-1}\in (\F^*)^2$. 
\end{lemma}
\begin{proof}
First assume that $f=(a_{ij}):  K_3^{14}(\beta_1)\to K_3^{14}(\beta_2)$ is an isomorphism. 
Then $f(\res x_2)=f(x_2)^{[p]}$ which implies that 
$\beta_1 a_{11}^2a_{22} x_4=\beta_2 a_{22}^p x_4$. Hence, $\beta_1/\beta_2=a_{22}^{p-1}a_{11}^{-2}\in (\F^*)^2$. To prove the converse, suppose that $\beta_1/\beta_2=\epsilon^2\in (\F^*)^2$. Then the following is an isomorphism from
 $K_3^{14}(\beta_1)$ to $K_3^{14}(\beta_2)$:
\begin{align*}
\begin{pmatrix}
\epsilon^{-1/p}&0&0&0 & 0\\
0& \epsilon^{2/p}& 0&0& 0\\
0&0&\epsilon^{1/p} &0& 0\\
0&0&0&1& 0\\
0&0&0&0& \epsilon\\
\end{pmatrix}.
\end{align*}

\end{proof}

\section{Extensions of  $L=\frac {L_{5,3}}{\la x_5\ra}$}

In this section we find all non-isomorphic p-maps on $L_{5,3}$ such that $\res x_5=0$.  We let
$$L=\frac {L_{5,3}}{\la x_5\ra}\cong L_{4,3},$$ where $L_{4,3} = \la x_1,x_2,x_3,x_4 \mid [x_1 , x_2]=x_3 , [x_1,x_3]=x_4 \ra$. The group $\Aut(L)$ consists of invertible matrices of the form
$$\begin{pmatrix}
 a_{11}  & 0 & 0 & 0 \\
a_{21} & a_{22} & 0 & 0\\
a_{31} & a_{32} & r & 0 \\
a_{41} & a_{42} & a_{11}a_{32}& a_{11}r
\end{pmatrix},$$
where $r = a_{11}a_{22} \neq 0$.

\begin{lemma}\label{L53-x5}
Let $K=L_{5,3}$ and $[p]:K\to K$ be a $p$-map on $K$ such that $\res x_5=0$ and let $L=\frac{K}{M}$ where $M=\la x_5\ra_{\F}$. Then $K\cong L_{\theta}$ where $\theta=(0,\omega)\in Z^2(L,M)$.
\end{lemma}
\begin{proof}
 Let $\pi :K\rightarrow L$ be the projection map. We have the exact sequence 
$$ 0\rightarrow M\rightarrow K\rightarrow L\rightarrow 0.$$
Let $\sigma :L\rightarrow K$ such that $x_i\mapsto x_i$, $1\leq i\leq 4$. Then $\sigma$ is an injective linear map and $\pi \sigma =1_L$. Now, we define $\phi : L\times L\rightarrow M$ by $\phi (x_i,x_j)=[\sigma (x_i),\sigma (x_j)]-\sigma ([x_i,x_j])$, $1\leq i,j\leq 4$ and $\omega: L\rightarrow M$ by $\omega (x)=\res {\sigma (x)} -\sigma (\res x)$. Note that
\begin{align*}
&\phi(x_1,x_2)=[\sigma(x_1),\sigma(x_2)]-\sigma([x_1,x_2])=0;\\
&\phi(x_1,x_3)=[\sigma(x_1),\sigma(x_3)]-\sigma([x_1,x_3])=0.
\end{align*}
Similarly, we can show that $\phi(x_1,x_4)=\phi(x_2,x_3)=\phi(x_2,x_4)=\phi(x_3,x_4)=0$. Therefore, $\phi=0$.
Now, by Lemma \ref{K=L-theta}, we have $\theta=(0,\omega)\in Z^2(L,M)$ and $K\cong L_{\theta}$.
\end{proof}

 Note that by Theorem \ref{4-dim}, there are four non-isomorphic restricted Lie algebra structures on $L$ given by the following  $p$-maps:
\begin{enumerate}
\item[II.1]Trivial $p$-map;
\item[II.2]$\res x_1=x_4$;
\item[II.3]$\res x_2 =\xi x_4 $;
\item[II.4]$\res x_3 =x_4$.
\end{enumerate}

In the following subsections, we make $L$ into a restricted Lie algebra by equipping it with each of the above $p$-maps. Then,  in  each  case,  we  find all possible orbit representatives of the form $(0,\omega)$ under the action of $\Aut_p(L)$ on $H^2(L,\F)$. By  Lemma \ref{L53-x5}, we do get all possible $p$-maps on $K_3$ with the property that $x_5^{[p]}=0$.

\subsection{Extensions of ($L$, trivial $p$-map)}\label{L_{5,3} mod x5 trivial pmap}
 First, we find a basis for $Z^2(L,\F)$. Let $(\phi,\omega)=(a\Delta _{12}+b\Delta_{13}+c\Delta_{14}+d\Delta_{23}+e\Delta_{24}+f\Delta_{34},\alpha f_1+\beta f_2+ \gamma f_3+\delta f_4)\in Z^2(L,\F)$. Then we must have $\delta^2\phi(x,y,z) =0$ and $\phi(x,\res y)=0$, for all $x,y,z \in L$. Therefore,
\begin{align*}
&0=(\delta^2\phi)(x_1,x_2,x_4)=\phi([x_1,x_2],x_4)+\phi([x_2,x_4],x_1)+\phi([x_4,x_1],x_2)=\phi(x_3,x_4), \text{ and }\\
&0=(\delta^2\phi)(x_1,x_2,x_3)=\phi([x_1,x_2],x_3)+\phi([x_2,x_3],x_1)+\phi([x_3,x_1],x_2)=\phi(x_2,x_4).
\end{align*}
Thus, we get $f=e=0$.
Since the $p$-map is trivial, $\phi(x,\res y)=\phi(x,0)=0$, for all $x,y\in L$. Therefore, a basis for $Z^2(L,\F)$ is as follows:
$$
 (\Delta_{12},0),(\Delta_{13},0),(\Delta_{14},0),(\Delta_{23},0),(0,f_1),
(0,f_2),(0,f_3),(0,f_4).
$$
Next, we find a basis for $B^2(L,\F)$. Let $(\phi,\omega)\in B^2(L,\F)$. Since $B^2(L,\F)\subseteq Z^2(L,\F)$, we have $(\phi,\omega)=(a\Delta _{12}+b\Delta_{13}+c\Delta_{14}+d\Delta_{23},\alpha f_1+\beta f_2+ \gamma f_3+\delta f_4)$. So, there exists a linear  map $\psi:L\to \F$ such that $\delta^1\psi(x,y)=\phi(x,y)$ and $\tilde \psi(x)=\omega(x)$, for all $x,y \in L$. So, we have
\begin{align*}
&a=\phi(x_1,x_2)=\delta^1\psi(x_1,x_2)=\psi([x_1,x_2])=\psi(x_3), \text{ and }\\
&b=\phi(x_1,x_3)=\delta^1\psi(x_1,x_3)=\psi([x_1,x_3])=\psi(x_4), \text{ and }\\
&c=\phi(x_1,x_4)=\delta^1\psi(x_1,x_4)=\psi([x_1,x_4])=0
\end{align*}
Similarly, we can show that $d=0$. Also, we have
\begin{align*}
\alpha=\omega(x_1)=\tilde \psi(x_1)=\psi(\res x_1)=0.
\end{align*}
Similarly, we can show that $\beta=\gamma=\delta=0$. Therefore, $(\phi,\omega)=(a\Delta_{12}+b\Delta_{13},0)$ and hence
 $B^2(L,\F)=\la(\Delta_{12},0),(\Delta_{13},0)\ra_{\F}$. We deduce that a basis for $H^2(L,\F)$ is as follows:
$$
[(\Delta_{14},0)],[(\Delta_{23},0)],[(0,f_1)],
[(0,f_2)],[(0,f_3)],[(0,f_4)].
$$
Let $[\theta]=[(\phi,\omega)]\in H^2(L,\F)$. Since we want $L_{\theta}$ and  $L_{5,3}$ to be isomorphic as  Lie algebras, we should have $\phi=0$. Since $0$ is preserved under $\Aut_p(L)$, it is enough to find $\Aut_p(L)$-representatives of the $\omega$'s.

Let $\omega=\alpha f_1+\beta f_2+\gamma f_3+\delta f_4$, for some $\alpha,\beta,\gamma,\delta \in \F$. Suppose that $A\omega =\alpha'f_1+\beta' f_2+\gamma'f_3+\delta' f_4$,  for some $\alpha',\beta',\gamma',\delta' \in \F.$ 
We can easily verify that the action of $\Aut_p(L)$ on the $\omega$'s in the matrix form is as follows:
\begin{align*} 
\begin{pmatrix}
\alpha' \\ 
\beta'\\
\gamma'\\
\delta'
\end{pmatrix}=
\begin{pmatrix}
a_{11}^p& a_{21}^p&a_{31}^p&a_{41}^p\\
0& a_{22}^p& a_{32}^p&a_{42}^p\\
0&0&r^p&a_{11}^pa_{32}^p\\
0&0&0&a_{11}^pr^p
\end{pmatrix}
\begin{pmatrix}
\alpha\\ 
\beta \\
\gamma \\
\delta
\end{pmatrix}.
\end{align*}
Now we find the representatives of the orbits of this action. 
Let $\nu = \begin{pmatrix}
\alpha\\ 
\beta \\
\gamma \\
\delta
\end{pmatrix} \in \F^4$. If $\nu= \begin{pmatrix}
0\\ 
0 \\
0 \\
0
\end{pmatrix}$, then $\{\nu \}$ is clearly an $\Aut_p(L)$-orbit. Suppose that $\delta \neq 0$. Then 
\begin{align*}
&\begin{pmatrix}
1& 0&0&-\alpha/\delta\\
0& 1& 0&-\beta/\delta\\
0&0&1&0\\
0&0&0&1
\end{pmatrix}
\begin{pmatrix}
\alpha\\ 
\beta \\
\gamma \\
\delta
\end{pmatrix}=
\begin{pmatrix}
0 \\ 
0\\
\gamma\\
\delta
\end{pmatrix}, \text{ and }\\
&\begin{pmatrix}
1& 0&0&0\\
0& 1& -\gamma/\delta&\gamma^2/\delta^2\\
0&0&1&-\gamma/\delta\\
0&0&0&1
\end{pmatrix}
\begin{pmatrix}
0\\ 
0 \\
\gamma \\
\delta
\end{pmatrix}=
\begin{pmatrix}
0 \\ 
0\\
0\\
\delta
\end{pmatrix}, \text{ and }\\
&\begin{pmatrix}
1& 0&0&0\\
0& 1/\delta& 0&0\\
0&0&1/\delta&0\\
0&0&0&1/\delta
\end{pmatrix}
\begin{pmatrix}
0\\ 
0 \\
0 \\
\delta
\end{pmatrix}=
\begin{pmatrix}
0 \\ 
0\\
0\\
1
\end{pmatrix}.
\end{align*}
 Next, if $\delta=0$, but $\gamma \neq 0$, then
\begin{align*}
&\begin{pmatrix}
1& 0&-\alpha /\gamma&0\\
0& 1& -\beta/\gamma&0\\
0&0&1&-\beta/\gamma\\
0&0&0&1
\end{pmatrix}
\begin{pmatrix}
\alpha\\ 
\beta \\
\gamma \\
0
\end{pmatrix}=
\begin{pmatrix}
0 \\ 
0\\
\gamma\\
0
\end{pmatrix}, \text{ and }\\
&\begin{pmatrix}
(1/\gamma)^2& 0&0&0\\
0& \gamma& 0&0\\
0&0&1/\gamma&0\\
0&0&0&(1/\gamma)^3
\end{pmatrix}
\begin{pmatrix}
0\\ 
0 \\
\gamma \\
0
\end{pmatrix}=
\begin{pmatrix}
0 \\ 
0\\
1\\
0
\end{pmatrix}.
\end{align*}
Next, if $\gamma =\delta=0$, but $\beta \neq 0$, then
\begin{align*}
&\begin{pmatrix}
1& -\alpha/\beta&0&0\\
0& 1& 0&0\\
0&0&1&0\\
0&0&0&1
\end{pmatrix}
\begin{pmatrix}
\alpha\\ 
\beta \\
0 \\
0
\end{pmatrix}=
\begin{pmatrix}
0 \\ 
\beta\\
0\\
0
\end{pmatrix}, \text{ and }\\
&\begin{pmatrix}
(1/\beta )^{-2}& 0&0&0\\
0& 1/\beta& 0&0\\
0&0&(1/\beta )^{-1}&0\\
0&0&0&(1/\beta )^{-3}
\end{pmatrix}
\begin{pmatrix}
0\\ 
\beta \\
0 \\
0
\end{pmatrix}=
\begin{pmatrix}
0 \\ 
1\\
0\\
0
\end{pmatrix}.
\end{align*}
Finally, if $\delta=\gamma =\beta =0$, but $\alpha \neq 0$, then
\begin{align*}
\begin{pmatrix}
1/\alpha& 0&0&0\\
0& 1& 0&0\\
0&0&1/\alpha&0\\
0&0&0&(1/\alpha )^2
\end{pmatrix}
\begin{pmatrix}
\alpha\\ 
0 \\
0 \\
0
\end{pmatrix}=
\begin{pmatrix}
1 \\ 
0\\
0\\
0
\end{pmatrix}.
\end{align*}
Thus the following elements are $\Aut_p(L)$-orbit representatives: 
\begin{align*}
\begin{pmatrix}
0\\ 
0\\
0\\
0
\end{pmatrix},
\begin{pmatrix}
1 \\ 
0\\
0\\
0
\end{pmatrix},
\begin{pmatrix}
0 \\ 
1\\
0\\
0
\end{pmatrix},
\begin{pmatrix}
0\\ 
0\\
1\\
0
\end{pmatrix},
\begin{pmatrix}
0 \\ 
0\\
0\\
1
\end{pmatrix}.
\end{align*}
 
Therefore, the corresponding restricted Lie algebra structures are as follows:
\begin{align*}
&K_3^{16}=\langle x_1,\ldots,x_5\mid [x_1,x_2]=x_3,[x_1,x_3]=x_4  \rangle;\\
&K_3^{17}=\langle x_1,\ldots,x_5\mid [x_1,x_2]=x_3,[x_1,x_3]=x_4, 
\res x_1=x_5 \rangle;\\
&K_3^{18}=\langle x_1,\ldots,x_5\mid [x_1,x_2]=x_3,[x_1,x_3]=x_4, 
\res x_2=x_5  \rangle;\\
&K_3^{19}=\langle x_1,\ldots,x_5\mid [x_1,x_2]=x_3,[x_1,x_3]=x_4, 
\res x_3=x_5 \rangle;\\
&K_3^{20}=\langle x_1,\ldots,x_5\mid [x_1,x_2]=x_3,[x_1,x_3]=x_4, 
\res x_4=x_5 \rangle.
\end{align*}

\subsection{Extensions of $(L,  \res x_1 = x_4)$}
 First, we find a basis for $Z^2(L,\F)$. Let $(\phi,\omega)=(a\Delta _{12}+b\Delta_{13}+c\Delta_{14}+d\Delta_{23}+e\Delta_{24}+f\Delta_{34},\alpha f_1+\beta f_2+ \gamma f_3+\delta f_4)\in Z^2(L,\F)$. Then we must have $\delta^2\phi(x,y,z) =0$ and $\phi(x,\res y)=0$, for all $x,y,z \in L$. Therefore, we have
\begin{align*}
&0=(\delta^2\phi)(x_1,x_2,x_4)=\phi([x_1,x_2],x_4)+\phi([x_2,x_4],x_1)+\phi([x_4,x_1],x_2)=\phi(x_3,x_4), \text{ and }\\
&0=(\delta^2\phi)(x_1,x_2,x_3)=\phi([x_1,x_2],x_3)+\phi([x_2,x_3],x_1)+\phi([x_3,x_1],x_2)=\phi(x_2,x_4).
\end{align*}
Thus, we get $f=e=0$.
Also, we have $\phi(x,\res y)=0$. Therefore, $\phi(x,x_4)=0$, for all $x\in L$ and hence $\phi(x_1,x_4)=\phi(x_2,x_4)=\phi(x_3,x_4)=0$ which implies that $c=e=f=0$. Therefore, $Z^2(L,\F)$
has a basis consisting of:
$$
 (\Delta_{12},0),(\Delta_{13},0),(\Delta_{23},0),(0,f_1),
(0,f_2),(0,f_3),(0,f_4).
$$
Next, we find a basis for $B^2(L,\F)$. Let $(\phi,\omega)\in B^2(L,\F)$. Since $B^2(L,\F)\subseteq Z^2(L,\F)$, we have $(\phi,\omega)=(a\Delta _{12}+b\Delta_{13}+c\Delta_{23},\alpha f_1+\beta f_2+ \gamma f_3+\delta f_4)$. Note that there exists a linear  map $\psi:L\to \F$ such that $\delta^1\psi(x,y)=\phi(x,y)$ and $\tilde \psi(x)=\omega(x)$, for all $x,y \in L$. So, we have
\begin{align*}
&a=\phi(x_1,x_2)=\delta^1\psi(x_1,x_2)=\psi([x_1,x_2])=\psi(x_3), \text{ and }\\
&b=\phi(x_1,x_3)=\delta^1\psi(x_1,x_3)=\psi([x_1,x_3])=\psi(x_4), \text{ and }\\
&c=\phi(x_2,x_3)=\delta^1\psi(x_2,x_3)=\psi([x_2,x_3])=0.
\end{align*}
Also, we have
\begin{align*}
\alpha=\omega(x_1)&=\tilde \psi(x_1)=\psi(\res x_1)=\psi(x_4), \text{ and }\\
\beta=\omega(x_2)&=\tilde \psi(x_2)=\psi(\res x_2)=0.
\end{align*}
Similarly, we can show that $\gamma=\delta=0$. Note that $\psi(x_4)=b=\alpha$.  Therefore, $(\phi,\omega)=(a\Delta_{12}+b\Delta_{13},bf_1)$ and hence $B^2(L,\F)=\la(\Delta_{13},f_1),(\Delta_{12},0)\ra_{\F}$. Note that
$$
\bigg\{[(\Delta_{13},0)],[(\Delta_{23},0)],[(0,f_1)],
[(0,f_2)],[(0,f_3)],[(0,f_4)]\bigg\}
$$
spans $H^2(L,\F)$.
Since $[(\Delta_{13},0)]+[(0,f_1)]=[(\Delta_{13},f_1)]=[0]$, then $[(0,f_1)]$ is an scalar multiple of $[(\Delta_{13},0)]$
in $H^2(L,\F)$. Note that $\dim H^2=\dim Z^2-\dim B^2=5$. Therefore, 
$$
\bigg \{[(\Delta_{13},0)],[(\Delta_{23},0)],[(0,f_2)],[(0,f_3)],[(0,f_4)]\bigg\}
$$
 forms a basis for $H^2(L,\F)$.

 Note that he group $\Aut_p(L)$ in this case consists of invertible matrices of the form
$$\begin{pmatrix}
 a_{11}  & 0 & 0 & 0 \\
a_{21} & a_{22} & 0 & 0\\
a_{31} & a_{32} & r & 0 \\
a_{41} & a_{42} & a_{11}a_{32}& a_{11}r
\end{pmatrix},$$
where $r = a_{11}a_{22} \neq 0$ and $a_{11}r=a_{11}^p$.
 
Let $[\theta]=[(\phi,\omega)]\in H^2(L,\F)$. As in Section \ref{L_{5,3} mod x5 trivial pmap},  it is enough to find $\Aut_p(L)$-representatives of the $\omega$'s.

Let $\omega =\beta f_2+\gamma f_3+\delta f_4$, for some $\beta,\gamma,\delta \in \F$. Suppose that $A\omega =\beta' f_2+\gamma' f_3+\delta' f_4$, for some  $\beta',\gamma',\delta' \in \F$. We have
\begin{align*}
&A\omega (x_2)=a_{22}^p\beta+a_{32}^p\gamma +a_{42}^p\delta;\\
&A\omega (x_3)=r^p\gamma +a_{11}^p a_{32}^p\delta;\\
&A\omega (x_4)=a_{11}^p r^p \delta.
\end{align*}
In the matrix form we can write this as
\begin{align*}
\begin{pmatrix} 
\beta'\\
\gamma'\\
\delta'
\end{pmatrix}=
\begin{pmatrix}
 a_{22}^p& a_{32}^p&a_{42}^p\\
0&r^p&a_{11}^pa_{32}^p\\
0&0&a_{11}^pr^p
\end{pmatrix}
\begin{pmatrix} 
\beta \\
\gamma \\
\delta
\end{pmatrix}.
\end{align*}
Note that we need to have $A\omega(x_1)=0$ which implies that $a_{21}^p\beta+a_{31}^p\gamma+a_{41}^p\delta=0$.
Now we find the representatives of the orbits of this action. Note that we take $a_{21}=a_{31}=a_{41}=0$. Therefore, $a_{21}^p\beta+a_{31}^p\gamma+a_{41}^p\delta=0$.
Let $\nu = \begin{pmatrix}
\beta \\
\gamma \\
\delta
\end{pmatrix} \in \F^3$. If $\nu= \begin{pmatrix}
0 \\
0 \\
0
\end{pmatrix}$, then $\{\nu \}$ is clearly an $\Aut_p(L)$-orbit. Suppose that $\delta \neq 0$. Then 
\begin{align*}
&\begin{pmatrix}
1&0&-\beta/\delta\\
0& 1& 0\\
0&0&1\\
\end{pmatrix}
\begin{pmatrix} 
\beta \\
\gamma \\
\delta
\end{pmatrix}=
\begin{pmatrix}
0\\
\gamma\\
\delta
\end{pmatrix}, \text{ and } \\
&\begin{pmatrix}
1& -\gamma/\delta&\gamma^2/\delta^2\\
0& 1& -\gamma/\delta\\
0&0&1\\
\end{pmatrix}
\begin{pmatrix} 
0 \\
\gamma \\
\delta
\end{pmatrix}=
\begin{pmatrix}
0\\
0\\
\delta
\end{pmatrix}, \text{ and }\\
&\begin{pmatrix}
(1/\delta)^\frac{p-2}{p}& 0&0\\
0& (1/\delta)^\frac {p-1}{p}& 0\\
0&0&1/\delta
\end{pmatrix}
\begin{pmatrix}
0 \\
0 \\
\delta
\end{pmatrix}=
\begin{pmatrix} 
0\\
0\\
1
\end{pmatrix}.
\end{align*}
 Next, if $\delta=0$, but $\gamma \neq 0$, then
 $$
\begin{pmatrix}
1& -\beta /\gamma&0\\
0& 1& -\beta/\gamma\\
0&0&1
\end{pmatrix}
\begin{pmatrix}
\beta \\
\gamma \\
0
\end{pmatrix}=
\begin{pmatrix}
 
0\\
\gamma\\
0
\end{pmatrix}.$$

 Finally, if $\gamma =\delta=0$, but $\beta \neq 0$, then we have
$\begin{pmatrix}
\beta \\
0 \\
0
\end{pmatrix}$.
Thus the following elements are $\Aut (L)$-orbit representatives:
\begin{align*}
\begin{pmatrix}
0\\
0\\
0
\end{pmatrix},
\begin{pmatrix}
\beta \\ 
0\\
0
\end{pmatrix},
\begin{pmatrix}
0 \\ 
\gamma\\
0
\end{pmatrix},
\begin{pmatrix}
0\\ 
0\\
1
\end{pmatrix}.
\end{align*}
Therefore, the corresponding restricted Lie algebra structures are as follows: 
\begin{align*}
&K_3^{21}=\langle x_1,\ldots,x_5\mid [x_1,x_2]=x_3,[x_1,x_3]=x_4 ,
\res x_1=x_4 \rangle;\\
&K_3^{22}(\beta)=\langle x_1,\ldots,x_5\mid [x_1,x_2]=x_3,[x_1,x_3]=x_4, 
\res x_1=x_4, \res x_2=\beta x_5 \rangle;\\
&K_3^{23}(\gamma)=\langle x_1,\ldots,x_5\mid [x_1,x_2]=x_3,[x_1,x_3]=x_4, 
\res x_1=x_4, \res x_3=\gamma x_5 \rangle;\\
&K_3^{24}=\langle x_1,\ldots,x_5\mid [x_1,x_2]=x_3,[x_1,x_3]=x_4, 
\res x_1=x_4, \res x_4=x_5 \rangle.
\end{align*}

\begin{lemma}\label{lem 5301}
We have 
$
K_3^{22}(\beta)\cong \langle x_1,\ldots,x_5\mid [x_1,x_2]=x_3,[x_1,x_3]=x_4, 
\res x_1=x_4, \res x_2= x_5 \rangle,
$
for every $\beta\in \F^*$.
\end{lemma}
\begin{proof}
The following automorphism of $L_{5,3}$ gives us the desired isomorphism:
$$\begin{pmatrix}
1&0&0&0&0\\
0& 1&0&0&0\\
0&0&1&0&0\\
0&0&0&1&0\\
0&0&0&0&1/\beta\\
\end{pmatrix}.$$
\end{proof}

\begin{lemma}
We have
$K_3^{23}(\gamma)\cong \langle x_1,\ldots,x_5\mid [x_1,x_2]=x_3,[x_1,x_3]=x_4, 
\res x_1=x_4, \res x_3= x_5 \rangle$,
for for every $\gamma\in \F^*$.
\end{lemma}
\begin{proof}
The following automorphism of $L_{5,3}$ gives us the desired isomorphism:
$$\begin{pmatrix}
1&0&0&0&0\\
0& 1&0&0&0\\
0&0&1&0&0\\
0&0&0&1&0\\
0&0&0&0&1/\gamma\\
\end{pmatrix}.$$
\end{proof}

\subsection{Extensions of $(L,  \res x_2 =\xi  x_4)$}
 First, we find a basis for $Z^2(L,\F)$. Let $(\phi,\omega)=(a\Delta _{12}+b\Delta_{13}+c\Delta_{14}+d\Delta_{23}+e\Delta_{24}+f\Delta_{34},\alpha f_1+\beta f_2+ \gamma f_3+\delta f_4)\in Z^2(L,\F)$. Then we must have $\delta^2\phi(x,y,z) =0$ and $\phi(x,\res y)=0$, for all $x,y,z \in L$. Therefore, we have
\begin{align*}
&0=(\delta^2\phi)(x_1,x_2,x_4)=\phi([x_1,x_2],x_4)+\phi([x_2,x_4],x_1)+\phi([x_4,x_1],x_2)=\phi(x_3,x_4), \text{ and }\\
&0=(\delta^2\phi)(x_1,x_2,x_3)=\phi([x_1,x_2],x_3)+\phi([x_2,x_3],x_1)+\phi([x_3,x_1],x_2)=\phi(x_2,x_4).
\end{align*}
Thus, we get $f=e=0$.
Also, we have $\phi(x,\res y)=0$. Therefore, $\phi(x,x_4)=0$, for all $x\in L$ and hence $\phi(x_1,x_4)=\phi(x_2,x_4)=\phi(x_3,x_4)=0$ which implies that $c=e=f=0$. Therefore, $Z^2(L,\F)$
has a basis consisting of:
$$
 (\Delta_{12},0),(\Delta_{13},0),(\Delta_{23},0),(0,f_1),
(0,f_2),(0,f_3),(0,f_4).
$$
Next, we find a basis for $B^2(L,\F)$. Let $(\phi,\omega)\in B^2(L,\F)$. Since $B^2(L,\F)\subseteq Z^2(L,\F)$, we have $(\phi,\omega)=(a\Delta _{12}+b\Delta_{13}+c\Delta_{23},\alpha f_1+\beta f_2+ \gamma f_3+\delta f_4)$. Note that there exists a linear  map $\psi:L\to \F$ such that $\delta^1\psi(x,y)=\phi(x,y)$ and $\tilde \psi(x)=\omega(x)$, for all $x,y \in L$. So, we have
\begin{align*}
&a=\phi(x_1,x_2)=\delta^1\psi(x_1,x_2)=\psi([x_1,x_2])=\psi(x_3), \text{ and }\\
&b=\phi(x_1,x_3)=\delta^1\psi(x_1,x_3)=\psi([x_1,x_3])=\psi(x_4), \text{ and }\\
&c=\phi(x_2,x_3)=\delta^1\psi(x_2,x_3)=\psi([x_2,x_3])=0.
\end{align*}
Also, we have
\begin{align*}
\beta=\omega(x_2)&=\tilde \psi(x_2)=\psi(\res x_2)=\xi \psi(x_4), \text{ and }\\
\alpha=\omega(x_1)&=\tilde \psi(x_1)=\psi(\res x_1)=0.
\end{align*}
Similarly, we can show that $\gamma=\delta=0$.  Note that $\psi(x_4)=b=\beta\xi^{-1}$.  Therefore, $(\phi,\omega)=(a\Delta_{12}+b\Delta_{13},b\xi f_2)$ and hence $B^2(L,\F)=\la(\Delta_{13},\xi f_2),(\Delta_{12},0)\ra_{\F}$. Note that
$$
\bigg\{[(\Delta_{13},0)],[(\Delta_{23},0)],[(0,f_1)],
[(0,f_2)],[(0,f_3)],[(0,f_4)]\bigg\}
$$
spans $H^2(L,\F)$.
Since $[(\Delta_{13},0)]+\xi [(0,f_2)]=[(\Delta_{13},\xi f_2)]=[0]$, then $[(0,f_2)]$ is an scalar multiple of $[(\Delta_{13},0)]$
in $H^2(L,\F)$. Note that $\dim H^2=\dim Z^2-\dim B^2=5$. Therefore, 
$$
\bigg \{[(\Delta_{13},0)],[(\Delta_{23},0)],[(0,f_1)],[(0,f_3)],[(0,f_4)]\bigg\}
$$
 forms a basis for $H^2(L,\F)$.

  Note that he group $\Aut_p(L)$ in this case consists of invertible matrices of the form
$$\begin{pmatrix}
 a_{11}  & 0 & 0 & 0 \\
a_{21} & a_{22} & 0 & 0\\
a_{31} & a_{32} & r & 0 \\
a_{41} & a_{42} & a_{11}a_{32}& a_{11}r
\end{pmatrix},$$
where $r = a_{11}a_{22} \neq 0$, $a_{21}=0$ and $a_{11}r=a_{22}^p$.

Let $[\theta]=[(\phi,\omega)]\in H^2(L,\F)$. As in Section \ref{L_{5,3} mod x5 trivial pmap},  it is enough to find $\Aut_p(L)$-representatives of the $\omega$'s.

Let $\omega =\alpha f_1+\gamma f_3+\delta f_4$, for some $\alpha,\gamma,\delta \in \F$. Suppose that $A\omega =\alpha' f_1+\gamma' f_3+\delta' f_4$, for some  $\alpha',\gamma',\delta' \in \F$. We have
\begin{align*}
&A\omega (x_1)=a_{11}^p\alpha+a_{31}^p\gamma +a_{41}^p\delta;\\
&A\omega (x_3)=r^p\gamma +a_{11}^p a_{32}^p\delta;\\
&A\omega (x_4)=a_{11}^p r^p \delta.
\end{align*}
In the matrix form we can write this as
\begin{align*}
\begin{pmatrix} 
\alpha'\\
\gamma'\\
\delta'
\end{pmatrix}=
\begin{pmatrix}
 a_{11}^p& a_{31}^p&a_{41}^p\\
0&r^p&a_{11}^pa_{32}^p\\
0&0&a_{11}^pr^p
\end{pmatrix}
\begin{pmatrix} 
\alpha \\
\gamma \\
\delta
\end{pmatrix}.
\end{align*}

Now we find the representatives of the orbits of this action. Note that we need to have $A\omega(x_2)=0$ which implies that $a_{32}^p\gamma+a_{42}^p\delta=0$.

Let $\nu = \begin{pmatrix}
\alpha \\
\gamma \\
\delta
\end{pmatrix} \in \F^3$. If $\nu= \begin{pmatrix}
0 \\
0 \\
0
\end{pmatrix}$, then $\{\nu \}$ is clearly an $\Aut_p(L)$-orbit. Suppose that $\delta \neq 0$. Then 
\begin{align*}
&\begin{pmatrix}
1&0&-\alpha/\delta\\
0& 1& -\gamma/\delta\\
0&0&1\\
\end{pmatrix}
\begin{pmatrix} 
\alpha\\
\gamma \\
\delta
\end{pmatrix}=
\begin{pmatrix}
0\\
0\\
\delta
\end{pmatrix}.
\end{align*}
 Next, if $\delta=0$, but $\gamma \neq 0$, then
 $$
\begin{pmatrix}
1& -\alpha /\gamma&0\\
0& 1&0\\
0&0&1
\end{pmatrix}
\begin{pmatrix}
\alpha\\
\gamma \\
0
\end{pmatrix}=
\begin{pmatrix}
0\\
\gamma\\
0
\end{pmatrix}.$$
 Finally, if $\gamma =\delta=0$, but $\alpha \neq 0$, then we have
$
\begin{pmatrix}
\alpha\\
0 \\
0
\end{pmatrix}.$
Thus the following elements are $\Aut (L)$-orbit representatives:
\begin{align*}
\begin{pmatrix}
0 \\ 
0\\
0
\end{pmatrix},
\begin{pmatrix}
\alpha\\ 
0\\
0
\end{pmatrix},
\begin{pmatrix}
0 \\ 
\gamma\\
0
\end{pmatrix},
\begin{pmatrix}
0\\ 
0\\
\delta
\end{pmatrix}.
\end{align*}

Therefore, the corresponding restricted Lie algebra structures are as follows: 
\begin{align*}
&K_3^{25}(\xi)=\langle x_1,\ldots,x_5\mid [x_1,x_2]=x_3,[x_1,x_3]=x_4, 
\res x_2=\xi x_4  \rangle;\\
&K_3^{26}(\alpha,\xi)=\langle x_1,\ldots,x_5\mid [x_1,x_2]=x_3,[x_1,x_3]=x_4, 
\res x_1=\alpha x_5, \res x_2=\xi x_4 \rangle;\\
&K_3^{27}(\gamma,\xi)=\langle x_1,\ldots,x_5\mid [x_1,x_2]=x_3,[x_1,x_3]=x_4, 
\res x_2=\xi x_4, \res x_3=\gamma x_5 \rangle;\\
&K_3^{28}(\delta,\xi)=\langle x_1,\ldots,x_5\mid [x_1,x_2]=x_3,[x_1,x_3]=x_4, 
\res x_2=\xi x_4, \res x_4=\delta x_5  \rangle\\
\end{align*}
where $\alpha, \gamma, \delta, \xi \in \F^*$.

\begin{lemma}\label{lem 5326}
We have 
$
K_3^{26}(\alpha,\xi)\cong K_3^{26}(1,\xi),
$
for every $\alpha, \xi\in \F^*$.
\end{lemma}
\begin{proof}
The following automorphism of $L_{5,3}$ gives us the desired isomorphism:
$$\begin{pmatrix}
1&0&0&0&0\\
0& 1&0&0&0\\
0&0&1&0&0\\
0&0&0&1&0\\
0&0&0&0&1/\alpha 
\end{pmatrix}.$$
\end{proof}

\begin{lemma}\label{lem 5327}
We have 
$
K_3^{27}(\gamma,\xi)\cong K_3^{27}(1,\xi),
$
for every $\gamma, \xi\in \F^*$.
\end{lemma}
\begin{proof}
The following automorphism of $L_{5,3}$ gives us the desired isomorphism:
$$\begin{pmatrix}
1&0&0&0&0\\
0& 1&0&0&0\\
0&0&1&0&0\\
0&0&0&1&0\\
0&0&0&0&1/\gamma
\end{pmatrix}.$$
\end{proof}

\begin{lemma}\label{lem 5328}
We have 
$
K_3^{28}(\delta,\xi) \cong  K_3^{28}(1,\xi),
$
for every $\delta, \xi\in \F^*$.
\end{lemma}
\begin{proof}
The following automorphism of $L_{5,3}$ gives us the desired isomorphism:
$$\begin{pmatrix}
1&0&0&0&0\\
0& 1&0&0&0\\
0&0&1&0&0\\
0&0&0&1&0\\
0&0&0&0&1/\delta
\end{pmatrix}.$$
\end{proof}

Therefore, the corresponding restricted Lie algebra structures are as follows: 
\begin{align*}
&K_3^{25}(\xi)=\langle x_1,\ldots,x_5\mid [x_1,x_2]=x_3,[x_1,x_3]=x_4, 
\res x_2=\xi x_4  \rangle;\\
&K_3^{26}(\xi)=\langle x_1,\ldots,x_5\mid [x_1,x_2]=x_3,[x_1,x_3]=x_4, 
\res x_1=x_5, \res x_2=\xi x_4 \rangle;\\
&K_3^{27}(\xi)=\langle x_1,\ldots,x_5\mid [x_1,x_2]=x_3,[x_1,x_3]=x_4, 
\res x_2=\xi x_4, \res x_3=x_5 \rangle;\\
&K_3^{28}(\xi)=\langle x_1,\ldots,x_5\mid [x_1,x_2]=x_3,[x_1,x_3]=x_4, 
\res x_2=\xi x_4, \res x_4= x_5  \rangle\\
\end{align*}
where $\xi \in \F^*$.
Note that $K_3^{25}(\xi)$ is identical to $K_3^{4}(\beta)$, $K_3^{26}(\xi)$ is identical to $K_3^{8}(\beta)$, and $K_3^{27}(\xi)$ is identical to $K_3^{14}(\beta)$.
\begin{lemma}\label{lemma-K327}
We have  $K_3^{28}(\xi_1)\cong K_3^{28}(\xi_2)$ if and only if $\xi_2\xi_1^{-1} \in (\F^*)^2$.
\end{lemma}
\begin{proof}
If  $f:K_3^{28}(\xi_1)\to K_3^{28}(\xi_2)$ is  an isomorphism, then $f(\res x_2)=\res{f(x_2)}$ which implies that 
 $\xi_1a_{11}^2a_{22}x_4=a_{22}^p\xi_2x_4+a_{42}^px_5$. We deduce that 
 $\xi_2\xi_1^{-1}=a_{11}^2a_{22}^{1-p}\in (\F^*)^2.$ 
Conversely, assume  that $\xi_2\xi_1^{-1}=\epsilon^2$ with some $\epsilon \in \F^*$.
Then the following automorphism of $L_{5,3}$
$$\begin{pmatrix}
\epsilon&0&0&0&0\\
0&1&0&0&0\\
0&0&\epsilon&0&0\\
0&0&0&\epsilon^2&0\\
0&0&0&0&\epsilon^{2p}
\end{pmatrix}$$ 
is an isomorphism between $K_3^{28}(\xi_1)$ and $K_3^{28}(\xi_2)$.
\end{proof}

\subsection{Extensions of $(L,  \res x_3 = x_4)$}
 First, we find a basis for $Z^2(L,\F)$. Let $(\phi,\omega)=(a\Delta _{12}+b\Delta_{13}+c\Delta_{14}+d\Delta_{23}+e\Delta_{24}+f\Delta_{34},\alpha f_1+\beta f_2+ \gamma f_3+\delta f_4)\in Z^2(L,\F)$. Then we must have $\delta^2\phi(x,y,z) =0$ and $\phi(x,\res y)=0$, for all $x,y,z \in L$. Therefore, we have
\begin{align*}
&0=(\delta^2\phi)(x_1,x_2,x_4)=\phi([x_1,x_2],x_4)+\phi([x_2,x_4],x_1)+\phi([x_4,x_1],x_2)=\phi(x_3,x_4), \text{ and }\\
&0=(\delta^2\phi)(x_1,x_2,x_3)=\phi([x_1,x_2],x_3)+\phi([x_2,x_3],x_1)+\phi([x_3,x_1],x_2)=\phi(x_2,x_4).
\end{align*}
Thus, we get $f=e=0$.
Also, we have $\phi(x,\res y)=0$. Therefore, $\phi(x,x_4)=0$, for all $x\in L$ and hence $\phi(x_1,x_4)=\phi(x_2,x_4)=\phi(x_3,x_4)=0$ which implies that $c=e=f=0$. Therefore, $Z^2(L,\F)$
has a basis consisting of:
$$
 (\Delta_{12},0),(\Delta_{13},0),(\Delta_{23},0),(0,f_1),
(0,f_2),(0,f_3),(0,f_4).
$$
Next, we find a basis for $B^2(L,\F)$. Let $(\phi,\omega)\in B^2(L,\F)$. Since $B^2(L,\F)\subseteq Z^2(L,\F)$, we have $(\phi,\omega)=(a\Delta _{12}+b\Delta_{13}+c\Delta_{23},\alpha f_1+\beta f_2+ \gamma f_3+\delta f_4)$. Note that there exists a linear  map $\psi:L\to \F$ such that $\delta^1\psi(x,y)=\phi(x,y)$ and $\tilde \psi(x)=\omega(x)$, for all $x,y \in L$. So, we have
\begin{align*}
&a=\phi(x_1,x_2)=\delta^1\psi(x_1,x_2)=\psi([x_1,x_2])=\psi(x_3), \text{ and }\\
&b=\phi(x_1,x_3)=\delta^1\psi(x_1,x_3)=\psi([x_1,x_3])=\psi(x_4), \text{ and }\\
&c=\phi(x_2,x_3)=\delta^1\psi(x_2,x_3)=\psi([x_2,x_3])=0.
\end{align*}
Also, we have
\begin{align*}
\gamma=\omega(x_3)&=\tilde \psi(x_3)=\psi(\res x_3)=\psi(x_4), \text{ and }\\
\alpha=\omega(x_1)&=\tilde \psi(x_1)=\psi(\res x_1)=0.
\end{align*}
Similarly, we can show that $\beta=\delta=0$. Note that $\psi(x_4)=b=\gamma$.  Therefore, $(\phi,\omega)=(a\Delta_{12}+b\Delta_{13},bf_3)$ and hence $B^2(L,\F)=\la(\Delta_{13},f_3),(\Delta_{12},0)\ra_{\F}$. Note that
$$
\bigg\{[(\Delta_{13},0)],[(\Delta_{23},0)],[(0,f_1)],
[(0,f_2)],[(0,f_3)],[(0,f_4)]\bigg\}
$$
spans $H^2(L,\F)$.
Since $[(\Delta_{13},0)]+[(0,f_3)]=[(\Delta_{13},f_3)]=[0]$, then $[(0,f_3)]$ is an scalar multiple of $[(\Delta_{13},0)]$
in $H^2(L,\F)$. Note that $\dim H^2=\dim Z^2-\dim B^2=5$. Therefore, 
$$
\bigg \{[(\Delta_{13},0)],[(\Delta_{23},0)],[(0,f_1)],[(0,f_2)],[(0,f_4)]\bigg\}
$$
 forms a basis for $H^2(L,\F)$.
 
  Note that he group $\Aut_p(L)$ in this case consists of invertible matrices of the form
$$\begin{pmatrix}
 a_{11}  & 0 & 0 & 0 \\
a_{21} & a_{22} & 0 & 0\\
a_{31} & a_{32} & r & 0 \\
a_{41} & a_{42} & a_{11}a_{32}& a_{11}r
\end{pmatrix},$$
where $r = a_{11}a_{22} \neq 0$, $a_{31}=a_{32}=0$ and $a_{11}r=r^p$.

Let $[\theta]=[(\phi,\omega)]\in H^2(L,\F)$. As in Section \ref{L_{5,3} mod x5 trivial pmap},  it is enough to find $\Aut_p(L)$-representatives of the $\omega$'s.
Let $\omega =\alpha f_1+\beta f_2+\delta f_4$, for some $\alpha,\beta,\delta \in \F$. Suppose that $A\omega =\alpha' f_1+\beta' f_2+\delta' f_4$, for some  $\alpha',\beta',\delta' \in \F$. We have
\begin{align*}
&A\omega (x_1)=a_{11}^p\alpha+a_{21}^p\beta+a_{41}^p\delta;\\
&A\omega (x_2)=a_{22}^p\beta+a_{42}^p\delta;\\
&A\omega (x_4)=a_{11}^p r^p \delta.
\end{align*}
In the matrix form we can write this as
\begin{align*}
\begin{pmatrix} 
\alpha'\\
\beta'\\
\delta'
\end{pmatrix}=
\begin{pmatrix}
 a_{11}^p& a_{21}^p&a_{41}^p\\
0&a_{22}^p&a_{42}^p\\
0&0&a_{11}^pr^p
\end{pmatrix}
\begin{pmatrix} 
\alpha \\
\beta\\
\delta
\end{pmatrix}.
\end{align*}
Note that we need to have $A\omega(x_3)=0$ which implies that $a_{11}^pa_{32}^p\delta=0$.
Now we find the representatives of the orbits of this action. Note that we take $a_{32}=0$. Therefore, $a_{11}^pa_{32}^p\delta=0$.
Let $\nu = \begin{pmatrix}
\alpha \\
\beta\\
\delta
\end{pmatrix} \in \F^3$. If $\nu= \begin{pmatrix}
0 \\
0 \\
0
\end{pmatrix}$, then $\{\nu \}$ is clearly an $\Aut_p(L)$-orbit. Suppose that $\delta \neq 0$. Then 
\begin{align*}
&\begin{pmatrix}
1&0&-\alpha/\delta\\
0& 1& -\beta/\delta\\
0&0&1\\
\end{pmatrix}
\begin{pmatrix} 
\alpha\\
\beta\\
\delta
\end{pmatrix}=
\begin{pmatrix}
0\\
0\\
\delta
\end{pmatrix}, \text{ and } \\
&\begin{pmatrix}
(1/\delta)^\frac{p-1}{p}& 0&0\\
0&(1/\delta)^\frac{2-p}{p}& 0\\
0&0&1/\delta
\end{pmatrix}
\begin{pmatrix}
0 \\
0 \\
\delta
\end{pmatrix}=
\begin{pmatrix} 
0\\
0\\
1
\end{pmatrix}.
\end{align*}
 Next, if $\delta=0$, but $\beta \neq 0$, then
 $$
\begin{pmatrix}
1& -\alpha /\beta&0\\
0& 1&0\\
0&0&1
\end{pmatrix}
\begin{pmatrix}
\alpha\\
\beta\\
0
\end{pmatrix}=
\begin{pmatrix}
0\\
\beta\\
0
\end{pmatrix}. $$

 Finally, if $\beta =\delta=0$, but $\alpha \neq 0$, then we have
$
\begin{pmatrix}
\alpha\\
0 \\
0
\end{pmatrix}.$
Thus the following elements are $\Aut (L)$-orbit representatives:
\begin{align*}
\begin{pmatrix}
0\\ 
0\\
0
\end{pmatrix},
\begin{pmatrix}
\alpha \\ 
0\\
0
\end{pmatrix},
\begin{pmatrix}
0 \\ 
\beta\\
0
\end{pmatrix},
\begin{pmatrix}
0\\ 
0\\
1
\end{pmatrix}.
\end{align*}
Therefore, the corresponding restricted Lie algebra structures are as follows:
\begin{align*}
&K_3^{29}=\langle x_1,\ldots,x_5\mid [x_1,x_2]=x_3,[x_1,x_3]=x_4, 
\res x_3= x_4 \rangle;\\
&K_3^{30}(\alpha)=\langle x_1,\ldots,x_5\mid [x_1,x_2]=x_3,[x_1,x_3]=x_4, 
\res x_1=\alpha x_5, \res x_3= x_4 \rangle;\\
&K_3^{31}(\beta)=\langle x_1,\ldots,x_5\mid [x_1,x_2]=x_3,[x_1,x_3]=x_4, 
\res x_2=\beta x_5, \res x_3=x_4 \rangle;\\
&K_3^{32}=\langle x_1,\ldots,x_5\mid [x_1,x_2]=x_3,[x_1,x_3]=x_4, 
\res x_3=x_4, \res x_4=x_5 \rangle.
\end{align*}

\begin{lemma}
We have
$K_3^{30}(\alpha)\cong \langle x_1,\ldots,x_5\mid [x_1,x_2]=x_3,[x_1,x_3]=x_4, 
\res x_1= x_5, \res x_3= x_4 \rangle$,
for for every $\alpha\in \F^*$.
\end{lemma}
\begin{proof}
The following automorphism of $L_{5,3}$ gives us the desired isomorphism:
$$\begin{pmatrix}
1&0&0&0&0\\
0& 1&0&0&0\\
0&0&1&0&0\\
0&0&0&1&0\\
0&0&0&0&1/\alpha\\
\end{pmatrix}.$$
\end{proof}

\begin{lemma}
We have
$K_3^{31}(\beta)\cong \langle x_1,\ldots,x_5\mid [x_1,x_2]=x_3,[x_1,x_3]=x_4, 
\res x_2=x_5, \res x_3=x_4 \rangle$,
for for every $\beta\in \F^*$.
\end{lemma}
\begin{proof}
The following automorphism of $L_{5,3}$ gives us the desired isomorphism:
$$\begin{pmatrix}
1&0&0&0&0\\
0& 1&0&0&0\\
0&0&1&0&0\\
0&0&0&1&0\\
0&0&0&0&1/\beta\\
\end{pmatrix}.$$
\end{proof}

Therefore, the corresponding restricted Lie algebra structures are as follows:
\begin{align*}
&K_3^{29}=\langle x_1,\ldots,x_5\mid [x_1,x_2]=x_3,[x_1,x_3]=x_4, 
\res x_3= x_4 \rangle;\\
&K_3^{30}=\langle x_1,\ldots,x_5\mid [x_1,x_2]=x_3,[x_1,x_3]=x_4, 
\res x_1= x_5, \res x_3= x_4 \rangle;\\
&K_3^{31}=\langle x_1,\ldots,x_5\mid [x_1,x_2]=x_3,[x_1,x_3]=x_4, 
\res x_2= x_5, \res x_3=x_4 \rangle;\\
&K_3^{32}=\langle x_1,\ldots,x_5\mid [x_1,x_2]=x_3,[x_1,x_3]=x_4, 
\res x_3=x_4, \res x_4=x_5 \rangle.
\end{align*}
\subsection{A list of restricted Lie algebra structures on $L_{5,3}$}

Therefore, the list of all (possibly redundant) restricted Lie algebra structures on $L_{5,3}$ is as follows:
\begin{align*}
&K_3^{1}=\langle x_1,\ldots,x_5\mid [x_1,x_2]=x_3,[x_1,x_3]=x_4 \rangle;\\
&K_3^{2}=\langle x_1,\ldots,x_5\mid [x_1,x_2]=x_3,[x_1,x_3]=x_4, 
 \res x_3=x_4, \res x_5=x_4 \rangle;\\
&K_3^{3}=\langle x_1,\ldots,x_5\mid [x_1,x_2]=x_3,[x_1,x_3]=x_4, 
 \res x_3=x_4 \rangle;\\
&K_3^{4}(\beta)=\langle x_1,\ldots,x_5\mid [x_1,x_2]=x_3,[x_1,x_3]=x_4,
 \res x_2=\beta x_4 \rangle;\\
&K_3^{5}=\langle x_1,\ldots,x_5\mid [x_1,x_2]=x_3,[x_1,x_3]=x_4, 
 \res x_1=x_4 \rangle;\\
&K_3^{6}=\langle x_1,\ldots,x_5\mid [x_1,x_2]=x_3,[x_1,x_3]=x_4, 
 \res x_5=x_4 \rangle;\\
 &K_3^{7}=\langle x_1,\ldots,x_5\mid [x_1,x_2]=x_3,[x_1,x_3]=x_4 ,
\res x_1=x_5 \rangle;\\
&K_3^{8}(\beta)=\langle x_1,\ldots,x_5\mid [x_1,x_2]=x_3,[x_1,x_3]=x_4, 
\res x_1=x_5, \res x_2=\beta x_4 \rangle;\\
&K_3^{9}=\langle x_1,\ldots,x_5\mid [x_1,x_2]=x_3,[x_1,x_3]=x_4, 
\res x_1=x_5, \res x_3=x_4 \rangle;\\
&K_3^{10}=\langle x_1,\ldots,x_5\mid [x_1,x_2]=x_3,[x_1,x_3]=x_4,
\res x_1=x_5, \res x_5=x_4 \rangle;\\
&K_3^{11}(\gamma)=\langle x_1,\ldots,x_5\mid [x_1,x_2]=x_3,[x_1,x_3]=x_4,
\res x_1=x_5, \res x_3=\gamma x_4 ,\res x_5=x_4 \rangle;\\
&K_3^{12}=\langle x_1,\ldots,x_5\mid [x_1,x_2]=x_3,[x_1,x_3]=x_4 ,
\res x_3=x_5 \rangle;\\
&K_3^{13}=\langle x_1,\ldots,x_5\mid [x_1,x_2]=x_3,[x_1,x_3]=x_4, 
\res x_1=x_4, \res x_3=x_5 \rangle;\\
&K_3^{14}(\beta)=\langle x_1,\ldots,x_5\mid [x_1,x_2]=x_3,[x_1,x_3]=x_4, 
\res x_2=\beta x_4, \res x_3=x_5 \rangle;\\
&K_3^{15}=\langle x_1,\ldots,x_5\mid [x_1,x_2]=x_3,[x_1,x_3]=x_4,
\res x_3=x_5, \res x_5=x_4 \rangle;\\
&K_3^{16}=\langle x_1,\ldots,x_5\mid [x_1,x_2]=x_3,[x_1,x_3]=x_4  \rangle;\\
&K_3^{17}=\langle x_1,\ldots,x_5\mid [x_1,x_2]=x_3,[x_1,x_3]=x_4, 
\res x_1=x_5 \rangle;\\
&K_3^{18}=\langle x_1,\ldots,x_5\mid [x_1,x_2]=x_3,[x_1,x_3]=x_4, 
\res x_2=x_5  \rangle;\\
&K_3^{19}=\langle x_1,\ldots,x_5\mid [x_1,x_2]=x_3,[x_1,x_3]=x_4, 
\res x_3=x_5 \rangle;\\
&K_3^{20}=\langle x_1,\ldots,x_5\mid [x_1,x_2]=x_3,[x_1,x_3]=x_4, 
\res x_4=x_5 \rangle;\\
&K_3^{21}=\langle x_1,\ldots,x_5\mid [x_1,x_2]=x_3,[x_1,x_3]=x_4 ,
\res x_1=x_4 \rangle;\\
&K_3^{22}=\langle x_1,\ldots,x_5\mid [x_1,x_2]=x_3,[x_1,x_3]=x_4, 
\res x_1=x_4, \res x_2=x_5 \rangle;\\
&K_3^{23}=\langle x_1,\ldots,x_5\mid [x_1,x_2]=x_3,[x_1,x_3]=x_4, 
\res x_1=x_4, \res x_3= x_5 \rangle;\\
&K_3^{24}=\langle x_1,\ldots,x_5\mid [x_1,x_2]=x_3,[x_1,x_3]=x_4, 
\res x_1=x_4, \res x_4=x_5 \rangle;\\
&K_3^{25}(\xi)=\langle x_1,\ldots,x_5\mid [x_1,x_2]=x_3,[x_1,x_3]=x_4, 
\res x_2=\xi x_4  \rangle;\\
&K_3^{26}(\xi)=\langle x_1,\ldots,x_5\mid [x_1,x_2]=x_3,[x_1,x_3]=x_4, 
\res x_1= x_5, \res x_2=\xi x_4 \rangle;\\
&K_3^{27}(\xi)=\langle x_1,\ldots,x_5\mid [x_1,x_2]=x_3,[x_1,x_3]=x_4, 
\res x_2=\xi x_4, \res x_3=x_5 \rangle;\\
&K_3^{28}(\xi)=\langle x_1,\ldots,x_5\mid [x_1,x_2]=x_3,[x_1,x_3]=x_4, 
\res x_2=\xi x_4, \res x_4= x_5  \rangle;\\
&K_3^{29}=\langle x_1,\ldots,x_5\mid [x_1,x_2]=x_3,[x_1,x_3]=x_4, 
\res x_3= x_4 \rangle;\\
&K_3^{30}=\langle x_1,\ldots,x_5\mid [x_1,x_2]=x_3,[x_1,x_3]=x_4, 
\res x_1=x_5, \res x_3= x_4 \rangle;\\
&K_3^{31}=\langle x_1,\ldots,x_5\mid [x_1,x_2]=x_3,[x_1,x_3]=x_4, 
\res x_2= x_5, \res x_3=x_4 \rangle;\\
&K_3^{32}=\langle x_1,\ldots,x_5\mid [x_1,x_2]=x_3,[x_1,x_3]=x_4, 
\res x_3=x_4, \res x_4=x_5 \rangle.
\end{align*}
\section{Detecting isomorphisms}
We can easily see that some of the algebras above are identical.

\begin{theorem}\label{thm L53}
The list of all restricted Lie algebra structures on $L_{5,3}$, up to isomorphism, is as follows:
\begin{align*}
&L_{5,3}^{1}=\langle x_1,\ldots,x_5\mid [x_1,x_2]=x_3,[x_1,x_3]=x_4 \rangle;\\
&L_{5,3}^{2}=\langle x_1,\ldots,x_5\mid [x_1,x_2]=x_3,[x_1,x_3]=x_4, 
 \res x_3=x_4, \res x_5=x_4 \rangle;\\
&L_{5,3}^{3}=\langle x_1,\ldots,x_5\mid [x_1,x_2]=x_3,[x_1,x_3]=x_4, 
 \res x_3=x_4 \rangle;\\
&L_{5,3}^{4}(\beta)=\langle x_1,\ldots,x_5\mid [x_1,x_2]=x_3,[x_1,x_3]=x_4,
 \res x_2=\beta x_4 \rangle;\\
&L_{5,3}^{5}=\langle x_1,\ldots,x_5\mid [x_1,x_2]=x_3,[x_1,x_3]=x_4, 
 \res x_1=x_4 \rangle;\\
&L_{5,3}^{6}=\langle x_1,\ldots,x_5\mid [x_1,x_2]=x_3,[x_1,x_3]=x_4, 
 \res x_5=x_4 \rangle;\\
&L_{5,3}^{7}=\langle x_1,\ldots,x_5\mid [x_1,x_2]=x_3,[x_1,x_3]=x_4 ,
\res x_1=x_5 \rangle;\\
&L_{5,3}^{8}(\beta)=\langle x_1,\ldots,x_5\mid [x_1,x_2]=x_3,[x_1,x_3]=x_4, 
\res x_1=x_5, \res x_2=\beta x_4 \rangle;\\
&L_{5,3}^{9}=\langle x_1,\ldots,x_5\mid [x_1,x_2]=x_3,[x_1,x_3]=x_4, 
\res x_1=x_5, \res x_3=x_4 \rangle;\\
&L_{5,3}^{10}=\langle x_1,\ldots,x_5\mid [x_1,x_2]=x_3,[x_1,x_3]=x_4,
\res x_1=x_5, \res x_5=x_4 \rangle;\\
&L_{5,3}^{11}=\langle x_1,\ldots,x_5\mid [x_1,x_2]=x_3,[x_1,x_3]=x_4 ,
\res x_3=x_5 \rangle;\\
&L_{5,3}^{12}=\langle x_1,\ldots,x_5\mid [x_1,x_2]=x_3,[x_1,x_3]=x_4, 
\res x_1=x_4, \res x_3=x_5 \rangle;\\
&L_{5,3}^{13}(\beta)=\langle x_1,\ldots,x_5\mid [x_1,x_2]=x_3,[x_1,x_3]=x_4, 
\res x_2=\beta x_4, \res x_3=x_5 \rangle;\\
&L_{5,3}^{14}=\langle x_1,\ldots,x_5\mid [x_1,x_2]=x_3,[x_1,x_3]=x_4,
\res x_3=x_5, \res x_5=x_4 \rangle;\\
&L_{5,3}^{15}=\langle x_1,\ldots,x_5\mid [x_1,x_2]=x_3,[x_1,x_3]=x_4, 
\res x_2=x_5  \rangle;\\
&L_{5,3}^{16}=\langle x_1,\ldots,x_5\mid [x_1,x_2]=x_3,[x_1,x_3]=x_4, 
\res x_4=x_5 \rangle;\\
&L_{5,3}^{17}=\langle x_1,\ldots,x_5\mid [x_1,x_2]=x_3,[x_1,x_3]=x_4, 
\res x_1=x_4, \res x_2=x_5 \rangle;\\
&L_{5,3}^{18}=\langle x_1,\ldots,x_5\mid [x_1,x_2]=x_3,[x_1,x_3]=x_4, 
\res x_1=x_4, \res x_4=x_5 \rangle;\\
&L_{5,3}^{19}(\xi)=\langle x_1,\ldots,x_5\mid [x_1,x_2]=x_3,[x_1,x_3]=x_4, 
\res x_2=\xi x_4, \res x_4=x_5  \rangle;\\
&L_{5,3}^{20}=\langle x_1,\ldots,x_5\mid [x_1,x_2]=x_3,[x_1,x_3]=x_4, 
\res x_2=x_5, \res x_3=x_4 \rangle;\\
&L_{5,3}^{21}=\langle x_1,\ldots,x_5\mid [x_1,x_2]=x_3,[x_1,x_3]=x_4, 
\res x_3=x_4, \res x_4=x_5 \rangle;\\
&L_{5,3}^{22}(\gamma)=\langle x_1,\ldots,x_5\mid [x_1,x_2]=x_3,[x_1,x_3]=x_4,
\res x_1=x_5, \res x_3=\gamma x_4,\res x_5=x_4 \rangle
\end{align*}
where $\beta, \xi \in T_2$.
\end{theorem}

In the remaining of this section we establish that the algebras given in Theorem \ref{thm L53} are pairwise non-isomorphic, thereby
completing the proof of Theorem \ref{thm L53}.

It is clear that $L_{5,3}^1$ is not isomorphic to the other restricted Lie algebras.
We claim that $L_{5,3}^2$ and $L_{5,3}^{3}$ are not isomorphic. Suppose to the contrary that there exists an isomorphism $A:L_{5,3}^2\to L_{5,3}^3$. Then 
\begin{align*}
&A(\res x_5)=\res {A(x_5)}\\
&A(x_4)=\res {(a_{45}x_4+a_{55}x_5)}\\
 &a_{11}^2a_{22}x_4=0.
 \end{align*}
Therefore, $a_{11}=0$ or $a_{22}=0$ which is a contradiction.
  
Next, we claim that $L_{5,3}^2$ and $L_{5,3}^{4}(\beta)$ are not isomorphic. Suppose to the contrary that there exists an isomorphism $A:L_{5,3}^2\to L_{5,3}^4(\beta)$. Then 
\begin{align*}
&A(\res x_5)=\res {A(x_5)}\\
&A(x_4)=\res {(a_{45}x_4+a_{55}x_5)}\\
 &a_{11}^2a_{22}x_4=0.
 \end{align*}
Therefore, $a_{11}=0$ or $a_{22}=0$ which is a contradiction.
 
Next, we claim that $L_{5,3}^2$ and $L_{5,3}^{5}$ are not isomorphic. Suppose to the contrary that there exists an isomorphism $A:L_{5,3}^2\to L_{5,3}^5$. Then 
\begin{align*}
&A(\res x_5)=\res {A(x_5)}\\
&A(x_4)=\res {(a_{45}x_4+a_{55}x_5)}\\
 &a_{11}^2a_{22}x_4=0.
 \end{align*}
Therefore, $a_{11}=0$ or $a_{22}=0$ which is a contradiction.

Next, we claim that $L_{5,3}^2$ and $L_{5,3}^6$ are not isomorphic. Suppose to the contrary that there exists an isomorphism $A:L_{5,3}^2\to L_{5,3}^6$. Then 
\begin{align*}
&A(\res x_3)=\res {A(x_3)}\\
&A(x_4)=\res {(a_{11}a_{22}x_3+a_{11}a_{32}x_4)}\\
 &a_{11}^2a_{22}x_4=0,\\ 
 \end{align*}
 which implies that $a_{11}^2a_{22}=0$. Therefore, we have $a_{11}=0$ or $a_{22}=0$ which is a contradiction.

Next, we claim that $L_{5,3}^2$ and $L_{5,3}^{7}$ are not isomorphic. Suppose to the contrary that there exists an isomorphism $A:L_{5,3}^2\to L_{5,3}^7$. Then 
\begin{align*}
&A(\res x_5)=\res {A(x_5)}\\
&A(x_4)=\res {(a_{45}x_4+a_{55}x_5)}\\
 &a_{11}^2a_{22}x_4=0.
 \end{align*}
Therefore, $a_{11}=0$ or $a_{22}=0$ which is a contradiction.
 
 Next, we claim that $L_{5,3}^2$ and $L_{5,3}^{11}$ are not isomorphic. Suppose to the contrary that there exists an isomorphism $A:L_{5,3}^2\to L_{5,3}^{11}$. Then 
\begin{align*}
&A(\res x_5)=\res {A(x_5)}\\
&A(x_4)=\res {(a_{45}x_4+a_{55}x_5)}\\
 &a_{11}^2a_{22}x_4=0.
 \end{align*}
Therefore, $a_{11}=0$ or $a_{22}=0$ which is a contradiction.
 
 Next, we claim that $L_{5,3}^2$ and $L_{5,3}^{15}$ are not isomorphic. Suppose to the contrary that there exists an isomorphism $A:L_{5,3}^2\to L_{5,3}^{15}$. Then 
\begin{align*}
&A(\res x_5)=\res {A(x_5)}\\
&A(x_4)=\res {(a_{45}x_4+a_{55}x_5)}\\
 &a_{11}^2a_{22}x_4=0.
 \end{align*}
Therefore, $a_{11}=0$ or $a_{22}=0$ which is a contradiction.
 
Next, we claim that $L_{5,3}^2$ and $L_{5,3}^{16}$ are not isomorphic. Suppose to the contrary that there exists an isomorphism $A:L_{5,3}^2\to L_{5,3}^{16}$. Then 
\begin{align*}
&A(\res x_4)=\res {A(x_4)}\\
&0=\res {(a_{11}^2a_{22}x_4)}\\
&0=a_{11}^{2p}a_{22}^px_5. 
 \end{align*}
 Therefore, $a_{11}=0$ or $a_{22}=0$ which is a contradiction.
 
It is clear that $L_{5,3}^2$ is not isomorphic to the other restricted Lie algebras. 
 
Next, we claim that $L_{5,3}^3$ and $L_{5,3}^{4}(\beta)$ are not isomorphic. Suppose to the contrary that there exists an isomorphism $A:L_{5,3}^3\to L_{5,3}^{4}(\beta)$. Then 
\begin{align*}
&A(\res x_2)=\res {A(x_2)}\\
&0=\res {(a_{22}x_2+a_{32}x_3+a_{42}x_4+a_{52}x_5)}\\
&0=a_{22}^p\beta x_4,
 \end{align*}
 which implies that $a_{22}^p\beta=0$. Since, $\beta \neq 0$, we have $a_{22}=0$ which is a contradiction.
 
Next, we claim that $L_{5,3}^3$ and $L_{5,3}^{5}$ are not isomorphic. Suppose to the contrary that there exists an isomorphism $A:L_{5,3}^3\to L_{5,3}^{5}$. Then 
\begin{align*}
&A(\res x_1)=\res {A(x_1)}\\
&0=\res {(a_{11}x_1+a_{21}x_2+a_{31}x_3+a_{41}x_4+a_{51}x_5)}\\
 &0=a_{11}^px_4.
 \end{align*}
Therefore, $a_{11}=0$ which is a contradiction.

Next, we claim that $L_{5,3}^3$ and $L_{5,3}^{6}$ are not isomorphic. Suppose to the contrary that there exists an isomorphism $A:L_{5,3}^3\to L_{5,3}^{6}$. Then 
\begin{align*}
&A(\res x_3)=\res {A(x_3)}\\
&A(x_4)=\res {(a_{11}a_{22}x_3+a_{11}a_{32}x_4)}\\
 &a_{11}^2a_{22}x_4=0,\\ 
 \end{align*}
 which implies that $a_{11}^2a_{22}=0$. Therefore, we have $a_{11}=0$ or $a_{22}=0$ which is a contradiction.
 
Next, we claim that $L_{5,3}^3$ and $L_{5,3}^{7}$ are not isomorphic. Suppose to the contrary that there exists an isomorphism $A:L_{5,3}^3\to L_{5,3}^{7}$. Then 
\begin{align*}
&A(\res x_1)=\res {A(x_1)}\\
&0=\res {(a_{11}x_1+a_{21}x_2+a_{31}x_3+a_{41}x_4+a_{51}x_5)}\\
 &0=a_{11}^px_5.
 \end{align*}
Therefore, $a_{11}=0$ which is a contradiction.
 
Next, we claim that $L_{5,3}^3$ and $L_{5,3}^{11}$ are not isomorphic. Suppose to the contrary that there exists an isomorphism $A:L_{5,3}^3\to L_{5,3}^{11}$. Then 
\begin{align*}
&A(\res x_3)=\res {A(x_3)}\\
&A(x_4)=\res {(a_{11}a_{22}x_3+a_{11}a_{32}x_4)}\\
 &a_{11}^2a_{22}x_4=a_{11}^pa_{22}^px_5,\\ 
 \end{align*}
 which implies that $a_{11}^2a_{22}=0$ and $a_{11}^pa_{22}^p=0$. Therefore, we have $a_{11}=0$ or $a_{22}=0$ which is a contradiction.
 
Next, we claim that $L_{5,3}^3$ and $L_{5,3}^{15}$ are not isomorphic. Suppose to the contrary that there exists an isomorphism $A:L_{5,3}^3\to L_{5,3}^{15}$. Then 
\begin{align*}
&A(\res x_2)=\res {A(x_2)}\\
&0=\res {(a_{22}x_2+a_{32}x_3+a_{42}x_4+a_{52}x_5)}\\
&0=a_{22}^p x_5.
 \end{align*}
Therefore, $a_{22}=0$ which is a contradiction.
 
Next, we claim that $L_{5,3}^3$ and $L_{5,3}^{16}$ are not isomorphic. Suppose to the contrary that there exists an isomorphism $A:L_{5,3}^3\to L_{5,3}^{16}$. Then 
\begin{align*}
&A(\res x_4)=\res {A(x_4)}\\
&0=\res {(a_{11}^2a_{22}x_4)}\\
&0=a_{11}^{2p}a_{22}^px_5. 
 \end{align*}
 Therefore, $a_{11}=0$ or $a_{22}=0$ which is a contradiction.
 
 It is clear that $L_{5,3}^3$ is not isomorphic to the other restricted Lie algebras.

Next, we claim that $L_{5,3}^4(\beta)$ and $L_{5,3}^{5}$ are not isomorphic. Suppose to the contrary that there exists an isomorphism $A:L_{5,3}^4(\beta)\to L_{5,3}^{5}$. Then 
\begin{align*}
&A(\res x_1)=\res {A(x_1)}\\
&0=\res {(a_{11}x_1+a_{21}x_2+a_{31}x_3+a_{41}x_4+a_{51}x_5)}\\
 &0=a_{11}^px_4.
 \end{align*}
Therefore, $a_{11}=0$ which is a contradiction.
  
Next, we claim that $L_{5,3}^4(\beta)$ and $L_{5,3}^{6}$ are not isomorphic. Suppose to the contrary that there exists an isomorphism $A:L_{5,3}^6\to L_{5,3}^{4}(\beta)$. Then 
\begin{align*}
&A(\res x_5)=\res {A(x_5)}\\
&A(x_4)=\res {(a_{45}x_4+a_{55}x_5)}\\
 &a_{11}^2a_{22}x_4=0.
 \end{align*}
Therefore, $a_{11}=0$ or $a_{22}=0$ which is a contradiction.
  
Next, we claim that $L_{5,3}^4(\beta)$ and $L_{5,3}^{7}$ are not isomorphic. Suppose to the contrary that there exists an isomorphism $A:L_{5,3}^4(\beta)\to L_{5,3}^{7}$. Then 
\begin{align*}
&A(\res x_1)=\res {A(x_1)}\\
&0=\res {(a_{11}x_1+a_{21}x_2+a_{31}x_3+a_{41}x_4+a_{51}x_5)}\\
 &0=a_{11}^px_5.
 \end{align*}
Therefore, $a_{11}=0$ which is a contradiction.
  
Next, we claim that $L_{5,3}^4(\beta)$ and $L_{5,3}^{11}$ are not isomorphic. Suppose to the contrary that there exists an isomorphism $A:L_{5,3}^4(\beta)\to L_{5,3}^{11}$. Then 
\begin{align*}
&A(\res x_3)=\res {A(x_3)}\\
&0=\res {(a_{11}a_{22}x_3+a_{11}a_{32}x_4)}\\
 &0=a_{11}^pa_{22}^px_5,\\ 
 \end{align*}
 which implies that $a_{11}^pa_{22}^p=0$. Therefore, we have $a_{11}=0$ or $a_{22}=0$ which is a contradiction.
  
Next, we claim that $L_{5,3}^4(\beta)$ and $L_{5,3}^{15}$ are not isomorphic. Suppose to the contrary that there exists an isomorphism $A:L_{5,3}^4(\beta)\to L_{5,3}^{15}$. Then 
\begin{align*}
&A(\res x_2)=\res {A(x_2)}\\
&A(x_4)=\res {(a_{22}x_2+a_{32}x_3+a_{42}x_4+a_{52}x_5)}\\
&a_{11}^2a_{22}x_4=a_{22}^p x_5.
 \end{align*}
Therefore, $a_{22}=0$ which is a contradiction.
 
Next, we claim that $L_{5,3}^4(\beta)$ and $L_{5,3}^{16}$ are not isomorphic. Suppose to the contrary that there exists an isomorphism $A:L_{5,3}^4(\beta)\to L_{5,3}^{16}$. Then 
\begin{align*}
&A(\res x_4)=\res {A(x_4)}\\
&0=\res {(a_{11}^2a_{22}x_4)}\\
&0=a_{11}^{2p}a_{22}^px_5. 
 \end{align*}
 Therefore, $a_{11}=0$ or $a_{22}=0$ which is a contradiction.
 
  It is clear that $L_{5,3}^4(\beta)$ is not isomorphic to the other restricted Lie algebras. 
 
Next, we claim that $L_{5,3}^5$ and $L_{5,3}^{6}$ are not isomorphic. Suppose to the contrary that there exists an isomorphism $A:L_{5,3}^6\to L_{5,3}^5$. Then 
\begin{align*}
&A(\res x_5)=\res {A(x_5)}\\
&A(x_4)=\res {(a_{45}x_4+a_{55}x_5)}\\
 &a_{11}^2a_{22}x_4=0.
 \end{align*}
Therefore, $a_{11}=0$ or $a_{22}=0$ which is a contradiction.
  
Next, we claim that $L_{5,3}^5$ and $L_{5,3}^{7}$ are not isomorphic. Suppose to the contrary that there exists an isomorphism $A:L_{5,3}^5\to L_{5,3}^{7}$. Then 
\begin{align*}
&A(\res x_1)=\res {A(x_1)}\\
&A(x_4)=\res {(a_{11}x_1+a_{21}x_2+a_{31}x_3+a_{41}x_4+a_{51}x_5)}\\
 &a_{11}^2a_{22}x_4=a_{11}^px_5.
 \end{align*}
Therefore, $a_{11}=0$ which is a contradiction.
   
Next, we claim that $L_{5,3}^5$ and $L_{5,3}^{11}$ are not isomorphic. Suppose to the contrary that there exists an isomorphism $A:L_{5,3}^5\to L_{5,3}^{11}$. Then 
\begin{align*}
&A(\res x_3)=\res {A(x_3)}\\
&0=\res {(a_{11}a_{22}x_3+a_{11}a_{32}x_4)}\\
 &0=a_{11}^pa_{22}^px_5,\\ 
 \end{align*}
 which implies that $a_{11}^pa_{22}^p=0$. Therefore, we have $a_{11}=0$ or $a_{22}=0$ which is a contradiction.
   
Next, we claim that $L_{5,3}^5$ and $L_{5,3}^{15}$ are not isomorphic. Suppose to the contrary that there exists an isomorphism $A:L_{5,3}^5\to L_{5,3}^{15}$. Then 
\begin{align*}
&A(\res x_2)=\res {A(x_2)}\\
&0=\res {(a_{22}x_2+a_{32}x_3+a_{42}x_4+a_{52}x_5)}\\
&0=a_{22}^p x_5.
 \end{align*}
Therefore, $a_{22}=0$ which is a contradiction.
  
Next, we claim that $L_{5,3}^5$ and $L_{5,3}^{16}$ are not isomorphic. Suppose to the contrary that there exists an isomorphism $A:L_{5,3}^5\to L_{5,3}^{16}$. Then 
\begin{align*}
&A(\res x_4)=\res {A(x_4)}\\
&0=\res {(a_{11}^2a_{22}x_4)}\\
&0=a_{11}^{2p}a_{22}^px_5. 
 \end{align*}
 Therefore, $a_{11}=0$ or $a_{22}=0$ which is a contradiction.
 
It is clear that $L_{5,3}^5$ is not isomorphic to the other restricted Lie algebras. 
 
Next, we claim that $L_{5,3}^6$ and $L_{5,3}^{7}$ are not isomorphic. Suppose to the contrary that there exists an isomorphism $A:L_{5,3}^6\to L_{5,3}^{7}$. Then 
\begin{align*}
&A(\res x_1)=\res {A(x_1)}\\
&0=\res {(a_{11}x_1+a_{21}x_2+a_{31}x_3+a_{41}x_4+a_{51}x_5)}\\
 &0=a_{11}^px_5.
 \end{align*}
Therefore, $a_{11}=0$ which is a contradiction.
   
Next, we claim that $L_{5,3}^6$ and $L_{5,3}^{11}$ are not isomorphic. Suppose to the contrary that there exists an isomorphism $A:L_{5,3}^6\to L_{5,3}^{11}$. Then 
\begin{align*}
&A(\res x_3)=\res {A(x_3)}\\
&0=\res {(a_{11}a_{22}x_3+a_{11}a_{32}x_4)}\\
 &0=a_{11}^pa_{22}^px_5,\\ 
 \end{align*}
 which implies that $a_{11}^pa_{22}^p=0$. Therefore, we have $a_{11}=0$ or $a_{22}=0$ which is a contradiction.
 
 Next, we claim that $L_{5,3}^6$ and $L_{5,3}^{15}$ are not isomorphic. Suppose to the contrary that there exists an isomorphism $A:L_{5,3}^6\to L_{5,3}^{15}$. Then 
\begin{align*}
&A(\res x_2)=\res {A(x_2)}\\
&0=\res {(a_{22}x_2+a_{32}x_3+a_{42}x_4+a_{52}x_5)}\\
&0=a_{22}^p x_5.
 \end{align*}
Therefore, $a_{22}=0$ which is a contradiction.
  
Next, we claim that $L_{5,3}^6$ and $L_{5,3}^{16}$ are not isomorphic. Suppose to the contrary that there exists an isomorphism $A:L_{5,3}^6\to L_{5,3}^{16}$. Then 
\begin{align*}
&A(\res x_4)=\res {A(x_4)}\\
&0=\res {(a_{11}^2a_{22}x_4)}\\
&0=a_{11}^{2p}a_{22}^px_5. 
 \end{align*}
 Therefore, $a_{11}=0$ or $a_{22}=0$ which is a contradiction.
  
It is clear that $L_{5,3}^6$ is not isomorphic to the other restricted Lie algebras.  
 
Next, we claim that $L_{5,3}^7$ and $L_{5,3}^{11}$ are not isomorphic. Suppose to the contrary that there exists an isomorphism $A:L_{5,3}^7\to L_{5,3}^{11}$. Then 
\begin{align*}
&A(\res x_3)=\res {A(x_3)}\\
&0=\res {(a_{11}a_{22}x_3+a_{11}a_{32}x_4)}\\
 &0=a_{11}^pa_{22}^px_5,\\ 
 \end{align*}
 which implies that $a_{11}^pa_{22}^p=0$. Therefore, we have $a_{11}=0$ or $a_{22}=0$ which is a contradiction.
 
 Next, we claim that $L_{5,3}^7$ and $L_{5,3}^{15}$ are not isomorphic. Suppose to the contrary that there exists an isomorphism $A:L_{5,3}^7\to L_{5,3}^{15}$. Then 
\begin{align*}
&A(\res x_2)=\res {A(x_2)}\\
&0=\res {(a_{22}x_2+a_{32}x_3+a_{42}x_4+a_{52}x_5)}\\
&0=a_{22}^p x_5.
 \end{align*}
Therefore, $a_{22}=0$ which is a contradiction.

Next, we claim that $L_{5,3}^7$ and $L_{5,3}^{16}$ are not isomorphic. Suppose to the contrary that there exists an isomorphism $A:L_{5,3}^7\to L_{5,3}^{16}$. Then 
\begin{align*}
&A(\res x_4)=\res {A(x_4)}\\
&0=\res {(a_{11}^2a_{22}x_4)}\\
&0=a_{11}^{2p}a_{22}^px_5. 
 \end{align*}
 Therefore, $a_{11}=0$ or $a_{22}=0$ which is a contradiction.
  
It is clear that $L_{5,3}^7$ is not isomorphic to the other restricted Lie algebras.  
 
Next, we claim that $L_{5,3}^8(\beta)$ and $L_{5,3}^{9}$ are not isomorphic. Suppose to the contrary that there exists an isomorphism $A:L_{5,3}^8(\beta)\to L_{5,3}^{9}$. Then 
\begin{align*}
&A(\res x_3)=\res {A(x_3)}\\
&0=\res {(a_{11}a_{22}x_3+a_{11}a_{32}x_4)}\\
 &0=a_{11}^pa_{22}^px_4,\\ 
 \end{align*}
 which implies that $a_{11}^pa_{22}^p=0$. Therefore, we have $a_{11}=0$ or $a_{22}=0$ which is a contradiction.
 
Note that $L_{5,3}^{8}(\beta)$ is not isomorphic to any of $L_{5,3}^{10}$, $L_{5,3}^{14}$, $L_{5,3}^{18}$, $L_{5,3}^{19}(\xi)$, $L_{5,3}^{21}$, $L_{5,3}^{22}(\gamma)$
 because 
$(L_{5,3}^8)^{[p]^2}=0$  but this is not true for those restricted Lie algebras.

  
Next, we claim that $L_{5,3}^8(\beta)$ and $L_{5,3}^{12}$ are not isomorphic. Suppose to the contrary that there exists an isomorphism $A:L_{5,3}^8(\beta)\to L_{5,3}^{12}$. Then 
\begin{align*}
&A(\res x_3)=\res {A(x_3)}\\
&0=\res {(a_{11}a_{22}x_3+a_{11}a_{32}x_4)}\\
 &0=a_{11}^pa_{22}^px_5,\\ 
 \end{align*}
 which implies that $a_{11}^pa_{22}^p=0$. Therefore, we have $a_{11}=0$ or $a_{22}=0$ which is a contradiction.

Next, we claim that $L_{5,3}^8(\beta)$ and $L_{5,3}^{13}(\beta)$ are not isomorphic. Suppose to the contrary that there exists an isomorphism $A:L_{5,3}^8(\beta)\to L_{5,3}^{13}(\beta)$. Then 
\begin{align*}
&A(\res x_3)=\res {A(x_3)}\\
&0=\res {(a_{11}a_{22}x_3+a_{11}a_{32}x_4)}\\
 &0=a_{11}^pa_{22}^px_5,\\ 
 \end{align*}
 which implies that $a_{11}^pa_{22}^p=0$. Therefore, we have $a_{11}=0$ or $a_{22}=0$ which is a contradiction.


 Next, we claim that $L_{5,3}^8(\beta)$ and $L_{5,3}^{17}$ are not isomorphic. Suppose to the contrary that there exists an isomorphism $A:L_{5,3}^8(\beta)\to L_{5,3}^{17}$. Then 
\begin{align*}
&A(\res x_2)=\res {A(x_2)}\\
&A(\beta x_4)=\res {(a_{22}x_2+a_{32}x_3+a_{42}x_4+a_{52}x_5)}\\
&\beta a_{11}^2a_{22}x_4=a_{22}^p x_5.
 \end{align*}
Therefore, $a_{22}=0$ which is a contradiction.

  
  
Next, we claim that $L_{5,3}^8(\beta)$ and $L_{5,3}^{20}$ are not isomorphic. Suppose to the contrary that there exists an isomorphism $A:L_{5,3}^8(\beta)\to L_{5,3}^{20}$. Then 
\begin{align*}
&A(\res x_3)=\res {A(x_3)}\\
&0=\res {(a_{11}a_{22}x_3+a_{11}a_{32}x_4)}\\
 &0=a_{11}^pa_{22}^px_4,\\ 
 \end{align*}
 which implies that $a_{11}^pa_{22}^p=0$. Therefore, we have $a_{11}=0$ or $a_{22}=0$ which is a contradiction.


It is clear that $L_{5,3}^8(\beta)$ is not isomorphic to the other restricted Lie algebras.  

Note that $L_{5,3}^9$ is not isomorphic to any of $L_{5,3}^{10}$, $L_{5,3}^{14}$, $L_{5,3}^{18}$, $L_{5,3}^{19}(\xi)$, $L_{5,3}^{21}$, $L_{5,3}^{22}(\gamma)$
 because 
$(L_{5,3}^9)^{[p]^2}=0$  but this is not true for those restricted Lie algebras.


Next, we claim that $L_{5,3}^9$ and $L_{5,3}^{12}$ are not isomorphic. Suppose to the contrary that there exists an isomorphism $A:L_{5,3}^9\to L_{5,3}^{12}$. Then 
\begin{align*}
&A(\res x_3)=\res {A(x_3)}\\
&A(x_4)=\res {(a_{11}a_{22}x_3+a_{11}a_{32}x_4)}\\
 &a_{11}^2a_{22}x_4=a_{11}^pa_{22}^px_5,\\ 
 \end{align*}
 which implies that $a_{11}^2a_{22}=0$ and $a_{11}^pa_{22}^p=0$. Therefore, we have $a_{11}=0$ or $a_{22}=0$ which is a contradiction.

Next, we claim that $L_{5,3}^9$ and $L_{5,3}^{13}(\beta)$ are not isomorphic. Suppose to the contrary that there exists an isomorphism $A:L_{5,3}^9\to L_{5,3}^{13}(\beta)$. Then 
\begin{align*}
&A(\res x_2)=\res {A(x_2)}\\
&0=\res {(a_{22}x_2+a_{32}x_3+a_{42}x_4+a_{52}x_5)}\\
&0=a_{22}^p\beta x_4+a_{32}^px_5,
 \end{align*}
 which implies that $a_{22}^p\beta=0$. Since, $\beta \neq 0$ we have $a_{22}=0$ which is a contradiction.


Next, we claim that $L_{5,3}^9$ and $L_{5,3}^{17}$ are not isomorphic. Suppose to the contrary that there exists an isomorphism $A:L_{5,3}^9\to L_{5,3}^{17}$. Then 
\begin{align*}
&A(\res x_2)=\res {A(x_2)}\\
&0=\res {(a_{22}x_2+a_{32}x_3+a_{42}x_4+a_{52}x_5)}\\
&0=a_{22}^p x_5.
 \end{align*}
 Therefore, $a_{22}=0$ which is a contradiction.

  

Next, we claim that $L_{5,3}^9$ and $L_{5,3}^{20}$ are not isomorphic. Suppose to the contrary that there exists an isomorphism $A:L_{5,3}^9\to L_{5,3}^{20}$. Then 
\begin{align*}
&A(\res x_2)=\res {A(x_2)}\\
&0=\res {(a_{22}x_2+a_{32}x_3+a_{42}x_4+a_{52}x_5)}\\
&0=a_{22}^p x_5+a_{32}^px_4.
 \end{align*}
 Therefore, $a_{22}=0$ which is a contradiction.

  
It is clear that $L_{5,3}^9$ is not isomorphic to the other restricted Lie algebras. 

Note that $L_{5,3}^{10}$ is not isomorphic to any of $L_{5,3}^{12}$, $L_{5,3}^{13}(\beta)$, $L_{5,3}^{17}$, $L_{5,3}^{20}$
 because 
$(L_{5,3}^{10})^{[p]^2}\neq 0$  but this is not true for those restricted Lie algebras.



Next, we claim that $L_{5,3}^{10}$ and $L_{5,3}^{14}$ are not isomorphic. Suppose to the contrary that there exists an isomorphism $A:L_{5,3}^{10}\to L_{5,3}^{14}$. Then 
\begin{align*}
&A(\res x_3)=\res {A(x_3)}\\
&0=\res {(a_{11}a_{22}x_3+a_{11}a_{32}x_4)}\\
 &0=a_{11}^pa_{22}^px_5,\\ 
 \end{align*}
 which implies that $a_{11}^pa_{22}^p=0$. Therefore, we have $a_{11}=0$ or $a_{22}=0$ which is a contradiction.


Next, we claim that $L_{5,3}^{10}$ and $L_{5,3}^{18}$ are not isomorphic. Suppose to the contrary that there exists an isomorphism $A:L_{5,3}^{10}\to L_{5,3}^{18}$. Then 
\begin{align*}
&A(\res x_4)=\res {A(x_4)}\\
&0=\res {(a_{11}^2a_{22}x_4)}\\
&0=a_{11}^{2p}a_{22}^px_5. 
 \end{align*}
 Therefore, $a_{11}=0$ or $a_{22}=0$ which is a contradiction.

Next, we claim that $L_{5,3}^{10}$ and $L_{5,3}^{19}(\xi)$ are not isomorphic. Suppose to the contrary that there exists an isomorphism $A:L_{5,3}^{10}\to L_{5,3}^{19}(\xi)$. Then 
\begin{align*}
&A(\res x_4)=\res {A(x_4)}\\
&0=\res {(a_{11}^2a_{22}x_4)}\\
&0=a_{11}^{2p}a_{22}^px_5. 
 \end{align*}
 Therefore, $a_{11}=0$ or $a_{22}=0$ which is a contradiction.


Next, we claim that $L_{5,3}^{10}$ and $L_{5,3}^{21}$ are not isomorphic. Suppose to the contrary that there exists an isomorphism $A:L_{5,3}^{10}\to L_{5,3}^{21}$. Then 
\begin{align*}
&A(\res x_3)=\res {A(x_3)}\\
&0=\res {(a_{11}a_{22}x_3+a_{11}a_{32}x_4)}\\
 &0=a_{11}^pa_{22}^px_4+a_{11}^pa_{32}^px_5,\\ 
 \end{align*}
 which implies that $a_{11}^pa_{22}^p=0$. Therefore, we have $a_{11}=0$ or $a_{22}=0$ which is a contradiction.
 
Next, we claim that $L_{5,3}^{10}$ and $L_{5,3}^{22}(\gamma)$ are not isomorphic. Suppose to the contrary that there exists an isomorphism $A:L_{5,3}^{10}\to L_{5,3}^{22}(\gamma)$. Then 
\begin{align*}
&A(\res x_3)=\res {A(x_3)}\\
&0=\res {(a_{11}a_{22}x_3+a_{11}a_{32}x_4)}\\
 &0=a_{11}^pa_{22}^p\gamma x_4,\\ 
 \end{align*}
 which implies that $a_{11}^pa_{22}^p=0$. Therefore, we have $a_{11}=0$ or $a_{22}=0$ which is a contradiction.

It is clear that $L_{5,3}^{10}$ is not isomorphic to the other restricted Lie algebras.  

Next, we claim that $L_{5,3}^{11}$ and $L_{5,3}^{15}$ are not isomorphic. Suppose to the contrary that there exists an isomorphism $A:L_{5,3}^{11}\to L_{5,3}^{15}$. Then 
\begin{align*}
&A(\res x_2)=\res {A(x_2)}\\
&0=\res {(a_{22}x_2+a_{32}x_3+a_{42}x_4+a_{52}x_5)}\\
&0=a_{22}^p x_5.
 \end{align*}
 Therefore, $a_{22}=0$ which is a contradiction.

Next, we claim that $L_{5,3}^{11}$ and $L_{5,3}^{16}$ are not isomorphic. Suppose to the contrary that there exists an isomorphism $A:L_{5,3}^{11}\to L_{5,3}^{16}$. Then 
\begin{align*}
&A(\res x_4)=\res {A(x_4)}\\
&0=\res {(a_{11}^2a_{22}x_4)}\\
&0=a_{11}^{2p}a_{22}^px_5. 
 \end{align*}
 Therefore, $a_{11}=0$ or $a_{22}=0$ which is a contradiction.

It is clear that $L_{5,3}^{11}$ is not isomorphic to the other restricted Lie algebras.  

Next, we claim that $L_{5,3}^{12}$ and $L_{5,3}^{13}(\beta)$ are not isomorphic. Suppose to the contrary that there exists an isomorphism $A:L_{5,3}^{12}\to L_{5,3}^{13}(\beta)$. Then 
\begin{align*}
&A(\res x_2)=\res {A(x_2)}\\
&0=\res {(a_{22}x_2+a_{32}x_3+a_{42}x_4+a_{52}x_5)}\\
&0=a_{22}^p \beta x_4+a_{32}^px_5,
 \end{align*}
 which implies that $a_{22}^p\beta \neq 0$. Since, $\beta \neq 0$ we have $a_{22}=0$ which is a contradiction.

Note that $L_{5,3}^{12}$ is not isomorphic to any of $L_{5,3}^{14}$, $L_{5,3}^{18}$, $L_{5,3}^{19}(\xi)$, $L_{5,3}^{21}$,
 $L_{5,3}^{22}(\gamma)$ because 
$(L_{5,3}^{12})^{[p]^2}=0$  but this is not true for those restricted Lie algebras.

  
Next, we claim that $L_{5,3}^{12}$ and $L_{5,3}^{17}$ are not isomorphic. Suppose to the contrary that there exists an isomorphism $A:L_{5,3}^{12}\to L_{5,3}^{17}$. Then 
\begin{align*}
&A(\res x_2)=\res {A(x_2)}\\
&0=\res {(a_{22}x_2+a_{32}x_3+a_{42}x_4+a_{52}x_5)}\\
&0=a_{22}^p x_5.
 \end{align*}
Therefore, $a_{22}=0$ which is a contradiction.



Next, we claim that $L_{5,3}^{12}$ and $L_{5,3}^{20}$ are not isomorphic. Suppose to the contrary that there exists an isomorphism $A:L_{5,3}^{12}\to L_{5,3}^{20}$. Then 
\begin{align*}
&A(\res x_2)=\res {A(x_2)}\\
&0=\res {(a_{22}x_2+a_{32}x_3+a_{42}x_4+a_{52}x_5)}\\
&0=a_{22}^p x_5+a_{32}^px_4.
 \end{align*}
Therefore, $a_{22}=0$ which is a contradiction.


It is clear that $L_{5,3}^{12}$ is not isomorphic to the other restricted Lie algebras. 

Note that $L_{5,3}^{13}$ is not isomorphic to any of $L_{5,3}^{14}$, $L_{5,3}^{18}$, $L_{5,3}^{19}(\xi)$, $L_{5,3}^{21}$,
$L_{5,3}^{22}(\gamma)$  because 
$(L_{5,3}^{13})^{[p]^2}=0$  but this is not true for those restricted Lie algebras.

  
Next, we claim that $L_{5,3}^{13}(\beta)$ and $L_{5,3}^{17}$ are not isomorphic. Suppose to the contrary that there exists an isomorphism $A:L_{5,3}^{13}(\beta)\to L_{5,3}^{17}$. Then 
\begin{align*}
&A(\res x_1)=\res {A(x_1)}\\
&0=\res {(a_{11}x_1+a_{21}x_2+a_{31}x_3+a_{41}x_4+a_{51}x_5)}\\
 &0=a_{11}^px_4+a_{21}^px_5.
 \end{align*}
Therefore, $a_{11}=0$ which is a contradiction.
   
   

Next, we claim that $L_{5,3}^{13}(\beta)$ and $L_{5,3}^{20}$ are not isomorphic. Suppose to the contrary that there exists an isomorphism $A:L_{5,3}^{13}(\beta)\to L_{5,3}^{20}$. Then 
\begin{align*}
&A(\res x_2)=\res {A(x_2)}\\
&A(\beta x_4)=\res {(a_{22}x_2+a_{32}x_3+a_{42}x_4+a_{52}x_5)}\\
&\beta a_{11}^2a_{22}x_4=a_{22}^p x_5+a_{32}^px_4.
 \end{align*}
Therefore, $a_{22}=0$ which is a contradiction.

 
 It is clear that $L_{5,3}^{13}(\beta)$ is not isomorphic to the other restricted Lie algebras. 
 
 Note that $L_{5,3}^{14}$ is not isomorphic to any of $L_{5,3}^{17}$, $L_{5,3}^{20}$
 because 
$(L_{5,3}^{14})^{[p]^2}\neq 0$  but this is not true for those restricted Lie algebras.


Next, we claim that $L_{5,3}^{14}$ and $L_{5,3}^{18}$ are not isomorphic. Suppose to the contrary that there exists an isomorphism $A:L_{5,3}^{14}\to L_{5,3}^{18}$. Then 
\begin{align*}
&A(\res x_4)=\res {A(x_4)}\\
&0=\res {(a_{11}^2a_{22}x_4)}\\
&0=a_{11}^{2p}a_{22}^px_5. 
 \end{align*}
Therefore, $a_{11}=0$ or $a_{22}=0$ which is a contradiction.
 
Next, we claim that $L_{5,3}^{14}$ and $L_{5,3}^{19}(\xi)$ are not isomorphic. Suppose to the contrary that there exists an isomorphism $A:L_{5,3}^{14}\to L_{5,3}^{19}(\xi)$. Then 
\begin{align*}
&A(\res x_4)=\res {A(x_4)}\\
&0=\res {(a_{11}^2a_{22}x_4)}\\
&0=a_{11}^{2p}a_{22}^px_5. 
 \end{align*}
Therefore, $a_{11}=0$ or $a_{22}=0$ which is a contradiction.
 

Next, we claim that $L_{5,3}^{14}$ and $L_{5,3}^{21}$ are not isomorphic. Suppose to the contrary that there exists an isomorphism $A:L_{5,3}^{14}\to L_{5,3}^{21}$. Then 
\begin{align*}
&A(\res x_4)=\res {A(x_4)}\\
&0=\res {(a_{11}^2a_{22}x_4)}\\
&0=a_{11}^{2p}a_{22}^px_5. 
 \end{align*}
Therefore, $a_{11}=0$ or $a_{22}=0$ which is a contradiction.

Next, we claim that $L_{5,3}^{14}$ and $L_{5,3}^{22}(\gamma)$ are not isomorphic. Suppose to the contrary that there exists an isomorphism $A:L_{5,3}^{22}(\gamma)\to L_{5,3}^{14}$. Then 
\begin{align*}
&A(\res x_3)=\res {A(x_3)}\\
&A(\gamma x_4)=\res {(a_{11}a_{22}x_3+a_{11}a_{32}x_4)}\\
 &a_{11}^2a_{22}\gamma x_4=a_{11}^pa_{22}^px_5,\\ 
 \end{align*}
 which implies that $a_{11}^pa_{22}^p=0$. Therefore, we have $a_{11}=0$ or $a_{22}=0$ which is a contradiction.

 It is clear that $L_{5,3}^{14}$ is not isomorphic to the other restricted Lie algebras. 

Next, we claim that $L_{5,3}^{15}$ and $L_{5,3}^{16}$ are not isomorphic. Suppose to the contrary that there exists an isomorphism $A:L_{5,3}^{15}\to L_{5,3}^{16}$. Then 
\begin{align*}
&A(\res x_4)=\res {A(x_4)}\\
&0=\res {(a_{11}^2a_{22}x_4)}\\
&0=a_{11}^{2p}a_{22}^px_5. 
 \end{align*}
Therefore, $a_{11}=0$ or $a_{22}=0$ which is a contradiction.
 
 It is clear that $L_{5,3}^{15}$ and $L_{5,3}^{16}$ are not isomorphic to the other restricted Lie algebras. 
 
 Note that $L_{5,3}^{17}$ is not isomorphic to any of $L_{5,3}^{18}$, $L_{5,3}^{19}(\xi)$, $L_{5,3}^{21}$, $L_{5,3}^{22}(\gamma)$
 because 
$(L_{5,3}^{17})^{[p]^2}=0$  but this is not true for those restricted Lie algebras.

 
 
Next, we claim that $L_{5,3}^{17}$ and $L_{5,3}^{20}$ are not isomorphic. Suppose to the contrary that there exists an isomorphism $A:L_{5,3}^{17}\to L_{5,3}^{20}$. Then 
\begin{align*}
&A(\res x_3)=\res {A(x_3)}\\
&0=\res {(a_{11}a_{22}x_3+a_{11}a_{32}x_4)}\\
 &0=a_{11}^pa_{22}^px_4,\\ 
 \end{align*}
 which implies that $a_{11}^pa_{22}^p=0$. Therefore, we have $a_{11}=0$ or $a_{22}=0$ which is a contradiction.

 
Next, we claim that $L_{5,3}^{18}$ and $L_{5,3}^{19}(\xi)$ are not isomorphic. Suppose to the contrary that there exists an isomorphism $A:L_{5,3}^{18}\to L_{5,3}^{19}(\xi)$. Then 
\begin{align*}
&A(\res x_2)=\res {A(x_2)}\\
&0=\res {(a_{22}x_2+a_{32}x_3+a_{42}x_4+a_{52}x_5)}\\
&0=a_{22}^p x_4+a_{42}^px_5.
 \end{align*}
Therefore, $a_{22}=0$ which is a contradiction.

Note that $L_{5,3}^{18}$ is not isomorphic to $L_{5,3}^{20}$
 because 
$(L_{5,3}^{18})^{[p]^2}\neq 0$  but this is not true for $L_{5,3}^{20}$.


Next, we claim that $L_{5,3}^{18}$ and $L_{5,3}^{21}$ are not isomorphic. Suppose to the contrary that there exists an isomorphism $A:L_{5,3}^{18}\to L_{5,3}^{21}$. Then 
\begin{align*}
&A(\res x_3)=\res {A(x_3)}\\
&0=\res {(a_{11}a_{22}x_3+a_{11}a_{32}x_4)}\\
 &0=a_{11}^pa_{22}^px_4+a_{11}^pa_{32}^px_5,\\ 
 \end{align*}
 which implies that $a_{11}^pa_{22}^p=0$. Therefore, we have $a_{11}=0$ or $a_{22}=0$ which is a contradiction.
 
 Next, we claim that $L_{5,3}^{18}$ and $L_{5,3}^{22}(\gamma)$ are not isomorphic. Suppose to the contrary that there exists an isomorphism $A:L_{5,3}^{18}\to L_{5,3}^{22}(\gamma)$. Then 
\begin{align*}
&A(\res x_3)=\res {A(x_3)}\\
&0=\res {(a_{11}a_{22}x_3+a_{11}a_{32}x_4)}\\
 &0=a_{11}^pa_{22}^p\gamma x_4,\\ 
 \end{align*}
 which implies that $a_{11}^pa_{22}^p=0$. Therefore, we have $a_{11}=0$ or $a_{22}=0$ which is a contradiction.

 Note that $L_{5,3}^{19}$ is not isomorphic to $L_{5,3}^{20}$
 because 
$(L_{5,3}^{19})^{[p]^2}\neq 0$  but this is not true for $L_{5,3}^{20}$.


Next, we claim that $L_{5,3}^{19}(\xi)$ and $L_{5,3}^{21}$ are not isomorphic. Suppose to the contrary that there exists an isomorphism $A:L_{5,3}^{19}(\xi)\to L_{5,3}^{21}$. Then 
\begin{align*}
&A(\res x_3)=\res {A(x_3)}\\
&0=\res {(a_{11}a_{22}x_3+a_{11}a_{32}x_4)}\\
 &0=a_{11}^pa_{22}^px_4+a_{11}^pa_{32}^px_5,\\ 
 \end{align*}
 which implies that $a_{11}^pa_{22}^p=0$. Therefore, we have $a_{11}=0$ or $a_{22}=0$ which is a contradiction.
 
 Next, we claim that $L_{5,3}^{19}$ and $L_{5,3}^{22}(\gamma)$ are not isomorphic. Suppose to the contrary that there exists an isomorphism $A:L_{5,3}^{19}\to L_{5,3}^{22}(\gamma)$. Then 
\begin{align*}
&A(\res x_3)=\res {A(x_3)}\\
&0=\res {(a_{11}a_{22}x_3+a_{11}a_{32}x_4)}\\
 &0=a_{11}^pa_{22}^p\gamma x_4,\\ 
 \end{align*}
 which implies that $a_{11}^pa_{22}^p=0$. Therefore, we have $a_{11}=0$ or $a_{22}=0$ which is a contradiction.

Note that $L_{5,3}^{20}$ is not isomorphic to $L_{5,3}^{21}$ and $L_{5,3}^{22}(\gamma)$
 because 
$(L_{5,3}^{20})^{[p]^2}=0$  but this is not true for $L_{5,3}^{21}$ and $L_{5,3}^{22}(\gamma)$.

Next, we claim that $L_{5,3}^{21}$ and $L_{5,3}^{22}(\gamma)$ are not isomorphic. Suppose to the contrary that there exists an isomorphism $A:L_{5,3}^{22}(\gamma)\to L_{5,3}^{21}$. Then 
\begin{align*}
&A(\res x_4)=\res {A(x_4)}\\
&0=\res {(a_{11}^2a_{22}x_4)}\\
 &0=a_{11}^{2p}a_{22}^px_5,\\ 
 \end{align*}
 which implies that $a_{11}^{2p}a_{22}^p=0$. Therefore, we have $a_{11}=0$ or $a_{22}=0$ which is a contradiction.


%% file: 1-dim.tex
\chapter{Restriction maps on Lie algebras of 1-dimensional centre}
\section{Restriction maps on $L_{5,4}$}
Let $
L_{5,4}=\langle x_1,\ldots,x_5\mid [x_1,x_2]=x_5, [x_3,x_4]=x_5 \rangle.$  We have $Z(L_{5,4})=\la x_5\ra_{\F}$ and so $x_5^{[p]}=0$.
Let
$$L=\frac {L_{5,4}}{\la x_5\ra_{\F}}\cong L_{4,1},$$
where $L_{4,1}=\la x_1,x_2,x_3,x_4\ra$.  Note that the group $\Aut(L)$ consists of invertible matrices of the form 
\[\begin{pmatrix} a_{11}  & a_{12} & a_{13}&a_{14} \\
a_{21} & a_{22} & a_{23}&a_{24}\\
a_{31} & a_{32} & a_{33} & a_{34}\\
a_{41} & a_{42} & a_{43} & a_{44}
\end{pmatrix}.\]

\begin{lemma}\label{LemmaL54}
Let $K=L_{5,4}$ and $[p]:K\to K$ be a $p$-map on $K$ and let $L=\frac{K}{M}$ where $M=\la x_5\ra_{\F}$. Then $K\cong L_{\theta}$ where $\theta=(\Delta_{12}+\Delta_{34},\omega)\in Z^2(L,\F)$.
\end{lemma}
\begin{proof}
 Let $\pi :K\rightarrow L$ be the projection map. We have the exact sequence 
$$ 0\rightarrow M\rightarrow K\rightarrow L\rightarrow 0.$$
Let $\sigma :L\rightarrow K$ such that $x_i\mapsto x_i$, $1\leq i\leq 4$. Then $\sigma$ is an injective linear map and $\pi \sigma =1_L$. Now, we define $\phi : L\times L\rightarrow M$ by $\phi (x_i,x_j)=[\sigma (x_i),\sigma (x_j)]-\sigma ([x_i,x_j])$, $1\leq i,j\leq 4$ and $\omega: L\rightarrow M$ by $\omega (x)=\res {\sigma (x)} -\sigma (\res x)$. Note that
\begin{align*}
&\phi(x_1,x_2)=[\sigma(x_1),\sigma(x_2)]-\sigma([x_1,x_2])=[x_1,x_2]=x_5;\\
&\phi(x_3,x_4)=[\sigma(x_3),\sigma(x_4)]-\sigma([x_3,x_4])=x_5;\\
&\phi(x_1,x_3)=[\sigma(x_1),\sigma(x_3)]-\sigma([x_1,x_3])=0;\\
\end{align*}
Similarly, we can show that $\phi(x_1,x_4)=\phi(x_2,x_3)=\phi(x_2,x_4)=0$. Therefore, $\phi=\Delta_{12}+\Delta_{34}$.
Now, by Lemma \ref{K=L-theta}, we have $\theta=(\Delta_{12}+\Delta_{34},\omega)\in Z^2(L,\F)$ and $K\cong L_{\theta}$.
\end{proof}

Note that by Theorem \ref{4-dim}, there are five non-isomorphic restricted Lie algebra structures on $L$ given by the following $p$-maps:
\begin{enumerate}
\item[I.1]Trivial $p$-map;
\item[I.2]$\res x_1=x_2$;
\item[I.3]$\res x_1 =x_2 , \res x_3=x_4$;
\item[I.4]$\res x_1 =x_2 , \res x_2=x_3$;
\item[I.5]$\res x_1 =x_2 , \res x_2=x_3, \res x_3=x_4$.
\end{enumerate}

We make $L$ into a restricted Lie algebra by equipping it with each of the above $p$-maps. Then,  in  each  case,  we  find all possible orbit representatives of the form $(\Delta_{12}+\Delta_{34},\omega)$ under the action of $\Aut_p(L)$ on $H^2(L,\F)$. By  Lemma \ref{LemmaL54}, we do get all possible $p$-maps on  $L_{5,5}$ with the property that $x_5^{[p]}=0$. 

Consider the case I.2 where the $p$-map of $L$ is given by $\res x_1 = x_2$.
Let  $[(\phi,\omega)]\in H^2(L,\F)$. Then we must have $\phi (x,\res y)=0$, for all $x,y\in L$, where $\phi=a\Delta _{12}+b\Delta_{13}+c\Delta_{14}+d\Delta_{23}+e\Delta_{24}+f\Delta_{34}$, for some $a,b,c,d,e,f \in \F$.
Since $\res L=\la x_2 \ra_{\F}$, we get  $\phi (x,x_2)=0$, for all $x\in L$. Therefore, 
\begin{align*}
\phi (x_1,x_2)&=a\Delta_{12}(x_1,x_2)+b\Delta_{13}(x_1,x_2)+c\Delta_{14}(x_1,x_2)\\
&+d\Delta_{23}(x_1,x_2)+a\Delta_{24}(x_1,x_2)+f\Delta_{34}(x_1,x_2)=0
\end{align*}
which implies that $a=0$. Since $\phi =\Delta_{12}+\Delta_{34}$ gives us $L_{5,4}$, we deduce by Lemma \ref{LemmaL54} that $L_{5,4}$ cannot be constructed in this case.  Similarly, in cases I.3, I.4, and I.5 we can show that  we  cannot construct $L_{5,4}$. In the following subsections, we consider the remaining cases.

\subsection{Extension of $L_{5,4}/\la x_5\ra$ via the trivial $p$-map}
 First, we find a basis for $Z^2(L,\F)$. Let $(\phi,\omega)=(a\Delta _{12}+b\Delta_{13}+c\Delta_{14}+d\Delta_{23}+e\Delta_{24}+f\Delta_{34},\alpha f_1+\beta f_2+ \gamma f_3+\delta f_4)\in Z^2(L,\F)$. Then we must have $\delta^2\phi(x,y,z) =0$ and $\phi(x,\res y)=0$, for all $x,y,z \in L$.
Since $L$ is abelian and the $p$-map is trivial,  $\delta^2\phi(x,y,z) =0$ and $\phi(x,\res y)=0$, for all $x,y,z \in L$.
Therefore, a basis for $Z^2(L,\F)$ is as follows:
$$
 (\Delta_{12},0),(\Delta_{13},0),(\Delta_{14},0),(\Delta_{23},0),(\Delta_{24},0),(\Delta_{34},0),(0,f_1),
(0,f_2),(0,f_3),(0,f_4).
$$
Next, we find a basis for $B^2(L,\F)$. Let $(\phi,\omega)\in B^2(L,\F)$. Since $B^2(L,\F)\subseteq Z^2(L,\F)$, we have $(\phi,\omega)=(a\Delta _{12}+b\Delta_{13}+c\Delta_{14}+d\Delta_{23}+e\Delta_{24}+f\Delta_{34},\alpha f_1+\beta f_2+ \gamma f_3+\delta f_4)$. So, there exists a linear map $\psi:L\to \F$ such that $\delta^1\psi(x,y)=\phi(x,y)$ and $\tilde \psi(x)=\omega(x)$, for all $x,y \in L$. So, we have
\begin{align*}
a=\phi(x_1,x_2)=\delta^1\psi(x_1,x_2)=\psi([x_1,x_2])=0.
\end{align*}
Similarly, we can show that $b=c=d=e=f=0$. Also, we have
\begin{align*}
\alpha=\omega(x_1)=\tilde \psi(x_1)=\psi(\res x_1)=0.
\end{align*}
Similarly, we can show that $\beta=\gamma=\delta=0$. Therefore, $(\phi,\omega)=(0,0)$ and hence
 $B^2(L,\F)=0$. We deduce that a basis for $H^2(L,\F)$ is as follows:
$$
[ (\Delta_{12},0)],[(\Delta_{13},0)],[(\Delta_{14},0)],[(\Delta_{23},0)],[(\Delta_{24},0)],[(\Delta_{34},0)],[(0,f_1)],
[(0,f_2)],[(0,f_3)],[(0,f_4)].
$$
Let $[(\phi,\omega)] \in H^2(L,\F)$. Then we have $\phi = a\Delta_{12} + b\Delta_{13}+c\Delta_{14}+d\Delta_{23}+e\Delta_{24}+f\Delta_{34}$, for some $a,b,c,d,e,f \in \F$. Suppose that $A\phi = a'\Delta_{12}+b'\Delta_{13}+c'\Delta_{14}+ d'\Delta_{23}+e'\Delta_{24}+f'\Delta_{34}$, for some $a',b',c',d',e',f' \in \F$. Then
\begin{align*}
A\phi(x_1,x_2)&=\phi(Ax_1,Ax_2)\\
&=\phi(a_{11}x_1+a_{21}x_2+a_{31}x_3+a_{41}x_4,a_{12}x_1+a_{22}x_2+a_{32}x_3+a_{42}x_4)\\
&=(a_{11}a_{22}-a_{12}a_{21})a+( a_{11}a_{32}-a_{31}a_{12})b+(a_{11}a_{42}-a_{41}a_{12})c\\
&+(a_{21}a_{32}-a_{31}a_{22})d+(a_{21}a_{42}-a_{41}a_{22})e+(a_{31}a_{42}-a_{41}a_{32})f, \text{ and }\\
\end{align*}
\begin{align*}
A\phi(x_1,x_3)&=\phi(Ax_1,Ax_3)\\
&=\phi(a_{11}x_1+a_{21}x_2+a_{31}x_3+a_{41}x_4,a_{13}x_1+a_{23}x_2+a_{33}x_3+a_{43}x_4)\\
&=(a_{11}a_{23}-a_{21}a_{13})a+(a_{11}a_{33}-a_{31}a_{13})b+(a_{11}a_{43}-a_{13}a_{41})c\\
&+( a_{21}a_{33}-a_{31}a_{23})d+(a_{21}a_{43}-a_{23}a_{41})e+(a_{31}a_{43}-a_{33}a_{41})f, \text{ and }
\end{align*}
\begin{align*}
A\phi(x_1,x_4)&=\phi(Ax_1,Ax_4)\\
&=\phi(a_{11}x_1+a_{21}x_2+a_{31}x_3+a_{41}x_4,a_{14}x_1+a_{24}x_2+a_{34}x_3+a_{44}x_4)\\
&=(a_{11}a_{24}-a_{14}a_{21})a+(a_{11}a_{34}-a_{14}a_{31})b+(a_{11}a_{44}-a_{14}a_{41})c\\
&+(a_{21}a_{34}-a_{24}a_{31})d+(a_{21}a_{44}-a_{24}a_{41})e+(a_{31}a_{44}-a_{34}a_{41})f, \text{ and }
\end{align*}
\begin{align*}
A\phi(x_2,x_3)&=\phi(Ax_2,Ax_3)\\
&=\phi(a_{12}x_1+a_{22}x_2+a_{32}x_3+a_{42}x_4,a_{13}x_1+a_{23}x_2+a_{33}x_3+a_{43}x_4)\\
&=(a_{12}a_{23}-a_{22}a_{13})a+(a_{12}a_{33}-a_{32}a_{13})b+(a_{12}a_{43}-a_{13}a_{42})c\\
&+(a_{22}a_{33}-a_{32}a_{23})d+(a_{22}a_{43}-a_{23}a_{42})e+(a_{32}a_{43}-a_{33}a_{42})f, \text{ and }
\end{align*}
\begin{align*}
A\phi(x_2,x_4)&=\phi(Ax_2,Ax_4)\\
&=\phi(a_{12}x_1+a_{22}x_2+a_{32}x_3+a_{42}x_4,a_{14}x_1+a_{24}x_2+a_{34}x_3+a_{44}x_4)\\
&=(a_{12}a_{24}-a_{14}a_{22})a+(a_{12}a_{34}-a_{14}a_{32})b+(a_{12}a_{44}-a_{14}a_{42})c\\
&+(a_{22}a_{34}-a_{24}a_{32})d+(a_{22}a_{44}-a_{24}a_{42})e+(a_{32}a_{44}-a_{34}a_{42})f, \text{ and }
\end{align*}
\begin{align*}
A\phi(x_3,x_4)&=\phi(Ax_3,Ax_4)\\
&=\phi(a_{13}x_1+a_{23}x_2+a_{33}x_3+a_{43}x_4,a_{14}x_1+a_{24}x_2+a_{34}x_3+a_{44}x_4)\\
&=(a_{13}a_{24}-a_{14}a_{23})a+(a_{13}a_{34}-a_{14}a_{33})b+(a_{13}a_{44}-a_{14}a_{43})c\\
&+(a_{23}a_{34}-a_{24}a_{33})d+(a_{23}a_{44}-a_{24}a_{43})e+(a_{33}a_{44}-a_{34}a_{43})f.
\end{align*}
In the matrix form we can write this as
\begin{align*}
&\tiny\begin{pmatrix}
a' \\ b'\\c'\\d'\\e'\\f'
\end{pmatrix}=\\
&\tiny\begin{pmatrix}
a_{11}a_{22}-a_{12}a_{21} & a_{11}a_{32}-a_{31}a_{12}&a_{11}a_{42}-a_{41}a_{12}&a_{21}a_{32}-a_{31}a_{22}&a_{21}a_{42}-a_{41}a_{22}&a_{31}a_{42}-a_{41}a_{32} \\
a_{11}a_{23}-a_{21}a_{13}&a_{11}a_{33}-a_{31}a_{13}&a_{11}a_{43}-a_{13}a_{41}& a_{21}a_{33}-a_{31}a_{23}&a_{21}a_{43}-a_{23}a_{41}&a_{31}a_{43}-a_{33}a_{41}\\
a_{11}a_{24}-a_{14}a_{21}&a_{11}a_{34}-a_{14}a_{31}&a_{11}a_{44}-a_{14}a_{41}&a_{21}a_{34}-a_{24}a_{31}&a_{21}a_{44}-a_{24}a_{41}&a_{31}a_{44}-a_{34}a_{41}\\
a_{12}a_{23}-a_{22}a_{13}&a_{12}a_{33}-a_{32}a_{13}&a_{12}a_{43}-a_{13}a_{42}&a_{22}a_{33}-a_{32}a_{23}&a_{22}a_{43}-a_{23}a_{42}&a_{32}a_{43}-a_{33}a_{42}\\
a_{12}a_{24}-a_{14}a_{22}&a_{12}a_{34}-a_{14}a_{32}&a_{12}a_{44}-a_{14}a_{42}&a_{22}a_{34}-a_{24}a_{32}&a_{22}a_{44}-a_{24}a_{42}&a_{32}a_{44}-a_{34}a_{42}\\
a_{13}a_{24}-a_{14}a_{23}&a_{13}a_{34}-a_{14}a_{33}&a_{13}a_{44}-a_{14}a_{43}&a_{23}a_{34}-a_{24}a_{33}&a_{23}a_{44}-a_{24}a_{43}&a_{33}a_{44}-a_{34}a_{43}
\end{pmatrix}
\begin{pmatrix}
a \\ b\\c\\d\\e\\f
\end{pmatrix}.
\end{align*}
The orbit with representative 
$\begin{pmatrix}
1\\ 0\\0\\0\\0\\1
\end{pmatrix}$ of this action  gives us $L_{5,4}$.

Also, we have $\omega=\alpha f_1+\beta f_2+\gamma f_3+\delta f_4$,  for some $\alpha ,\beta ,\gamma ,\delta \in \F$. Suppose that $A\omega = \alpha' f_1+\beta' f_2+\gamma' f_3+\delta' f_4$,  for some $\alpha' ,\beta' ,\gamma' ,\delta' \in \F$. Then we can verify that  the action of $\Aut(L)$ on the set of $\omega$'s in the matrix form is as follows:
\begin{align*}
\begin{pmatrix}
\alpha' \\ 
\beta'\\
\gamma'\\
\delta'
\end{pmatrix}=
\begin{pmatrix}
a_{11}^p& a_{21}^p&a_{31}^p&a_{41}^p\\
a_{12}^p& a_{22}^p& a_{32}^p&a_{42}^p\\
a_{13}^p&a_{23}^p&a_{33}^p&a_{43}^p\\
a_{14}^p&a_{24}^p&a_{34}^p&a_{44}^p
\end{pmatrix}
\begin{pmatrix}
\alpha\\ 
\beta \\
\gamma \\
\delta
\end{pmatrix}.
\end{align*}
Now we find the representatives of the orbits of the action of $\Aut (L)$ on the set of $\omega$'s such that 
the orbit represented by
 $\begin{pmatrix}
1\\ 0\\0\\0\\0\\1
\end{pmatrix}$
 is preserved under the action of $\Aut (L)$ on the set of $\phi$'s.
Let $\nu = \begin{pmatrix}
\alpha\\ 
\beta \\
\gamma \\
\delta
\end{pmatrix} \in \F^4$. If $\nu= \begin{pmatrix}
0\\ 
0 \\
0 \\
0
\end{pmatrix}$, then $\{\nu \}$ is clearly an $\Aut(L)$-orbit. Let $\nu \neq 0$. Suppose that $\delta \neq 0$, then
\begin{align*}
&\footnotesize\bigg[\begin{pmatrix}
1&0&(-\beta/\delta)^{1/p}&0&(\alpha/\delta)^{1/p}&0\\
(-\alpha/\delta)^{1/p}& 1&(\alpha\beta/\delta^2)^{1/p}&0&(-\alpha^2/\delta^2)^{1/p}&(\alpha/\delta)^{1/p}\\
0&0&1&0&0&0\\
(-\beta/\delta)^{1/p}&0&(\beta^2/\delta^2)^{1/p}&1&(-\alpha\beta/\delta^2)^{1/p}&(\beta/\delta)^{1/p}\\
0&0&0&0&1&0\\
0&0&(\beta/\delta)^{1/p}&0&(-\alpha/\delta)^{1/p}&1
\end{pmatrix},
\begin{pmatrix}
1& 0&0&-\alpha/\delta\\
0& 1&0&-\beta/\delta\\
\beta/\delta&-\alpha/\delta&1&0\\
0&0&0&1
\end{pmatrix}\bigg]\\
&\bigg[\begin{pmatrix}
1 \\ 0\\0\\0\\0\\1
\end{pmatrix},
\begin{pmatrix}
\alpha\\ 
\beta \\
\gamma \\
\delta
\end{pmatrix}\bigg]
=\bigg[\begin{pmatrix}
1\\0\\0\\0\\0\\1
\end{pmatrix},
\begin{pmatrix}
0 \\ 
0\\
\gamma\\
\delta
\end{pmatrix}\bigg], \text{ and }
\end{align*}
\begin{align*}
& \bigg[\begin{pmatrix}
1&0& 0&0&0&0\\
0& \delta^{1/p}&(-\gamma)^{1/p}&0&0&0\\
0&0&(1/\delta)^{1/p}&0&0&0\\
0&0&0&\delta^{1/p}&(-\gamma)^{1/p}&0\\
0&0&0&0&(1/\delta)^{1/p}&0\\
0&0&0&0&0&1
\end{pmatrix},
\begin{pmatrix}
1& 0&0&0\\
0& 1&0&0\\
0&0&\delta&-\gamma\\
0&0&0&1/\delta
\end{pmatrix}\bigg]
\bigg[\begin{pmatrix}
1 \\ 0\\0\\0\\0\\1
\end{pmatrix},
\begin{pmatrix}
0\\ 
0 \\
\gamma\\
\delta
\end{pmatrix}\bigg]\\
&=
\bigg[\begin{pmatrix}
1\\0\\0\\0\\0\\1
\end{pmatrix},
\begin{pmatrix}
0 \\ 
0\\
0\\
1
\end{pmatrix}\bigg].
\end{align*}

Next, suppose that $\delta =0$, but $\beta \neq 0$, then
\begin{align*}
&\footnotesize\bigg[\begin{pmatrix}
1&0&0&0&(\gamma/\beta)^{1/p}&0\\
(-\gamma/\beta)^{1/p}& 1&0&(-\alpha/\beta)^{1/p}&(-\gamma^2/\beta^2)^{1/p}&(\gamma/\beta)^{1/p}\\
0&0&1&0&(-\alpha/\beta)^{1/p}&0\\
0&0&0&1&0&0\\
0&0&0&0&1&0\\
0&0&0&0&(-\gamma/\beta)^{1/p}&1
\end{pmatrix},
\begin{pmatrix}
1& -\alpha/\beta&0&-\gamma/\beta\\
0& 1&0&0\\
0&-\gamma/\beta&1&0\\
0&0&0&1
\end{pmatrix}\bigg]\\
&\bigg[\begin{pmatrix}
1 \\ 0\\0\\0\\0\\1
\end{pmatrix},
\begin{pmatrix}
\alpha\\ 
\beta \\
\gamma \\
0
\end{pmatrix}\bigg]
=\bigg[\begin{pmatrix}
1\\0\\0\\0\\0\\1
\end{pmatrix},
\begin{pmatrix}
0 \\ 
\beta\\
0\\
0
\end{pmatrix}\bigg], \text{ and }\\
&\bigg[\begin{pmatrix}
1&0& 0&0&0&0\\
0& \beta^{1/p}&0&0&0&0\\
0&0&\beta^{1/p}&0&0&0\\
0&0&0&(1/\beta)^{1/p}&0&0\\
0&0&0&0&(1/\beta)^{1/p}&0\\
0&0&0&0&0&1
\end{pmatrix},
\begin{pmatrix}
\beta& 0&0&0\\
0& 1/\beta&0&0\\
0&0&1&0\\
0&0&0&1
\end{pmatrix}\bigg]
\bigg[\begin{pmatrix}
1 \\ 0\\0\\0\\0\\1
\end{pmatrix},
\begin{pmatrix}
0\\ 
\beta \\
0\\
0
\end{pmatrix}\bigg]\\
&=
\bigg[\begin{pmatrix}
1\\0\\0\\0\\0\\1
\end{pmatrix},
\begin{pmatrix}
0 \\ 
1\\
0\\
0
\end{pmatrix}\bigg].
\end{align*}
Next, if $\delta=\beta=0$, but $\alpha \neq 0$, then
\begin{align*}
&\small
\bigg[\begin{pmatrix}
1&0& (\gamma/\alpha)^{1/p}&0&0&0\\
0& 1&0&0&0&0\\
0&0&1&0&0&0\\
(\gamma/\alpha)^{1/p}&0&(\gamma^2/\alpha^2)^{1/p}&1&0&(-\gamma/\alpha)^{1/p}\\
0&0&0&0&1&0\\
0&0&(-\gamma/\alpha)^{1/p}&0&0&1
\end{pmatrix},
\begin{pmatrix}
1& 0&0&0\\
0& 1&0&\gamma/\alpha\\
-\gamma/\alpha&0&1&0\\
0&0&0&1
\end{pmatrix}\bigg]
\bigg[\begin{pmatrix}
1 \\ 0\\0\\0\\0\\1
\end{pmatrix},
\begin{pmatrix}
\alpha\\ 
0\\
\gamma \\
0
\end{pmatrix}\bigg]\\
&=\bigg[\begin{pmatrix}
1\\0\\0\\0\\0\\1
\end{pmatrix},
\begin{pmatrix}
\alpha \\ 
0\\
0\\
0
\end{pmatrix}\bigg], \text{ and }
\end{align*}
\begin{align*}
&\small
 \bigg[\begin{pmatrix}
1&0& 0&0&0&0\\
0& (1/\alpha)^{1/p}&0&0&0&0\\
0&0&(1/\alpha)^{1/p}&0&0&0\\
0&0&0&\alpha^{1/p}&0&0\\
0&0&0&0&\alpha^{1/p}&0\\
0&0&0&0&0&1
\end{pmatrix},
\begin{pmatrix}
1/\alpha& 0&0&0\\
0& \alpha&0&0\\
0&0&1&0\\
0&0&0&1
\end{pmatrix}\bigg]
\bigg[\begin{pmatrix}
1 \\ 0\\0\\0\\0\\1
\end{pmatrix},
\begin{pmatrix}
\alpha\\ 
0 \\
0\\
0
\end{pmatrix}\bigg]=
\bigg[\begin{pmatrix}
1\\0\\0\\0\\0\\1
\end{pmatrix},
\begin{pmatrix}
1 \\ 
0\\
0\\
0
\end{pmatrix}\bigg].
\end{align*}
Next, if $\delta=\beta=\alpha=0$, but $\gamma \neq 0$, then
\begin{align*}
\small
 \bigg[\begin{pmatrix}
1&0& 0&0&0&0\\
0& (1/\gamma)^{1/p}&0&0&0&0\\
0&0&\gamma^{1/p}&0&0&0\\
0&0&0&(1/\gamma)^{1/p}&0&0\\
0&0&0&0&\gamma^{1/p}&0\\
0&0&0&0&0&1
\end{pmatrix},
\begin{pmatrix}
1& 0&0&0\\
0& 1&0&0\\
0&0&1/\gamma&0\\
0&0&0&\gamma
\end{pmatrix}\bigg]
\bigg[\begin{pmatrix}
1 \\ 0\\0\\0\\0\\1
\end{pmatrix},
\begin{pmatrix}
0\\ 
0 \\
\gamma\\
0
\end{pmatrix}\bigg]=
\bigg[\begin{pmatrix}
1\\0\\0\\0\\0\\1
\end{pmatrix},
\begin{pmatrix}
0 \\ 
0\\
1\\
0
\end{pmatrix}\bigg].
\end{align*}
Now we claim that the following elements are in the same $\Aut(L)$-orbit representatives.
\begin{align*}
\bigg[\begin{pmatrix}
1\\0\\0\\0\\0\\1
\end{pmatrix},
\begin{pmatrix}
1 \\ 
0\\
0\\
0
\end{pmatrix}\bigg],
\bigg[\begin{pmatrix}
1\\0\\0\\0\\0\\1
\end{pmatrix},
\begin{pmatrix}
0 \\ 
1\\
0\\
0
\end{pmatrix}\bigg],
\bigg[\begin{pmatrix}
1\\0\\0\\0\\0\\1
\end{pmatrix},
\begin{pmatrix}
0 \\ 
0\\
1\\
0
\end{pmatrix}\bigg],
\bigg[\begin{pmatrix}
1\\0\\0\\0\\0\\1
\end{pmatrix},
\begin{pmatrix}
0 \\ 
0\\
0\\
1
\end{pmatrix}\bigg].
\end{align*}
Indeed,
\begin{align*}
&\bigg[\begin{pmatrix}
1&0& 0&0&0&0\\
0& 0&0&-1&0&0\\
0&0&0&0&-1&0\\
0&1&0&0&0&0\\
0&0&1&0&0&0\\
0&0&0&0&0&1
\end{pmatrix},
\begin{pmatrix}
0& -1&0&0\\
1& 0&0&0\\
0&0&1&0\\
0&0&0&1
\end{pmatrix}\bigg]
\bigg[\begin{pmatrix}
1 \\ 0\\0\\0\\0\\1
\end{pmatrix},
\begin{pmatrix}
1\\ 
0 \\
0\\
0
\end{pmatrix}\bigg]=
\bigg[\begin{pmatrix}
1\\0\\0\\0\\0\\1
\end{pmatrix},
\begin{pmatrix}
0 \\ 
1\\
0\\
0
\end{pmatrix}\bigg],\\
&\bigg[\begin{pmatrix}
0&0& 0&0&0&1\\
0& -1&0&0&0&0\\
0&0&0&-1&0&0\\
0&0&-1&0&0&0\\
0&0&0&0&-1&0\\
1&0&0&0&0&0
\end{pmatrix},
\begin{pmatrix}
0& 0&1&0\\
0& 0&0&1\\
1&0&0&0\\
0&1&0&0
\end{pmatrix}\bigg]
\bigg[\begin{pmatrix}
1 \\ 0\\0\\0\\0\\1
\end{pmatrix},
\begin{pmatrix}
1\\ 
0 \\
0\\
0
\end{pmatrix}\bigg]=
\bigg[\begin{pmatrix}
1\\0\\0\\0\\0\\1
\end{pmatrix},
\begin{pmatrix}
0 \\ 
0\\
1\\
0
\end{pmatrix}\bigg],\\
&\bigg[\begin{pmatrix}
0&0& 0&0&0&1\\
0& 0&0&1&0&0\\
0&-1&0&0&0&0\\
0&0&0&0&1&0\\
0&0&-1&0&0&0\\
1&0&0&0&0&0
\end{pmatrix},
\begin{pmatrix}
0& 0&1&0\\
0& 0&0&1\\
0&-1&0&0\\
1&0&0&0
\end{pmatrix}\bigg]
\bigg[\begin{pmatrix}
1 \\ 0\\0\\0\\0\\1
\end{pmatrix},
\begin{pmatrix}
1\\ 
0 \\
0\\
0
\end{pmatrix}\bigg]=
\bigg[\begin{pmatrix}
1\\0\\0\\0\\0\\1
\end{pmatrix},
\begin{pmatrix}
0 \\ 
0\\
0\\
1
\end{pmatrix}\bigg].
\end{align*}
Therefore, the following elements are $\Aut(L)$-orbit representatives:
\begin{align*}
\bigg[\begin{pmatrix}
1\\0\\0\\0\\0\\1
\end{pmatrix},
\begin{pmatrix}
0 \\ 
0\\
0\\
0
\end{pmatrix}\bigg],
\bigg[\begin{pmatrix}
1\\0\\0\\0\\0\\1
\end{pmatrix},
\begin{pmatrix}
1 \\ 
0\\
0\\
0
\end{pmatrix}\bigg].
\end{align*}

\begin{theorem}
The list of all restricted Lie algebra structures on $L_{5,4}$, up to isomorphism, is as follows:
\begin{align*}
& L_{5,4}^1=\langle x_1,\ldots,x_5\mid [x_1,x_2]=x_5,[x_3,x_4]=x_5\rangle;\\
&L_{5,4}^2=\langle x_1,\ldots,x_5\mid [x_1,x_2]=x_5,[x_3,x_4]=x_5, \res x_1=x_5\rangle.
 \end{align*}
\end{theorem}

\section{Restriction maps on $L_{5,5}$}
Let $
L_{5,5}=\langle x_1,\ldots,x_5\mid [x_1,x_2]=x_3, [x_1,x_3]=x_5, [x_2,x_4]=x_5\rangle.$  We have $Z(L_{5,5})=\la x_5\ra_{\F}$. 
Let
$$L=\frac {L_{5,5}}{\la x_5\ra_{\F}}\cong  L_{4,2},$$
 where $L_{4,2}=\la x_1,x_2,x_3,x_4 \mid [x_1,x_2]=x_3\ra $. Note that the group $\Aut(L)$ consists of invertible matrices of the form 
$$\begin{pmatrix}
 a_{11}  & a_{12} & 0 & 0 \\
a_{21} & a_{22} & 0 & 0\\
a_{31} & a_{32} & r & a_{34} \\
a_{41} & a_{42} & 0& a_{44}
\end{pmatrix},$$
where $r = a_{11}a_{22}-a_{12}a_{21}\neq 0$. 

\begin{lemma}\label{LemmaL55}
Let $K=L_{5,5}$ and $[p]:K\to K$ be a $p$-map on $K$ and let $L=\frac{K}{M}$ where $M=\la x_5\ra_{\F}$. Then $K\cong L_{\theta}$ where $\theta=(\Delta_{13}+\Delta_{24},\omega)\in Z^2(L,\F)$.
\end{lemma}
\begin{proof}
 Let $\pi :K\rightarrow L$ be the projection map. We have the exact sequence 
$$ 0\rightarrow M\rightarrow K\rightarrow L\rightarrow 0.$$
Let $\sigma :L\rightarrow K$ such that $x_i\mapsto x_i$, $1\leq i\leq 4$. Then $\sigma$ is an injective linear map and $\pi \sigma =1_L$. Now, we define $\phi : L\times L\rightarrow M$ by $\phi (x_i,x_j)=[\sigma (x_i),\sigma (x_j)]-\sigma ([x_i,x_j])$, $1\leq i,j\leq 4$ and $\omega: L\rightarrow M$ by $\omega (x)=\res {\sigma (x)} -\sigma (\res x)$. Note that
\begin{align*}
&\phi(x_1,x_3)=[\sigma(x_1),\sigma(x_3)]-\sigma([x_1,x_3])=[x_1,x_3]=x_5;\\
&\phi(x_2,x_4)=[\sigma(x_2),\sigma(x_4)]-\sigma([x_2,x_4])=x_5;\\
&\phi(x_1,x_2)=[\sigma(x_1),\sigma(x_2)]-\sigma([x_1,x_2])=0;\\
\end{align*}
Similarly, we can show that $\phi(x_1,x_4)=\phi(x_2,x_3)=\phi(x_3,x_4)=0$. Therefore, $\phi=\Delta_{13}+\Delta_{24}$.
Now, by Lemma \ref{K=L-theta}, we have $\theta=(\Delta_{13}+\Delta_{24},\omega)\in Z^2(L,\F)$ and $K\cong L_{\theta}$.
\end{proof}

Note that by Theorem \ref{4-dim}, there are eight non-isomorphic restricted Lie algebra structures on $L$ given by the following $p$-maps:
\begin{enumerate}
\item[II.1]Trivial $p$-map;
\item[II.2]$\res x_1=x_3;$
\item[II.3]$\res x_1 =x_4;$
\item[II.4]$ \res x_1=x_3, \res x_2=x_4;$
\item[II.5]$\res x_3=x_4;$
\item[II.6]$\res x_3=x_4, \res x_2=x_3;$
\item[II.7]$\res x_4=x_3;$
\item[II.8]$\res x_4=x_3, \res x_2=x_4.$
\end{enumerate}

 We make $L$ into a restricted Lie algebra by equipping it with each of the above $p$-maps. Then,  in  each  case,  we  find all possible orbit representatives of the form $(\Delta_{13}+\Delta_{24},\omega)$ under the action of $\Aut_p(L)$ on $H^2(L,\F)$. By  Lemma \ref{LemmaL55}, we do get all possible $p$-maps on $L_{5,5}$  with the property that $x_5^{[p]}=0$.

Consider the case I.2 where the $p$-map of $L$ is given by $\res x_1 = x_3$.
Note that $\res L=\la x_3 \ra_{\F}$. Let  $[(\phi,\omega)]\in H^2(L,\F)$. Then we must have $\phi (x,\res y)=0$, for all $x,y\in L$, where $\phi=a\Delta _{12}+b\Delta_{13}+c\Delta_{14}+d\Delta_{23}+e\Delta_{24}+f\Delta_{34}$, for some $a,b,c,d,e,f \in \F$.
Hence, $\phi (x_1,x_3)=0$ which implies that $b=0$. Since $\phi =\Delta_{13}+\Delta_{24}$ gives us $L_{5,5}$, we deduce by Lemma \ref{LemmaL55} that $L_{5,5}$ cannot be constructed in this case.  Similarly, in cases II3-II8 we can show that  we  cannot construct $L_{5,5}$. It remains to consider case II.1.

\subsection{Extension of $L_{5,5}/\la x_5\ra$ via the trivial $p$-map}
 First, we find a basis for $Z^2(L,\F)$. Let $(\phi,\omega)=(a\Delta _{12}+b\Delta_{13}+c\Delta_{14}+d\Delta_{23}+e\Delta_{24}+f\Delta_{34},\alpha f_1+\beta f_2+ \gamma f_3+\delta f_4)\in Z^2(L,\F)$. Then we must have $\delta^2\phi(x,y,z) =0$ and $\phi(x,\res y)=0$, for all $x,y,z \in L$. Therefore,
\begin{align*}
0=(\delta^2\phi)(x_1,x_2,x_4)&=\phi([x_1,x_2],x_4)+\phi([x_2,x_4],x_1)+\phi([x_4,x_1],x_2)=\phi(x_3,x_4).
\end{align*}
Thus, we get $f=0$.
Since the $p$-map is trivial, $\phi(x,\res y)=\phi(x,0)=0$, for all $x,y\in L$. Therefore, a basis for $Z^2(L,\F)$ is as follows:
$$
 (\Delta_{12},0),(\Delta_{13},0),(\Delta_{14},0),(\Delta_{23},0),(\Delta_{24},0),(0,f_1),
(0,f_2),(0,f_3),(0,f_4).
$$
Next, we find a basis for $B^2(L,\F)$. Let $(\phi,\omega)\in B^2(L,\F)$. Since $B^2(L,\F)\subseteq Z^2(L,\F)$, we have $(\phi,\omega)=(a\Delta _{12}+b\Delta_{13}+c\Delta_{14}+d\Delta_{23}+e\Delta_{24},\alpha f_1+\beta f_2+ \gamma f_3+\delta f_4)$. So, there exists a linear map $\psi:L\to \F$ such that $\delta^1\psi(x,y)=\phi(x,y)$ and $\tilde \psi(x)=\omega(x)$, for all $x,y \in L$. So, we have
\begin{align*}
b=\phi(x_1,x_3)=\delta^1\psi(x_1,x_3)=\psi([x_1,x_3])=0.
\end{align*}
Similarly, we can show that $c=d=e=0$. Also, we have
\begin{align*}
\alpha=\omega(x_1)=\tilde \psi(x_1)=\psi(\res x_1)=0.
\end{align*}
Similarly, we can show that $\beta=\gamma=\delta=0$. Therefore, $(\phi,\omega)=(a\Delta_{12},0)$ and hence
 $B^2(L,\F)=\la(\Delta_{12},0)\ra_{\F}$. We deduce that a basis for $H^2(L,\F)$ is as follows:
$$
[(\Delta_{13},0)],[(\Delta_{14},0)],[(\Delta_{23},0)],[(\Delta_{24},0)],[(0,f_1)],
[(0,f_2)],[(0,f_3)],[(0,f_4)].
$$
Let $[(\phi,\omega)] \in H^2(L,\F)$. Then we have $\phi = a\Delta_{13} + b\Delta_{14}+c\Delta_{23}+d\Delta_{24}$,  for some  $a,b,c,d \in \F$. Suppose that $A\phi =a'\Delta_{13} + b'\Delta_{14}+c'\Delta_{23}+d'\Delta_{24}$, for some $a',b',c',d' \in \F$. We can verify that the action of $\Aut(L)$ on the set of $\phi$'s in the matrix form is as follows:
\begin{align}\label{12}
\begin{pmatrix}
a' \\ b'\\c'\\d'
\end{pmatrix}=
\begin{pmatrix}
ra_{11} & 0& r a_{21}&0\\
a_{11}a_{34}& a_{11}a_{44}&a_{21}a_{34}&a_{21}a_{44}\\
r a_{12}&0&r a_{22}&0\\
a_{12}a_{34}&a_{12}a_{44}&a_{22}a_{34}&a_{22}a_{44}
\end{pmatrix}
\begin{pmatrix}
a \\ b\\c\\d
\end{pmatrix}.
\end{align}
 The orbit with representative 
$\begin{pmatrix}
1\\ 0\\0\\1
\end{pmatrix}$ of this action  gives us $L_{5,5}$.

Also, we have $\omega=\alpha f_1+\beta f_2+\gamma f_3+\delta f_4$,  for some $\alpha ,\beta ,\gamma ,\delta \in \F$. Suppose that $A\omega = \alpha' f_1+\beta' f_2+\gamma' f_3+\delta' f_4$, for some $\alpha',\beta',\gamma',\delta' \in \F$. We can verify that the action of $\Aut(L)$ on the set of $\omega$'s in the matrix form is as follows: 
\begin{align}\label{23}
\begin{pmatrix}
\alpha' \\ 
\beta'\\
\gamma'\\
\delta'
\end{pmatrix}=
\begin{pmatrix}
a_{11}^p& a_{21}^p&a_{31}^p&a_{41}^p\\
a_{12}^p& a_{22}^p& a_{32}^p&a_{42}^p\\
0&0&r ^p&0\\
0&0&a_{34}^p&a_{44}^p
\end{pmatrix}
\begin{pmatrix}
\alpha\\ 
\beta \\
\gamma \\
\delta
\end{pmatrix}.
\end{align}
Thus, we can write Equations \eqref{12} and \eqref{23} together as follows:
\begin{align*}
&\bigg[\begin{pmatrix}
r a_{11} & 0& r a_{21}&0\\
a_{11}a_{34}& a_{11}a_{44}&a_{21}a_{34}&a_{21}a_{44}\\
r a_{12}&0&r a_{22}&0\\
a_{12}a_{34}&a_{12}a_{44}&a_{22}a_{34}&a_{22}a_{44}
\end{pmatrix},
\begin{pmatrix}
a_{11}^p& a_{21}^p&a_{31}^p&a_{41}^p\\
a_{12}^p& a_{22}^p& a_{32}^p&a_{42}^p\\
0&0&r ^p&0\\
0&0&a_{34}^p&a_{44}^p
\end{pmatrix}\bigg]
\bigg[\begin{pmatrix}
a\\b\\c\\d
\end{pmatrix},
\begin{pmatrix}
\alpha\\ 
\beta \\
\gamma \\
\delta
\end{pmatrix}\bigg]\\
&=
\bigg[\begin{pmatrix}
a' \\ b'\\c'\\d'
\end{pmatrix},
\begin{pmatrix}
\alpha' \\ 
\beta'\\
\gamma'\\
\delta'
\end{pmatrix}\bigg].
\end{align*}
Now we find the representatives of the orbits of the action of $\Aut (L)$ on the set of $\omega$'s  such that 
the orbit represented by 
$\begin{pmatrix}
1 \\ 0\\0\\1
\end{pmatrix}$ is preserved under the action of $\Aut (L)$ on the set of $\phi$'s.

Let $\nu = \begin{pmatrix}
\alpha\\ 
\beta \\
\gamma \\
\delta
\end{pmatrix} \in \F^4$. If $\nu= \begin{pmatrix}
0\\ 
0 \\
0 \\
0
\end{pmatrix}$, then $\{\nu \}$ is clearly an $\Aut (L)$-orbit. Let $\nu \neq 0$. Suppose that $\gamma \neq 0$. Then 
\begin{align*}
&\bigg[\begin{pmatrix}
1& 0&0&0\\
0& 1& 0&0\\
0&0&1&0\\
0&0&0&1
\end{pmatrix},
\begin{pmatrix}
1& 0&-\alpha /\gamma&0\\
0& 1& -\beta /\gamma&0\\
0&0&1&0\\
0&0&0&1
\end{pmatrix}\bigg]
\bigg[\begin{pmatrix}
1\\
0\\
0\\
1
\end{pmatrix},
\begin{pmatrix}
\alpha\\ 
\beta \\
\gamma \\
\delta
\end{pmatrix}\bigg]=
\bigg[\begin{pmatrix}
1\\0\\0\\1
\end{pmatrix},
\begin{pmatrix}
0\\ 
0 \\
\gamma \\
\delta
\end{pmatrix}\bigg], \text{ and }\\
&\bigg[\begin{pmatrix}
1& 0&0&0\\
0& \gamma^{3/p}& 0&0\\
0&0&\gamma^{-3/p}&0\\
0&0&0&1
\end{pmatrix},
\begin{pmatrix}
\gamma& 0&0&0\\
0& \gamma^{-2}& 0&0\\
0&0&\gamma^{-1}&0\\
0&0&0&\gamma^2
\end{pmatrix}\bigg]
\bigg[\begin{pmatrix}
1\\
0\\
0\\
1
\end{pmatrix},
\begin{pmatrix}
0\\ 
0 \\
\gamma \\
\delta
\end{pmatrix}\bigg]=
\bigg[\begin{pmatrix}
1\\0\\0\\1
\end{pmatrix},
\begin{pmatrix}
0\\ 
0 \\
1 \\
\gamma^2\delta
\end{pmatrix}\bigg].
\end{align*}
Furthermore,
\begin{align*}
\bigg[\begin{pmatrix}
1& 0&\delta^{1/p}&0\\
-\delta^{1/p}&1& -\delta^{2/p}&\delta^{1/p}\\
0&0&1&0\\
0&0&-\delta^{1/p}&1
\end{pmatrix},
\begin{pmatrix}
1& \delta&0&0\\
0& 1& 0&0\\
0&0&1&0\\
0&0&-\delta&1
\end{pmatrix}\bigg]
\bigg[\begin{pmatrix}
1\\
0\\
0\\
1
\end{pmatrix},
\begin{pmatrix}
0\\ 
0 \\
1\\
\delta
\end{pmatrix}\bigg]=
\bigg[\begin{pmatrix}
1\\0\\0\\1
\end{pmatrix},
\begin{pmatrix}
0\\ 
0 \\
1 \\
0
\end{pmatrix}\bigg].
\end{align*}
Next, if $\gamma =0$, but $\delta \neq 0$, then
\begin{align*}
&\bigg[\begin{pmatrix}
1& 0&0&0\\
0& 1& 0&0\\
0&0&1&0\\
0&0&0&1
\end{pmatrix},
\begin{pmatrix}
1& 0&0&-\alpha/\delta\\
0& 1& 0&-\beta/\delta\\
0&0&1&0\\
0&0&0&1
\end{pmatrix}\bigg]
\bigg[\begin{pmatrix}
1\\
0\\
0\\
1
\end{pmatrix},
\begin{pmatrix}
\alpha\\ 
\beta \\
0 \\
\delta
\end{pmatrix}\bigg]=
\bigg[\begin{pmatrix}
1\\0\\0\\1
\end{pmatrix},
\begin{pmatrix}
0\\ 
0 \\
0 \\
\delta
\end{pmatrix}\bigg].
\end{align*}
Next, if $\gamma=\delta=0$, but $\beta \neq 0$, then
\begin{align*}
\small \bigg[\begin{pmatrix}
1& 0&(-\alpha/\beta)^{1/p}&0\\
(\alpha/\beta)^{1/p}&1& (-\alpha/\beta)^{2/p}&(-\alpha/\beta)^{1/p}\\
0&0&1&0\\
0&0&(\alpha/\beta)^{1/p}&1
\end{pmatrix},
\begin{pmatrix}
1& -\alpha/\beta&0&0\\
0&1& 0&0\\
0&0&1&0\\
0&0&\alpha/\beta&1
\end{pmatrix}\bigg]
\bigg[\begin{pmatrix}
1\\
0\\
0\\
1
\end{pmatrix},
\begin{pmatrix}
\alpha\\ 
\beta \\
0 \\
0
\end{pmatrix}\bigg]=
\bigg[\begin{pmatrix}
1\\0\\0\\1
\end{pmatrix},
\begin{pmatrix}
0\\ 
\beta\\
0\\
0
\end{pmatrix}\bigg].
\end{align*}
Finally, if $\gamma=\delta=\beta=0$, but $\alpha \neq 0$, then 
\begin{align*}
\bigg[\begin{pmatrix}
1& 0&0&0\\
0& \alpha^{-3/p}& 0&0\\
0&0&\alpha^{3/p}&0\\
0&0&0&1
\end{pmatrix},
\begin{pmatrix}
\alpha^{-1}& 0&0&0\\
0& \alpha^2& 0&0\\
0&0&\alpha&0\\
0&0&0&\alpha^{-2}
\end{pmatrix}\bigg]
\bigg[\begin{pmatrix}
1\\
0\\
0\\
1
\end{pmatrix},
\begin{pmatrix}
\alpha\\ 
0 \\
0 \\
0
\end{pmatrix}\bigg]=
\bigg[\begin{pmatrix}
1\\0\\0\\1
\end{pmatrix},
\begin{pmatrix}
1\\ 
0 \\
0 \\
0
\end{pmatrix}\bigg].
\end{align*}
Thus the following elements are $\Aut (L)$-orbit representatives:
\begin{align*}
\begin{pmatrix}
0\\ 
0\\
0\\
0
\end{pmatrix},
\begin{pmatrix}
1\\ 
0\\
0\\
0
\end{pmatrix},
\begin{pmatrix}
0 \\ 
\beta\\
0\\
0
\end{pmatrix},
\begin{pmatrix}
0 \\ 
0\\
1\\
0
\end{pmatrix},
\begin{pmatrix}
0 \\ 
0\\
0\\
\delta
\end{pmatrix}.
\end{align*}

\begin{theorem}
The list of all restricted Lie algebra structures on $L_{5,5}$, up to isomorphism, is as follows:
\begin{align*}
&L_{5,5}^{1}=\langle x_1,\ldots,x_5\mid [x_1,x_2]=x_3, [x_1,x_3]=x_5, [x_2,x_4]=x_5 \rangle;\\
&L_{5,5}^{2}=\langle x_1,\ldots,x_5\mid [x_1,x_2]=x_3, [x_1,x_3]=x_5, [x_2,x_4]=x_5, 
\res x_1= x_5  \rangle;\\
&L_{5,5}^{3}(\beta)=\langle x_1,\ldots,x_5\mid [x_1,x_2]=x_3, [x_1,x_3]=x_5, [x_2,x_4]=x_5,
\res x_2= \beta x_5 \rangle;\\
&L_{5,5}^{4}=\langle x_1,\ldots,x_5\mid [x_1,x_2]=x_3, [x_1,x_3]=x_5, [x_2,x_4]=x_5,
\res x_3= x_5  \rangle;\\
&L_{5,5}^{5}(\delta)=\langle x_1,\ldots,x_5\mid [x_1,x_2]=x_3, [x_1,x_3]=x_5, [x_2,x_4]=x_5,
\res x_4=\delta x_5 \rangle
\end{align*}
where $\beta,\gamma \in \F^*$. 
\end{theorem}

The automorphism group of $L_{5,5}$ consists of 
 
$$ \begin{pmatrix} a_{11} & 0 & 0 & 0 & 0 \\
                   a_{21} & a_{22} & 0 & 0 & 0\\
                   a_{31} & a_{32} & a_{11}a_{22} & -a_{11}a_{21} & 0 \\
                   a_{41} & a_{42} & 0 & a_{11}^2 & 0 \\
                   a_{51} & a_{52} & d & a_{54} & a_{11}^2a_{22}
\end{pmatrix},$$
with $d = a_{11}a_{32}+a_{21}a_{42}-a_{41}a_{22}$.

\begin{lemma}
We have $L_{5,5}^{3}(\beta_1)\cong L_{5,5}^{3}(\beta_2)$ if and only if $\beta_1\beta_2^{-1}\in (\F^*)^2$.
\end{lemma}
\begin{proof}
Suppose that $f=(a_{ij}): L_{5,5}^3(\beta_1) \to  L_{5,5}^3(\beta_2)$ 
is an isomorphism.  Then we have $f(\res x_2)=f(x_2)^{[p]}$ which implies that
$\beta_1/\beta_2=a_{11}^{-2}a_{22}^{p-1}\in (\F^*)^2$.
To prove the converse, suppose that 
 $\beta_1/ \beta_2=\epsilon^2$, for some $\epsilon \in \F^*$. It is easy to see that the 
 following is an isomorphism from   $L_{5,5}^3(\beta_1)$ to  $L_{5,5}^3(\beta_2)$:
 \begin{align*}
\begin{pmatrix}
\epsilon^{-1/p} & 0 & 0 & 0 & 0\\
 0 & \epsilon^{2/p} &0 & 0 & 0\\
 0 & 0 & \epsilon^{1/p} & 0 & 0\\
 0 & 0 & 0 &\epsilon^{-2/p} &  0\\
0 & 0 & 0 & 0 & 1\\
\end{pmatrix} 
  \end{align*}
  
\end{proof}

\begin{lemma}
We have 
$L_{5,5}^{5}(\delta)\cong \langle x_1,\ldots,x_5\mid [x_1,x_2]=x_3, [x_1,x_3]=x_5, [x_2,x_4]=x_5,
\res x_4= x_5 \rangle$, for every $\delta\in \F^*$.
\end{lemma}
\begin{proof}
\begin{align*}
\begin{pmatrix}
1 & 0 & 0 & 0 & 0\\
 0 & \delta^{-1} &0 & 0 & 0\\
 0 & 0 &\delta^{-1}  & 0 & 0\\
 0 & 0 & 0 &1 &  0\\
0 & 0 & 0 & 0 & \delta^{-1} \\
\end{pmatrix}.
  \end{align*}

\end{proof}

\section{Restriction maps on $L_{5,6}$}
Let $ L_{5,6}=\langle x_1,\ldots,x_5\mid [x_1,x_2]=x_3, [x_1,x_3]=x_4, [x_1,x_4]=x_5, [x_2,x_3]=x_5\rangle.$  We have $Z(L_{5,6})=\la x_5\ra_{\F}$.  Let
$$L=\frac {L_{5,6}}{\la x_5\ra_{\F}} = \la x_1,x_2,x_3,x_4 \mid [x_1,x_2] =x_3, [x_1,x_3]=x_4 \ra \cong L_{4,3}.$$

The group $\Aut(L)$ consists of invertible matrices of the form
\[\begin{pmatrix} a_{11}  & 0 & 0 & 0 \\
a_{21} & a_{22} & 0 & 0\\
a_{31} & a_{32} & r& 0 \\
a_{41} & a_{42} & a_{11}a_{32}& a_{11}r
\end{pmatrix},\]
where $r= a_{11}a_{22} \neq 0$.

\begin{lemma}\label{LemmaL56}
Let $K=L_{5,6}$ and $[p]:K\to K$ be a $p$-map on $K$ and let $L=\frac{K}{M}$ where $M=\la x_5\ra_{\F}$. Then $K\cong L_{\theta}$ where $\theta=(\Delta_{14}+\Delta_{23},\omega)\in Z^2(L,\F)$.
\end{lemma}
\begin{proof}
 Let $\pi :K\rightarrow L$ be the projection map. We have the exact sequence 
$$ 0\rightarrow M\rightarrow K\rightarrow L\rightarrow 0.$$
Let $\sigma :L\rightarrow K$ such that $x_i\mapsto x_i$, $1\leq i\leq 4$. Then $\sigma$ is an injective linear map and $\pi \sigma =1_L$. Now, we define $\phi : L\times L\rightarrow M$ by $\phi (x_i,x_j)=[\sigma (x_i),\sigma (x_j)]-\sigma ([x_i,x_j])$, $1\leq i,j\leq 4$ and $\omega: L\rightarrow M$ by $\omega (x)=\res {\sigma (x)} -\sigma (\res x)$. Note that
\begin{align*}
&\phi(x_1,x_4)=[\sigma(x_1),\sigma(x_4)]-\sigma([x_1,x_4])=[x_1,x_4]=x_5;\\
&\phi(x_2,x_3)=[\sigma(x_2),\sigma(x_3)]-\sigma([x_2,x_3])=x_5;\\
&\phi(x_1,x_2)=[\sigma(x_1),\sigma(x_2)]-\sigma([x_1,x_2])=0;\\
\end{align*}
Similarly, we can show that $\phi(x_1,x_3)=\phi(x_2,x_4)=\phi(x_3,x_4)=0$. Therefore, $\phi=\Delta_{14}+\Delta_{23}$.
Now, by Lemma \ref{K=L-theta}, we have $\theta=(\Delta_{14}+\Delta_{23},\omega)\in Z^2(L,\F)$ and $K\cong L_{\theta}$.
\end{proof}

Note that by Theorem \ref{4-dim}, there are four non-isomorphic restricted Lie algebra structures on $L$ given by the following $p$-maps:
\begin{enumerate}
\item[I.1]Trivial $p$-map;
\item[I.2] $\res x_1=x_4$;
\item[I.3]  $\res x_2 =\xi x_4$;
\item[I.4]  $ \res x_3 =x_4$.
\end{enumerate}

We make $L$ into a restricted Lie algebra by equipping it with each of the above $p$-maps. Then,  in  each  case,  we  find all possible orbit representatives of the form $(\Delta_{14}+\Delta_{23},\omega)$ under the action of $\Aut_p(L)$ on $H^2(L,\F)$. By  Lemma \ref{LemmaL56}, we do get all possible $p$-maps on  $L_{5,6}$ with the property that $x_5^{[p]}=0$. 

Consider the case I.2 where the $p$-map of $L$ is given by $\res x_1 = x_4$.
Note that $\res L=\la x_4 \ra_{\F}$.  Let  $[(\phi,\omega)]\in H^2(L,\F)$. Then we must have $\phi (x,\res y)=0$, for all $x,y\in L$, where $\phi=a\Delta _{12}+b\Delta_{13}+c\Delta_{14}+d\Delta_{23}+e\Delta_{24}+f\Delta_{34}$, for some $a,b,c,d,e,f \in \F$.
Hence,  $\phi (x_1,x_4)=0$ which implies that $c=0$. Since $\phi =\Delta_{14}+\Delta_{23}$ gives us $L_{5,6}$, we deduce by Lemma \ref{LemmaL56} that $L_{5,6}$ cannot be constructed in this case.  Similarly we can show that in cases I.3 and I.4  we also get $c=0$. It remains to consider the case I.1.

\subsection{Extensions of $L_{5,6}/\la x_5\ra$ via the trivial $p$-map}
 First, we find a basis for $Z^2(L,\F)$. Let $(\phi,\omega)=(a\Delta _{12}+b\Delta_{13}+c\Delta_{14}+d\Delta_{23}+e\Delta_{24}+f\Delta_{34},\alpha f_1+\beta f_2+ \gamma f_3+\delta f_4)\in Z^2(L,\F)$. Then we must have $\delta^2\phi(x,y,z) =0$ and $\phi(x,\res y)=0$, for all $x,y,z \in L$. Therefore,
\begin{align*}
0=(\delta^2\phi)(x_1,x_2,x_3)&=\phi([x_1,x_2],x_3)+\phi([x_2,x_3],x_1)+\phi([x_3,x_1],x_2)=\phi(x_2,x_4), \text{ and }\\
0=(\delta^2\phi)(x_1,x_2,x_4)&=\phi([x_1,x_2],x_4)+\phi([x_2,x_4],x_1)+\phi([x_4,x_1],x_2)=\phi(x_3,x_4).
\end{align*}
Thus, we get $e=f=0$.
Since the $p$-map is trivial, $\phi(x,\res y)=\phi(x,0)=0$, for all $x,y\in L$. Therefore, a basis for $Z^2(L,\F)$ is as follows:
$$
 (\Delta_{12},0),(\Delta_{13},0),(\Delta_{14},0),(\Delta_{23},0),(0,f_1),
(0,f_2),(0,f_3),(0,f_4).
$$
Next, we find a basis for $B^2(L,\F)$. Let $(\phi,\omega)\in B^2(L,\F)$. Since $B^2(L,\F)\subseteq Z^2(L,\F)$, we have $(\phi,\omega)=(a\Delta _{12}+b\Delta_{13}+c\Delta_{14}+d\Delta_{23},\alpha f_1+\beta f_2+ \gamma f_3+\delta f_4)$. So, there exists a linear map $\psi:L\to \F$ such that $\delta^1\psi(x,y)=\phi(x,y)$ and $\tilde \psi(x)=\omega(x)$, for all $x,y \in L$. So, we have
\begin{align*}
c=\phi(x_1,x_4)=\delta^1\psi(x_1,x_4)=\psi([x_1,x_4])=0.
\end{align*}
Similarly, we can show that $d=0$. Also, we have
\begin{align*}
\alpha=\omega(x_1)=\tilde \psi(x_1)=\psi(\res x_1)=0.
\end{align*}
Similarly, we can show that $\beta=\gamma=\delta=0$. Therefore, $(\phi,\omega)=(a\Delta_{12}+b\Delta_{13},0)$ and hence
 $B^2(L,\F)=\la(\Delta_{12},0),(\Delta_{13},0)\ra_{\F}$. We deduce that a basis for $H^2(L,\F)$ is as follows:
$$[(\Delta_{14},0)],[(\Delta_{23},0)],[(0,f_1)],
[(0,f_2)],[(0,f_3)],[(0,f_4)].
$$
Let $[(\phi,\omega)] \in H^2(L,\F)$. Then we have $\phi = a\Delta_{14} + b\Delta_{23}$, for some $a, b\in \F$. Suppose that $A\phi = a'\Delta_{14}+b'\Delta_{23}$, for some $a', b'\in \F$. We determine $a', b'$. 
Note that 
\begin{align*}
A\phi (x_1,x_4)=&\phi (Ax_1,Ax_4)=\phi (a_{11}x_1+a_{21}x_2+a_{31}x_3+a_{41}x_4,a_{11}rx_4)=a_{11}^2ra; \\
A\phi (x_2,x_3)=&\phi (Ax_2,Ax_3)=\phi (a_{22}x_2+a_{32}x_3+a_{42}x_4,rx_3+a_{11}a_{32}x_4)=a_{22}rb.
\end{align*}
In the matrix form we can write this as
\begin{align}\label{34}
\begin{pmatrix}
a' \\ b'
\end{pmatrix}=
\begin{pmatrix}
a_{11}^2r& 0\\
0& a_{22}r
\end{pmatrix}
\begin{pmatrix}
a \\ b
\end{pmatrix}.
\end{align}
 The orbit with representative 
$\begin{pmatrix}
1 \\ 1
\end{pmatrix}$ of this action  gives us $L_{5,6}$.

Also, we have $\omega=\alpha f_1+\beta f_2+\gamma f_3+\delta f_4$, for some $\alpha, \beta, \gamma, \delta\in \F$.
 Suppose that $A\omega = \alpha' f_1+\beta' f_2+\gamma' f_3+\delta' f_4$, for some $\alpha', \beta', \gamma', \delta'\in \F$. We have
\begin{align*}
A\omega (x_1) =& \omega (Ax_1)=\omega (a_{11}x_1+a_{21}x_2+a_{31}x_3+a_{41}x_4)=a_{11}^p \alpha+a_{21}^p \beta +a_{31}^p \gamma+a_{41}^p \delta; \\
A\omega (x_2)=& a_{22}^p \beta +a_{32}^p \gamma + a_{42}^p \delta;\\
A\omega (x_3)=& r^p \gamma + a_{11}^pa_{32}^p \delta;\\
A\omega(x_4)=&a_{11}^pr^p \delta.
\end{align*}
In the matrix form we can write this as
\begin{align}\label{45}
\begin{pmatrix}
\alpha' \\ 
\beta'\\
\gamma'\\
\delta'
\end{pmatrix}
=
\begin{pmatrix}
a_{11}^p& a_{21}^p&a_{31}^p&a_{41}^p\\
0& a_{22}^p& a_{32}^p&a_{42}^p\\
0&0&r^p&a_{11}^pa_{32}^p\\
0&0&0&a_{11}^pr^p
\end{pmatrix}
\begin{pmatrix}
\alpha\\ 
\beta \\
\gamma \\
\delta
\end{pmatrix}.
\end{align}
Thus, we can write Equations \eqref{34} and \eqref{45} together as follows:
\begin{align*}
\bigg[
r\begin{pmatrix}
a_{11}^2& 0\\
0& a_{22}
\end{pmatrix},
\begin{pmatrix}
a_{11}^p& a_{21}^p&a_{31}^p&a_{41}^p\\
0& a_{22}^p& a_{32}^p&a_{42}^p\\
0&0&r^p&a_{11}^pa_{32}^p\\
0&0&0&a_{11}^pr^p
\end{pmatrix}
\bigg]
\bigg[
\begin{pmatrix}
a \\ b
\end{pmatrix},
\begin{pmatrix}
\alpha\\ 
\beta \\
\gamma \\
\delta
\end{pmatrix}
\bigg]
=
\bigg[
\begin{pmatrix}
a' \\ b'
\end{pmatrix},
\begin{pmatrix}
\alpha' \\ 
\beta'\\
\gamma'\\
\delta'
\end{pmatrix}
\bigg].
\end{align*}
Now we find the representatives of the orbits of the action of $\Aut (L)$ on the set of $\omega$'s such that 
the orbit represented by 
$\begin{pmatrix}
1 \\ 1
\end{pmatrix}$ is preserved under the action of $\Aut (L)$ on the set of $\phi$'s.

Let $\nu = \begin{pmatrix}
\alpha\\ 
\beta \\
\gamma \\
\delta
\end{pmatrix} \in \F^4$. If $\nu= \begin{pmatrix}
0\\ 
0 \\
0 \\
0
\end{pmatrix}$, then $\{\nu \}$ is clearly an $\Aut (L)$-orbit. Let $\nu \neq 0$. Suppose that $\delta \neq 0$. Then 
\begin{align*}
\bigg[\begin{pmatrix}
1& 0\\
0& 1
\end{pmatrix},
\begin{pmatrix}
1& 0&0&-\alpha /\delta\\
0& 1& 0&-\beta /\delta\\
0&0&1&0\\
0&0&0&1
\end{pmatrix}\bigg]
\bigg[\begin{pmatrix}
1 \\ 1
\end{pmatrix},
\begin{pmatrix}
\alpha\\ 
\beta \\
\gamma \\
\delta
\end{pmatrix}\bigg]
=&
\bigg[\begin{pmatrix}
1\\1
\end{pmatrix},
\begin{pmatrix}
0 \\ 
0\\
\gamma\\
\delta
\end{pmatrix}\bigg], \text{ and }\\
\bigg[\begin{pmatrix}
1& 0\\
0& 1
\end{pmatrix},
\begin{pmatrix}
1& 0&0&0\\
0& 1& -\gamma /\delta&\gamma^2/\delta^2\\
0&0&1&-\gamma /\delta\\
0&0&0&1
\end{pmatrix}\bigg]
\bigg[\begin{pmatrix}
1 \\ 1
\end{pmatrix},
\begin{pmatrix}
0\\ 
0 \\
\gamma \\
\delta
\end{pmatrix}\bigg]
=&
\bigg[\begin{pmatrix}
1\\1
\end{pmatrix},
\begin{pmatrix}
0 \\ 
0\\
0\\
\delta
\end{pmatrix}\bigg].
\end{align*}
Next, if $\delta= 0$ , but $\gamma \neq 0$, then
\begin{align*}
&\bigg[\begin{pmatrix}
1& 0\\
0& 1
\end{pmatrix},
\begin{pmatrix}
1& 0&-\alpha /\gamma&0\\
0& 1& -\beta /\gamma &0\\
0&0&1&-\beta /\gamma\\
0&0&0&1
\end{pmatrix}\bigg]
\bigg[\begin{pmatrix}
1 \\ 1
\end{pmatrix},
\begin{pmatrix}
\alpha\\ 
\beta\\
\gamma \\
0
\end{pmatrix}\bigg]=
\bigg[\begin{pmatrix}
1\\1
\end{pmatrix},
\begin{pmatrix}
0 \\ 
0\\
\gamma\\
0
\end{pmatrix}\bigg].
\end{align*}
Next, if $\gamma=\delta =0$, but $\beta \neq 0$, then 
\begin{align*}
&\bigg[\begin{pmatrix}
1& 0\\
0& 1
\end{pmatrix},
\begin{pmatrix}
1& -\alpha /\beta&0&0\\
0& 1& 0&0\\
0&0&1&0\\
0&0&0&1
\end{pmatrix}\bigg]
\bigg[\begin{pmatrix}
1 \\ 1
\end{pmatrix},
\begin{pmatrix}
\alpha\\ 
\beta\\
0 \\
0
\end{pmatrix}\bigg]=
\bigg[\begin{pmatrix}
1\\1
\end{pmatrix},
\begin{pmatrix}
0 \\ 
\beta\\
0\\
0
\end{pmatrix}\bigg].
\end{align*}
Finally, if  $\beta=\gamma=\delta =0$, but $\alpha \neq 0$, then we have
$\bigg[\begin{pmatrix}
1\\1
\end{pmatrix},
\begin{pmatrix}
\alpha \\ 
0\\
0\\
0
\end{pmatrix}\bigg]$.\\
Thus the following elements are $\Aut (L)$-orbit representatives: 
\begin{align*}
\begin{pmatrix}
0\\ 
0 \\
0 \\
0
\end{pmatrix}, 
\begin{pmatrix}
\alpha \\ 
0\\
0\\
0
\end{pmatrix}, 
\begin{pmatrix}
0 \\ 
\beta\\
0\\
0
\end{pmatrix}, 
\begin{pmatrix}
0\\ 
0\\
\gamma\\
0
\end{pmatrix}, 
\begin{pmatrix}
0 \\ 
0\\
0\\
\delta
\end{pmatrix}.
\end{align*}

\begin{theorem}
The list of all restricted Lie algebra structures on $L_{5,6}$, up to isomorphism, is as follows:
\begin{align*}
& L_{5,6}^1=\langle x_1,\ldots,x_5\mid [x_1,x_2]=x_3,[x_1,x_3]=x_4, [x_1,x_4]=x_5, [x_2,x_3]=x_5 \rangle;\\
& L_{5,6}^2(\alpha)=\langle x_1,\ldots,x_5\mid [x_1,x_2]=x_3,[x_1,x_3]=x_4, [x_1,x_4]=x_5, [x_2,x_3]=x_5 ,
\res x_1= \alpha x_5\rangle;\\
& L_{5,6}^3(\beta)=\langle x_1,\ldots,x_5\mid [x_1,x_2]=x_3,[x_1,x_3]=x_4,[x_1,x_4]=x_5,  [x_2,x_3]=x_5 ,
\res x_2=\beta x_5\rangle;\\
& L_{5,6}^4(\gamma)=\langle x_1,\ldots,x_5\mid [x_1,x_2]=x_3,[x_1,x_3]=x_4, [x_1,x_4]=x_5, [x_2,x_3]=x_5 ,
\res x_3=\gamma x_5\rangle;\\
& L_{5,6}^5(\delta)=\langle x_1,\ldots,x_5\mid [x_1,x_2]=x_3,[x_1,x_3]=x_4, [x_1,x_4]=x_5, [x_2,x_3]=x_5 ,
\res x_4= \delta x_5\rangle
\end{align*} 
where $\alpha, \gamma, \delta \in T_{2,5}$ and $\beta \in T_{3,5}$.
\end{theorem}

The automorphism group consists of
$$ \begin{pmatrix} a_{11} & 0 & 0 & 0 & 0 \\
                   a_{21} & a_{11}^2 & 0 & 0 & 0\\
                   a_{31} & a_{32} & a_{11}^3 & 0 & 0 \\
                   a_{41} & a_{42} & a_{11}a_{32} & a_{11}^4 & 0 \\
                   a_{51} & a_{52} & u & v & a_{11}^5
\end{pmatrix},$$
where $u=a_{11}a_{42}+a_{21}a_{32}-a_{31}a_{11}^2$, $v=a_{21}a_{11}^3
+a_{32}a_{11}^2$. \par

\begin{lemma}\label{lemma-K6-2}
We have $L_{5,6}^{2}(\alpha_1)\cong L_{5,6}^{2}(\alpha_2)$ if and only if  $\alpha_1/\alpha_2=\epsilon^{p-5}$, for some $\epsilon \in \F^*$.
\end{lemma}
\begin{proof}
Suppose that $f=(a_{ij}): L_{5,6}^{2}(\alpha_1)\to L_{5,6}^{2}(\alpha_2)$ 
is an isomorphism.  Then we have $f(\res x_1)=f(x_1)^{[p]}$ which implies that
$\alpha_1/\alpha_2=a_{11}^{p-5}$.
To prove the converse, suppose that 
 $\alpha_1/\alpha_2=\epsilon^{p-5}$, for some $\epsilon \in \F^*$. It is easy to see that the 
 following is an isomorphism from   $L_{5,6}^{2}(\alpha_1)$  to $L_{5,6}^{2}(\alpha_2)$:
 \begin{align*}
\begin{pmatrix}
\epsilon & 0 & 0 & 0 & 0\\
 0 & \epsilon^{2} &0 & 0 & 0\\
 0 & 0 & \epsilon^{3} & 0 & 0\\
 0 & 0 & 0 &\epsilon^{4} &  0\\
0 & 0 & 0 & 0 & \epsilon^{5}\\
\end{pmatrix} 
  \end{align*}
  
\end{proof}

\begin{lemma}\label{lemma-K6-3}
We have $L_{5,6}^{3}(\beta_1)\cong L_{5,6}^{3}(\beta_2)$ if and only if 
 $\beta_1/\beta_2=\epsilon^{2p-5}$, for some $\epsilon \in \F^*$.
\end{lemma}
\begin{proof}
Suppose that $f=(a_{ij}): L_{5,6}^{3}(\beta_1)\to L_{5,6}^{3}(\beta_2)$ 
is an isomorphism.  Then we have $f(\res x_2)=f(x_2)^{[p]}$ which implies that
$\alpha_1/\alpha_2=a_{11}^{2p-5}$.
To prove the converse, suppose that 
 $\beta_1/\beta_2=\epsilon^{2p-5}$, for some $\epsilon \in \F^*$. It is easy to see that the 
 following is an isomorphism from   $L_{5,6}^{2}(\alpha_1)$  to $L_{5,6}^{2}(\alpha_2)$:
 \begin{align*}
\begin{pmatrix}
\epsilon & 0 & 0 & 0 & 0\\
 0 & \epsilon^{2} &0 & 0 & 0\\
 0 & 0 & \epsilon^{3} & 0 & 0\\
 0 & 0 & 0 &\epsilon^{4} &  0\\
0 & 0 & 0 & 0 & \epsilon^{5}\\
\end{pmatrix} 
  \end{align*}
  
\end{proof}

\begin{lemma}\label{lemma-K6-4}
We have $L_{5,6}^{4}(\gamma_1)\cong L_{5,6}^{4}(\gamma_2)$ if and only if 
 $\gamma_1/\gamma_2=\epsilon^{3p-5}$, for some $\epsilon \in \F^*$.
\end{lemma}
\begin{proof}
Suppose that $f=(a_{ij}): L_{5,6}^{4}(\gamma_1)\to L_{5,6}^{4}(\gamma_2)$ 
is an isomorphism.  Then we have $f(\res x_3)=f(x_3)^{[p]}$ which implies that
$\gamma_1/\gamma_2=a_{11}^{3p-5}$.
To prove the converse, suppose that 
 $\gamma_1/\gamma_2=a_{11}^{3p-5}$, for some $\epsilon \in \F^*$. It is easy to see that the 
 following is an isomorphism from   $L_{5,6}^{4}(\gamma_1)$  to $L_{5,6}^{4}(\gamma_2)$:
 \begin{align*}
\begin{pmatrix}
\epsilon & 0 & 0 & 0 & 0\\
 0 & \epsilon^{2} &0 & 0 & 0\\
 0 & 0 & \epsilon^{3} & 0 & 0\\
 0 & 0 & 0 &\epsilon^{4} &  0\\
0 & 0 & 0 & 0 & \epsilon^{5}\\
\end{pmatrix} 
  \end{align*}
  
\end{proof}

\begin{lemma}\label{lemma-K6-5}
We have $L_{5,6}^{5}(\delta_1)\cong L_{5,6}^{5}(\delta_2)$ if and only if 
 $\delta_1/\delta_2=\epsilon^{4p-5}$, for some $\epsilon \in \F^*$.
\end{lemma}
\begin{proof}
Suppose that $f=(a_{ij}): L_{5,6}^{5}(\delta_1) \to L_{5,6}^{5}(\delta_2)$ 
is an isomorphism.  Then we have $f(\res x_4)=f(x_4)^{[p]}$ which implies that
$\delta_1/\delta_2=a_{11}^{4p-5}$.
To prove the converse, suppose that 
 $\delta_1/\delta_2=\epsilon^{4p-5}$, for some $\epsilon \in \F^*$. It is easy to see that the 
 following is an isomorphism from   $L_{5,6}^{5}(\delta_1)$  to $L_{5,6}^{5}(\delta_2)$:
 \begin{align*}
\begin{pmatrix}
\epsilon & 0 & 0 & 0 & 0\\
 0 & \epsilon^{2} &0 & 0 & 0\\
 0 & 0 & \epsilon^{3} & 0 & 0\\
 0 & 0 & 0 &\epsilon^{4} &  0\\
0 & 0 & 0 & 0 & \epsilon^{5}\\
\end{pmatrix} 
  \end{align*}
  
\end{proof}

\section{Restriction maps on $L_{5,7}$} Let 
$L_{5,7} = \langle x_1, \ldots, x_5\mid  [x_1,x_2]=x_3 , [x_1,x_3]=x_4 , [x_1,x_4]=x_5\rangle. $
We have $Z(L_{5,7})=\la x_5\ra_{\F} $. Let
$$L=\frac {L_{5,7}}{\la x_5\ra}\cong L_{4,3},
$$ where $L_{4,3} = \la x_1,\ldots,x_4 \mid [x_1 , x_2]=x_3 , [x_1 , x_3]=x_4\ra $. 
Note that the group $\Aut(L)$ consists of invertible matrices of the form
\[\begin{pmatrix} a_{11}  & 0 & 0 & 0 \\
a_{21} & a_{22} & 0 & 0\\
a_{31} & a_{32} & r& 0 \\
a_{41} & a_{42} & a_{11}a_{32}& a_{11}r
\end{pmatrix},\]
where $r= a_{11}a_{22} \neq 0$.

\begin{lemma}\label{LemmaL57}
Let $K=L_{5,7}$ and $[p]:K\to K$ be a $p$-map on $K$ and let $L=\frac{K}{M}$ where $M=\la x_5\ra_{\F}$. Then $K\cong L_{\theta}$ where $\theta=(\Delta_{14},\omega)\in Z^2(L,\F)$.
\end{lemma}
\begin{proof}
 Let $\pi :K\rightarrow L$ be the projection map. We have the exact sequence 
$$ 0\rightarrow M\rightarrow K\rightarrow L\rightarrow 0.$$
Let $\sigma :L\rightarrow K$ such that $x_i\mapsto x_i$, $1\leq i\leq 4$. Then $\sigma$ is an injective linear map and $\pi \sigma =1_L$. Now, we define $\phi : L\times L\rightarrow M$ by $\phi (x_i,x_j)=[\sigma (x_i),\sigma (x_j)]-\sigma ([x_i,x_j])$, $1\leq i,j\leq 4$ and $\omega: L\rightarrow M$ by $\omega (x)=\res {\sigma (x)} -\sigma (\res x)$. Note that
\begin{align*}
&\phi(x_1,x_4)=[\sigma(x_1),\sigma(x_4)]-\sigma([x_1,x_4])=[x_1,x_4]=x_5;\\
&\phi(x_1,x_2)=[\sigma(x_1),\sigma(x_2)]-\sigma([x_1,x_2])=0;\\
\end{align*}
Similarly, we can show that $\phi(x_1,x_3)=\phi(x_2,x_3)=\phi(x_2,x_4)=\phi(x_3,x_4)=0$. Therefore, $\phi=\Delta_{14}$.
Now, by Lemma \ref{K=L-theta}, we have $\theta=(\Delta_{14},\omega)\in Z^2(L,\F)$ and $K\cong L_{\theta}$.
\end{proof}

Note that by Theorem \ref{4-dim}, there are four non-isomorphic restricted Lie algebra structures on $L_{4,3}$ given by the following $p$-maps:
\begin{enumerate}
\item[I.1]Trivial $p$-map;
\item[I.2] $\res x_1=x_4$;
\item[I.3]  $\res x_2 =\xi x_4$;
\item[I.4]  $ \res x_3 =x_4 $.
\end{enumerate}

In the following subsections, we make $L$ into a restricted Lie algebra by equipping it with each of the above $p$-maps. Then,  in  each  case,  we  find all possible orbit representatives of the form $(\Delta_{14},\omega)$ under the action of $\Aut_p(L)$ on $H^2(L,\F)$. By  Lemma \ref{LemmaL57}, we do get all possible $p$-maps on  $L_{5,7}$ with the property that $x_5^{[p]}=0$.

Consider the case I.2 where the $p$-map of $L$ is given by $\res x_1 = x_4$.
 Let  $[(\phi,\omega)]\in H^2(L,\F)$. Then we must have $\phi (x,\res y)=0$, for all $x,y\in L$, where $\phi=a\Delta _{12}+b\Delta_{13}+c\Delta_{14}+d\Delta_{23}+e\Delta_{24}+f\Delta_{34}$, for some $a,b,c,d,e,f \in \F$.
Since $\res L=\la x_4 \ra_{\F}$, we get  $\phi (x_1,x_4)=0$ which implies that $c=0$. Since $\phi =\Delta_{14}$ gives us $L_{5,7}$, we deduce by Lemma \ref{LemmaL57} that $L_{5,7}$ cannot be constructed in this case.  Similarly, we can show in cases I.3 and I.4  we also get $c=0$. It remains to consider the case I.1.

\subsection{Extensions of $L_{5,7}/\la x_5\ra$ via the trivial $p$-map}
 First, we find a basis for $Z^2(L,\F)$. Let $(\phi,\omega)=(a\Delta _{12}+b\Delta_{13}+c\Delta_{14}+d\Delta_{23}+e\Delta_{24}+f\Delta_{34},\alpha f_1+\beta f_2+ \gamma f_3+\delta f_4)\in Z^2(L,\F)$. Then we must have $\delta^2\phi(x,y,z) =0$ and $\phi(x,\res y)=0$, for all $x,y,z \in L$. Therefore,
\begin{align*}
0=(\delta^2\phi)(x_1,x_2,x_3)&=\phi([x_1,x_2],x_3)+\phi([x_2,x_3],x_1)+\phi([x_3,x_1],x_2)=\phi(x_2,x_4), \text{ and }\\
0=(\delta^2\phi)(x_1,x_2,x_4)&=\phi([x_1,x_2],x_4)+\phi([x_2,x_4],x_1)+\phi([x_4,x_1],x_2)=\phi(x_3,x_4).
\end{align*}
Thus, we get $e=f=0$.
Since the $p$-map is trivial, $\phi(x,\res y)=\phi(x,0)=0$, for all $x,y\in L$. Therefore, a basis for $Z^2(L,\F)$ is as follows:
$$
 (\Delta_{12},0),(\Delta_{13},0),(\Delta_{14},0),(\Delta_{23},0),(0,f_1),
(0,f_2),(0,f_3),(0,f_4).
$$
Next, we find a basis for $B^2(L,\F)$. Let $(\phi,\omega)\in B^2(L,\F)$. Since $B^2(L,\F)\subseteq Z^2(L,\F)$, we have $(\phi,\omega)=(a\Delta _{12}+b\Delta_{13}+c\Delta_{14}+d\Delta_{23},\alpha f_1+\beta f_2+ \gamma f_3+\delta f_4)$. So, there exists a linear map $\psi:L\to \F$ such that $\delta^1\psi(x,y)=\phi(x,y)$ and $\tilde \psi(x)=\omega(x)$, for all $x,y \in L$. So, we have
\begin{align*}
c=\phi(x_1,x_4)=\delta^1\psi(x_1,x_4)=\psi([x_1,x_4])=0.
\end{align*}
Similarly, we can show that $d=0$. Also, we have
\begin{align*}
\alpha=\omega(x_1)=\tilde \psi(x_1)=\psi(\res x_1)=0.
\end{align*}
Similarly, we can show that $\beta=\gamma=\delta=0$. Therefore, $(\phi,\omega)=(a\Delta_{12}+b\Delta_{13},0)$ and hence
 $B^2(L,\F)=\la(\Delta_{12},0),(\Delta_{13},0)\ra_{\F}$. We deduce that a basis for $H^2(L,\F)$ is as follows:
$$[(\Delta_{14},0)],[(\Delta_{23},0)],[(0,f_1)],
[(0,f_2)],[(0,f_3)],[(0,f_4)].
$$
Let $[(\phi,\omega)] \in H^2(L,\F)$. Then we have $\phi = a\Delta_{14} + b\Delta_{23}$, for some $a, b\in \F$. Suppose that $A\phi = a'\Delta_{14}+b'\Delta_{23}$, for some $a', b'\in \F$. We determine $a', b'$. 
Note that 
\begin{align*}
A\phi (x_1,x_4)=&\phi (Ax_1,Ax_4)=\phi (a_{11}x_1+a_{21}x_2+a_{31}x_3+a_{41}x_4,a_{11}rx_4)=a_{11}^2ra; \\
A\phi (x_2,x_3)=&\phi (Ax_2,Ax_3)=\phi (a_{22}x_2+a_{32}x_3+a_{42}x_4,rx_3+a_{11}a_{32}x_4)=a_{22}rb.
\end{align*}
In the matrix form we can write this as
\begin{align}\label{56}
\begin{pmatrix}
a' \\ b'
\end{pmatrix}=
\begin{pmatrix}
a_{11}^2r& 0\\
0& a_{22}r
\end{pmatrix}
\begin{pmatrix}
a \\ b
\end{pmatrix}.
\end{align}
 The orbit with representative 
$\begin{pmatrix}
1 \\ 0
\end{pmatrix}$ of this action  gives us $L_{5,7}$.\\
Also, we have $\omega=\alpha f_1+\beta f_2+\gamma f_3+\delta f_4$, for some $\alpha, \beta, \gamma, \delta\in \F$.
 Suppose that $A\omega = \alpha' f_1+\beta' f_2+\gamma' f_3+\delta' f_4$, for some $\alpha', \beta', \gamma', \delta'\in \F$. We have
\begin{align*}
A\omega (x_1) =& \omega (Ax_1)=\omega (a_{11}x_1+a_{21}x_2+a_{31}x_3+a_{41}x_4)=a_{11}^p \alpha+a_{21}^p \beta +a_{31}^p \gamma+a_{41}^p \delta; \\
A\omega (x_2)=& a_{22}^p \beta +a_{32}^p \gamma + a_{42}^p \delta;\\
A\omega (x_3)=& r^p \gamma + a_{11}^pa_{32}^p \delta;\\
A\omega(x_4)=&a_{11}^pr^p \delta.
\end{align*}
In the matrix form we can write this as
\begin{align}\label{67}
\begin{pmatrix}
\alpha' \\ 
\beta'\\
\gamma'\\
\delta'
\end{pmatrix}
=
\begin{pmatrix}
a_{11}^p& a_{21}^p&a_{31}^p&a_{41}^p\\
0& a_{22}^p& a_{32}^p&a_{42}^p\\
0&0&r^p&a_{11}^pa_{32}^p\\
0&0&0&a_{11}^pr^p
\end{pmatrix}
\begin{pmatrix}
\alpha\\ 
\beta \\
\gamma \\
\delta
\end{pmatrix}.
\end{align}
Thus, we can write Equations \eqref{56} and \eqref{67} together as follows:
\begin{align*}
\bigg[
r\begin{pmatrix}
a_{11}^2& 0\\
0& a_{22}
\end{pmatrix},
\begin{pmatrix}
a_{11}^p& a_{21}^p&a_{31}^p&a_{41}^p\\
0& a_{22}^p& a_{32}^p&a_{42}^p\\
0&0&r^p&a_{11}^pa_{32}^p\\
0&0&0&a_{11}^pr^p
\end{pmatrix}
\bigg]
\bigg[
\begin{pmatrix}
a \\ b
\end{pmatrix},
\begin{pmatrix}
\alpha\\ 
\beta \\
\gamma \\
\delta
\end{pmatrix}
\bigg]
=
\bigg[
\begin{pmatrix}
a' \\ b'
\end{pmatrix},
\begin{pmatrix}
\alpha' \\ 
\beta'\\
\gamma'\\
\delta'
\end{pmatrix}
\bigg].
\end{align*}
Now we find the representatives of the orbits of the action of $\Aut (L)$ on the set of $\omega$'s such that 
the orbit represented by 
$\begin{pmatrix}
1 \\ 0
\end{pmatrix}$ is preserved under the action of $\Aut (L)$ on the set of $\phi$'s.
Let $\nu = \begin{pmatrix}
\alpha\\ 
\beta \\
\gamma \\
\delta
\end{pmatrix} \in \F^4$. If $\nu= \begin{pmatrix}
0\\ 
0 \\
0 \\
0
\end{pmatrix}$, then $\{\nu \}$ is clearly an $\Aut (L)$-orbit. Let $\nu \neq 0$. Suppose that $\delta \neq 0$. Then 
\begin{align*}
&\bigg[\begin{pmatrix}
1& 0\\
0& 1
\end{pmatrix},
\begin{pmatrix}
1& 0&0&-\alpha /\delta\\
0& 1& 0&-\beta /\delta\\
0&0&1&0\\
0&0&0&1
\end{pmatrix}\bigg]
\bigg[\begin{pmatrix}
1 \\ 0
\end{pmatrix},
\begin{pmatrix}
\alpha\\ 
\beta \\
\gamma \\
\delta
\end{pmatrix}\bigg]
=
\bigg[\begin{pmatrix}
1\\0
\end{pmatrix},
\begin{pmatrix}
0 \\ 
0\\
\gamma\\
\delta
\end{pmatrix}\bigg], \text{ and }\\
&\bigg[\begin{pmatrix}
1& 0\\
0& 1
\end{pmatrix},
\begin{pmatrix}
1& 0&0&0\\
0& 1& -\gamma /\delta&\gamma^2/\delta^2\\
0&0&1&-\gamma /\delta\\
0&0&0&1
\end{pmatrix}\bigg]
\bigg[\begin{pmatrix}
1 \\ 0
\end{pmatrix},
\begin{pmatrix}
0\\ 
0 \\
\gamma \\
\delta
\end{pmatrix}\bigg]
=
\bigg[\begin{pmatrix}
1\\0
\end{pmatrix},
\begin{pmatrix}
0 \\ 
0\\
0\\
\delta
\end{pmatrix}\bigg], \text{ and }\\
&\bigg[\delta^{-2/p}\begin{pmatrix}
\delta^{2/p}& 0\\
0& \delta^{-3/p}
\end{pmatrix},
\begin{pmatrix}
\delta& 0&0&0\\
0& \delta^{-3}& 0\\
0&0&\delta^{-2}&0\\
0&0&0&\delta^{-1}
\end{pmatrix}\bigg]
\bigg[\begin{pmatrix}
1 \\ 0
\end{pmatrix},
\begin{pmatrix}
0\\ 
0 \\
0\\
\delta
\end{pmatrix}\bigg]
=
\bigg[\begin{pmatrix}
1\\0
\end{pmatrix},
\begin{pmatrix}
0 \\ 
0\\
0\\
1
\end{pmatrix}\bigg].
\end{align*}
Next, if $\delta= 0$ , but $\gamma \neq 0$, then
\begin{align*}
&\bigg[\begin{pmatrix}
1& 0\\
0& 1
\end{pmatrix},
\begin{pmatrix}
1& 0&-\alpha /\gamma&0\\
0& 1& -\beta /\gamma &0\\
0&0&1&-\beta/\gamma\\
0&0&0&1
\end{pmatrix}\bigg]
\bigg[\begin{pmatrix}
1\\ 0
\end{pmatrix},
\begin{pmatrix}
\alpha\\ 
\beta\\
\gamma \\
0
\end{pmatrix}\bigg]=
\bigg[\begin{pmatrix}
1\\0
\end{pmatrix},
\begin{pmatrix}
0 \\ 
0\\
\gamma\\
0
\end{pmatrix}\bigg].
\end{align*}
Next, if $\gamma=\delta =0$, but $\beta \neq 0$, then 
\begin{align*}
&\bigg[\begin{pmatrix}
1& 0\\
0& 1
\end{pmatrix},
\begin{pmatrix}
1& -\alpha /\beta&0&0\\
0& 1& 0&0\\
0&0&1&0\\
0&0&0&1
\end{pmatrix}\bigg]
\bigg[\begin{pmatrix}
1 \\ 0
\end{pmatrix},
\begin{pmatrix}
\alpha\\ 
\beta\\
0 \\
0
\end{pmatrix}\bigg]=
\bigg[\begin{pmatrix}
1\\0
\end{pmatrix},
\begin{pmatrix}
0 \\ 
\beta\\
0\\
0
\end{pmatrix}\bigg].
\end{align*}
Finally, if  $\beta=\gamma=\delta =0$, but $\alpha \neq 0$, then we have
\begin{align*}
 \bigg[\alpha^{2/p}\begin{pmatrix}
\alpha^{-2/p}& 0\\
0& \alpha^{3/p}
\end{pmatrix},
\begin{pmatrix}
\alpha^{-1}& 0&0&0\\
0& \alpha^3&0 &0\\
0&0&\alpha^2&0\\
0&0&0&\alpha
\end{pmatrix}\bigg]
\bigg[\begin{pmatrix}
1 \\ 0
\end{pmatrix},
\begin{pmatrix}
\alpha\\ 
0\\
0 \\
0
\end{pmatrix}\bigg]=
\bigg[\begin{pmatrix}
1\\0
\end{pmatrix},
\begin{pmatrix}
1 \\ 
0\\
0\\
0
\end{pmatrix}\bigg].
\end{align*}
Thus the following elements are $\Aut (L)$-orbit representatives: 
\begin{align*}
\begin{pmatrix}
0\\ 
0 \\
0 \\
0
\end{pmatrix}, 
\begin{pmatrix}
1 \\ 
0\\
0\\
0
\end{pmatrix}, 
\begin{pmatrix}
0 \\ 
\beta\\
0\\
0
\end{pmatrix}, 
\begin{pmatrix}
0\\ 
0\\
\gamma\\
0
\end{pmatrix}, 
\begin{pmatrix}
0 \\ 
0\\
0\\
1
\end{pmatrix}.
\end{align*}

\begin{theorem}
The list of all restricted Lie algebra structures on $L_{5,7}$, up to isomorphism, is as follows:
\begin{align*}
& L_{5,7}^1=\langle x_1,\ldots,x_5\mid [x_1,x_2]=x_3,[x_1,x_3]=x_4, [x_1,x_4]=x_5 \rangle;\\
& L_{5,7}^2=\langle x_1,\ldots,x_5\mid [x_1,x_2]=x_3,[x_1,x_3]=x_4, [x_1,x_4]=x_5 ,
\res x_1=x_5\rangle;\\
& L_{5,7}^3(\beta)=\langle x_1,\ldots,x_5\mid [x_1,x_2]=x_3,[x_1,x_3]=x_4, [x_1,x_4]=x_5 ,
\res x_2=\beta x_5\rangle;\\
& L_{5,7}^4(\gamma)=\langle x_1,\ldots,x_5\mid [x_1,x_2]=x_3,[x_1,x_3]=x_4, [x_1,x_4]=x_5 ,
\res x_3=\gamma x_5\rangle;\\
& L_{5,7}^5=\langle x_1,\ldots,x_5\mid [x_1,x_2]=x_3,[x_1,x_3]=x_4, [x_1,x_4]=x_5 ,
\res x_4=x_5\rangle
\end{align*}
where $\beta, \gamma \in\F^*$.
\end{theorem}

The automorphism group $\Aut(L_{5,7})$ consists of
$$ \begin{pmatrix} a_{11} & 0 & 0 & 0 & 0 \\
                   a_{21} & a_{22} & 0 & 0 & 0\\
                   a_{31} & a_{32} & a_{11}a_{22} & 0 & 0 \\
                   a_{41} & a_{42} & a_{11}a_{32} & a_{11}^2a_{22} & 0 \\
                   a_{51} & a_{52} & a_{11}a_{42} & a_{11}^2a_{32} & 
a_{11}^3a_{22}
\end{pmatrix}.$$

\begin{lemma}\label{lemma-K7-3}
We have $L_{5,7}^{3}(\beta_1)\cong L_{5,7}^{3}(\beta_2)$ if and only if the exists non-zero $x, y\in \F^*$ such that $\beta_1/\beta_2=x^3y^{p-1}$.
\end{lemma}
\begin{proof}
Let $f=(a_{ij})\in \Aut(L_{5,7})$. 
We just need to observe that  $f$ is an isomorphism between  $L_{5,7}^3(\beta_1)$
and $L_{5,7}^3(\beta_2)$  if and only if $f(\res x_2)=f(x_2)^{[p]}$ which in turn is equivalent to saying that  
$\beta_1/\beta_2=a_{11}^{-3}a_{22}^{p-1}$.

  \end{proof}
  
Note that the solutions to  Equation   $\beta_1/\beta_2=x^3y^{p-1}$ in Lemma 
\ref{lemma-K7-3} depends on the underlying field. For example, over the prime field $\F_p$, we  have
that $L_{5,7}^{3}(\beta_1)\cong L_{5,7}^{3}(\beta_2)$ if and only if the ratio  $\beta_1/\beta_2$ is a cubic.
  
  \begin{lemma}\label{lemma-K7-4}
We have $L_{5,7}^{4}(\gamma_1)\cong L_{5,7}^{4}(\gamma_2)$ if and only if $\gamma_1/\gamma_2\in (\F^*)^2$.
\end{lemma}
\begin{proof}
Suppose that $f=(a_{ij}): L_{5,7}^{4}(\gamma_1)\to L_{5,7}^{4}(\gamma_2)$ 
is an isomorphism.  Then we have $f(\res x_3)=f(x_3)^{[p]}$ which implies that
$\gamma_1/\gamma_2=a_{11}^{p-3}a_{22}^{p-1}\in (\F^*)^2$.
To prove the converse, suppose that 
 $\gamma_1/\gamma_2=\epsilon^2$, for some $\epsilon \in \F^*$. It is easy to see that the 
 following is an isomorphism from   $L_{5,7}^{4}(\gamma_1)$ to  $L_{5,7}^{4}(\gamma_2)$:
 \begin{align*}
\begin{pmatrix}
\epsilon^{-1/p} & 0 & 0 & 0 & 0\\
 0 & \epsilon^{3/p} &0 & 0 & 0\\
 0 & 0 & \epsilon^{2/p} & 0 & 0\\
 0 & 0 & 0 &\epsilon^{1/p} &  0\\
0 & 0 & 0 & 0 & 1\\
\end{pmatrix}.
  \end{align*}
  \end{proof}

%% file: 5,8.tex
\chapter{Restriction maps on $L_{5,8}$}
Let 
$$
K_8=L_{5,8}=\langle x_1,\ldots,x_5\mid [x_1,x_2]=x_4, [x_1,x_3]=x_5\rangle.
$$  
Then  $Z(L_{5,8})=\la x_4,x_5\ra_{\F}$ and the group $\Aut(L_{5,8})$ consists of invertible matrices of the form 
\[\begin{pmatrix} a_{11} & 0 & 0 & 0 & 0 \\
a_{21} & a_{22} & a_{23}& 0 & 0\\
a_{31} & a_{32} & a_{33} & 0 & 0 \\
a_{41} & a_{42} & a_{43}& a_{11}a_{22}
& a_{11}a_{23} \\
a_{51} & a_{52} & a_{53}& a_{11}a_{32}
& a_{11}a_{33}
\end{pmatrix}.\]
Note that there exists an element $\alpha x_4+\beta x_5 \in Z(L_{5,8})$ such that 
 $\res {(\alpha x_4+\beta x_5)}=0,$ for some $\alpha ,\beta \in \F$. 
If $\alpha \neq 0$ then consider 
 $$ 
 K=\langle y_1,\ldots,y_5\mid [y_1,y_2]=y_4, [y_1,y_3]=y_5\rangle,
 $$
 where $y_1=x_1, y_2=\alpha x_2 +\beta x_3, y_3=x_3, y_4=\alpha x_4+\beta x_5, y_5= x_5$. Let $\phi :K_8\to K$ given by  $x_i\mapsto y_i$, for $1\leq i\leq 5$. It is easy to see that $\phi$ is an isomorphism. Therefore, in this case we can suppose that $\res x_4=0$.  
 If $\alpha=0$ then $\beta\neq 0$ and we rescale $x_5$ so that $\res x_5=0$.
Hence we can assume either $\res x_4=0$ or $\res x_5=0$. Consider the automorphism of 
$K_8$ given by $x_1\mapsto x_1, x_2\mapsto x_3, x_3\mapsto x_2, x_4\mapsto x_5$ and $x_5\mapsto x_4$.
Using this automorphism, we deduce that it is enough to determine all the $p$-maps on $K_8$ for which $\res x_5=0$. 
\begin{lemma}\label{Lemma58}
Let $K=L_{5,8}$ and $[p]:K\to K$ be a $p$-map on $K$ such that $\res x_5=0$ and let $L=\frac{K}{M}$ where $M=\la x_5\ra_{\F}$. Then $K\cong L_{\theta}$ where $\theta=(\Delta_{13},\omega)\in Z^2(L,\F)$.
\end{lemma}
\begin{proof}
 Let $\pi :K\rightarrow L$ be the projection map. We have the exact sequence 
$$ 0\rightarrow M\rightarrow K\rightarrow L\rightarrow 0.$$
Let $\sigma :L\rightarrow K$ such that $x_i\mapsto x_i$, $1\leq i\leq 4$. Then $\sigma$ is an injective linear map and $\pi \sigma =1_L$. Now, we define $\phi : L\times L\rightarrow M$ by $\phi (x_i,x_j)=[\sigma (x_i),\sigma (x_j)]-\sigma ([x_i,x_j])$, $1\leq i,j\leq 4$ and $\omega: L\rightarrow M$ by $\omega (x)=\res {\sigma (x)} -\sigma (\res x)$. Note that
\begin{align*}
&\phi(x_1,x_3)=[\sigma(x_1),\sigma(x_3)]-\sigma([x_1,x_3])=[x_1,x_3]=x_5;\\
&\phi(x_1,x_2)=[\sigma(x_1),\sigma(x_2)]-\sigma([x_1,x_2])=0.
\end{align*}
Similarly, we can show that $\phi(x_1,x_4)=\phi(x_2,x_3)=\phi(x_2,x_4)=\phi(x_3,x_4)=0$. Therefore, $\phi=\Delta_{13}$.
Now, by Lemma, \ref{K=L-theta} we have $\theta=(\Delta_{13},\omega)\in Z^2(L,\F)$ and $K\cong L_{\theta}$.
\end{proof}

 Let 
$$
L=\frac {L_{5,8}}{\la x_5\ra} = \la x_1,x_2,x_3,x_4 \mid [x_1,x_2] =x_4 \ra. 
$$ 
Then $L\cong L_{4,2}$ and  the group $\Aut(L)$ consists of invertible matrices of the form
$$\begin{pmatrix}
 a_{11}  & a_{12} & 0 & 0 \\
a_{21} & a_{22} & 0 & 0\\
a_{31} & a_{32}& a_{33} & 0 \\
a_{41} & a_{42} & a_{43}& r
\end{pmatrix},$$
where $r= a_{11}a_{22}-a_{12}a_{21}\neq 0$.   
By Theorem \ref{4-dim}, there are eight non-isomorphic restricted Lie algebra structures on $L$ given by the following  $p$-maps:
\begin{enumerate}
\item[I.1]Trivial $p$-map;
\item[I.2]$\res x_1=x_4$;
\item[I.3]$\res x_1 =x_3$;
\item[I.4]$ \res x_1=x_4, \res x_2=x_3$;
\item[I.5]$\res x_4=x_3$;
\item[I.6]$\res x_4=x_3, \res x_2=x_4$;
\item[I.7]$\res x_3=x_4$;
\item[I.8]$\res x_3=x_4, \res x_2=x_3$.
\end{enumerate}

 We make $L$ into a restricted Lie algebra by equipping it with each of the above $p$-maps. Then,  in  each  case,  we  find all possible orbit representatives of the form $(\Delta_{13},\omega)$ under the action of $\Aut_p(L)$ on $H^2(L,\F)$. By  Lemma \ref{Lemma58}, we do get all possible $p$-maps on $K_8$.

Consider the case I.3 where the $p$-map of $L$ is given by $\res x_1 = x_3$.
Let  $(\phi,\omega)\in Z^2(L,\F)$. Then we must have $\phi (x,\res y)=0$, for all $x,y\in L$, where $\phi=a\Delta _{12}+b\Delta_{13}+c\Delta_{14}+d\Delta_{23}+e\Delta_{24}+f\Delta_{34}$, for some $a,b,c,d,e,f \in \F$. Since $\res L=\la x_3\ra$, we get $\phi (x_1,x_3) =0$. So, $b=0$. Since $\phi =\Delta_{13}$ gives us $L_{5,8}$, we deduce by Lemma \ref{Lemma58} that $L_{5,8}$ cannot be constructed in this case. Similarly, we can show that in cases I.4, I.5, I.6, and I.8  we also get $b=0$ and so we cannot put a $p$-map on $L_{5,8}$. We  consider 
the remaining cases in  the following sections.

\section{Extensions of ($L$, trivial $p$-map)}
 First, we find a basis for $Z^2(L,\F)$. Let $(\phi,\omega)=(a\Delta _{12}+b\Delta_{13}+c\Delta_{14}+d\Delta_{23}+e\Delta_{24}+f\Delta_{34},\alpha f_1+\beta f_2+ \gamma f_3+\delta f_4)\in Z^2(L,\F)$. Then we must have $\delta^2\phi(x,y,z) =0$ and $\phi(x,\res y)=0$, for all $x,y,z \in L$. Therefore,
\begin{align*}
0=(\delta^2\phi)(x_1,x_2,x_3)&=\phi([x_1,x_2],x_3)+\phi([x_2,x_3],x_1)+\phi([x_3,x_1],x_2)=\phi(x_4,x_3).
\end{align*}
Thus, we get $f=0$.
Since the $p$-map is trivial, $\phi(x,\res y)=\phi(x,0)=0$, for all $x,y\in L$. Therefore, a basis for $Z^2(L,\F)$ is as follows:
$$
 (\Delta_{12},0),(\Delta_{13},0),(\Delta_{14},0),(\Delta_{23},0),(\Delta_{24},0),(0,f_1),
(0,f_2),(0,f_3),(0,f_4).
$$
Next, we find a basis for $B^2(L,\F)$. Let $(\phi,\omega)\in B^2(L,\F)$. Since $B^2(L,\F)\subseteq Z^2(L,\F)$, we have $(\phi,\omega)=(a\Delta _{12}+b\Delta_{13}+c\Delta_{14}+d\Delta_{23}+e\Delta_{24},\alpha f_1+\beta f_2+ \gamma f_3+\delta f_4)$. So, there exists a linear map $\psi:L\to \F$ such that $\delta^1\psi(x,y)=\phi(x,y)$ and $\tilde \psi(x)=\omega(x)$, for all $x,y \in L$. So, we have
\begin{align*}
b=\phi(x_1,x_3)=\delta^1\psi(x_1,x_3)=\psi([x_1,x_3])=0.
\end{align*}
Similarly, we can show that $c=d=e=0$. Also, we have
\begin{align*}
\alpha=\omega(x_1)=\tilde \psi(x_1)=\psi(\res x_1)=0.
\end{align*}
Similarly, we can show that $\beta=\gamma=\delta=0$. Therefore, $(\phi,\omega)=(a\Delta_{12},0)$ and hence
 $B^2(L,\F)=\la(\Delta_{12},0)\ra_{\F}$. We deduce that a basis for $H^2(L,\F)$ is as follows:
$$
[(\Delta_{13},0)],[(\Delta_{14},0)],[(\Delta_{23},0)],[(\Delta_{24},0)],[(0,f_1)],
[(0,f_2)],[(0,f_3)],[(0,f_4)].
$$
Let $[(\phi,\omega)] \in H^2(L,\F)$. Then we have $\phi = a\Delta _{13}+b\Delta_{14}+c\Delta_{23}+d\Delta_{24}$, for some $a,b,c,d \in \F$. Suppose that $A\phi = a'\Delta _{13}+b'\Delta_{14}+c'\Delta_{23}+d'\Delta_{24},$ for some $a',b',c',d' \in \F$. We determine $a',b',c',d'$. Note that
\begin{align*}
&A\phi (x_1,x_3)
=a_{11}a_{33}a+a_{11}a_{43}b+a_{21}a_{33}c+a_{21}a_{43}d;\\
&A\phi (x_1,x_4)
=a_{11}rb+a_{21}rd;\\
&A\phi (x_2,x_3)
=a_{12}a_{33}a+a_{12}a_{43}b+a_{22}a_{33}c+a_{22}a_{43}d;\\
&A\phi (x_2,x_4)
=a_{12}rb+a_{22}rd.
\end{align*}
In the matrix form we can write this as
\begin{align}\label{78}
\begin{pmatrix}
a' \\ b'\\c'\\d'
\end{pmatrix}=
\begin{pmatrix}
a_{11}a_{33} & a_{11}a_{43}& a_{21}a_{33}&a_{21}a_{43}\\
0& a_{11}r&0&a_{21}r\\
a_{12}a_{33}&a_{12}a_{43}& a_{22}a_{33}&a_{22}a_{43}\\
0&a_{12}r&0&a_{22}r
\end{pmatrix}
\begin{pmatrix}
a \\ b\\c\\d
\end{pmatrix}.
\end{align}
The orbit with representative 
$\begin{pmatrix}
1\\ 0\\0\\0
\end{pmatrix}$ of this action  gives us $L_{5,8}$.

Also, we have $\omega=\alpha f_1+\beta f_2+\gamma f_3+\delta f_4$, for some $\alpha , \beta , \gamma, \delta \in \F$. Suppose that $A\omega = \alpha' f_1+\beta' f_2+\gamma' f_3+\delta' f_4$,  for some $\alpha' , \beta' , \gamma', \delta' \in \F$. We determine $\alpha' , \beta' , \gamma', \delta'$. Note that
\begin{align*}
&A\omega (x_1) = \omega (Ax_1)=\omega (a_{11}x_1+a_{21}x_2+a_{31}x_3+a_{41}x_4)=a_{11}^p \alpha +a_{21}^p\beta + a_{31}^p \gamma+a_{41}^p \delta;\\
&A\omega (x_2)=a_{12}^p\alpha+a_{22}^p\beta +a_{32}^p\gamma+ a_{42}^p \delta;\\
&A\omega (x_3)= a_{33}^p \gamma +a_{43}^p\delta;\\
&A\omega(x_4)=r^p \delta.
\end{align*}
In the matrix form we can write this as
\begin{align}\label{89}
\begin{pmatrix}
\alpha' \\ 
\beta'\\
\gamma'\\
\delta'
\end{pmatrix}=
\begin{pmatrix}
a_{11}^p& a_{21}^p&a_{31}^p&a_{41}^p\\
a_{12}^p& a_{22}^p& a_{32}^p&a_{42}^p\\
0&0&a_{33}^p&a_{43}^p\\
0&0&0&r^p
\end{pmatrix}
\begin{pmatrix}
\alpha\\ 
\beta \\
\gamma \\
\delta
\end{pmatrix}.
\end{align}
Thus, we can write Equations \eqref{78}  and \eqref{89} together as follows:
\begin{align*}
\small{
\bigg[\begin{pmatrix}
a_{11}a_{33} & a_{11}a_{43}& a_{21}a_{33}&a_{21}a_{43}\\
0& a_{11}r&0&a_{21}r\\
a_{12}a_{33}&a_{12}a_{43}& a_{22}a_{33}&a_{22}a_{43}\\
0&a_{12}r&0&a_{22}r
\end{pmatrix},
\begin{pmatrix}
a_{11}^p& a_{21}^p&a_{31}^p&a_{41}^p\\
a_{12}^p& a_{22}^p& a_{32}^p&a_{42}^p\\
0&0&a_{33}^p&a_{43}^p\\
0&0&0&r^p
\end{pmatrix}\bigg]
\bigg[\begin{pmatrix}
a \\ b\\c\\d
\end{pmatrix},
\begin{pmatrix}
\alpha \\ 
\beta\\
\gamma\\
\delta
\end{pmatrix}\bigg]=
\bigg[\begin{pmatrix}
a' \\ b'\\c'\\d'
\end{pmatrix},
\begin{pmatrix}
\alpha' \\ 
\beta'\\
\gamma'\\
\delta'
\end{pmatrix}\bigg].}
\end{align*}
Now we find the representatives of the orbits of the action of $\Aut (L)$ on the set of $\omega$'s such that
 the orbit represented by
$\begin{pmatrix}
1\\ 0\\0\\0
\end{pmatrix}$
 is preserved under the action of $\Aut (L)$ on the set of $\phi$'s.

Let $\nu = \begin{pmatrix}
\alpha\\ 
\beta \\
\gamma \\
\delta
\end{pmatrix} \in \F^4$. If $\nu= \begin{pmatrix}
0\\ 
0 \\
0 \\
0
\end{pmatrix}$, then $\{\nu \}$ is clearly an $\Aut(L)$-orbit. Let $\nu \neq 0$. Suppose that $\delta \neq 0$. Then 
\begin{align*}
&\bigg[\begin{pmatrix}
1& -\gamma/\delta&0&0\\
0& 1&0&0\\
0&0&1&-\gamma/\delta\\
0&0&0&1
\end{pmatrix},
\begin{pmatrix}
1& 0&0&-\alpha/\delta\\
0& 1& 0&-\beta/\delta\\
0&0&1&-\gamma/\delta\\
0&0&0&1
\end{pmatrix}\bigg]
\bigg[\begin{pmatrix}
1 \\ 0\\0\\0
\end{pmatrix},
\begin{pmatrix}
\alpha\\ 
\beta \\
\gamma \\
\delta
\end{pmatrix}\bigg]=
\bigg[\begin{pmatrix}
1\\0\\0\\0
\end{pmatrix},
\begin{pmatrix}
0 \\ 
0\\
0\\
\delta
\end{pmatrix}\bigg], \text{ and }\\
&\bigg[\begin{pmatrix}
1& 0&0&0\\
0& \delta ^{-2/p}&0&0\\
0&0&\delta ^{1/p}&0\\
0&0&0&\delta^{-1/p}
\end{pmatrix},
\begin{pmatrix}
1/\delta& 0&0&0\\
0& 1& 0&0\\
0&0&\delta&0\\
0&0&0&1/\delta
\end{pmatrix}\bigg]
\bigg[\begin{pmatrix}
1\\ 0\\0\\0
\end{pmatrix},
\begin{pmatrix}
0\\ 
0 \\
0 \\
\delta
\end{pmatrix}\bigg]=
\bigg[\begin{pmatrix}
1\\0\\0\\0
\end{pmatrix},
\begin{pmatrix}
0 \\ 
0\\
0\\
1
\end{pmatrix}\bigg].
\end{align*}
Next, if $\delta =0$, but $\gamma \neq 0$, then
\begin{align*}
&\bigg[\begin{pmatrix}
1& 0&0&0\\
0& 1&0&0\\
0&0&1&0\\
0&0&0&1
\end{pmatrix},
\begin{pmatrix}
1& 0&-\alpha/\gamma&0\\
0& 1& -\beta/\gamma&0\\
0&0&1&0\\
0&0&0&1
\end{pmatrix}\bigg]
\bigg[\begin{pmatrix}
1\\ 0\\0\\0
\end{pmatrix},
\begin{pmatrix}
\alpha\\ 
\beta\\
\gamma \\
0
\end{pmatrix}\bigg]=
\bigg[\begin{pmatrix}
1\\0\\0\\0
\end{pmatrix},
\begin{pmatrix}
0 \\ 
0\\
\gamma\\
0
\end{pmatrix}\bigg], \text{ and }\\
&\bigg[\begin{pmatrix}
1& 0&0&0\\
0& \gamma^{2/p}&0&0\\
0&0&\gamma^{-1/p}&0\\
0&0&0&\gamma ^{1/p}
\end{pmatrix},
\begin{pmatrix}
\gamma& 0&0&0\\
0& 1& 0&0\\
0&0&1/\gamma&0\\
0&0&0&\gamma
\end{pmatrix}\bigg]
\bigg[\begin{pmatrix}
1\\ 0\\0\\0
\end{pmatrix},
\begin{pmatrix}
0\\ 
0\\
\gamma \\
0
\end{pmatrix}\bigg]=
\bigg[\begin{pmatrix}
1\\0\\0\\0
\end{pmatrix},
\begin{pmatrix}
0 \\ 
0\\
1\\
0
\end{pmatrix}\bigg].
\end{align*}
Next, if $\delta =\gamma=0$, but $\beta \neq 0$, then
\begin{align*}
&\small
\bigg[\begin{pmatrix}
1& 0&(-\alpha /\beta )^{1/p}&0\\
0& 1&0&(-\alpha /\beta )^{1/p}\\
0&0&1&0\\
0&0&0&1
\end{pmatrix},
\begin{pmatrix}
1& -\alpha/\beta&0&0\\
0& 1& 0&0\\
0&0&1&0\\
0&0&0&1
\end{pmatrix}\bigg]
\bigg[\begin{pmatrix}
1\\ 0\\0\\0
\end{pmatrix},
\begin{pmatrix}
\alpha\\ 
\beta\\
0\\
0
\end{pmatrix}\bigg]=
\bigg[\begin{pmatrix}
1\\0\\0\\0
\end{pmatrix},
\begin{pmatrix}
0 \\ 
\beta\\
0\\
0
\end{pmatrix}\bigg], \text{ and }\\
&\small
\bigg[\begin{pmatrix}
1& 0&0&0\\
0& (1/\beta )^{1/p}&0&0\\
0&0&(1/\beta )^{1/p}&0\\
0&0&0&(1/\beta )^{2/p}
\end{pmatrix},
\begin{pmatrix}
1& 0&0&0\\
0& 1/\beta& 0&0\\
0&0&1&0\\
0&0&0&1/\beta
\end{pmatrix}\bigg]
\bigg[\begin{pmatrix}
1\\ 0\\0\\0
\end{pmatrix},
\begin{pmatrix}
0\\ 
\beta\\
0\\
0
\end{pmatrix}\bigg]=
\bigg[\begin{pmatrix}
1\\0\\0\\0
\end{pmatrix},
\begin{pmatrix}
0 \\ 
1\\
0\\
0
\end{pmatrix}\bigg].
\end{align*}
Finally, if $\delta =\gamma =\beta =0$, but $\alpha \neq 0$, then
\begin{align*}
\footnotesize\bigg[\begin{pmatrix}
1& 0&0&0\\
0& \alpha^{-2/p}&0&0\\
0&0&\alpha^{1/p}&0\\
0&0&0&\alpha^{-1/p}
\end{pmatrix},
\begin{pmatrix}
1/\alpha& 0&0&0\\
0& 1& 0&0\\
0&0&\alpha&0\\
0&0&0&1/\alpha
\end{pmatrix}\bigg]
\bigg[\begin{pmatrix}
1\\ 0\\0\\0
\end{pmatrix},
\begin{pmatrix}
\alpha\\ 
0\\
0\\
0
\end{pmatrix}\bigg]=
\bigg[\begin{pmatrix}
1\\0\\0\\0
\end{pmatrix},
\begin{pmatrix}
1\\ 
0\\
0\\
0
\end{pmatrix}\bigg].
\end{align*}
Thus the following elements are $\Aut(L)$-orbit representatives: 
\begin{align*}
\begin{pmatrix}
0\\ 
0\\
0\\
0
\end{pmatrix},
\begin{pmatrix}
1 \\ 
0\\
0\\
0
\end{pmatrix},
\begin{pmatrix}
0 \\ 
1\\
0\\
0
\end{pmatrix},
\begin{pmatrix}
0 \\ 
0\\
1\\
0
\end{pmatrix},
\begin{pmatrix}
0\\ 
0\\
0\\
1
\end{pmatrix}.
\end{align*}
Therefore, the corresponding restricted Lie algebra structures are as follows:
\begin{align*}
&K_8^{1}=\langle x_1,\ldots,x_5\mid [x_1,x_2]=x_4,[x_1,x_3]=x_5 \rangle;\\
&K_8^{2}=\langle x_1,\ldots,x_5\mid [x_1,x_2]=x_4,[x_1,x_3]=x_5, 
 \res x_1=x_5 \rangle;\\
&K_8^{3}=\langle x_1,\ldots,x_5\mid [x_1,x_2]=x_4,[x_1,x_3]=x_5, 
 \res x_2=x_5 \rangle;\\
&K_8^{4}=\langle x_1,\ldots,x_5\mid [x_1,x_2]=x_4,[x_1,x_3]=x_5,
\res x_3=x_5 \rangle;\\
&K_8^{5}=\langle x_1,\ldots,x_5\mid [x_1,x_2]=x_4,[x_1,x_3]=x_5, 
\res x_4=x_5 \rangle.
\end{align*}
\section{Extensions of ($L,  \res x_1 = x_4$)}
 First, we find a basis for $Z^2(L,\F)$. Let $(\phi,\omega)=(a\Delta _{12}+b\Delta_{13}+c\Delta_{14}+d\Delta_{23}+e\Delta_{24}+f\Delta_{34},\alpha f_1+\beta f_2+ \gamma f_3+\delta f_4)\in Z^2(L,\F)$. Then we must have $\delta^2\phi(x,y,z) =0$ and $\phi(x,\res y)=0$, for all $x,y,z \in L$. Therefore, we have
\begin{align*}
0=(\delta^2\phi)(x_1,x_2,x_3)&=\phi([x_1,x_2],x_3)+\phi([x_2,x_3],x_1)+\phi([x_3,x_1],x_2)=\phi(x_4,x_3).
\end{align*}
Thus, we get $f=0$. Also, we have $\phi(x,\res y)=0$. Therefore, $\phi(x,x_4)=0$, for all $x\in L$ and hence $\phi(x_1,x_4)=\phi(x_2,x_4)=\phi(x_3,x_4)=0$ which implies that $c=e=f=0$. Therefore, $Z^2(L,\F)$ has a basis consisting of:
$$ (\Delta_{12},0),(\Delta_{13},0),(\Delta_{23},0),(0,f_1),
(0,f_2),(0,f_3),(0,f_4).$$
Next, we find a basis for $B^2(L,\F)$. Let $(\phi,\omega)\in B^2(L,\F)$. Since $B^2(L,\F)\subseteq Z^2(L,\F)$, we have $(\phi,\omega)=(a\Delta _{12}+b\Delta_{13}+c\Delta_{23},\alpha f_1+\beta f_2+ \gamma f_3+\delta f_4)$. Note that there exists a linear map $\psi:L\to \F$  such that $\delta^1\psi(x,y)=\phi(x,y)$ and $\tilde \psi(x)=\omega(x)$, for all $x,y \in L$. So, we have
\begin{align*}
&a=\phi(x_1,x_2)=\delta^1\psi(x_1,x_2)=\psi([x_1,x_2])=\psi(x_4), \text{ and }\\
&b=\phi(x_1,x_3)=\delta^1\psi(x_1,x_3)=\psi([x_1,x_3])=0.
\end{align*}
Similarly, we can show that $c=0$. Also, we have
\begin{align*}
&\alpha=\omega(x_1)=\tilde \psi(x_1)=\psi(\res x_1)=\psi(x_4), \text{ and }\\
&\beta=\omega(x_2)=\tilde \psi(x_2)=\psi(\res x_2)=0.
\end{align*}
Similarly, we can show that $\gamma=\delta=0$.  Note that $\psi(x_4)=a=\alpha$. Therefore, $(\phi,\omega)=(a\Delta_{12},a f_1)$ and hence $B^2(L,\F)=\la(\Delta_{12},f_1)\ra_{\F}$. Note that 
$$[(\Delta_{12},0)],[(\Delta_{13},0)],[(\Delta_{23},0)],[0,f_1],[(0,f_2)],[(0,f_3)],[(0,f_4)]$$
spans $H^2(L,\F)$.
Since $[(\Delta_{12},0)]+[(0,f_1)]=[(\Delta_{12},f_1)]=[0]$, then $[(0,f_1)]$ is an scalar multiple of $[(\Delta_{12},0)]$
in $H^2(L,\F)$. Note that $\dim H^2=\dim Z^2-\dim B^2=6$. Therefore, 
$$[(\Delta_{12},0)],[(\Delta_{13},0)],[(\Delta_{23},0)],[(0,f_2)],[(0,f_3)],[(0,f_4)].$$
forms a basis for $H^2(L,\F)$.

Note that the group $\Aut(L)$ in this case consists of invertible matrices of the form
$$\begin{pmatrix}
 a_{11}  & a_{12} & 0 & 0 \\
a_{21} & a_{22} & 0 & 0\\
a_{31} & a_{32}& a_{33} & 0 \\
a_{41} & a_{42} & a_{43}& r
\end{pmatrix},$$
where $r= a_{11}a_{22}-a_{12}a_{21}=a_{11}^p$ and $a_{12}=0$.

Let $[(\phi,\omega)] \in H^2(L,\F)$. Then we have $\phi = a\Delta_{13}+b\Delta_{12}+c\Delta_{23}$, for some $a,b,c \in \F$. Suppose that $A\phi=a'\Delta_{13}+b'\Delta_{12}+c'\Delta_{23}$, for some $a',b',c' \in \F$. We determine $a',b',c'$. Note that
\begin{align*}
&A\phi (x_1,x_3)= \phi (a_{11}x_1+a_{21}x_2+a_{31}x_3+a_{41}x_4 , a_{33}x_3+a_{43}x_4)=a_{11}a_{33}a+a_{21}a_{33}c;\\
&A\phi (x_1,x_2)= \phi (a_{11}x_1+a_{21}x_2+a_{31}x_3+a_{41}x_4 , a_{12}x_1+a_{22}x_2+a_{32}x_3+a_{42}x_4)\\
&=(a_{11}a_{32}-a_{31}a_{12})a+(a_{11}a_{22}-a_{12}a_{21})b+(a_{21}a_{32}-a_{31}a_{22})c;\\
&A\phi (x_2,x_3)= \phi (a_{12}x_1+a_{22}x_2+a_{32}x_3+a_{42}x_4 , a_{33}x_3+a_{43}x_4)=a_{12}a_{33}a+a_{22}a_{33}c.
\end{align*}
In the matrix form we can write this as



\begin{align}\label{i}
\begin{pmatrix}
a'\\b' \\c'
\end{pmatrix}=
\begin{pmatrix}
a_{11}a_{33} &0&  a_{21}a_{33}\\
a_{11}a_{32}&a_{11}a_{22}&a_{21}a_{32}-a_{31}a_{22}\\
0&0& a_{22}a_{33}\\
\end{pmatrix}
\begin{pmatrix}
a \\b\\c
\end{pmatrix}.
\end{align}
 The orbit with representative 
$\begin{pmatrix}
1\\ 0\\0
\end{pmatrix}$ of this action  gives us $L_{5,8}$.

Also, we have $\omega=\beta f_2+\gamma f_3+\delta f_4$,  for some $\beta,\gamma,\delta \in \F$. Suppose that $A\omega =\beta' f_2+\gamma' f_3+\delta' f_4$,  for some $\beta',\gamma',\delta' \in \F$. We have
\begin{align*}
&A\omega (x_2)=a_{22}^p\beta +a_{32}^p\gamma+ a_{42}^p \delta;\\
&A\omega (x_3)= a_{33}^p \gamma +a_{43}^p\delta;\\
&A\omega(x_4)=r^p \delta.
\end{align*}
In the matrix form we can write this as
\begin{align}\label{j}
\begin{pmatrix} 
\beta'\\
\gamma'\\
\delta'
\end{pmatrix}=
\begin{pmatrix}
 a_{22}^p& a_{32}^p&a_{42}^p\\
0&a_{33}^p&a_{43}^p\\
0&0&a_{11}^{p^2}
\end{pmatrix}
\begin{pmatrix}
\beta \\
\gamma \\
\delta
\end{pmatrix}.
\end{align}
Thus, we can write Equations \eqref{i} and \eqref{j} together as follows:
\begin{align*}
\small{
\bigg[\begin{pmatrix}
a_{11}a_{33} &0&  a_{21}a_{33}\\
a_{11}a_{32}&a_{11}a_{22}&a_{21}a_{32}{-}a_{31}a_{22}\\
0&0& a_{22}a_{33}\\
\end{pmatrix}
{,}
\begin{pmatrix}
 a_{22}^p& a_{32}^p&a_{42}^p\\
0&a_{33}^p&a_{43}^p\\
0&0&a_{11}^{p^2}
\end{pmatrix}\bigg]
\bigg[\begin{pmatrix}
a\\b \\c
\end{pmatrix},
\begin{pmatrix}
\beta\\
\gamma\\
\delta
\end{pmatrix}\bigg]
{=}
\bigg[\begin{pmatrix}
a' \\b'\\ c'
\end{pmatrix},
\begin{pmatrix}
\beta'\\
\gamma'\\
\delta'
\end{pmatrix}\bigg]}.
\end{align*}
Note that $A\phi(x_1,x_4)=A\phi(x_2,x_4)=A\phi(x_3,x_4)=0$ and $A\omega(x_1)=0$ which imply that $a_{21}^p\beta+a_{31}^p\gamma+a_{41}^p\delta=0$.
Now we find the representatives of the orbits of the action of $\Aut (L)$ on the set of $\omega$'s such that the orbit represented by
$\begin{pmatrix}
1\\ 0\\0
\end{pmatrix}$
 is preserved under the action of $\Aut (L)$ on the set of $\phi$'s. Note that we need to have $a_{21}^p\beta+a_{31}^p\gamma+a_{41}^p\delta=0$.

Let $\nu = \begin{pmatrix}
\beta \\
\gamma \\
\delta
\end{pmatrix} \in \F^3$. If $\nu= \begin{pmatrix}
0 \\
0 \\
0
\end{pmatrix}$, then $\{\nu \}$ is clearly an $\Aut(L)$-orbit. Let $\nu \neq 0$. Suppose that $\delta \neq 0$. Then 
\begin{align*}
&\bigg[\begin{pmatrix}
1&0&0\\
0&1&0\\
0&0&1
\end{pmatrix},
\begin{pmatrix}
 1& 0&-\beta/\delta\\
0&1&-\gamma/\delta\\
0&0&1
\end{pmatrix}\bigg]
\bigg[\begin{pmatrix}
1 \\ 0\\0
\end{pmatrix},
\begin{pmatrix} 
\beta \\
\gamma \\
\delta
\end{pmatrix}\bigg]=
\bigg[\begin{pmatrix}
1\\0\\0
\end{pmatrix},
\begin{pmatrix}
0\\
0\\
\delta
\end{pmatrix}\bigg], \text{ and }\\
&\bigg[\begin{pmatrix}
1&0&0\\
0&(1/\delta )^{1/p}&0\\
0&0&\delta^\frac{p-2}{p^2}
\end{pmatrix},
\begin{pmatrix}
 (1/\delta)^\frac{p-1}{p}& 0&0\\
0&\delta^{1/p}&0\\
0&0&1/\delta
\end{pmatrix}\bigg]
\bigg[\begin{pmatrix}
1\\ 0\\0
\end{pmatrix},
\begin{pmatrix}
0 \\
0 \\
\delta
\end{pmatrix}\bigg]=
\bigg[\begin{pmatrix}
1\\0\\0
\end{pmatrix},
\begin{pmatrix}
0\\
0\\
1
\end{pmatrix}\bigg].
\end{align*}
Next, if $\delta =0$, but $\gamma \neq 0$, then
\begin{align*}
\bigg[\begin{pmatrix}
1& 0&0\\
0&\gamma&0\\
0&0&\gamma^\frac{p-2}{p}
\end{pmatrix},
\begin{pmatrix}
 \gamma^{p-1}&0&0\\
0&1/\gamma&0\\
0&0&\gamma^p
\end{pmatrix}\bigg]
\bigg[\begin{pmatrix}
1\\ 0\\0
\end{pmatrix},
\begin{pmatrix} 
\beta\\
\gamma \\
0
\end{pmatrix}\bigg]=
\bigg[\begin{pmatrix}
1\\0\\0
\end{pmatrix},
\begin{pmatrix}
\gamma^{p-1}\beta\\
1\\
0
\end{pmatrix}\bigg].
\end{align*}

If $\beta =0$, then we have 
$\begin{pmatrix}
0\\
1\\
0
\end{pmatrix}$. If $\beta \neq 0$, then we rename $\gamma^{p-1}\beta$ with $\beta$ and we have 
$
\begin{pmatrix} 
\beta\\
1 \\
0
\end{pmatrix}$.
Finally, if $\delta =\gamma=0$, but $\beta \neq 0$, then we have 
$
\begin{pmatrix}
\beta\\
0\\
0
\end{pmatrix}$.
Thus the following elements are $\Aut (L)$-orbit representatives:
\begin{align*}
\begin{pmatrix}
0\\
0\\
0
\end{pmatrix},
\begin{pmatrix}
\beta\\ 
0\\
0
\end{pmatrix},
\begin{pmatrix}
0 \\ 
1\\
0
\end{pmatrix},
\begin{pmatrix}
0 \\ 
0\\
1
\end{pmatrix},
\begin{pmatrix}
\beta\\ 
1\\
0
\end{pmatrix}.
\end{align*}
\begin{lemma}\label{lemma-K38-2}
The vectors 
$\begin{pmatrix}
\beta_1\\
1\\
0
\end{pmatrix}$ and
$\begin{pmatrix}
\beta_2\\
1\\
0
\end{pmatrix}$
are in the same $\Aut (L)$-orbit if and only if $\beta_1=\beta_2$. 
\end{lemma}
\begin{proof}
First assume that 
$\begin{pmatrix}
\beta_1\\
1\\
0
\end{pmatrix}$ and
$\begin{pmatrix}
\beta_2\\
1\\
0
\end{pmatrix}$
are in the same $\Aut (L)$-orbit. Then
\begin{align*}
\small{
\bigg[\begin{pmatrix}
a_{11}a_{33} &0&  a_{21}a_{33}\\
a_{11}a_{32}&a_{11}a_{22}&a_{21}a_{32}{-}a_{31}a_{22}\\
0&0& a_{22}a_{33}\\
\end{pmatrix}
{,}
\begin{pmatrix}
 a_{22}^p& a_{32}^p&a_{42}^p\\
0&a_{33}^p&a_{43}^p\\
0&0&a_{11}^{p^2}
\end{pmatrix}\bigg]
\bigg[\begin{pmatrix}
1\\0 \\0
\end{pmatrix},
\begin{pmatrix}
\beta_1\\
1\\
0
\end{pmatrix}\bigg]
{=}
\bigg[\begin{pmatrix}1\\0\\0
\end{pmatrix},
\begin{pmatrix}
\beta_2\\
1\\
0
\end{pmatrix}\bigg]}.
\end{align*}
From this we obtain that 
\begin{align}
& a_{11}a_{33}=1\label{58-d1}\\
&a_{11}a_{32}=0 \label{58-d2} \\
 &a_{22}^p\beta_1+a_{32}^p=\beta_2 \label{58-d3}\\
 &a_{33}^p=1. \label{58-d4}
\end{align} 
Then using Equation \eqref{58-d2} and \eqref{58-d3}, we get that $\beta_2\beta_1^{-1}=a_{22}^p$. Note that we have  $r=a_{11}^p$ which get that $a_{22}=a_{11}^{p-1}$ and hence $\beta_2\beta_1^{-1}=a_{22}^p=(a_{11}^p)^{p-1}$. Now using Equation \eqref{58-d1} and \eqref{58-d4}, we obtain that $a_{11}^p=1$ which implies that $\beta_2\beta_1^{-1}=1$. Therefore, $\beta_1=\beta_2$.
The converse is clear.
\end{proof}
Therefore, the corresponding restricted Lie algebra structures are as follows:
\begin{align*}
&K_8^{6}=\langle x_1,\ldots,x_5\mid [x_1,x_2]=x_4,[x_1,x_3]=x_5 ,
\res x_1=x_4\rangle;\\
&K_8^{7}(\beta)=\langle x_1,\ldots,x_5\mid [x_1,x_2]=x_4,[x_1,x_3]=x_5, 
\res x_1=x_4 , \res x_2=\beta x_5 \rangle;\\
&K_8^{8}=\langle x_1,\ldots,x_5\mid [x_1,x_2]=x_4,[x_1,x_3]=x_5, 
\res x_1=x_4, \res x_3=x_5 \rangle;\\
&K_8^{9}=\langle x_1,\ldots,x_5\mid [x_1,x_2]=x_4,[x_1,x_3]=x_5,
\res x_1=x_4, \res x_4=x_5 \rangle;\\
&K_8^{10}(\beta)=\langle x_1,\ldots,x_5\mid [x_1,x_2]=x_4,[x_1,x_3]=x_5,
\res x_1=x_4, \res x_2=\beta x_5, \res x_3=x_5 \rangle
\end{align*}
where $\beta \in \F^*$.

\begin{lemma}
We have 
$
K_8^{7}(\beta)\cong \langle x_1,\ldots,x_5\mid [x_1,x_2]=x_4,[x_1,x_3]=x_5, 
\res x_1=x_4 , \res x_2=x_5 \rangle
$, for every $\beta\in \F^*$.
\end{lemma}
\begin{proof}
Note that the following automorphism of $L_{5,8}$ gives us the required result:
$$\begin{pmatrix}
1&0&0&0&0\\
0& 1&0&0&0\\
0&0&1/\beta&0&0\\
0&0&0&1&0\\
0&0&0&0&1/\beta\\
\end{pmatrix}.$$
\end{proof}

\begin{lemma}
We have $K_8^{10}(\beta_1)\cong K_8^{10}(\beta_2)$ if and only if $\beta_1/\beta_2= \epsilon^{p-2}$, for some $\epsilon\in \F^*$.  
\end{lemma}
\begin{proof}
Let $f=(a_{ij})\in \Aut(K_8^{10})$. Then $f$ is an isomorphism from  $K_8^{10}(\beta_1)$
to  $K_8^{10}(\beta_2)$ if and only if 
\begin{align*}
a_{22}=a_{11}^{p-1},\quad a_{11}=a_{33}^{p-1}, \quad \beta_1/\beta_2=a_{22}^p a_{11}^{-1} a_{33}^{-1},
\end{align*}
and this in turn simplifies to having 
\begin{align*}
 \beta_1/\beta_2= a_{33}^{p^2(p-2)}.
\end{align*}
To prove the converse, suppose that $\beta_1/\beta_2= \epsilon^{p-2}$, for some $\epsilon\in \F^*$.  Then, it is easy to see that the following is an isomorphism from  $K_8^{10}(\beta_1)$
to  $K_8^{10}(\beta_2)$:

$$\begin{pmatrix}
\epsilon^{(p-1)/p^2}&0&0&0&0\\
0& \epsilon^{(p-1)^2/p^2}&0&0&0\\
0&0&\epsilon^{1/p^2}&0&0\\
0&0&0&\epsilon^{(p-1)^3/p^2}&0\\
0&0&0&0&\epsilon^{1/p}\\
\end{pmatrix}.$$

\end{proof}

\section{Extensions of ($L,  \res x_3 = x_4$)}
 First, we find a basis for $Z^2(L,\F)$. Let $(\phi,\omega)=(a\Delta _{12}+b\Delta_{13}+c\Delta_{14}+d\Delta_{23}+e\Delta_{24}+f\Delta_{34},\alpha f_1+\beta f_2+ \gamma f_3+\delta f_4)\in Z^2(L,\F)$. Then we must have $\delta^2\phi(x,y,z) =0$ and $\phi(x,\res y)=0$, for all $x,y,z \in L$. Therefore, we have
\begin{align*}
0=(\delta^2\phi)(x_1,x_2,x_3)&=\phi([x_1,x_2],x_3)+\phi([x_2,x_3],x_1)+\phi([x_3,x_1],x_2)=\phi(x_4,x_3).
\end{align*}
Thus, we get $f=0$. Also, we have $\phi(x,\res y)=0$. Therefore, $\phi(x,x_4)=0$, for all $x\in L$ and hence $\phi(x_1,x_4)=\phi(x_2,x_4)=\phi(x_3,x_4)=0$ which implies that $c=e=f=0$. Therefore, $Z^2(L,\F)$ has a basis consisting of:
$$ (\Delta_{12},0),(\Delta_{13},0),(\Delta_{23},0),(0,f_1),
(0,f_2),(0,f_3),(0,f_4).$$
Next, we find a basis for $B^2(L,\F)$. Let $(\phi,\omega)\in B^2(L,\F)$. Since $B^2(L,\F)\subseteq Z^2(L,\F)$, we have $(\phi,\omega)=(a\Delta _{12}+b\Delta_{13}+c\Delta_{23},\alpha f_1+\beta f_2+ \gamma f_3+\delta f_4)$. Note that there exists a linear map $\psi:L\to \F$  such that $\delta^1\psi(x,y)=\phi(x,y)$ and $\tilde \psi(x)=\omega(x)$, for all $x,y \in L$. So, we have
\begin{align*}
&a=\phi(x_1,x_2)=\delta^1\psi(x_1,x_2)=\psi([x_1,x_2])=\psi(x_4), \text{ and }\\
&b=\phi(x_1,x_3)=\delta^1\psi(x_1,x_3)=\psi([x_1,x_3])=0.
\end{align*}
Similarly, we can show that $c=0$. Also, we have
\begin{align*}
&\gamma=\omega(x_3)=\tilde \psi(x_3)=\psi(\res x_3)=\psi(x_4), \text{ and }\\
&\alpha=\omega(x_1)=\tilde \psi(x_1)=\psi(\res x_1)=0.
\end{align*}
Similarly, we can show that $\beta=\delta=0$.  Note that $\psi(x_4)=a=\gamma$. Therefore, $(\phi,\omega)=(a\Delta_{12},a f_3)$ and hence $B^2(L,\F)=\la(\Delta_{12},f_3)\ra_{\F}$. Note that 
$$[(\Delta_{12},0)],[(\Delta_{13},0)],[(\Delta_{23},0)],[0,f_1],[(0,f_2)],[(0,f_3)],[(0,f_4)]$$
spans $H^2(L,\F)$.
Since $[(\Delta_{12},0)]+[(0,f_3)]=[(\Delta_{12},f_3)]=[0]$, then $[(0,f_3)]$ is an scalar multiple of $[(\Delta_{12},0)]$
in $H^2(L,\F)$. Note that $\dim H^2=\dim Z^2-\dim B^2=6$. Therefore, 
$$[(\Delta_{12},0)],[(\Delta_{13},0)],[(\Delta_{23},0)],[(0,f_1)],[(0,f_2)],[(0,f_4)].$$
forms a basis for $H^2(L,\F)$.

Note that the group $\Aut(L)$ in this case consists of invertible matrices of the form
$$\begin{pmatrix}
 a_{11}  & a_{12} & 0 & 0 \\
a_{21} & a_{22} & 0 & 0\\
a_{31} & a_{32}& a_{33} & 0 \\
a_{41} & a_{42} & a_{43}& r
\end{pmatrix},$$
where $r= a_{11}a_{22}-a_{12}a_{21}=a_{33}^p$ and $a_{31}=a_{32}=0$.

Let $[(\phi,\omega)] \in H^2(L,\F)$. Then we have $\phi = a\Delta_{13}+b\Delta_{12}+c\Delta_{23}$, for some $a,b,c \in \F$. Suppose that $A\phi=a'\Delta_{13}+b'\Delta_{12}+c'\Delta_{23}$, for some $a',b',c' \in \F$. We determine $a',b',c'$. Note that
\begin{align*}
&A\phi (x_1,x_3)= \phi (a_{11}x_1+a_{21}x_2+a_{31}x_3+a_{41}x_4 , a_{33}x_3+a_{43}x_4)=a_{11}a_{33}a+a_{21}a_{33}c;\\
&A\phi (x_1,x_2)= \phi (a_{11}x_1+a_{21}x_2+a_{31}x_3+a_{41}x_4 , a_{12}x_1+a_{22}x_2+a_{32}x_3+a_{42}x_4)\\
&=(a_{11}a_{32}-a_{31}a_{12})a+(a_{11}a_{22}-a_{12}a_{21})b+(a_{21}a_{32}-a_{31}a_{22})c;\\
&A\phi (x_2,x_3)= \phi (a_{12}x_1+a_{22}x_2+a_{32}x_3+a_{42}x_4 , a_{33}x_3+a_{43}x_4)=a_{12}a_{33}a+a_{22}a_{33}c.
\end{align*}
In the matrix form we can write this as
\begin{align}\label{k}
\begin{pmatrix}
a'\\b' \\c'
\end{pmatrix}=
\begin{pmatrix}
a_{11}a_{33} &0&  a_{21}a_{33}\\
0&a_{11}a_{22}-a_{12}a_{21}&0\\
a_{12}a_{33}&0& a_{22}a_{33}\\
\end{pmatrix}
\begin{pmatrix}
a \\b\\c
\end{pmatrix}.
\end{align}
 The orbit with representative 
$\begin{pmatrix}
1\\ 0\\0
\end{pmatrix}$ of this action  gives us $L_{5,8}$.

Also, we have $\omega=\alpha f_1+\beta f_2+\delta f_4$, for some $\alpha, \beta, \delta \in \F$. Suppose that $A\omega =\alpha' f_1+\beta' f_2+\delta' f_4$, for some $\alpha', \beta', \delta' \in \F$. We have
\begin{align*}
&A\omega (x_1)= a_{11}^p \alpha + a_{21}^p\beta +a_{41}^p\delta;\\
&A\omega (x_2)=a_{12}^p \alpha+a_{22}^p\beta + a_{42}^p \delta;\\
&A\omega(x_4)=r^p \delta.
\end{align*}
In the matrix form we can write this as
\begin{align}\label{h}
\begin{pmatrix} 
\alpha'\\
\beta'\\
\delta'
\end{pmatrix}=
\begin{pmatrix}
 a_{11}^p& a_{21}^p&a_{41}^p\\
a_{12}^p&a_{22}^p&a_{42}^p\\
0&0&a_{33}^{p^2}
\end{pmatrix}
\begin{pmatrix}
\alpha\\
\beta \\
\delta
\end{pmatrix}.
\end{align}
Thus, we can write Equations \eqref{k} and \eqref{h} together as follows:
\begin{align*}
&\bigg[\begin{pmatrix}
a_{11}a_{33} &0&  a_{21}a_{33}\\
0&a_{11}a_{22}-a_{12}a_{21}&0\\
a_{12}a_{33}&0& a_{22}a_{33}\\
\end{pmatrix},
\begin{pmatrix}
 a_{11}^p& a_{21}^p&a_{41}^p\\
a_{12}^p&a_{22}^p&a_{42}^p\\
0&0&a_{33}^{p^2}
\end{pmatrix}\bigg]
\bigg[\begin{pmatrix}
a \\b\\c
\end{pmatrix},
\begin{pmatrix}
\alpha\\
\beta\\
\delta
\end{pmatrix}\bigg]\\
&=
\bigg[\begin{pmatrix}
a'\\b' \\ c'
\end{pmatrix},
\begin{pmatrix}
\alpha'\\
\beta'\\
\delta'
\end{pmatrix}\bigg].
\end{align*}
Note that $A\phi(x_1,x_4)=A\phi(x_2,x_4)=A\phi(x_3,x_4)=0$ and $A\omega(x_3)=0$ which imply that $a_{43}^p\delta=0$.
Now we find the representatives of the orbits of the action of $\Aut (L)$ on the set of $\omega$'s such that the orbit represented by
$\begin{pmatrix}
1\\ 0\\0
\end{pmatrix}$
 is preserved under the action of $\Aut (L)$ on the set of $\phi$'s.
Note that we take $a_{43}=0$. Therefore, $a_{43}^p\delta=0$.
Let $\nu = \begin{pmatrix}
\alpha \\
\beta \\
\delta
\end{pmatrix} \in \F^3$. If $\nu= \begin{pmatrix}
0 \\
0 \\
0
\end{pmatrix}$, then $\{\nu \}$ is clearly an $\Aut(L)$-orbit. Let $\nu \neq 0$. Suppose that $\delta \neq 0$. Then 
\begin{align*}
&\bigg[\begin{pmatrix}
1&0&0\\
0&1&0\\
0&0&1
\end{pmatrix},
\begin{pmatrix}
 1& 0&-\alpha/\delta\\
0&1&-\beta/\delta\\
0&0&1
\end{pmatrix}\bigg]
\bigg[\begin{pmatrix}
1 \\ 0\\0
\end{pmatrix},
\begin{pmatrix} 
\alpha \\
\beta \\
\delta
\end{pmatrix}\bigg]=
\bigg[\begin{pmatrix}
1\\0\\0
\end{pmatrix},
\begin{pmatrix}
0\\
0\\
\delta
\end{pmatrix}\bigg], \text{ and }\\
&\bigg[\begin{pmatrix}
1&0&0\\
0&(1/\delta)^{1/p}&0\\
0&0&(1/\delta)^\frac{p+2}{p^2}
\end{pmatrix},
\begin{pmatrix}
 \delta^{1/p}& 0&0\\
0&(1/\delta)^\frac{p+1}{p}&0\\
0&0&1/\delta
\end{pmatrix}\bigg]
\bigg[\begin{pmatrix}
1\\ 0\\0
\end{pmatrix},
\begin{pmatrix}
0 \\
0 \\
\delta
\end{pmatrix}\bigg]=
\bigg[\begin{pmatrix}
1\\0\\0
\end{pmatrix},
\begin{pmatrix}
0\\
0\\
1
\end{pmatrix}\bigg].
\end{align*}
Next, if $\delta =0$, but $\beta \neq 0$, then
\begin{align*}
&\bigg[\begin{pmatrix}
1&0& (-\alpha/\beta)^{1/p}\\
0& 1&0\\
0&0&1
\end{pmatrix},
\begin{pmatrix}
 1& -\alpha/\beta&0\\
0&1&0\\
0&0&1
\end{pmatrix}\bigg]
\bigg[\begin{pmatrix}
1\\ 0\\0
\end{pmatrix},
\begin{pmatrix} 
\alpha\\
\beta \\
0
\end{pmatrix}\bigg]=
\bigg[\begin{pmatrix}
1\\0\\0
\end{pmatrix},
\begin{pmatrix}
0\\
\beta\\
0
\end{pmatrix}\bigg].
\end{align*}
Finally, if $\delta =\beta=0$, but $\alpha \neq 0$, then
\begin{align*}
&\bigg[\begin{pmatrix}
1& 0&0\\
0&\alpha&0\\
0&0&\alpha^\frac{p+2}{p}
\end{pmatrix},
\begin{pmatrix}
 1/\alpha& 0&0\\
0&\alpha^{p+1}&0\\
0&0&\alpha^p
\end{pmatrix}\bigg]
\bigg[\begin{pmatrix}
1\\ 0\\0
\end{pmatrix},
\begin{pmatrix}
\alpha\\
0\\
0
\end{pmatrix}\bigg]=
\bigg[\begin{pmatrix}
1\\0\\0
\end{pmatrix},
\begin{pmatrix}
1\\
0\\
0
\end{pmatrix}\bigg].
\end{align*}
Thus the following elements are $\Aut (L)$-orbit representatives:
\begin{align*}
\begin{pmatrix}
0\\
0\\
0
\end{pmatrix}, 
\begin{pmatrix}
1 \\ 
0\\
0
\end{pmatrix},
\begin{pmatrix}
0 \\ 
\beta\\
0
\end{pmatrix},
\begin{pmatrix}
0 \\ 
0\\
1
\end{pmatrix}.
\end{align*}
Therefore, the corresponding restricted Lie algebra structures are as follows:
\begin{align*}
&K_8^{11}=\langle x_1,\ldots,x_5\mid [x_1,x_2]=x_4,[x_1,x_3]=x_5 ,
\res x_3=x_4 \rangle;\\
&K_8^{12}=\langle x_1,\ldots,x_5\mid [x_1,x_2]=x_4,[x_1,x_3]=x_5, 
\res x_1=x_5 , \res x_3=x_4 \rangle;\\
&K_8^{13}(\beta)=\langle x_1,\ldots,x_5\mid [x_1,x_2]=x_4,[x_1,x_3]=x_5, 
\res x_2=\beta x_5, \res x_3=x_4 \rangle;\\
&K_8^{14}=\langle x_1,\ldots,x_5\mid [x_1,x_2]=x_4,[x_1,x_3]=x_5,
\res x_3=x_4, \res x_4=x_5 \rangle
\end{align*}
where $\beta \in \F^*$.
\begin{lemma}
We have $K_8^{13}(\beta_1)$ and $K_8^{13}(\beta_2)$ are isomorphic if and only if $\beta_1/\beta_2= \epsilon^{p+1}$, for some $\epsilon\in \F^*$. 
\end{lemma}
\begin{proof}
Let $f=(a_{ij})\in \Aut(K_8^{10})$. Then $f$ is an isomorphism from  $K_8^{13}(\beta_1)$
to  $K_8^{13}(\beta_2)$ if and only if 
\begin{align*}
a_{11}=a_{33}^{p}a_{22}^{-1}, \quad \beta_1/\beta_2=a_{22}^p a_{11}^{-1} a_{33}^{-1},
\end{align*}
and this in turn simplifies to having 
\begin{align*}
 \beta_1/\beta_2= (\frac{a_{22}}{a_{33}})^{p+1}.
\end{align*}
To prove the converse, suppose that $\beta_1/\beta_2= \epsilon^{p+1}$, for some $\epsilon\in \F^*$.  Then, it is easy to see that the following is an isomorphism from  $K_8^{13}(\beta_1)$
to  $K_8^{13}(\beta_2)$:

$$\begin{pmatrix}
\epsilon^{-1/p}&0&0&0&0\\
0& \epsilon^{1/p}&0&0&0\\
0&0&1&0&0\\
0&0&0&1&0\\
0&0&0&0&\epsilon^{-1/p}\\
\end{pmatrix}.$$

\end{proof}

\section{Detecting isomorphisms}\label{iso-L58}
 The following is the list of all restricted Lie algebra structures on $L_{5,8}$ and yet, as we shall see below, we prove that some of them are isomorphic.
\begin{align*}
&K_8^{1}=\langle x_1,\ldots,x_5\mid [x_1,x_2]=x_4,[x_1,x_3]=x_5 \rangle;\\
&K_8^{2}=\langle x_1,\ldots,x_5\mid [x_1,x_2]=x_4,[x_1,x_3]=x_5, 
 \res x_1=x_5 \rangle;\\
&K_8^{3}=\langle x_1,\ldots,x_5\mid [x_1,x_2]=x_4,[x_1,x_3]=x_5, 
 \res x_2=x_5 \rangle;\\
&K_8^{4}=\langle x_1,\ldots,x_5\mid [x_1,x_2]=x_4,[x_1,x_3]=x_5,
\res x_3=x_5 \rangle;\\
&K_8^{5}=\langle x_1,\ldots,x_5\mid [x_1,x_2]=x_4,[x_1,x_3]=x_5, 
\res x_4=x_5 \rangle;\\
&K_8^{6}=\langle x_1,\ldots,x_5\mid [x_1,x_2]=x_4,[x_1,x_3]=x_5 ,
\res x_1=x_4\rangle;\\
&K_8^{7}=\langle x_1,\ldots,x_5\mid [x_1,x_2]=x_4,[x_1,x_3]=x_5, 
\res x_1=x_4 , \res x_2= x_5 \rangle;\\
&K_8^{8}=\langle x_1,\ldots,x_5\mid [x_1,x_2]=x_4,[x_1,x_3]=x_5, 
\res x_1=x_4, \res x_3=x_5 \rangle;\\
&K_8^{9}=\langle x_1,\ldots,x_5\mid [x_1,x_2]=x_4,[x_1,x_3]=x_5,
\res x_1=x_4, \res x_4=x_5 \rangle;\\
&K_8^{10}(\beta)=\langle x_1,\ldots,x_5\mid [x_1,x_2]=x_4,[x_1,x_3]=x_5,
\res x_1=x_4, \res x_2=\beta x_5, \res x_3=x_5 \rangle;\\
&K_8^{11}=\langle x_1,\ldots,x_5\mid [x_1,x_2]=x_4,[x_1,x_3]=x_5 ,
\res x_3=x_4 \rangle;\\
&K_8^{12}=\langle x_1,\ldots,x_5\mid [x_1,x_2]=x_4,[x_1,x_3]=x_5, 
\res x_1=x_5 , \res x_3=x_4 \rangle;\\
&K_8^{13}(\beta)=\langle x_1,\ldots,x_5\mid [x_1,x_2]=x_4,[x_1,x_3]=x_5, 
\res x_2=\beta x_5, \res x_3=x_4 \rangle;\\
&K_8^{14}=\langle x_1,\ldots,x_5\mid [x_1,x_2]=x_4,[x_1,x_3]=x_5,
\res x_3=x_4, \res x_4=x_5 \rangle.
\end{align*}

Note that the following automorphism of $L_{5,8}$ that maps $x_1\mapsto x_1$, $x_2 \mapsto x_3$, $x_3 \mapsto x_2$, $x_4 \mapsto x_5$ and $x_5\mapsto x_4$
$$\begin{pmatrix} 1& 0 & 0 & 0 & 0 \\
0 &0& 1 & 0&0\\
0 & 1& 0 & 0 & 0 \\
0& 0 & 0& 0
&1 \\
0 & 0 & 0 & 1
& 0
\end{pmatrix}$$
implies that
\begin{align*}
K_8^2 \cong K_8^6, \quad K_8^{3} \cong K_8^{11}, \quad K_8^{7} \cong K_8^{12}.
\end{align*}
Moreover, $K_8^{10}(\beta)\cong K_8^{8} $ via the automorphism
$$\begin{pmatrix}
1 &0& 0& 0&0\\
0&1&0&0&0\\
\beta^{1/p^2}& \beta^{1/p}& 1 & 0 & 0 \\
0& 0& 0& 1
&0 \\
0 & 0 & 0 & \beta^{1/p}
& 1
\end{pmatrix}.$$
\begin{theorem}\label{thmL58}
The list of all  the restricted Lie algebra structures on $L_{5,8}$, up to isomorphism, is as follows:
\begin{align*}
&L_{5,8}^{1}=\langle x_1,\ldots,x_5\mid [x_1,x_2]=x_4,[x_1,x_3]=x_5 \rangle;\\
&L_{5,8}^{2}=\langle x_1,\ldots,x_5\mid [x_1,x_2]=x_4,[x_1,x_3]=x_5, 
 \res x_1=x_5 \rangle;\\
&L_{5,8}^{3}=\langle x_1,\ldots,x_5\mid [x_1,x_2]=x_4,[x_1,x_3]=x_5, 
 \res x_2=x_5 \rangle;\\
&L_{5,8}^{4}=\langle x_1,\ldots,x_5\mid [x_1,x_2]=x_4,[x_1,x_3]=x_5,
\res x_3=x_5 \rangle;\\
&L_{5,8}^{5}=\langle x_1,\ldots,x_5\mid [x_1,x_2]=x_4,[x_1,x_3]=x_5, 
\res x_4=x_5 \rangle;\\
&L_{5,8}^{6}=\langle x_1,\ldots,x_5\mid [x_1,x_2]=x_4,[x_1,x_3]=x_5, 
\res x_1=x_4 , \res x_2=x_5 \rangle;\\
&L_{5,8}^{7}=\langle x_1,\ldots,x_5\mid [x_1,x_2]=x_4,[x_1,x_3]=x_5, 
\res x_1=x_4, \res x_3=x_5 \rangle;\\
&L_{5,8}^{8}=\langle x_1,\ldots,x_5\mid [x_1,x_2]=x_4,[x_1,x_3]=x_5,
\res x_1=x_4, \res x_4=x_5 \rangle;\\
&L_{5,8}^{9}(\beta)=\langle x_1,\ldots,x_5\mid [x_1,x_2]=x_4,[x_1,x_3]=x_5, 
\res x_2=\beta x_5, \res x_3=x_4 \rangle;\\
&L_{5,8}^{10}=\langle x_1,\ldots,x_5\mid [x_1,x_2]=x_4,[x_1,x_3]=x_5,
\res x_3=x_4, \res x_4=x_5 \rangle
\end{align*}
where $\beta \in \F^*$.
\end{theorem}

In the remaining of this section we establish that the algebras given in Theorem \ref{thmL58} are pairwise non-isomorphic, thereby
completing the proof of Theorem \ref{thmL58}. 

It is clear that $L_{5,8}^1$ is not isomorphic to the other restricted Lie algebras.

We claim that $L_{5,8}^2$ and $L_{5,8}^3$ are not isomorphic. Suppose to the contrary that there exists an isomorphism $A: L_{5,8}^3\to L_{5,8}^2$. Then 
\begin{align*}
&A(\res x_1)=\res {A(x_1)}\\
&0=\res {(a_{11}x_1+a_{21}x_2+a_{31}x_3+a_{41}x_4+a_{51}x_5)}\\
&0=a_{11}^px_5.
\end{align*}
Therefore, $a_{11}=0$ which is a contradiction.

Next, we claim that $L_{5,8}^2$ and $L_{5,8}^4$ are not isomorphic. Suppose to the contrary that there exists an isomorphism $A: L_{5,8}^4\to L_{5,8}^2$. Then 
\begin{align*}
&A(\res x_1)=\res {A(x_1)}\\
&0=\res {(a_{11}x_1+a_{21}x_2+a_{31}x_3+a_{41}x_4+a_{51}x_5)}
&0=a_{11}^px_5.
\end{align*}
Therefore, $a_{11}=0$ which is a contradiction.

Next, we claim that $L_{5,8}^2$ and $L_{5,8}^5$ are not isomorphic. Suppose to the contrary that there exists an isomorphism $A: L_{5,8}^5\to L_{5,8}^2$. Then 
\begin{align*}
&A(\res x_1)=\res {A(x_1)}\\
&0=\res {(a_{11}x_1+a_{21}x_2+a_{31}x_3+a_{41}x_4+a_{51}x_5)}\\
&0=a_{11}^px_5.
\end{align*}
Therefore, $a_{11}=0$ which is a contradiction.

It is clear that $L_{5,8}^2$ is not isomorphic to the other restricted Lie algebras.

Next, we claim that $L_{5,8}^3$ and $L_{5,8}^4$ are not isomorphic. Suppose to the contrary that there exists an isomorphism $A: L_{5,8}^4\to L_{5,8}^3$. Then 
\begin{align*}
&A(\res x_3)=\res {A(x_3)}\\
&A(x_5)=\res {(a_{23}x_2+a_{33}x_3+a_{43}x_4+a_{53}x_5)}\\
&a_{11}a_{23}x_4+a_{11}a_{33}x_5=a_{23}^px_5,
\end{align*}
which implies that $a_{11}a_{23}=0$ and $a_{11}a_{33}=a_{23}^p$.  Therefore, $a_{11}=0$ or $a_{23}=0$. First, if $a_{11}=0$, we have a contradiction. Next, if $a_{23}=0$ then $a_{11}a_{33}=0$ and $a_{11}a_{23}=0$ which is a contradiction.

Next, we claim that $L_{5,8}^3$ and $L_{5,8}^5$ are not isomorphic. Suppose to the contrary that there exists an isomorphism $A: L_{5,8}^5\to L_{5,8}^3$. Then 
\begin{align*}
&A(\res x_4)=\res {A(x_4)}\\
&A(x_5)=\res {(a_{11}a_{22}x_4+a_{11}a_{32}x_5)}\\
&a_{11}a_{23}x_4+a_{11}a_{33}x_5=0.
\end{align*}
Therefore, $a_{11}a_{23}=0$ and $a_{11}a_{33}=0$ which is a contradiction.

It is clear that $L_{5,8}^3$ is not isomorphic to the other restricted Lie algebras.

Next, we claim that $L_{5,8}^4$ and $L_{5,8}^5$ are not isomorphic. Suppose to the contrary that there exists an isomorphism $A: L_{5,8}^5\to L_{5,8}^4$. Then 
\begin{align*}
&A(\res x_4)=\res {A(x_4)}\\
&A(x_5)=\res {(a_{11}a_{22}x_4+a_{11}a_{32}x_5)}\\
&a_{11}a_{23}x_4+a_{11}a_{33}x_5=0.
\end{align*}
Therefore, $a_{11}a_{23}=0$ and $a_{11}a_{33}=0$ which is a contradiction.

It is clear that $L_{5,8}^4$ and $L_{5,8}^5$ are not isomorphic to the other restricted Lie algebras.

Next, we claim that $L_{5,8}^6$ and $L_{5,8}^7$ are not isomorphic. Suppose to the contrary that there exists an isomorphism $A: L_{5,8}^7\to L_{5,8}^6$. Then 
\begin{align*}
&A(\res x_3)=\res {A(x_3)}\\
&A(x_5)=\res {(a_{23}x_2+a_{33}x_3+a_{43}x_4+a_{53}x_5)}\\
&a_{11}a_{23}x_4+a_{11}a_{33}x_5=a_{23}^px_5,
\end{align*}
which implies that $a_{11}a_{23}=0$ and $a_{11}a_{33}=a_{23}^p$.  Therefore, $a_{11}=0$ or $a_{23}=0$. First, if $a_{11}=0$, we have a contradiction. Next, if $a_{23}=0$ then $a_{11}a_{33}=0$ and $a_{11}a_{23}=0$ which is a contradiction.

Note  that $L_{5,8}^6$ is not isomorphic to any of $L_{5,8}^{8}$, $L_{5,8}^{10}$
 because 
$(L_{5,8}^6)^{[p]^2}=0$  but $(L_{5,8}^8)^{[p]^2}\neq 0$, $(L_{5,8}^{10})^{[p]^2}\neq 0$.

Next, we claim that $L_{5,8}^6$ and $L_{5,8}^{9}(\beta)$ are not isomorphic. Suppose to the contrary that there exists an isomorphism $A: L_{5,8}^{9}(\beta)\to L_{5,8}^6$. Then 
\begin{align*}
&A(\res x_1)=\res {A(x_1)}\\
&0=\res {(a_{11}x_1+a_{21}x_2+a_{31}x_3+a_{41}x_4+a_{51}x_5)}\\
&0=a_{11}^px_4+a_{21}^px_5.
\end{align*}
Therefore, $a_{11}=0$ which is a contradiction.

Note  that $L_{5,8}^7$ is not isomorphic to any of $L_{5,8}^{8}$, $L_{5,8}^{10}$
 because 
$(L_{5,8}^7)^{[p]^2}=0$  but $(L_{5,8}^8)^{[p]^2}\neq 0$, $(L_{5,8}^{10})^{[p]^2}\neq 0$.

Next, we claim that $L_{5,8}^7$ and $L_{5,8}^{9}(\beta)$ are not isomorphic. Suppose to the contrary that there exists an isomorphism $A: L_{5,8}^{9}(\beta)\to L_{5,8}^7$. Then 
\begin{align*}
&A(\res x_1)=\res {A(x_1)}\\
&0=\res {(a_{11}x_1+a_{21}x_2+a_{31}x_3+a_{41}x_4+a_{51}x_5)}\\
&0=a_{11}^px_4+a_{31}^px_5.
\end{align*}
Therefore, $a_{11}=0$ which is a contradiction.

Note  that $L_{5,8}^8$ is not isomorphic to $L_{5,8}^{9}(\beta)$,
 because 
$(L_{5,8}^{9}(\beta))^{[p]^2}=0$ but $(L_{5,8}^8)^{[p]^2}\neq 0$.

Next, we claim that $L_{5,8}^8$ and $L_{5,8}^{10}$ are not isomorphic. Suppose to the contrary that there exists an isomorphism $A: L_{5,8}^{10}\to L_{5,8}^8$. Then 
\begin{align*}
&A(\res x_1)=\res {A(x_1)}\\
&0=\res {(a_{11}x_1+a_{21}x_2+a_{31}x_3+a_{41}x_4+a_{51}x_5)}\\
&0=a_{11}^px_4+a_{41}^px_5.
\end{align*}
Therefore, $a_{11}=0$ which is a contradiction.

Note  that $L_{5,8}^{9}(\beta)$ is not isomorphic to $L_{5,8}^{10}$,
 because 
$(L_{5,8}^{9}(\beta))^{[p]^2}=0$ but $(L_{5,8}^{10})^{[p]^2}\neq 0$.\\

%% file: 5,9.tex
\chapter{Restriction maps on $L_{5,9}$}
Let 
$$ K_9=L_{5,9} = \langle x_1, \ldots, x_5\mid  [x_1,x_2]=x_3 , [x_1,x_3]=x_4 , [x_2,x_3]=x_5\rangle. $$

Then $Z({K_9})=\la x_4,x_5\ra_{\F} $ and  the group $\Aut(L_{5,9})$ consists of invertible matrices of the form 
\begin{align}
\begin{pmatrix} 
a_{11} & a_{12} & 0 & 0 & 0 \\
a_{21} & a_{22} & 0 & 0 & 0\\
a_{31} & a_{32} & u & 0 & 0 \\
a_{41} & a_{42} & a_{11}a_{32}-a_{12}a_{31} & a_{11}u
& a_{12}u \\
a_{51} & a_{52} & a_{21}a_{32}-a_{22}a_{31} & a_{21}u
& a_{22}u
\end{pmatrix},
\label{aut-59}
\end{align}
where $u = a_{11}a_{22}-a_{12}a_{21}\neq 0$.
Note that there exists an element $\alpha x_4+\beta x_5 \in Z(L_{5,9})$ such that 
 $\res {(\alpha x_4+\beta x_5)}=0,$ for some $\alpha ,\beta \in \F$. 
If $\alpha \neq 0$ then consider 
$$ K=\langle x_1',\ldots,x_5'\mid [x_1',x_2']=x_3', [x_1',x_3']=x_4', [x_2',x_3']=x_5'\rangle,$$
 where $x_1'=\alpha x_1+\beta x_2, x_2'=\alpha^{-1}x_2, x_3'=x_3, x_4'=\alpha x_4+\beta x_5, x_5'=\alpha^{-1}x_5$.
 Let $\phi :K_9\to K$ given by  $x_i\mapsto x_i'$, for $1\leq i\leq 5$. It is easy to see that $\phi$ is an isomorphism. Therefore, in this case we can suppose that $\res x_4=0$.  
 If $\alpha=0$ then $\beta\neq 0$ and we rescale $x_5$ so that $\res x_5=0$.
Hence we can assume either $\res x_4=0$ or $\res x_5=0$. 
Furthermore, we claim that it is enough to consider the case where $\res x_5=0$. Indeed, suppose that there exists a 
$p$-map of $L_{5,9}$ such that $\res x_4=0$. Then the image of the $x_i$'s under the following  automorphism of $L_{5,9}$   yields  another basis $y_i$ of $L_{5,9}$ such that $y_5^{p]}=0$:
\[\begin{pmatrix} 
0 & 1 & 0 & 0 & 0 \\
1 & 0 & 0 & 0 & 0\\
0 & 0 & -1 & 0 & 0 \\
0 & 0 & 0 & 0& -1 \\
0 & 0 & 0 & -1
& 0
\end{pmatrix}.\]

Now we  let 
$$
L=\frac {L_{5,9}}{\la x_5\ra_{\F}}\cong L_{4,3},
$$ 
where $L_{4,3} = \la x_1,\ldots,x_4 \mid [x_1 , x_2]=x_3 , [x_1 , x_3]=x_4\ra $. 
The group $\Aut(L)$ consists of invertible matrices of the form
\[\begin{pmatrix} a_{11}  & 0 & 0 & 0 \\
a_{21} & a_{22} & 0 & 0\\
a_{31} & a_{32} & d& 0 \\
a_{41} & a_{42} & a_{11}a_{32}& a_{11}d
\end{pmatrix},\]
where $d= a_{11}a_{22}$. 
\begin{lemma}\label{LemmaL59}
Let $K=L_{5,9}$ and $[p]:K\to K$ be a $p$-map on $K$ such that $\res x_5=0$ and let $L=\frac{K}{M}$ where $M=\la x_5\ra_{\F}$. Then $K\cong L_{\theta}$ where $\theta=(\Delta_{23},\omega)\in Z^2(L,\F)$.
\end{lemma}
\begin{proof}
 Let $\pi :K\rightarrow L$ be the projection map. We have the exact sequence 
$$ 0\rightarrow M\rightarrow K\rightarrow L\rightarrow 0.$$
Let $\sigma :L\rightarrow K$ such that $x_i\mapsto x_i$, $1\leq i\leq 4$. Then $\sigma$ is an injective linear map and $\pi \sigma =1_L$. Now, we define $\phi : L\times L\rightarrow M$ by $\phi (x_i,x_j)=[\sigma (x_i),\sigma (x_j)]-\sigma ([x_i,x_j])$, $1\leq i,j\leq 4$ and $\omega: L\rightarrow M$ by $\omega (x)=\res {\sigma (x)} -\sigma (\res x)$. Note that
\begin{align*}
&\phi(x_2,x_3)=[\sigma(x_2),\sigma(x_3)]-\sigma([x_2,x_3])=[x_2,x_3]=x_5;\\
&\phi(x_1,x_2)=[\sigma(x_1),\sigma(x_2)]-\sigma([x_1,x_2])=0.
\end{align*}
Similarly, we can show that $\phi(x_1,x_3)=\phi(x_1,x_4)=\phi(x_2,x_4)=\phi(x_3,x_4)=0$. Therefore, $\phi=\Delta_{23}$.
Now, by Lemma \ref{K=L-theta}, we have $\theta=(\Delta_{23},\omega)\in Z^2(L,\F)$ and $K\cong L_{\theta}$.
\end{proof}
We deduce that any $p$-map on $K$ can be obtained by an extension of $L$ via $\theta=(\Delta_{23},\omega)$, for some $\omega$.

Note that by Theorem \ref{4-dim},  there are four non-isomorphic restricted Lie algebra structures on $L_{4,3}$ given by the following $p$-maps:
\begin{enumerate}
\item[I.1] Trivial $p$-map;
\item[I.2] $\res x_1=x_4 $;
\item[I.3]  $\res x_2 =\xi x_4 $;
\item[I.4]  $ \res x_3 =x_4$.
\end{enumerate}

In the following subsections we consider each of the above cases and find all possible 
$[\theta]=[(\phi,\omega)]\in H^2(L,\F)$ and construct $L_{\theta}$. Note that by Lemma \ref{LemmaL59}, it suffices to assume $[\theta]=[(\Delta_{23},\omega)]$ and find all non-isomorphic restricted Lie algebra structures on $L_{5,9}$.

\section{Extensions of ($L$, trivial $p$-map)}
 First, we find a basis for $Z^2(L,\F)$. Let $(\phi,\omega)=(a\Delta _{12}+b\Delta_{13}+c\Delta_{14}+d\Delta_{23}+e\Delta_{24}+f\Delta_{34},\alpha f_1+\beta f_2+ \gamma f_3+\delta f_4)\in Z^2(L,\F)$. Then we must have $\delta^2\phi(x,y,z) =0$ and $\phi(x,\res y)=0$, for all $x,y,z \in L$. Therefore,
\begin{align*}
0=(\delta^2\phi)(x_1,x_2,x_3)&=\phi([x_1,x_2],x_3)+\phi([x_2,x_3],x_1)+\phi([x_3,x_1],x_2)=\phi(x_2,x_4);\\
0=(\delta^2\phi)(x_1,x_2,x_4)&=\phi([x_1,x_2],x_4)+\phi([x_2,x_4],x_1)+\phi([x_4,x_1],x_2)=\phi(x_3,x_4).
\end{align*}
Thus, we get $e=f=0$.
Since the $p$-map is trivial, $\phi(x,\res y)=\phi(x,0)=0$, for all $x,y\in L$. Therefore, a basis for $Z^2(L,\F)$ is as follows:
$$
 (\Delta_{12},0),(\Delta_{13},0),(\Delta_{14},0),(\Delta_{23},0),(0,f_1),
(0,f_2),(0,f_3),(0,f_4).
$$
Next, we find a basis for $B^2(L,\F)$. Let $(\phi,\omega)\in B^2(L,\F)$. Since $B^2(L,\F)\subseteq Z^2(L,\F)$, we have $(\phi,\omega)=(a\Delta _{12}+b\Delta_{13}+c\Delta_{14}+d\Delta_{23},\alpha f_1+\beta f_2+ \gamma f_3+\delta f_4)$. So, there exists a linear map $\psi:L\to \F$ such that $\delta^1\psi(x,y)=\phi(x,y)$ and $\tilde \psi(x)=\omega(x)$, for all $x,y \in L$. So, we have
\begin{align*}
c=\phi(x_1,x_4)=\delta^1\psi(x_1,x_4)=\psi([x_1,x_4])=0.
\end{align*}
Similarly, we can show that $d=0$. Also, we have
\begin{align*}
\alpha=\omega(x_1)=\tilde \psi(x_1)=\psi(\res x_1)=0.
\end{align*}
Similarly, we can show that $\beta=\gamma=\delta=0$. Therefore, $(\phi,\omega)=(a\Delta_{12}+b\Delta_{13},0)$ and hence
 $B^2(L,\F)=\la(\Delta_{12},0),(\Delta_{13},0)\ra_{\F}$. We deduce that a basis for $H^2(L,\F)$ is as follows:
$$
[(\Delta_{14},0)],[(\Delta_{23},0)],[(0,f_1)],
[(0,f_2)],[(0,f_3)],[(0,f_4)].
$$
Let $[(\phi,\omega)] \in H^2(L,\F)$. Then we have $\phi = a\Delta_{14} + b\Delta_{23}$, for some $a, b\in \F$. Suppose that $A\phi = a'\Delta_{14}+b'\Delta_{23}$, for some $a', b'\in \F$. We determine $a', b'$. 
Note that 
\begin{align*}
A\phi (x_1,x_4)=&\phi (Ax_1,Ax_4)=\phi (a_{11}x_1+a_{21}x_2+a_{31}x_3+a_{41}x_4,a_{11}dx_4)=a_{11}^2da; \\
A\phi (x_2,x_3)=&\phi (Ax_2,Ax_3)=\phi (a_{22}x_2+a_{32}x_3+a_{42}x_4,dx_3+a_{11}a_{32}x_4)=a_{22}db.
\end{align*}
In the matrix form we can write this as
\begin{align}\label{cocycle}
\begin{pmatrix}
a' \\ b'
\end{pmatrix}=
\begin{pmatrix}
a_{11}^2d & 0\\
0& a_{22}d
\end{pmatrix}
\begin{pmatrix}
a \\ b
\end{pmatrix}.
\end{align}
 The orbit with representative 
$\begin{pmatrix}
0 \\ 1
\end{pmatrix}$ of this action  gives us $L_{5,9}$.

Also we have $\omega=\alpha f_1+\beta f_2+\gamma f_3+\delta f_4$, for some $\alpha, \beta, \gamma, \delta\in \F$.
 Suppose that $A\omega = \alpha' f_1+\beta' f_2+\gamma' f_3+\delta' f_4$, for some $\alpha', \beta', \gamma', \delta'\in \F$. We have
\begin{align*}
A\omega (x_1) =& \omega (Ax_1)=\omega (a_{11}x_1+a_{21}x_2+a_{31}x_3+a_{41}x_4)=a_{11}^p \alpha+a_{21}^p \beta +a_{31}^p \gamma+a_{41}^p \delta; \\
A\omega (x_2)=& a_{22}^p \beta +a_{32}^p \gamma + a_{42}^p \delta;\\
A\omega (x_3)=& d^p \gamma + a_{11}^pa_{32}^p \delta;\\
A\omega(x_4)=&a_{11}^pd^p \delta.
\end{align*}
In the matrix form we can write this as
\begin{align}\label{omega}
\begin{pmatrix}
\alpha' \\ 
\beta'\\
\gamma'\\
\delta'
\end{pmatrix}
=
\begin{pmatrix}
a_{11}^p& a_{21}^p&a_{31}^p&a_{41}^p\\
0& a_{22}^p& a_{32}^p&a_{42}^p\\
0&0&d^p&a_{11}^pa_{32}^p\\
0&0&0&a_{11}^pd^p
\end{pmatrix}
\begin{pmatrix}
\alpha\\ 
\beta \\
\gamma \\
\delta
\end{pmatrix}.
\end{align}
Thus, we can write Equations \eqref{cocycle} and \eqref{omega} together as follows:
\begin{align*}
\bigg[
d\begin{pmatrix}
a_{11}^2& 0\\
0& a_{22}
\end{pmatrix},
\begin{pmatrix}
a_{11}^p& a_{21}^p&a_{31}^p&a_{41}^p\\
0& a_{22}^p& a_{32}^p&a_{42}^p\\
0&0&d^p&a_{11}^pa_{32}^p\\
0&0&0&a_{11}^pd^p
\end{pmatrix}
\bigg]
\bigg[
\begin{pmatrix}
a \\ b
\end{pmatrix},
\begin{pmatrix}
\alpha\\ 
\beta \\
\gamma \\
\delta
\end{pmatrix}
\bigg]
=
\bigg[
\begin{pmatrix}
a' \\ b'
\end{pmatrix},
\begin{pmatrix}
\alpha' \\ 
\beta'\\
\gamma'\\
\delta'
\end{pmatrix}
\bigg].
\end{align*}
Now we find the representatives of the orbits of the action of $\Aut_p (L)$ on the set of $\omega$'s such that 
the orbit represented by 
$\begin{pmatrix}
0 \\ 1
\end{pmatrix}$ is preserved under the action of $\Aut_p (L)$ on the set of $\phi$'s.

Let $\nu = \begin{pmatrix}
\alpha\\ 
\beta \\
\gamma \\
\delta
\end{pmatrix} \in \F^4$. If $\nu= \begin{pmatrix}
0\\ 
0 \\
0 \\
0
\end{pmatrix}$, then $\{\nu \}$ is clearly an $\Aut_p (L)$-orbit. Suppose that $\delta \neq 0$. Then 
\begin{align*}
\bigg[\begin{pmatrix}
1& 0\\
0& 1
\end{pmatrix},
\begin{pmatrix}
1& 0&0&-\alpha /\delta\\
0& 1& 0&-\beta /\delta\\
0&0&1&0\\
0&0&0&1
\end{pmatrix}\bigg]
\bigg[\begin{pmatrix}
0 \\ 1
\end{pmatrix},
\begin{pmatrix}
\alpha\\ 
\beta \\
\gamma \\
\delta
\end{pmatrix}\bigg]
=&
\bigg[\begin{pmatrix}
0\\1
\end{pmatrix},
\begin{pmatrix}
0 \\ 
0\\
\gamma\\
\delta
\end{pmatrix}\bigg], \text{ and }\\
\bigg[\begin{pmatrix}
1& 0\\
0& 1
\end{pmatrix},
\begin{pmatrix}
1& 0&0&0\\
0& 1& -\gamma /\delta&\gamma^2/\delta^2\\
0&0&1&-\gamma /\delta\\
0&0&0&1
\end{pmatrix}\bigg]
\bigg[\begin{pmatrix}
0 \\ 1
\end{pmatrix},
\begin{pmatrix}
0\\ 
0 \\
\gamma \\
\delta
\end{pmatrix}\bigg]
=&
\bigg[\begin{pmatrix}
0\\1
\end{pmatrix},
\begin{pmatrix}
0 \\ 
0\\
0\\
\delta
\end{pmatrix}\bigg].
\end{align*}

Hence the set of vectors $\begin{pmatrix}
\alpha\\ 
\beta \\
\gamma \\
\delta
\end{pmatrix}$ with $\delta \neq 0$ form an single orbit with orbit representative
$ \begin{pmatrix}
0 \\ 
0\\
0\\
\delta
\end{pmatrix}$. Next, if $\delta= 0$, but $\gamma \neq 0$, then
\begin{align*}
&\bigg[\begin{pmatrix}
1& 0\\
0& 1
\end{pmatrix},
\begin{pmatrix}
1& 0&-\alpha /\gamma&0\\
0& 1& -\beta /\gamma &0\\
0&0&1&-\beta /\gamma\\
0&0&0&1
\end{pmatrix}\bigg]
\bigg[\begin{pmatrix}
0 \\ 1
\end{pmatrix},
\begin{pmatrix}
\alpha\\ 
\beta\\
\gamma \\
0
\end{pmatrix}\bigg]=
\bigg[\begin{pmatrix}
0\\1
\end{pmatrix},
\begin{pmatrix}
0 \\ 
0\\
\gamma\\
0
\end{pmatrix}\bigg], 
\text{ and }\\
&\bigg[(1/\gamma)^{1/p}
\begin{pmatrix}
(1/\gamma)^{4/p}& 0\\
0& (1/\gamma)^{-1/p}
\end{pmatrix},
\begin{pmatrix}
\gamma^{-2}& 0&0&0\\
0& \gamma&0 &0\\
0&0&1/\gamma&0\\
0&0&0&\gamma^{-3}
\end{pmatrix}\bigg]
\bigg[\begin{pmatrix}
0 \\ 1
\end{pmatrix},
\begin{pmatrix}
0\\ 
0\\
\gamma \\
0
\end{pmatrix}\bigg]=
\bigg[\begin{pmatrix}
0\\1
\end{pmatrix},
\begin{pmatrix}
0 \\ 
0\\
1\\
0
\end{pmatrix}\bigg].
\end{align*}
Next, if $\gamma=\delta =0$, but $\beta \neq 0$, then 
\begin{align*}
&\bigg[\begin{pmatrix}
1& 0\\
0& 1
\end{pmatrix},
\begin{pmatrix}
1& -\alpha /\beta&0&0\\
0& 1& 0&0\\
0&0&1&0\\
0&0&0&1
\end{pmatrix}\bigg]
\bigg[\begin{pmatrix}
0 \\ 1
\end{pmatrix},
\begin{pmatrix}
\alpha\\ 
\beta\\
0 \\
0
\end{pmatrix}\bigg]=
\bigg[\begin{pmatrix}
0\\1
\end{pmatrix},
\begin{pmatrix}
0 \\ 
\beta\\
0\\
0
\end{pmatrix}\bigg], \text{ and }\\
& \bigg[\beta^{1/p}\begin{pmatrix}
\beta^{4/p}& 0\\
0& \beta^{-1/p}
\end{pmatrix},
\begin{pmatrix}
\beta^{2}& 0&0&0\\
0& 1/\beta&0 &0\\
0&0&\beta&0\\
0&0&0&\beta^{3}
\end{pmatrix}\bigg]
\bigg[\begin{pmatrix}
0 \\ 1
\end{pmatrix},
\begin{pmatrix}
0\\ 
\beta\\
0 \\
0
\end{pmatrix}\bigg]=
\bigg[\begin{pmatrix}
0\\1
\end{pmatrix},
\begin{pmatrix}
0 \\ 
1\\
0\\
0
\end{pmatrix}\bigg].
\end{align*}
Finally, if  $\beta=\gamma=\delta =0$, but $\alpha \neq 0$, then we have
$\bigg[\begin{pmatrix}
0\\1
\end{pmatrix},
\begin{pmatrix}
\alpha \\ 
0\\
0\\
0
\end{pmatrix}\bigg]$.

Thus the following elements are $\Aut_p (L)$-orbit representatives: 
\begin{align*}
\begin{pmatrix}
0\\ 
0 \\
0 \\
0
\end{pmatrix}, 
\begin{pmatrix}
\alpha \\ 
0\\
0\\
0
\end{pmatrix}, 
\begin{pmatrix}
0 \\ 
1\\
0\\
0
\end{pmatrix}, \begin{pmatrix}
0\\ 
0\\
1\\
0
\end{pmatrix}, \begin{pmatrix}
0 \\ 
0\\
0\\
\delta
\end{pmatrix}.
\end{align*}

\begin{lemma}\label{lemma-K9-5}
The vectors
$\begin{pmatrix}
0\\ 
0\\
0\\
\delta_1
\end{pmatrix}$ and
$\begin{pmatrix}
0 \\ 
0\\
0\\
\delta_2
\end{pmatrix}$
are in the same $\Aut_p (L)$-orbit if and only if $\delta_2\delta_1^{-1}\in (\F^*)^3$.
\end{lemma}
\begin{proof}
 First assume that 
$\begin{pmatrix}
0 \\ 
0\\
0\\
\delta_1
\end{pmatrix}$ and
$\begin{pmatrix}
0\\ 
0\\
0\\
\delta_2
\end{pmatrix}$
are in the same $\Aut_p (L)$-orbit. Then
\begin{align*}
\bigg[
d\begin{pmatrix}
a_{11}^2& 0\\
0& a_{22}
\end{pmatrix},
\begin{pmatrix}
a_{11}^p& a_{21}^p&a_{31}^p&a_{41}^p\\
0& a_{22}^p& a_{32}^p&a_{42}^p\\
0&0&d^p&a_{11}^pa_{32}^p\\
0&0&0&a_{11}^pd^p
\end{pmatrix}
\bigg]
\bigg[
\begin{pmatrix}
0 \\ 1
\end{pmatrix},
\begin{pmatrix}
0\\ 
0 \\
0 \\
\delta_1
\end{pmatrix}
\bigg]
=
\bigg[
\begin{pmatrix}
0\\1
\end{pmatrix},
\begin{pmatrix}
0 \\
 0\\
0\\
\delta_2
\end{pmatrix}
\bigg].
\end{align*}
From this we obtain that $\delta_2=a_{11}^pd^p\delta_1$ and that $da_{22}=1$ which gives that $\delta_2\delta_1^{-1}=a_{22}^{-3p}\in (\F^*)^3$. Assume next that $\delta_2\delta_1^{-1}=\epsilon^3$, with some $\epsilon \in \F^*$. Then
\begin{align*}
\bigg[\epsilon ^{1/p}\begin{pmatrix}
(\epsilon)^{4/p}& 0\\
0& (\epsilon)^{-1/p}
\end{pmatrix},
\begin{pmatrix}
\epsilon^2& 0&0&0\\
0& \epsilon^{-1}&0 &0\\
0&0&\epsilon&0\\
0&0&0&\epsilon^3
\end{pmatrix}\bigg]
\bigg[\begin{pmatrix}
0 \\ 1
\end{pmatrix},
\begin{pmatrix}
0\\ 
0\\
0 \\
\delta_1
\end{pmatrix}\bigg]=
\bigg[\begin{pmatrix}
0\\1
\end{pmatrix},
\begin{pmatrix}
0\\ 
0\\
0\\
\delta_2
\end{pmatrix}\bigg],
\end{align*}
as required.
\end{proof}

Hence, the corresponding restricted Lie algebra structures are as follows:
\begin{align*}
& K_9^1=\langle x_1,\ldots,x_5\mid [x_1,x_2]=x_3,[x_1,x_3]=x_4, [x_2,x_3]=x_5 \rangle;\\
& K_9^2(\alpha)=\langle x_1,\ldots,x_5\mid [x_1,x_2]=x_3,[x_1,x_3]=x_4, [x_2,x_3]=x_5 ,
\res x_1= \alpha x_5\rangle;\\
& K_9^3=\langle x_1,\ldots,x_5\mid [x_1,x_2]=x_3,[x_1,x_3]=x_4, [x_2,x_3]=x_5 ,
\res x_2=x_5\rangle;\\
& K_9^4=\langle x_1,\ldots,x_5\mid [x_1,x_2]=x_3,[x_1,x_3]=x_4, [x_2,x_3]=x_5 ,
\res x_3=x_5\rangle;\\
& K_9^5(\delta)=\langle x_1,\ldots,x_5\mid [x_1,x_2]=x_3,[x_1,x_3]=x_4, [x_2,x_3]=x_5 ,
\res x_4= \delta x_5\rangle
\end{align*}
where $\alpha ,\delta \in \F^*$. 

\begin{lemma}\label{lemma-K9-2}
We have $K_9^2(\alpha_1) \cong K_9^2(\alpha_2)$ if and only if $\alpha_2 \alpha _1^{-1} \in (\F^*)^2$.
\end{lemma}
\begin{proof}
Let $f=(a_{ij})\in \Aut(L_{5,9})$ and suppose that 
$f: K_9^2(\alpha_1)\to K_9^2(\alpha_2)$ is an isomorphism. 
Then $f(\res x_1)=f(x_1)^{[p]}$ which implies that 
$\alpha_1 a_{11}a_{22}^2 x_5=\alpha_2 a_{11}^p x_5$. Hence, $\alpha_1/\alpha_2=a_{11}^{p-1}a_{22}^{-2}\in (\F^*)^2$. To prove the converse, suppose that $\alpha_1/\alpha_2=\epsilon^2\in (\F^*)^2$. Then the following is an isomorphism from
 $K_9^2(\alpha_1)$ to $K_9^2(\alpha_2)$:
\begin{align*}
\begin{pmatrix}
\epsilon^{2/p}&0&0&0 & 0\\
0& \epsilon^{-1/p}& 0&0& 0\\
0&0&\epsilon^{1/p} &0& 0\\
0&0&0&\epsilon^{3/p} & 0\\
0&0&0&0& 1\\
\end{pmatrix}.
\end{align*}
\end{proof}

The conditions on $\delta_1$ and $\delta_2$ such that  $K_9^5(\delta_1) \cong K_9^5(\delta_2)$ 
reduces to the solutions of an algebraic equation over $\F$ as the following lemma shows. However,  by Lemma \ref{lemma-K9-5} if $\delta_1/\delta_2\in (\F^*)^3$ then $K_9^5(\delta_1) \cong K_9^5(\delta_2)$.
 
\begin{lemma}
We have $K_9^5(\delta_1) \cong K_9^5(\delta_2)$ if and only if 
the equation $\delta_1/\delta_2=x^{2p-1} y^{p-2}$ has a solution  in $\F^*$. 
\end{lemma}
\begin{proof}
Let $f=(a_{ij})\in \Aut(L_{5,9})$ and suppose that 
$f: K_9^5(\delta_1)\to K_9^5(\delta_2)$ is an isomorphism. 
Then $f(\res x_4)=f(x_4)^{[p]}$ which implies that 
$\delta_1 a_{11}a_{22}^2 x_5=\delta_{2} a_{11}^{2p} a_{22}^p x_5$. Hence, $\delta_1/\delta_2=a_{11}^{2p-1} a_{22}^{p-2}$. 
\end{proof} 

\begin{corollary}
If $\delta_1/\delta_2\in \F_p^*$ then  $K_9^5(\delta_1) \cong K_9^5(\delta_2)$.
\end{corollary}

\section{Extensions of ($L, \res x_1=x_4$)}
 First, we find a basis for $Z^2(L,\F)$. Let $(\phi,\omega)=(a\Delta _{12}+b\Delta_{13}+c\Delta_{14}+d\Delta_{23}+e\Delta_{24}+f\Delta_{34},\alpha f_1+\beta f_2+ \gamma f_3+\delta f_4)\in Z^2(L,\F)$. Then we must have $\delta^2\phi(x,y,z) =0$ and $\phi(x,\res y)=0$, for all $x,y,z \in L$. Therefore,
\begin{align*}
0=(\delta^2\phi)(x_1,x_2,x_3)&=\phi([x_1,x_2],x_3)+\phi([x_2,x_3],x_1)+\phi([x_3,x_1],x_2)=\phi(x_2,x_4);\\
0=(\delta^2\phi)(x_1,x_2,x_4)&=\phi([x_1,x_2],x_4)+\phi([x_2,x_4],x_1)+\phi([x_4,x_1],x_2)=\phi(x_3,x_4).
\end{align*}
Thus, we get $e=f=0$.
Also, we have $\phi(x,\res y)=0$. Therefore, $\phi(x,x_4)=0$, for all $x\in L$ and hence $\phi(x_1,x_4)=\phi(x_2,x_4)=\phi(x_3,x_4)=0$ which implies that $c=e=f=0$. Therefore, $Z^2(L,\F)$ has a basis consisting of:
$$ (\Delta_{12},0),(\Delta_{13},0),(\Delta_{23},0),(0,f_1),(0,f_2),(0,f_3),(0,f_4).$$
Next, we find a basis for $B^2(L,\F)$. Let $(\phi,\omega)\in B^2(L,\F)$. Since $B^2(L,\F)\subseteq Z^2(L,\F)$, we have $(\phi,\omega)=(a\Delta _{12}+b\Delta_{13}+c\Delta_{23},\alpha f_1+\beta f_2+ \gamma f_3+\delta f_4)$. Note that there exists a linear map $\psi:L\to \F$  such that $\delta^1\psi(x,y)=\phi(x,y)$ and $\tilde \psi(x)=\omega(x)$, for all $x,y \in L$. So, we have
\begin{align*}
&a=\phi(x_1,x_2)=\delta^1\psi(x_1,x_2)=\psi([x_1,x_2])=\psi(x_3), \text{ and }\\
&b=\phi(x_1,x_3)=\delta^1\psi(x_1,x_3)=\psi([x_1,x_3])=\psi(x_4), \text{ and }\\
&c=\phi(x_2,x_3)=\delta^1\psi(x_2,x_3)=\psi([x_2,x_3])=0.
\end{align*}
Also, we have
\begin{align*}
&\alpha=\omega(x_1)=\tilde \psi(x_1)=\psi(\res x_1)=\psi(x_4), \text{ and }\\
&\beta=\omega(x_2)=\tilde \psi(x_2)=\psi(\res x_2)=0.
\end{align*}
Similarly, we can show that $\gamma=\delta=0$.  Note that $\psi(x_4)=b=\alpha$. Therefore, $(\phi,\omega)=(a\Delta_{12}+b\Delta_{13},b f_1)$ and hence $B^2(L,\F)=\la(\Delta_{12},0),(\Delta_{13},f_1)\ra_{\F}$. Note that 
$$[(\Delta_{13},0)],[(\Delta_{23},0)],[0,f_1],[(0,f_2)],[(0,f_3)],[(0,f_4)]$$
spans $H^2(L,\F)$.
Since $[(\Delta_{13},0)]+[(0,f_1)]=[(\Delta_{13},f_1)]=[0]$, then $[(0,f_1)]$ is an scalar multiple of $[(\Delta_{13},0)]$
in $H^2(L,\F)$. Note that $\dim H^2=\dim Z^2-\dim B^2=5$. Therefore, 
$$[(\Delta_{13},0)],[(\Delta_{23},0)],[(0,f_2)],[(0,f_3)],[(0,f_4)]$$
forms a basis for $H^2(L,\F)$.

Note that the group $\Aut(L)$ in this case consists of invertible matrices of the form
\[\begin{pmatrix} a_{11}  & 0 & 0 & 0 \\
a_{21} & a_{22} & 0 & 0\\
a_{31} & a_{32} & d& 0 \\
a_{41} & a_{42} & a_{11}a_{32}& a_{11}d
\end{pmatrix},\]
where $d= a_{11}a_{22}$ and $a_{11}d=a_{11}^p$. 

Let $[(\phi,\omega)] \in H^2(L,\F)$. Then we have $\phi=a\Delta_{13}+b\Delta_{23}$, for some $a,b\in \F$. Suppose that $A\phi =a'\Delta_{13}+b'\Delta_{23}$ for some $a',b'\in \F$. We determine $a,b'$. Note that
\begin{align*}
&A\phi (x_1,x_3)=\phi (Ax_1,Ax_3)=\phi (a_{11}x_1+a_{21}x_2+a_{31}x_3+a_{41}x_4,dx_3+a_{11}a_{32}x_4)=a_{11}da+a_{21}db; \text{ and }\\
&A\phi (x_2,x_3)=\phi (Ax_2,Ax_3)=\phi (a_{22}x_2+a_{32}x_3+a_{42}x_4,dx_3+a_{11}a_{32}x_4)=a_{22}db.
\end{align*}
In the matrix form we can write this as
\begin{align}\label{a}
\begin{pmatrix}
a'\\b'
\end{pmatrix}=
\begin{pmatrix}
a_{11}d&a_{21}d\\
0&a_{22}d
\end{pmatrix}
\begin{pmatrix}
a\\b
\end{pmatrix}.
 \end{align} 
The orbit with representative 
$\begin{pmatrix}
0\\1
\end{pmatrix}$  of this action gives us $L_{5,9}$.

Also, we have $\omega=\beta f_2+\gamma f_3+\delta f_4$, for some $ \beta, \gamma, \delta\in \F$.
 Suppose that $A\omega = \beta' f_2+\gamma' f_3+\delta' f_4$, for some $\beta', \gamma', \delta'\in \F$. We have
\begin{align*}
A\omega (x_2)=& a_{22}^p \beta +a_{32}^p \gamma + a_{42}^p \delta;\\
A\omega (x_3)=& d^p \gamma + a_{11}^pa_{32}^p \delta;\\
A\omega(x_4)=&a_{11}^pd^p \delta.
\end{align*}
In the matrix form we can write this as
\begin{align}\label{b}
\begin{pmatrix}
\beta'\\
\gamma'\\
\delta'
\end{pmatrix}
=
\begin{pmatrix}
a_{22}^p& a_{32}^p&a_{42}^p\\
0&d^p&a_{11}^pa_{32}^p\\
0&0&a_{11}^pd^p
\end{pmatrix}
\begin{pmatrix}
\beta \\
\gamma \\
\delta
\end{pmatrix}.
\end{align}

Thus, we can write Equations \eqref{a} and \eqref{b} together as follows:
\begin{align*}
\bigg[\begin{pmatrix}
a_{11}d&a_{21}d\\
0&a_{22}d
\end{pmatrix} , 
\begin{pmatrix}
a_{22}^p& a_{32}^p&a_{42}^p\\
0&d^p&a_{11}^pa_{32}^p\\
0&0&a_{11}^pd^p
\end{pmatrix}\bigg]
\bigg[\begin{pmatrix}
a\\b
\end{pmatrix},
\begin{pmatrix} 
\beta \\
\gamma \\
\delta
\end{pmatrix}\bigg]=
\bigg[\begin{pmatrix}
a'\\b'
\end{pmatrix},
\begin{pmatrix} 
\beta' \\
\gamma' \\
\delta'
\end{pmatrix}\bigg].
\end{align*}
Now we find the representatives of the orbits of the action of $\Aut_p (L)$ on the set of $\omega$'s such that the orbit represented by
$\begin{pmatrix}
0\\1
\end{pmatrix}$ is preserved under the action of $\Aut_p (L)$ on the set of $\phi$'s.
Note that we need to have $A\omega(x_1)=0$ which implies that $a_{21}^p\beta+a_{31}^p\gamma+a_{41}^p\delta=0$.

Let $\nu = \begin{pmatrix}
\beta \\
\gamma \\
\delta
\end{pmatrix} \in \F^3$. If $\nu= \begin{pmatrix}
0\\ 
0 \\
0 
\end{pmatrix}$, then $\{\nu \}$ is clearly an $\Aut_p (L)$-orbit. Suppose that $\delta \neq 0$. Then 
\begin{align*}
&\bigg[
\begin{pmatrix}
1&0\\
0&1
\end{pmatrix},
\begin{pmatrix}
 1& 0&-\beta /\delta\\
0&1&0\\
0&0&1
\end{pmatrix}\bigg]
\bigg[
\begin{pmatrix}
0\\1
\end{pmatrix},
\begin{pmatrix}
\beta \\
\gamma \\
\delta
\end{pmatrix}\bigg]=
\bigg[\begin{pmatrix}
0\\1
\end{pmatrix},
\begin{pmatrix}
0\\
\gamma\\
\delta
\end{pmatrix}\bigg], \text {and }\\
&\bigg[\begin{pmatrix}
1&0\\
0&1
\end{pmatrix},
\begin{pmatrix}
 1& -\gamma /\delta&\gamma^2 /\delta^2\\
0&1&-\gamma /\delta\\
0&0&1
\end{pmatrix}\bigg]
\bigg[\begin{pmatrix}
0\\1
\end{pmatrix},
\begin{pmatrix}
0\\
\gamma \\
\delta
\end{pmatrix}\bigg]=
\bigg[\begin{pmatrix}
0\\1
\end{pmatrix},
\begin{pmatrix}
0\\
0\\
\delta
\end{pmatrix}\bigg].
\end{align*}
Next, if $\delta =0$, but $\gamma \neq 0$, then
\begin{align*}
&\bigg[\begin{pmatrix}
1&0\\
0&1
\end{pmatrix},
\begin{pmatrix}
 1& -\beta /\gamma&0\\
0&1&-\beta /\gamma\\
0&0&1
\end{pmatrix}\bigg]
\bigg[\begin{pmatrix}
0\\1
\end{pmatrix},
\begin{pmatrix}
\beta\\
\gamma \\
0
\end{pmatrix}\bigg]=
\bigg[\begin{pmatrix}
0\\1
\end{pmatrix},
\begin{pmatrix}
0\\
\gamma\\
0
\end{pmatrix}\bigg].
\end{align*}

Finally, if $\gamma=\delta =0$, but $\beta \neq 0$, then we have 
$
\begin{pmatrix}
\beta\\
0 \\
0
\end{pmatrix}.
$
Thus the following elements are $\Aut_p (L)$-orbit representatives:
 
$\begin{pmatrix}
0\\ 
0\\
0
\end{pmatrix}$, $\begin{pmatrix} 
\beta\\
0\\
0
\end{pmatrix}$, $\begin{pmatrix}
0\\
\gamma\\
0
\end{pmatrix}$, $\begin{pmatrix}
0\\
0\\
\delta
\end{pmatrix}$.

\begin{lemma}\label{lemma-K9-7}
The vectors
$\begin{pmatrix}
\beta_1\\
0\\
0
\end{pmatrix}$ and
$\begin{pmatrix}
\beta_2\\
0\\
0
\end{pmatrix}$
are in the same $\Aut_p (L)$-orbit if and only if $\beta_2/\beta_1=\epsilon,$  
where $\epsilon^{2p-3}=1$. 
\end{lemma}
\begin{proof}
Let 
$\theta_1=\bigg[\begin{pmatrix}
0\\1
\end{pmatrix},
\begin{pmatrix}
\beta_1 \\
0 \\
0
\end{pmatrix}
\bigg]
$ 
and 
$\theta_2=\bigg[\begin{pmatrix}
0\\1
\end{pmatrix},
\begin{pmatrix}
\beta_2 \\
0 \\
0
\end{pmatrix}
\bigg]
$.
Then $\theta_1$ and $\theta_2$ are in the same $\Aut_p (L)$-orbit if and only if 
there exists $A=(a_{ij})\in \Aut_p (L)$ such that 
$
A\cdot 
\theta_1
=
\theta_2
$.
Equivalently, $\theta_1$ and $\theta_2$ are in the same $\Aut_p (L)$-orbit if and only if 
there exists $A=(a_{ij})\in \Aut_p (L)$ such that 
 $a_{11}d =a_{11}^p$ and
\begin{align*}
\bigg[
\begin{pmatrix}
a_{11}d&a_{21}d\\
0&a_{22}d
\end{pmatrix},
\begin{pmatrix}
a_{22}^p& a_{32}^p&a_{42}^p\\
0&d^p&a_{11}^pa_{32}^p\\
0&0&a_{11}^pd^p
\end{pmatrix}
\bigg]
\bigg[\begin{pmatrix}
0\\1
\end{pmatrix},
\begin{pmatrix}
\beta_1 \\
0 \\
0
\end{pmatrix}
\bigg]
=
\bigg[\begin{pmatrix}
0\\1
\end{pmatrix},
\begin{pmatrix}
 \beta_2\\
0\\
0
\end{pmatrix}
\bigg].
\end{align*}
Hence, $\theta_1$ and $\theta_2$ are in the same $\Aut_p (L)$-orbit if and only if 
\begin{align}
 \beta_2=a_{22}^p\beta_1 \label{59-b1}\\
 da_{22}=1 \Leftrightarrow a_{11}= a_{22}^{-2} \label{59-b2} \\
 a_{11}d =a_{11}^p\label{59-b3}\\
 a_{21}=0. 
\end{align} 

First, suppose that $\theta_1$ and $\theta_2$ are in the same $\Aut_p (L)$-orbit and set 
$\epsilon=a_{22}^{p}$. 
Then, by Equations \eqref{59-b1} and  \eqref{59-b3}, we have 
$\beta_2/\beta_1=a_{22}^p=\epsilon \in \F^*$.
Furthermore, by Equations \eqref{59-b2} and  \eqref{59-b3}, $a_{22}^{2p-3}=1$ and so $\epsilon^{2p-3}=1$.
Conversely, suppose that 
$\beta_2/\beta_1= \epsilon\in \F^*$, where $\epsilon^{2p-3}=1$. We let 
 $a_{22}=\epsilon^{1/p}, a_{11}= a_{22}^{-2}$, and $a_{21}=0$.
Then we have 
\begin{align*}
a_{11}d =a_{11}^p\Leftrightarrow a_{11}^2 a_{22}=a_{11}^p \Leftrightarrow  a_{22}^{-3}=a_{22}^{-2p}
\Leftrightarrow a_{22}^{2p-3}=1 \Leftrightarrow \epsilon^{2p-3}=1.
\end{align*}
It remains to verify that Equation \eqref{59-b1} is satisfied. We have:
\begin{align*}
\beta_2/\beta_1= \epsilon = a_{22}^{p},
\end{align*}
as required. 
\end{proof}

\begin{lemma}\label{lemma-K9-8}
The vectors
$\begin{pmatrix}
0\\
\gamma_1\\
0
\end{pmatrix}$ and
$\begin{pmatrix}
0\\
\gamma_2\\
0
\end{pmatrix}$
are in the same $\Aut_p (L)$-orbit if and only if 
$\gamma_2/\gamma_1=\epsilon^{2},$  where $\epsilon^{2p-3}=1$. 
\end{lemma}
\begin{proof}
Let 
$\theta_1=\bigg[\begin{pmatrix}
0\\1
\end{pmatrix},
\begin{pmatrix}
0\\
\gamma_1\\
0
\end{pmatrix}
\bigg]
$ 
and 
$\theta_2=\bigg[\begin{pmatrix}
0\\1
\end{pmatrix},
\begin{pmatrix}
0\\
\gamma_2\\
0
\end{pmatrix}
\bigg]
$.
Then $\theta_1$ and $\theta_2$ are in the same $\Aut_p (L)$-orbit if and only if 
there exists $A=(a_{ij})\in \Aut_p (L)$ such that 
$
A\cdot 
\theta_1
=
\theta_2
$.
Equivalently, $\theta_1$ and $\theta_2$ are in the same $\Aut_p (L)$-orbit if and only if 
there exists $A=(a_{ij})\in \Aut_p (L)$ such that 
\begin{align*}
\bigg[
\begin{pmatrix}
a_{11}d&a_{21}d\\
0&a_{22}d
\end{pmatrix},
\begin{pmatrix}
a_{22}^p& a_{32}^p&a_{42}^p\\
0&d^p&a_{11}^pa_{32}^p\\
0&0&a_{11}^pd^p
\end{pmatrix}
\bigg]
\bigg[\begin{pmatrix}
0\\1
\end{pmatrix},
\begin{pmatrix}
0\\
\gamma_1\\
0
\end{pmatrix}
\bigg]
=
\bigg[\begin{pmatrix}
0\\1
\end{pmatrix},
\begin{pmatrix}
0\\
\gamma_2\\
0
\end{pmatrix}
\bigg].
\end{align*}
Hence, $\theta_1$ and $\theta_2$ are in the same $\Aut_p (L)$-orbit if and only if 
\begin{align}
 \gamma_2=d^p\gamma_1 \label{59-g1}\\
 da_{22}=1 \Leftrightarrow a_{11}= a_{22}^{-2} \label{59-g2} \\
 a_{11}d =a_{11}^p\label{59-g3}\\
 a_{21}=a_{32}=0.
\end{align} 
First, suppose that $\theta_1$ and $\theta_2$ are in the same $\Aut_p (L)$-orbit and set 
$\epsilon=a_{11}^{(p-1)/2}$. 
Then, by Equations \eqref{59-g1} and  \eqref{59-g3}, we have 
$ (\gamma_2/\gamma_1)^{1/p}= d=a_{11}^{p-1}=\epsilon^2\in (\F^*)^2$. 
Furthermore, by Equations \eqref{59-g2} and  \eqref{59-g3}, $a_{22}^{2p-3}=1$ and so $\epsilon^{2p-3}=a_{11}^{(p-1)(2p-3)/2}=a_{22}^{(1-p)(2p-3)}=1$.
Conversely, suppose that 
$\gamma_2/\gamma_1= \epsilon^2\in (\F^*)^2$. We let 
 $a_{22}=\epsilon^{-2/p}, a_{11}= a_{22}^{-2}$, and $a_{21}=a_{32}=0$.
Then we have 
\begin{align*}
a_{11}d =a_{11}^p\Leftrightarrow a_{11}^2 a_{22}=a_{11}^p \Leftrightarrow  a_{22}^{-3}=a_{22}^{-2p}
\Leftrightarrow a_{22}^{2p-3}=1 \Leftrightarrow \epsilon^{2p-3}=1.
\end{align*}
It remains to verify that Equation \eqref{59-g1} is satisfied. We have:
\begin{align*}
\gamma_2/\gamma_1= \epsilon^{2} = a_{22}^{-p} =(a_{11}a_{22})^{p}=d^p,
\end{align*}
as required. 
\end{proof}

\begin{lemma}\label{lemma-K9-9}
The vectors
$\begin{pmatrix}
0\\
0\\
\delta_1
\end{pmatrix}$ and
$\begin{pmatrix}
0\\
0\\
\delta_2
\end{pmatrix}$
are in the same $\Aut_p (L)$-orbit if and only if $\delta_2/\delta_1=\epsilon^{3},$  
where $\epsilon^{2p-3}=1$. 
\end{lemma}
\begin{proof}
Let 
$\theta_1=\bigg[\begin{pmatrix}
0\\1
\end{pmatrix},
\begin{pmatrix}
0 \\
0 \\
\delta_1
\end{pmatrix}
\bigg]
$ 
and 
$\theta_2=\bigg[\begin{pmatrix}
0\\1
\end{pmatrix},
\begin{pmatrix}
0 \\
0 \\
\delta_2
\end{pmatrix}
\bigg]
$.
Then $\theta_1$ and $\theta_2$ are in the same $\Aut_p (L)$-orbit if and only if 
there exists $A=(a_{ij})\in \Aut_p (L)$ such that 
$
A\cdot 
\theta_1
=
\theta_2
$.
Equivalently, $\theta_1$ and $\theta_2$ are in the same $\Aut_p (L)$-orbit if and only if 
there exists $A=(a_{ij})\in \Aut_p (L)$ such that 
\begin{align*}
\bigg[
\begin{pmatrix}
a_{11}d&a_{21}d\\
0&a_{22}d
\end{pmatrix},
\begin{pmatrix}
a_{22}^p& a_{32}^p&a_{42}^p\\
0&d^p&a_{11}^pa_{32}^p\\
0&0&a_{11}^pd^p
\end{pmatrix}
\bigg]
\bigg[\begin{pmatrix}
0\\1
\end{pmatrix},
\begin{pmatrix}
0 \\
0 \\
\delta_1
\end{pmatrix}
\bigg]
=
\bigg[\begin{pmatrix}
0\\1
\end{pmatrix},
\begin{pmatrix}
 0\\
0\\
\delta_2
\end{pmatrix}
\bigg].
\end{align*}
Hence, $\theta_1$ and $\theta_2$ are in the same $\Aut_p (L)$-orbit if and only if 
\begin{align}
 \delta_2=a_{11}^pd^p\delta_1 \label{59-d1}\\
 da_{22}=1 \Leftrightarrow a_{11}= a_{22}^{-2} \label{59-d2} \\
 a_{11}d =a_{11}^p\label{59-d3}\\
 a_{21}=a_{32}=a_{42}=0.
\end{align} 
First, suppose that $\theta_1$ and $\theta_2$ are in the same $\Aut_p (L)$-orbit. 
Then, using Equation \eqref{59-d2}, we get that
$ a_{11}^p= a_{22}^{-2p}$ which using \eqref{59-d2} and \eqref{59-d3}  implies that 
$a_{22}^{2p-3}=1$. Now, by \eqref{59-d1} and \eqref{59-d2} , we have 
\begin{align*}
\delta_2/\delta_1=(a_{11}d)^{p}= (a_{11}^2a_{22})^{p}=a_{22}^{-3p}.
\end{align*}

To prove the sufficiency, we let $a_{22}=\epsilon^{-1/p}, a_{11}= a_{22}^{-2}$, and $a_{21}=a_{32}=a_{42}=0$.
Then we have 
\begin{align*}
a_{11}d =a_{11}^p\Leftrightarrow a_{11}^2 a_{22}=a_{11}^p \Leftrightarrow  a_{22}^{-3}=a_{22}^{-2p}
\Leftrightarrow a_{22}^{2p-3}=1 \Leftrightarrow \epsilon^{2p-3}=1.
\end{align*}
It remains to verify that Equation \eqref{59-d1} is satisfied. We have:
\begin{align*}
\delta_2/\delta_1= \epsilon^{3} = a_{22}^{-3p} =(a_{11}d)^{p},
\end{align*}
as required. 
\end{proof}

Hence, the corresponding restricted Lie algebra structurs are as follows:
\begin{align*}
& K_9^6=\langle x_1,\ldots,x_5\mid [x_1,x_2]=x_3,[x_1,x_3]=x_4, [x_2,x_3]=x_5 ,
\res x_1=x_4 \rangle;\\
& K_9^7(\beta)=\langle x_1,\ldots,x_5\mid [x_1,x_2]=x_3,[x_1,x_3]=x_4, [x_2,x_3]=x_5 ,
\res x_1=x_4, \res x_2=\beta x_5 \rangle; \\
&K_9^8(\gamma)=\langle x_1,\ldots,x_5\mid [x_1,x_2]=x_3,[x_1,x_3]=x_4, [x_2,x_3]=x_5 ,
\res x_1=x_4, \res x_3=\gamma x_5 \rangle;\\
&K_9^9(\delta)=\langle x_1,\ldots,x_5\mid [x_1,x_2]=x_3,[x_1,x_3]=x_4, [x_2,x_3]=x_5 ,
\res x_1=x_4, \res x_4=\delta x_5 \rangle
\end{align*}
where $\beta, \gamma, \delta \in \F^*$.

Lemma \ref{lemma-K9-7} provides us with some sufficient conditions as to when $K_9^7(\beta_1) \cong K_9^7(\beta_2)$. Over the prime field $\F_p$ we can say the following:

\begin{lemma}
Let $\beta_1, \beta_2\in \F_p^*$. Then    $K_9^7(\beta_1) \cong K_9^7(\beta_2)$ over $\F_p$  if and only if $\beta_1=\beta_2$.
\end{lemma}
\begin{proof}
Let $f=(a_{ij})\in \Aut(L_{5,9})$, where $a_{ij}\in \F_p$.  Then 
$f: K_9^7(\beta_1)\to K_9^7(\beta_2)$ is an isomorphism
if and only if 
\begin{align*}
u=a_{11}^{p-1}, \quad \beta_1/\beta_2=a_{22}^{p-1}u^{-1}\in (\F^*)^2. 
\end{align*}
Thus, $K_9^7(\beta_1) \cong K_9^7(\beta_2)$  if and only if $\beta_1=\beta_2$.

\end{proof}

Lemma \ref{lemma-K9-8} provides us with some sufficient conditions as to when 
$K_9^8(\gamma_1) \cong K_9^8(\gamma_2)$ and a characterization is given below:

\begin{lemma}
We have $K_9^8(\gamma_1) \cong K_9^8(\gamma_2)$ if and only if  $\gamma_1/\gamma_2=\epsilon^{p^2-3p+3}$, for some $\epsilon\in \F^*$.
\end{lemma}
\begin{proof}
Let $f=(a_{ij})\in \Aut(L_{5,9})$. Then 
$f: K_9^8(\gamma_1)\to K_9^8(\gamma_2)$ is an isomorphism 
 if and only if 
\begin{align*}
a_{12}=a_{21}=0, \quad u=a_{11}^{p-1}, \quad \gamma_1/\gamma_2 = u^{p-1}a_{22}^{-1}, 
\end{align*}
which is equivallent to saying that $\gamma_1/\gamma_2=a_{11}^{p^2-3p+3}$.

\end{proof}

Lemma \ref{lemma-K9-9} provides us with some sufficient conditions as to when 
$K_9^9(\delta_1) \cong K_9^9(\delta_2)$ and a characterization is given below:

\begin{lemma}
We have $K_9^9(\delta_1) \cong K_9^9(\delta_2)$ if and only if  $\delta_1/\delta_2=\epsilon^{p^2-2p+3}$, for some $\epsilon\in \F^*$.
\end{lemma}
\begin{proof}
Let $f=(a_{ij})\in \Aut(L_{5,9})$. Then 
$f: K_9^9(\delta_1)\to K_9^9(\delta_2)$ is an isomorphism 
 if and only if 
\begin{align*}
a_{12}=a_{21}=0, \quad u=a_{11}^{p-1}, \quad \delta_1/\delta_2 = \frac{(a_{11}u)^p}{a_{22}u},
\end{align*}
which is equivalent to saying that $\delta_1/\delta_2=a_{11}^{p^2-2p+3}$.

\end{proof}

\section{Extensions of ($L, \res x_2 =\xi x_4$)}
 First, we find a basis for $Z^2(L,\F)$. Let $(\phi,\omega)=(a\Delta _{12}+b\Delta_{13}+c\Delta_{14}+d\Delta_{23}+e\Delta_{24}+f\Delta_{34},\alpha f_1+\beta f_2+ \gamma f_3+\delta f_4)\in Z^2(L,\F)$. Then we must have $\delta^2\phi(x,y,z) =0$ and $\phi(x,\res y)=0$, for all $x,y,z \in L$. Therefore,
\begin{align*}
0=(\delta^2\phi)(x_1,x_2,x_3)&=\phi([x_1,x_2],x_3)+\phi([x_2,x_3],x_1)+\phi([x_3,x_1],x_2)=\phi(x_2,x_4);\\
0=(\delta^2\phi)(x_1,x_2,x_4)&=\phi([x_1,x_2],x_4)+\phi([x_2,x_4],x_1)+\phi([x_4,x_1],x_2)=\phi(x_3,x_4).
\end{align*}
Thus, we get $e=f=0$.
Also, we have $\phi(x,\res y)=0$. Therefore, $\phi(x,x_4)=0$, for all $x\in L$ and hence $\phi(x_1,x_4)=\phi(x_2,x_4)=\phi(x_3,x_4)=0$ which implies that $c=e=f=0$. Therefore, $Z^2(L,\F)$ has a basis consisting of:
$$ (\Delta_{12},0),(\Delta_{13},0),(\Delta_{23},0),(0,f_1),(0,f_2),(0,f_3),(0,f_4).$$
Next, we find a basis for $B^2(L,\F)$. Let $(\phi,\omega)\in B^2(L,\F)$. Since $B^2(L,\F)\subseteq Z^2(L,\F)$, we have $(\phi,\omega)=(a\Delta _{12}+b\Delta_{13}+c\Delta_{23},\alpha f_1+\beta f_2+ \gamma f_3+\delta f_4)$. Note that there exists a linear map $\psi:L\to \F$  such that $\delta^1\psi(x,y)=\phi(x,y)$ and $\tilde \psi(x)=\omega(x)$, for all $x,y \in L$. So, we have
\begin{align*}
&a=\phi(x_1,x_2)=\delta^1\psi(x_1,x_2)=\psi([x_1,x_2])=\psi(x_3), \text{ and }\\
&b=\phi(x_1,x_3)=\delta^1\psi(x_1,x_3)=\psi([x_1,x_3])=\psi(x_4), \text{ and }\\
&c=\phi(x_1,x_4)=\delta^1\psi(x_1,x_4)=\psi([x_1,x_4])=0.
\end{align*}
Also, we have
\begin{align*}
&\beta=\omega(x_2)=\tilde \psi(x_2)=\psi(\res x_2)=\xi\psi(x_4), \text{ and }\\
&\alpha=\omega(x_1)=\tilde \psi(x_1)=\psi(\res x_1)=0.
\end{align*}
Similarly, we can show that $\gamma=\delta=0$.  Note that $\psi(x_4)=b=\beta\xi^{-1}$. Therefore, $(\phi,\omega)=(a\Delta_{12}+b\Delta_{13},b\xi f_2)$ and hence $B^2(L,\F)=\la(\Delta_{12},0),(\Delta_{13},\xi f_2)\ra_{\F}$. Note that 
$$[(\Delta_{13},0)],[(\Delta_{23},0)],[0,f_1],[(0,f_2)],[(0,f_3)],[(0,f_4)]$$
spans $H^2(L,\F)$.
Since $[(\Delta_{13},0)]+\xi [(0,f_2)]=[(\Delta_{13},\xi f_2)]=[0]$, then $[(0,f_2)]$ is an scalar multiple of $[(\Delta_{13},0)]$
in $H^2(L,\F)$. Note that $\dim H^2=\dim Z^2-\dim B^2=5$. Therefore, 
$$[(\Delta_{13},0)],[(\Delta_{23},0)],[(0,f_1)],[(0,f_3)],[(0,f_4)]$$
forms a basis for $H^2(L,\F)$.

The group $\Aut(L)$ consists of invertible matrices of the form
\[\begin{pmatrix} a_{11}  & 0 & 0 & 0 \\
a_{21} & a_{22} & 0 & 0\\
a_{31} & a_{32} & d& 0 \\
a_{41} & a_{42} & a_{11}a_{32}& a_{11}d
\end{pmatrix},\]
where $d= a_{11}a_{22}$, $a_{21}=0$, and $a_{11}d=a_{22}^p$.

Let $[(\phi,\omega)] \in H^2(L,\F)$. Then we have $\phi=a\Delta_{13}+b\Delta_{23}$, for some $a,b\in \F$. Suppose that $A\phi =a'\Delta_{13}+b'\Delta_{23}$ for some $a',b'\in \F$.
Then we have
\begin{align}\label{c}
\begin{pmatrix}
a'\\b'
\end{pmatrix}=
\begin{pmatrix}
a_{11}d&a_{21}d\\
0&a_{22}d
\end{pmatrix}
\begin{pmatrix}
a\\b
\end{pmatrix}.
 \end{align} 
The orbit with representative 
$\begin{pmatrix}
0\\1
\end{pmatrix}$  of this action gives us $L_{5,9}$.

Also, we have $\omega=\alpha f_1+\gamma f_3+\delta f_4 $, for some $\alpha, \gamma, \delta \in \F$. Suppose that $A\omega =\alpha' f_1+\gamma' f_3+\delta' f_4$, for some $\alpha', \gamma',\delta' \in \F$. We have
\begin{align*}
&A\omega (x_1)=a_{11}^p \alpha + a_{31}^p \gamma + a_{41}^p \delta;\\ 
&A\omega (x_3)= d^p \gamma + a_{11}^pa_{32}^p \delta;\\
&A\omega(x_4)=a_{11}^pd^p \delta.
\end{align*}
In the matrix form we can write this as
\begin{align}\label{d}
\begin{pmatrix}
\alpha'\\
\gamma'\\
\delta'
\end{pmatrix}=
\begin{pmatrix}
a_{11}^p& a_{31}^p&a_{41}^p\\
0&d^p&a_{11}^pa_{32}^p\\
0&0&a_{11}^pd^p
\end{pmatrix}
\begin{pmatrix} 
\alpha \\
\gamma \\
\delta
\end{pmatrix}.
\end{align}

Thus, we can write Equations \eqref{c} and \eqref{d} together as follows:
\begin{align}
\bigg[\begin{pmatrix}
a_{11}d&a_{21}d\\
0&a_{22}d
\end{pmatrix}, 
\begin{pmatrix}
a_{11}^p& a_{31}^p&a_{41}^p\\
0&d^p&a_{11}^pa_{32}^p\\
0&0&a_{11}^pd^p
\end{pmatrix}\bigg]
\bigg[\begin{pmatrix}
a\\b
\end{pmatrix},
\begin{pmatrix} 
\alpha \\
\gamma \\
\delta
\end{pmatrix}\bigg]=
\bigg[\begin{pmatrix}
a'\\b'
\end{pmatrix},
\begin{pmatrix} 
\alpha' \\
\gamma' \\
\delta'
\end{pmatrix}\bigg]\label{59-action}
\end{align}
Now we find the representatives of the orbits of the action of $\Aut(L)$ on the set of $\omega$'s such that the orbit represented by 
$\begin{pmatrix}
0\\1
\end{pmatrix}$ is preserved under the action of $\Aut(L)$ on the set of $\phi$'s. Note that we need to have 
$A\omega(x_2)=0$ which implies that $a_{32}^p\gamma+a_{42}^p\delta=0$.

Let $\nu = \begin{pmatrix}
\alpha \\
\gamma \\
\delta
\end{pmatrix} \in \F^3$. If $\nu= \begin{pmatrix}
0\\ 
0 \\
0 
\end{pmatrix}$, then $\{\nu \}$ is clearly an $\Aut(L)$-orbit. Suppose that $\delta \neq 0$. Then
\begin{align*}
& \bigg[\begin{pmatrix}
1&0\\
0&1
\end{pmatrix},
\begin{pmatrix}
 1& 0&-\alpha /\delta\\
0&1&0\\
0&0&1
\end{pmatrix}\bigg]
\bigg[\begin{pmatrix}
0\\1
\end{pmatrix},
\begin{pmatrix}
\alpha \\
\gamma \\
\delta
\end{pmatrix}\bigg]=
\bigg[\begin{pmatrix}
0\\1
\end{pmatrix},
\begin{pmatrix}
0\\
\gamma\\
\delta
\end{pmatrix}\bigg], \text{ and }\\
&\bigg[\begin{pmatrix}
1&0\\
0&1
\end{pmatrix},
\begin{pmatrix}
 1& 0&0\\
0&1&-\gamma /\delta\\
0&0&1
\end{pmatrix}\bigg]
\bigg[\begin{pmatrix}
0\\1
\end{pmatrix},
\begin{pmatrix}
0\\
\gamma \\
\delta
\end{pmatrix}\bigg]=
\bigg[\begin{pmatrix}
0\\1
\end{pmatrix},
\begin{pmatrix}
0\\
0\\
\delta
\end{pmatrix}\bigg].
\end{align*}

Next, if $\delta =0$, but $\gamma \neq 0$, then
\begin{align*}
&\bigg[\begin{pmatrix}
1&0\\
0&1
\end{pmatrix},
\begin{pmatrix}
 1& -\alpha /\gamma&0\\
0&1&0\\
0&0&1
\end{pmatrix}\bigg]
\bigg[\begin{pmatrix}
0\\1
\end{pmatrix},
\begin{pmatrix}
\alpha\\
\gamma \\
0
\end{pmatrix}\bigg]=
\bigg[\begin{pmatrix}
0\\1
\end{pmatrix},
\begin{pmatrix}
0\\
\gamma\\
0
\end{pmatrix}\bigg]\\
&\bigg[\begin{pmatrix}
1&0\\
0&1
\end{pmatrix},
\begin{pmatrix}
 1& -\alpha /\gamma&0\\
0&1&0\\
0&0&1
\end{pmatrix}\bigg]
\bigg[\begin{pmatrix}
0\\1
\end{pmatrix},
\begin{pmatrix}
\alpha\\
\gamma \\
0
\end{pmatrix}\bigg]=
\bigg[\begin{pmatrix}
0\\1
\end{pmatrix},
\begin{pmatrix}
0\\
\gamma\\
0
\end{pmatrix}\bigg]
\end{align*}

Finally, if $\gamma=\delta =0$, but $\alpha \neq 0$, then we get
$\begin{pmatrix}
\alpha\\
0\\
0
\end{pmatrix}.$\\
Thus the following elements are $\Aut(L)$-orbit representatives: 
\begin{align*}
\begin{pmatrix}
0\\ 
0\\
0
\end{pmatrix}, 
\begin{pmatrix} 
\alpha\\
0\\
0
\end{pmatrix},
\begin{pmatrix}
0\\
\gamma\\
0
\end{pmatrix},
\begin{pmatrix}
0\\
0\\
\delta
\end{pmatrix}.
\end{align*}

Therefore, the corresponding restricted Lie algebra structures are as follows:
\begin{align*}
&K_9^{10}(\xi)=\langle x_1,\ldots,x_5\mid [x_1,x_2]=x_3,[x_1,x_3]=x_4, [x_2,x_3]=x_5 ,
 \res x_2=\xi x_4 \rangle;\\
&K_9^{11}(\xi, \alpha)=\langle x_1,\ldots,x_5\mid [x_1,x_2]=x_3,[x_1,x_3]=x_4, [x_2,x_3]=x_5 ,
\res x_1=\alpha x_5, \res x_2=\xi x_4\rangle;\\
&K_9^{12}(\xi, \gamma)=\langle x_1,\ldots,x_5\mid [x_1,x_2]=x_3,[x_1,x_3]=x_4, [x_2,x_3]=x_5 ,
\res x_2=\xi x_4, \res x_3=\gamma x_5\rangle;\\
&K_9^{13}(\xi, \delta)=\langle x_1,\ldots,x_5\mid [x_1,x_2]=x_3,[x_1,x_3]=x_4, [x_2,x_3]=x_5 ,
\res x_2=\xi  x_4, \res x_4=\delta x_5\rangle,
\end{align*}
where $\alpha, \gamma, \delta, \xi \in \F^*$.

\begin{lemma}\label{solutions-Fp}
Let $a, b\in \F_p$ and suppose  $a\neq 0$. Then the equation $x^2-ay^2=b$ has a solution in  $\F_p$.
\end{lemma}
\begin{proof}
Consider the two sets $A=\{x^2 \mid x\in \F_p\}$ and $B=\{ ay^2+b \mid y\in \F_p\}$. Since $a\neq 0$, both sets have $\frac{p+1}{2}$ elements. Indeed, let $\{a_1, \ldots, a_{p-1}\}$ be the nonzero elements of $F_p$. Then $a_i^2\equiv (p-a_i)^2$
 $\mod p$, for every$1\leq i\leq p-1$. Therefore, all quadratic residues must be among
$$a_1^2, \ldots, a_{\frac{p-1}{2}}^2.$$ Now, we prove that the above elements are distinct. For this, suppose that $a_i^2\equiv a_j^2$ $\mod p$, for $1\leq i,j \leq \frac{p-1}{2}$. Then, we have $(a_i-a_j)(a_i+a_j)\equiv 0$ $\mod p$, which implies that either $a_i-a_j\equiv 0$ $\mod p$, or $a_i+a_j\equiv 0$ $\mod p$. Note that if $a_i+a_j\equiv 0$ $\mod p$, then $p-a_i\equiv a_i\equiv -a_j$ $\mod p$, and hence $a_i-a_j\equiv p\equiv 0$ $\mod p$, which lead us to $a_i\equiv a_j\equiv 0$ $\mod p$, a contradiction. Therefore, $a_i-a_j\equiv 0$ $\mod p$, and so $a_i\equiv a_j$ $\mod p$. So, the number of elements of $A$ is $\frac{p+1}{2}$. Note that we can easily define a bijection from set $A$ to set $B$ and hence the number of elements of $B$ is also $\frac{p+1}{2}$. So, the sets $A$ and $B$ must overlap and hence a solution exists.
\end{proof} 

\begin{lemma}\label{lemma-K9-11}
Let $\xi_1, \gamma_1, \xi_2, \gamma_2\in \F_p$. Then  $K_9^{11}(\xi_1, \alpha_1)$ and $K_9^{11}(\xi_2, \alpha_2)$ are isomorphic over $\F_p$
if and only if  $\frac{\xi_2}{\xi_1}\frac{\alpha_1}{\alpha_2}\in (\F_p^*)^2$
\end{lemma}
\begin{proof}
Let $f=(a_{ij})\in \Aut(L_{5,9})$, where $a_{ij}\in\F_p$. Then 
$f: K_9^{11}(\delta_1)\to K_9^{11}(\delta_2)$ is an isomorphism 
 if and only if 
\begin{align}
&\quad a_{11}u\xi_1=\xi_2a_{22};\label{1}\\
&\quad a_{12}u\alpha_1=\xi_2a_{21};\label{2}\\
&\quad a_{21}u\xi_1=\alpha_2a_{12};\label{3}\\
&\quad a_{22}u\alpha_1=\alpha_2a_{11}.\label{4}
\end{align}
Hence, if $K_9^{11}(\xi_1, \alpha_1)\cong K_9^{11}(\xi_2, \alpha_2)$ then 
$\frac{\xi_2}{\xi_1}\frac{\alpha_1}{\alpha_2}=(\frac{a_{11}}{a_{22}})^2\in (\F_p^*)^2$. To prove the converse, suppose that $\frac{\xi_2}{\xi_1}\frac{\alpha_1}{\alpha_2}=\epsilon^2$, for some 
$\epsilon\in \F_p^*$.
Note that by Lemma \ref{solutions-Fp},  the equation 
\begin{align}
x^2-\frac{\xi_1}{\alpha_1}y^2=\frac{\alpha_2}{\alpha_1}
\end{align}
has a solution in $\F_p$. 
Now, we  take $a_{11}=\epsilon x$, $a_{21}=y$, and
set \begin{align}
a_{22}=\frac{a_{11}}{\epsilon}, \quad a_{12}=\frac{\xi_1}{\alpha_1}\epsilon a_{21}.\label{1-2-3-4}
 \end{align}
Then, $u=a_{11}a_{22}-a_{12}a_{21}=\epsilon\frac{\alpha_2}{\alpha_1}$ and  we can verify that Equations \eqref{1}-\eqref{4} are satisfied. Thus, $K_9^{11}(\xi_1, \alpha_1)\cong K_9^{11}(\xi_2, \alpha_2)$. The proof is complete.

\end{proof}

\begin{lemma}\label{K_9^{12}}
Let $\xi_1, \gamma_1, \xi_2, \gamma_2\in \F_p$. Then,   $K_9^{12}(\xi_1, \gamma_1)$ and 
$ K_9^{12}(\xi_2, \gamma_2)$ are isomorphic over $\F_p$  if and only if $\xi_1/\xi_2=\epsilon^2$, for some 
 $\epsilon\in  \F_p^*$.
\end{lemma}
\begin{proof}
Let $f=(a_{ij})\in \Aut(L_{5,9})$. Then 
$f: K_9^{12}(\xi_1, \gamma_1)\to K_9^{12}(\xi_2, \gamma_2)$ is an isomorphism 
 if and only if 
\begin{align}
&a_{12}=a_{21}=0\nonumber\\
&a_{11}u\xi_1=\xi_2a_{22}\label{9-12-1} \\
&a_{22}u\gamma_1=\gamma_2 u.\label{9-12-2}
\end{align}
Hence,  $\xi_1/\xi_2=\frac{1}{ a_{11}^2}$
and $\gamma_1/\gamma_2=\frac{1}{a_{22}}$.
The necessity is now clear. For sufficiency, we note that $ a_{11}$ 
is determined by the ratio $\xi_1/\xi_2$ and $a_{22}$ is determined from the equation
$\gamma_1/\gamma_2=\frac{1}{a_{22}}$.
\end{proof}

\begin{lemma}\label{K_9^{13}}
Let $\xi_1, \gamma_1, \xi_2, \gamma_2\in \F_p$. Then, 
 $K_9^{13}(\xi_1, \gamma_1)$ and  $K_9^{13}(\xi_2, \gamma_2)$ are isomorphic over $\F_p$  if and only if $\xi_1/\xi_2=\epsilon^2$,  for some
 $\epsilon \in  \F_p^*$.
\end{lemma}
\begin{proof}
Let $f=(a_{ij})\in \Aut(L_{5,9})$. Then 
$f: K_9^{13}(\xi_1, \delta_1)\to K_9^{13}(\xi_2, \delta_2)$ is an isomorphism 
 if and only if 
\begin{align*}
&a_{12}=a_{21}=0\\
&a_{11}u\xi_1=\xi_2a_{22} \\
&a_{22}u\delta_1=\delta_2 a_{11}u.
\end{align*}
Hence,  $\xi_1/\xi_2=\frac{1}{ a_{11}^2}$
and $\delta_1/\delta_2=\frac{a_{11}}{a_{22}}$.
The necessity is now clear. For sufficiency, we note that $ a_{11}$ 
is determined by the ratio $\xi_1/\xi_2$ and then $a_{22}$ is determined from the equation
$\delta_1/\delta_2=\frac{a_{11}}{a_{22}}$.
 
\end{proof}

\section{Extensions of ($L, \res x_3=x_4$)}
 First, we find a basis for $Z^2(L,\F)$. Let $(\phi,\omega)=(a\Delta _{12}+b\Delta_{13}+c\Delta_{14}+d\Delta_{23}+e\Delta_{24}+f\Delta_{34},\alpha f_1+\beta f_2+ \gamma f_3+\delta f_4)\in Z^2(L,\F)$. Then we must have $\delta^2\phi(x,y,z) =0$ and $\phi(x,\res y)=0$, for all $x,y,z \in L$. Therefore,
\begin{align*}
0=(\delta^2\phi)(x_1,x_2,x_3)&=\phi([x_1,x_2],x_3)+\phi([x_2,x_3],x_1)+\phi([x_3,x_1],x_2)=\phi(x_2,x_4);\\
0=(\delta^2\phi)(x_1,x_2,x_4)&=\phi([x_1,x_2],x_4)+\phi([x_2,x_4],x_1)+\phi([x_4,x_1],x_2)=\phi(x_3,x_4).
\end{align*}
Thus, we get $e=f=0$.
Also, we have $\phi(x,\res y)=0$. Therefore, $\phi(x,x_4)=0$, for all $x\in L$ and hence $\phi(x_1,x_4)=\phi(x_2,x_4)=\phi(x_3,x_4)=0$ which implies that $c=e=f=0$. Therefore, $Z^2(L,\F)$ has a basis consisting of:
$$ (\Delta_{12},0),(\Delta_{13},0),(\Delta_{23},0),(0,f_1),(0,f_2),(0,f_3),(0,f_4).$$
Next, we find a basis for $B^2(L,\F)$. Let $(\phi,\omega)\in B^2(L,\F)$. Since $B^2(L,\F)\subseteq Z^2(L,\F)$, we have $(\phi,\omega)=(a\Delta _{12}+b\Delta_{13}+c\Delta_{23},\alpha f_1+\beta f_2+ \gamma f_3+\delta f_4)$. Note that there exists a linear map $\psi:L\to \F$  such that $\delta^1\psi(x,y)=\phi(x,y)$ and $\tilde \psi(x)=\omega(x)$, for all $x,y \in L$. So, we have
\begin{align*}
&a=\phi(x_1,x_2)=\delta^1\psi(x_1,x_2)=\psi([x_1,x_2])=\psi(x_3), \text{ and }\\
&b=\phi(x_1,x_3)=\delta^1\psi(x_1,x_3)=\psi([x_1,x_3])=\psi(x_4), \text{ and }\\
&c=\phi(x_1,x_4)=\delta^1\psi(x_1,x_4)=\psi([x_1,x_4])=0.
\end{align*}
Also, we have
\begin{align*}
&\gamma=\omega(x_3)=\tilde \psi(x_3)=\psi(\res x_3)=\psi(x_4), \text{ and }\\
&\alpha=\omega(x_1)=\tilde \psi(x_1)=\psi(\res x_1)=0.
\end{align*}
Similarly, we can show that $\beta=\delta=0$.  Note that $\psi(x_4)=b=\gamma$. Therefore, $(\phi,\omega)=(a\Delta_{12}+b\Delta_{13},b f_3)$ and hence $B^2(L,\F)=\la(\Delta_{12},0),(\Delta_{13},f_3)\ra_{\F}$. Note that 
$$[(\Delta_{13},0)],[(\Delta_{23},0)],[0,f_1],[(0,f_2)],[(0,f_3)],[(0,f_4)]$$
spans $H^2(L,\F)$.
Since $[(\Delta_{13},0)]+[(0,f_3)]=[(\Delta_{13},f_3)]=[0]$, then $[(0,f_3)]$ is an scalar multiple of $[(\Delta_{13},0)]$
in $H^2(L,\F)$. Note that $\dim H^2=\dim Z^2-\dim B^2=5$. Therefore, 
$$[(\Delta_{13},0)],[(\Delta_{23},0)],[(0,f_1)],[(0,f_2)],[(0,f_4)]$$
forms a basis for $H^2(L,\F)$.

The group $\Aut(L)$ in this case consists of invertible matrices of the form
\[\begin{pmatrix} a_{11}  & 0 & 0 & 0 \\
a_{21} & a_{22} & 0 & 0\\
a_{31} & a_{32} & d& 0 \\
a_{41} & a_{42} & a_{11}a_{32}& a_{11}d
\end{pmatrix},\]
where $d= a_{11}a_{22}$, $a_{11}d=d^p$, and $a_{31}=a_{32}=0$.

Let $[(\phi,\omega)] \in H^2(L,\F)$. Then we have $\phi=a\Delta_{13}+b\Delta_{23}$, for some $a,b\in \F$. Suppose that $A\phi =a'\Delta_{13}+b'\Delta_{23}$ for some $a',b'\in \F$. We determine $a,b'$. Note that
\begin{align*}
&A\phi (x_1,x_3)=\phi (Ax_1,Ax_3)=\phi (a_{11}x_1+a_{21}x_2+a_{31}x_3+a_{41}x_4,dx_3+a_{11}a_{32}x_4)=a_{11}da+a_{21}db; \text{ and }\\
&A\phi (x_2,x_3)=\phi (Ax_2,Ax_3)=\phi (a_{22}x_2+a_{32}x_3+a_{42}x_4,dx_3+a_{11}a_{32}x_4)=a_{22}db.
\end{align*}
In the matrix form we can write this as
\begin{align}\label{e}
\begin{pmatrix}
a'\\b'
\end{pmatrix}=
\begin{pmatrix}
a_{11}d&a_{21}d\\
0&a_{22}d
\end{pmatrix}
\begin{pmatrix}
a\\b
\end{pmatrix}.
 \end{align} 
The orbit with representative 
$\begin{pmatrix}
0\\1
\end{pmatrix}$  of this action gives us $L_{5,9}$.

Also, we have  $\omega=\alpha f_1+\beta f_2+\delta f_4 $ for some $\alpha, \beta, \delta \in \F$. Suppose that $A\omega =\alpha' f_1+\beta' f_2+\delta' f_4$, for some $\alpha', \beta',\delta' \in \F$. We have
\begin{align*}
&A\omega (x_1)=a_{11}^p \alpha + a_{21}^p \beta+ a_{41}^p \delta;\\
&A\omega (x_2)=a_{22}^p\beta + a_{42}^p\delta;\\
&A\omega(x_4)=a_{11}^pd^p \delta.
\end{align*}
In the matrix form we can write this as
\begin{align}\label{f}
\begin{pmatrix}
\alpha'\\
\beta'\\
\delta'
\end{pmatrix}=
\begin{pmatrix}
a_{11}^p& a_{21}^p&a_{41}^p\\
0&a_{22}^p&a_{42}^p\\
0&0&a_{11}^pd^p
\end{pmatrix}
\begin{pmatrix} 
\alpha \\
\beta \\
\delta
\end{pmatrix}.
\end{align}

Thus, we can write Equations \eqref{e}  and \eqref{f} together as follows:
\begin{align*}
\bigg[\begin{pmatrix}
a_{11}d&a_{21}d\\
0&a_{22}d
\end{pmatrix}, 
\begin{pmatrix}
a_{11}^p& a_{21}^p&a_{41}^p\\
0&a_{22}^p&a_{42}^p\\
0&0&a_{11}^pd^p
\end{pmatrix}\bigg]
\bigg[\begin{pmatrix}
a\\b
\end{pmatrix},
\begin{pmatrix} 
\alpha \\
\beta\\
\delta
\end{pmatrix}\bigg]=
\bigg[\begin{pmatrix}
a'\\b'
\end{pmatrix},
\begin{pmatrix} 
\alpha' \\
\beta' \\
\delta'
\end{pmatrix}\bigg].
\end{align*}
Now we find the representatives of the orbits of the action of $\Aut(L)$ on the set of $\omega$'s such that the orbit represented by 
$\begin{pmatrix}
0\\1
\end{pmatrix}$ is preserved under the action of $\Aut(L)$ on the set of $\phi$'s.
Note that we need to have $A\omega(x_3)=0$ which implies that $a_{11}^pa_{32}^p\delta=0$.

Let $\nu = \begin{pmatrix}
\alpha \\
\beta\\
\delta
\end{pmatrix} \in \F^3$. If $\nu= \begin{pmatrix}
0\\ 
0 \\
0 
\end{pmatrix}$, then $\{\nu \}$ is clearly an $\Aut(L)$-orbit. Suppose that $\delta \neq 0$. Then
\begin{align*}
&\bigg[
\begin{pmatrix}
1&0\\
0&1
\end{pmatrix},
\begin{pmatrix}
 1& 0&-\alpha /\delta\\
0&1&0\\
0&0&1
\end{pmatrix}\bigg]
\bigg[\begin{pmatrix}
0\\1
\end{pmatrix},
\begin{pmatrix}
\alpha \\
\beta\\
\delta
\end{pmatrix}\bigg]=
\bigg[\begin{pmatrix}
0\\1
\end{pmatrix},
\begin{pmatrix}
0\\
\beta\\
\delta
\end{pmatrix}\bigg], \text{ and }\\
&\bigg[
\begin{pmatrix}
1&0\\
0&1
\end{pmatrix},
\begin{pmatrix}
 1& 0&0\\
0&1&-\beta /\delta\\
0&0&1
\end{pmatrix}\bigg]
\bigg[\begin{pmatrix}
0\\1
\end{pmatrix},
\begin{pmatrix}
0\\
\beta \\
\delta
\end{pmatrix}\bigg]=
\bigg[
\begin{pmatrix}
0\\1
\end{pmatrix},
\begin{pmatrix}
0\\
0\\
\delta
\end{pmatrix}\bigg].
\end{align*}

Next, if $\delta =0$ but $\beta \neq 0$, then we have 
\begin{align*}
&\bigg[
\begin{pmatrix}
1&0\\
0&1
\end{pmatrix},
\begin{pmatrix}
 1& 0&0\\
0&1&-\alpha /\beta\\
0&0&1
\end{pmatrix}\bigg]
\bigg[\begin{pmatrix}
0\\1
\end{pmatrix},
\begin{pmatrix}
\alpha \\
\beta\\
0
\end{pmatrix}\bigg]=
\bigg[\begin{pmatrix}
0\\1
\end{pmatrix},
\begin{pmatrix}
\alpha\\
0\\
0
\end{pmatrix}\bigg], 
\end{align*}

Finally, if  $\beta=\delta =0$, but $\alpha \neq 0$, then we have
$\begin{pmatrix}
\alpha\\
0\\
0
\end{pmatrix}$.\\
Thus the following elements are $\Aut_p(L)$-orbit representatives: 
\begin{align*}
\begin{pmatrix}
0\\ 
0\\
0
\end{pmatrix}, 
\begin{pmatrix} 
\alpha\\
0\\
0
\end{pmatrix},
\begin{pmatrix}
0\\
0\\
\delta
\end{pmatrix}.
\end{align*}

Therefore, the corresponding restricted Lie algebra structures are as follows:
\begin{align*}
&K_9^{14}=\langle x_1,\ldots,x_5\mid [x_1,x_2]=x_3,[x_1,x_3]=x_4, [x_2,x_3]=x_5 ,
\res x_3=x_4 \rangle;\\
&K_9^{15}(\alpha)=\langle x_1,\ldots,x_5\mid [x_1,x_2]=x_3,[x_1,x_3]=x_4, [x_2,x_3]=x_5 ,
\res x_1=\alpha x_5, \res x_3= x_4 \rangle;\\
&K_9^{16}(\delta)=\langle x_1,\ldots,x_5\mid [x_1,x_2]=x_3,[x_1,x_3]=x_4, [x_2,x_3]=x_5 ,
\res x_3=x_4, \res x_4=\delta x_5 \rangle.
\end{align*}
where $\alpha , \delta \in \F^*$.

\begin{lemma}\label{K_9^{15}}
Let $\alpha_1, \alpha_2 \in \F_p$. Then, 
 $K_9^{15}(\alpha_1)$ and  $K_9^{15}(\alpha_2)$ are isomorphic over $\F_p$  if and only if $\alpha_1/\alpha_2=\epsilon^2$,  for some
 $\epsilon \in  \F_p^*$.
\end{lemma}
\begin{proof}
Let $f=(a_{ij})\in \Aut(L_{5,9})$. Then 
$f: K_9^{15}(\alpha_1) \to K_9^{15}(\alpha_2)$ is an isomorphism 
 if and only if 
\begin{align*}
&a_{12}=a_{21}=0\\
&a_{22}u\alpha_1=\alpha_2a_{11} \\
&a_{11}u=u.
\end{align*}
Hence,  $a_{11}=1$ and 
$\alpha_1/\alpha_2=a_{22}^{-2}\in ( \F_p^*)^2$. 
To prove the converse, suppose that $\alpha_1/\alpha_2=\epsilon^2$. We take 
$a_{11}=1$ and $a_{22}=\epsilon^{-1}$.
\end{proof}

\begin{lemma}\label{K_9^{16}}
Let $\delta \in \F_p$. Then, 
 $K_9^{16}(\delta)\cong K_9^{16}(1)$.
\end{lemma}
\begin{proof}
We can take the following maps that yields the required automorphism 
$K_9^{16}(\delta)\to  K_9^{16}(1 )$:
$$
\begin{pmatrix}
1 & 0& 0 & 0 &0\\
0 & 1/\delta & 0 & 0 &0\\
0  & 0 & 1/\delta  & 0 &0\\
0 &  0 & 0 &1/\delta & 0\\
0 &  0 & 0 &0& 1/\delta^2 \\
\end{pmatrix}
 $$
\end{proof}

 \section{Detecting isomorphisms}\label{iso-L59}
 The following is the list of all  restricted Lie algebra structures on $L_{5,9}$ and yet as we shall see below we prove that some of them are isomorphic.
\begin{align*}
& K_9^1=\langle x_1,\ldots,x_5\mid [x_1,x_2]=x_3,[x_1,x_3]=x_4, [x_2,x_3]=x_5 \rangle;\\
& K_9^2(\alpha)=\langle x_1,\ldots,x_5\mid [x_1,x_2]=x_3,[x_1,x_3]=x_4, [x_2,x_3]=x_5 ,
\res x_1= \alpha x_5 \rangle;\\
& K_9^3=\langle x_1,\ldots,x_5\mid [x_1,x_2]=x_3,[x_1,x_3]=x_4, [x_2,x_3]=x_5 ,
\res x_2=x_5 \rangle;\\
& K_9^4=\langle x_1,\ldots,x_5\mid [x_1,x_2]=x_3,[x_1,x_3]=x_4, [x_2,x_3]=x_5 ,
\res x_3=x_5 \rangle;\\
& K_9^5(\delta)=\langle x_1,\ldots,x_5\mid [x_1,x_2]=x_3,[x_1,x_3]=x_4, [x_2,x_3]=x_5 ,
\res x_4= \delta x_5 \rangle;\\
& K_9^6=\langle x_1,\ldots,x_5\mid [x_1,x_2]=x_3,[x_1,x_3]=x_4, [x_2,x_3]=x_5 ,
\res x_1=x_4 \rangle;\\
& K_9^7(\beta)=\langle x_1,\ldots,x_5\mid [x_1,x_2]=x_3,[x_1,x_3]=x_4, [x_2,x_3]=x_5 ,
\res x_1=x_4, \res x_2=\beta x_5 \rangle; \\
&K_9^8(\gamma)=\langle x_1,\ldots,x_5\mid [x_1,x_2]=x_3,[x_1,x_3]=x_4, [x_2,x_3]=x_5 ,
\res x_1=x_4, \res x_3=\gamma x_5 \rangle;\\
&K_9^9(\delta)=\langle x_1,\ldots,x_5\mid [x_1,x_2]=x_3,[x_1,x_3]=x_4, [x_2,x_3]=x_5 ,
\res x_1=x_4, \res x_4=\delta x_5 \rangle;\\
&K_9^{10}(\xi)=\langle x_1,\ldots,x_5\mid [x_1,x_2]=x_3,[x_1,x_3]=x_4, [x_2,x_3]=x_5 ,
 \res x_2=\xi x_4 \rangle;\\
&K_9^{11}(\alpha ,\xi)=\langle x_1,\ldots,x_5\mid [x_1,x_2]=x_3,[x_1,x_3]=x_4, [x_2,x_3]=x_5 ,
\res x_1=\alpha x_5, \res x_2=\xi x_4 \rangle;\\
&K_9^{12}(\xi,\gamma)=\langle x_1,\ldots,x_5\mid [x_1,x_2]=x_3,[x_1,x_3]=x_4, [x_2,x_3]=x_5 ,
\res x_2=\xi x_4, \res x_3=\gamma x_5 \rangle;\\
&K_9^{13}(\xi ,\delta)=\langle x_1,\ldots,x_5\mid [x_1,x_2]=x_3,[x_1,x_3]=x_4, [x_2,x_3]=x_5 ,
\res x_2=\xi  x_4, \res x_4=\delta x_5 \rangle;\\
&K_9^{14}=\langle x_1,\ldots,x_5\mid [x_1,x_2]=x_3,[x_1,x_3]=x_4, [x_2,x_3]=x_5 ,
\res x_3=x_4 \rangle;\\
&K_9^{15}(\alpha)=\langle x_1,\ldots,x_5\mid [x_1,x_2]=x_3,[x_1,x_3]=x_4, [x_2,x_3]=x_5 ,
\res x_1=\alpha x_5, \res x_3= x_4 \rangle;\\
&K_9^{16}(\delta)=\langle x_1,\ldots,x_5\mid [x_1,x_2]=x_3,[x_1,x_3]=x_4, [x_2,x_3]=x_5 ,
\res x_3=x_4, \res x_4=\delta x_5 \rangle;
\end{align*}
Note that we have the  following isomorphisms:
$$\begin{pmatrix}
0&1&0&0&0\\
-1&0&0&0&0\\
0&0&1&0&0\\
0&0&0&0&1\\
0&0&0&-1&0
\end{pmatrix}: \quad K_9^3 \to K_9^6,\quad  \quad K_9^{4} \to  K_9^{14},
$$
 and 
$$\begin{pmatrix}
0&1&0&0&0\\
1&0&0&0&0\\
0&0&1&0&0\\
0&0&0&0&1\\
0&0&0&-1&0
\end{pmatrix}:\quad   K_9^{2}(\alpha)\to K_9^{10}(\alpha), \quad  K_9^{15}(\alpha) \to K_9^{12}(\alpha, 1).
$$

\begin{theorem}\label{thm L59}
The list of all restricted Lie algebra structures on $L_{5,9}$, up to isomorphism, is as follows:
\begin{align*}
& L_{5,9}^1=\langle x_1,\ldots,x_5\mid [x_1,x_2]=x_3,[x_1,x_3]=x_4, [x_2,x_3]=x_5 \rangle;\\
& L_{5,9}^2(\alpha)=\langle x_1,\ldots,x_5\mid [x_1,x_2]=x_3,[x_1,x_3]=x_4, [x_2,x_3]=x_5 ,
\res x_1= \alpha x_5 \rangle;\\
& L_{5,9}^3=\langle x_1,\ldots,x_5\mid [x_1,x_2]=x_3,[x_1,x_3]=x_4, [x_2,x_3]=x_5 ,
\res x_2=x_5 \rangle;\\
& L_{5,9}^4=\langle x_1,\ldots,x_5\mid [x_1,x_2]=x_3,[x_1,x_3]=x_4, [x_2,x_3]=x_5 ,
\res x_3=x_5 \rangle;\\
& L_{5,9}^5(\delta)=\langle x_1,\ldots,x_5\mid [x_1,x_2]=x_3,[x_1,x_3]=x_4, [x_2,x_3]=x_5 ,
\res x_4= \delta x_5 \rangle;\\
& L_{5,9}^6(\beta)=\langle x_1,\ldots,x_5\mid [x_1,x_2]=x_3,[x_1,x_3]=x_4, [x_2,x_3]=x_5 ,
\res x_1=x_4, \res x_2=\beta x_5 \rangle; \\
&L_{5,9}^7(\gamma)=\langle x_1,\ldots,x_5\mid [x_1,x_2]=x_3,[x_1,x_3]=x_4, [x_2,x_3]=x_5 ,
\res x_1=x_4, \res x_3=\gamma x_5 \rangle;\\
&L_{5,9}^8(\delta)=\langle x_1,\ldots,x_5\mid [x_1,x_2]=x_3,[x_1,x_3]=x_4, [x_2,x_3]=x_5 ,
\res x_1=x_4, \res x_4=\delta x_5 \rangle;\\
&L_{5,9}^{9}(\xi, \alpha)=\langle x_1,\ldots,x_5\mid [x_1,x_2]=x_3,[x_1,x_3]=x_4, 
[x_2,x_3]=x_5, \res x_1=\alpha x_5, \res x_2=\xi x_4 \rangle;\\
&L_{5,9}^{10}(\xi,\gamma)=\langle x_1,\ldots,x_5\mid [x_1,x_2]=x_3,[x_1,x_3]=x_4, [x_2,x_3]=x_5 ,
\res x_2=\xi x_4, \res x_3=\gamma x_5 \rangle;\\
&L_{5,9}^{11}(\xi ,\delta)=\langle x_1,\ldots,x_5\mid [x_1,x_2]=x_3,[x_1,x_3]=x_4, [x_2,x_3]=x_5 ,
\res x_2=\xi  x_4, \res x_4=\delta x_5 \rangle;\\
&L_{5,9}^{12}(\delta)=\langle x_1,\ldots,x_5\mid [x_1,x_2]=x_3,[x_1,x_3]=x_4, [x_2,x_3]=x_5 ,
\res x_3=x_4, \res x_4=\delta x_5 \rangle.
\end{align*}
\end{theorem}
In the remaining of this section we establish that the algebras given in Theorem \ref{thm L59} are pairwise non-isomorphic, thereby
completing the proof of Theorem \ref{thm L59}. 

It is clear that $L_{5,9}^1$ is not isomorphic to the other restricted Lie algebras.
We claim that $L_{5,9}^2(\alpha)$ and $L_{5,9}^3$ are not isomorphic. Suppose to the contrary that there exists an isomorphism $A: L_{5,9}^2(\alpha)\to L_{5,9}^3$.  Then 
$A(\res x_1)=\res {A(x_1)}$. So, 
\begin{align*}
&A(\alpha x_5)=\res {(a_{11}x_1+a_{21}x_2+a_{31}x_3+a_{41}x_4+a_{51}x_5)}\\ 
&\alpha a_{12}ux_4+\alpha a_{22}ux_5=a_{11}^p\res x_1 + a_{21}^p\res x_2+a_{31}^p\res x_3+a_{41}^p\res x_4+a_{51}^p\res x_5\\
&\alpha a_{12}ux_4+\alpha a_{22}ux_5=a_{21}^px_5,
\end{align*}
which implies that $\alpha a_{12}u=0$. Since $u\neq 0, \alpha \neq 0$, we have $a_{12}=0$. Thus, $u=a_{11}a_{22}$.
Also, we have 
\begin{align*}
A(\res x_2)=&\res {A(x_2)}\\
0=&\res {(a_{12}x_1+a_{22}x_2+a_{32}x_3+a_{42}x_4+a_{52}x_5)}\\
0=&a_{22}^p x_5,
\end{align*}
which implies that   $a_{22}=0$. Therefore $u=0$, which is a contradiction.\\

Next, We claim that $L_{5,9}^2(\alpha)$ and $L_{5,9}^4$ are not isomorphic. Suppose to the contrary that there exists an isomorphism $A: L_{5,9}^2(\alpha)\to L_{5,9}^4$.  Then 
\begin{align*}
A(\res x_3)=&\res {A(x_3)}\\
0=&\res {(ux_3+(a_{11}a_{32}-a_{31}a_{12})x_4+(a_{21}a_{32}-a_{31}a_{22}x_5)}\\
0=&u^p x_5.
\end{align*}
Therefore, $u=0$, which is a contradiction.\\

Next, we claim that $L_{5,9}^2(\alpha)$ and $L_{5,9}^5(\delta)$ are not isomorphic. Suppose to the contrary that there exists an isomorphism $A: L_{5,9}^2(\alpha)\to L_{5,9}^5(\delta)$.  Then 
$A(\res x_4)=\res {A(x_4)}$. So, 
\begin{align*}
0=&\res {(a_{11}ux_4+a_{21}ux_5)}\\ 
0=&a_{11}^pu^p\delta x_5,\\
\end{align*}
which implies that $\a_{11}^pu^p\delta =0$. Since $u\neq 0, \delta \neq 0$, we have $a_{11}=0$. Thus, $u=-a_{12}a_{21}$.
Also, we have 
\begin{align*}
A(\res x_5)=&\res {A(x_5)}\\
0=&\res {(a_{12}ux_4+a_{22}ux_5)}\\
0=&a_{12}^pu^p\delta x_5,
\end{align*}
which implies that   $a_{12}^pu^p\delta=0$. Since $u\neq 0, \delta \neq 0$, we have $a_{12}=0$. Therefore $u=0$, which is a contradiction.\\

It is clear that $L_{5,9}^2(\alpha)$ is not isomorphic to the other restricted Lie algebras.\\

Next, We claim that $L_{5,9}^3$ and $L_{5,9}^4$ are not isomorphic. Suppose to the contrary that there exists an isomorphism $A: L_{5,9}^3\to L_{5,9}^4$.  Then 
\begin{align*}
A(\res x_3)=&\res {A(x_3)}\\
0=&\res {(ux_3+(a_{11}a_{32}-a_{31}a_{12})x_4+(a_{21}a_{32}-a_{31}a_{22}x_5)}\\
0=&u^p x_5.
\end{align*}
Therefore, $u=0$, which is a contradiction.\\

Next, we claim that $L_{5,9}^3$ and $L_{5,9}^5(\delta)$ are not isomorphic. Suppose to the contrary that there exists an isomorphism $A: L_{5,9}^3\to L_{5,9}^5(\delta)$.  Then 
$A(\res x_4)=\res {A(x_4)}$. So, 
\begin{align*}
0=&\res {(a_{11}ux_4+a_{21}ux_5)}\\ 
0=&a_{11}^pu^p\delta x_5,\\
\end{align*}
which implies that $\a_{11}^pu^p\delta =0$. Since $u\neq 0, \delta \neq 0$, we have $a_{11}=0$. Thus, $u=-a_{12}a_{21}$.
Also, we have 
\begin{align*}
A(\res x_5)=&\res {A(x_5)}\\
0=&\res {(a_{12}ux_4+a_{22}ux_5)}\\
0=&a_{12}^pu^p\delta x_5,
\end{align*}
which implies that   $a_{12}^pu^p\delta=0$. Since $u\neq 0, \delta \neq 0$, we have $a_{12}=0$. Therefore $u=0$, which is a contradiction.\\

It is clear that $L_{5,9}^3$ is not isomorphic to the other restricted Lie algebras.\\

Next, we claim that $L_{5,9}^4$ and $L_{5,9}^5(\delta)$ are not isomorphic. Suppose to the contrary that there exists an isomorphism $A: L_{5,9}^4\to L_{5,9}^5(\delta)$.  Then 
$A(\res x_4)=\res {A(x_4)}$. So, 
\begin{align*}
0=&\res {(a_{11}ux_4+a_{21}ux_5)}\\ 
0=&a_{11}^pu^p\delta x_5,\\
\end{align*}
which implies that $\a_{11}^pu^p\delta =0$. Since $u\neq 0, \delta \neq 0$, we have $a_{11}=0$. Thus, $u=-a_{12}a_{21}$.
Also, we have 
\begin{align*}
A(\res x_5)=&\res {A(x_5)}\\
0=&\res {(a_{12}ux_4+a_{22}ux_5)}\\
0=&a_{12}^pu^p\delta x_5,
\end{align*}
which implies that   $a_{12}^pu^p\delta=0$. Since $u\neq 0, \delta \neq 0$, we have $a_{12}=0$. Therefore $u=0$, which is a contradiction.\\

It is clear that $L_{5,9}^4$ and $L_{5,9}^5(\delta)$ are not isomorphic to the other restricted Lie algebras.\\

Next, We claim that $L_{5,9}^6(\beta)$ and $L_{5,9}^7(\gamma)$ are not isomorphic. Suppose to the contrary that there exists an isomorphism $A: L_{5,9}^6(\beta)\to L_{5,9}^7(\gamma)$.  Then 
\begin{align*}
A(\res x_3)=&\res {A(x_3)}\\
0=&\res {(ux_3+(a_{11}a_{32}-a_{31}a_{12})x_4+(a_{21}a_{32}-a_{31}a_{22}x_5)}\\
0=&u^p\gamma x_5.
\end{align*}
Therefore, $u=0$, which is a contradiction.\\

Note that $L_{5,9}^6(\beta)$ is not isomorphic to any of $L_{5,9}^{8}(\delta)$, $L_{5,9}^{11}(\xi,\delta)$, $L_{5,9}^{12}(\delta)$
 because 
$(L_{5,9}^6(\beta))^{[p]^2}=0$  but $(L_{5,9}^8(\delta))^{[p]^2}\neq 0$, $(L_{5,9}^{11}(\xi,\delta))^{[p]^2}\neq 0$, 
$(L_{5,9}^{12}(\delta))^{[p]^2}\neq 0$.\\

We shall  compare  $L_{5,9}^6(\beta)$ and $L_{5,9}^9(\alpha,\xi)$ in Lemma \ref{L6-59}.

Next, We claim that $L_{5,9}^6(\beta)$ and $L_{5,9}^{10}(\xi,\gamma)$ are not isomorphic. Suppose to the contrary that there exists an isomorphism $A: L_{5,9}^6(\beta)\to L_{5,9}^{10}(\xi,\gamma)$.  Then 
\begin{align*}
A(\res x_3)=&\res {A(x_3)}\\
0=&\res {(ux_3+(a_{11}a_{32}-a_{31}a_{12})x_4+(a_{21}a_{32}-a_{31}a_{22}x_5)}\\
0=&u^p\gamma x_5.
\end{align*}
Therefore, $u=0$, which is a contradiction.\\

Note that $L_{5,9}^7(\gamma)$ is not isomorphic to any of $L_{5,9}^{8}(\delta)$, $L_{5,9}^{11}(\xi,\delta)$, $L_{5,9}^{12}(\delta)$
 because 
$(L_{5,9}^7(\gamma))^{[p]^2}=0$  but $(L_{5,9}^8(\delta))^{[p]^2}\neq 0$, $(L_{5,9}^{11}(\xi,\delta))^{[p]^2}\neq 0$, 
$(L_{5,9}^{12}(\delta))^{[p]^2}\neq 0$.\\

Next, We claim that $L_{5,9}^7(\gamma)$ and $L_{5,9}^{9}(\alpha,\xi)$ are not isomorphic. Suppose to the contrary that there exists an isomorphism $A: L_{5,9}^9(\alpha,\xi)\to L_{5,9}^{7}(\gamma)$.  Then 
\begin{align*}
A(\res x_3)=&\res {A(x_3)}\\
0=&\res {(ux_3+(a_{11}a_{32}-a_{31}a_{12})x_4+(a_{21}a_{32}-a_{31}a_{22}x_5)}\\
0=&u^p\gamma x_5.
\end{align*}
Therefore, $u=0$, which is a contradiction.\\

Next, We claim that $L_{5,9}^7(\gamma)$ and $L_{5,9}^{10}(\xi,\gamma)$ are not isomorphic. Suppose to the contrary that there exists an isomorphism $A: L_{5,9}^7(\gamma)\to L_{5,9}^{10}(\xi,\gamma)$.  Then 
\begin{align*}
A(\res x_2)=&\res {A(x_2)}\\
0=&\res{(a_{12}x_1+a_{22}x_2+a_{32}x_3+a_{42}x_4+a_{52}x_5)}\\
0=&a_{22}^p\xi x_4+a_{32}^p\gamma x_5,
\end{align*}
which implies that $a_{22}^p\xi =0$. Since $\xi \neq 0$, we have $a_{22}=0.$
Thus, $u=-a_{12}a_{21}$. Also, we have
\begin{align*}
A(\res x_3)=&\res {A(x_3)}\\
A(\gamma x_5)=&\res {(ux_3+(a_{11}a_{32}-a_{31}a_{12})x_4+(a_{21}a_{32}-a_{31}a_{22}x_5)}\\
\gamma a_{12}ux_4+\gamma a_{22}ux_5=&u^p\gamma x_5,
\end{align*}
which imples that $a_{12}u=0$. Since $u\neq 0$, we have $a_{12}=0$. Therefore, $u=0$, which is a contradiction.\\

Note that $L_{5,9}^8(\delta)$ is not isomorphic to any of $L_{5,9}^{9}(\alpha,\xi)$, $L_{5,9}^{10}(\xi,\gamma)$ 
 because 
$(L_{5,9}^9(\alpha,\xi))^{[p]^2}=(L_{5,9}^{10}(\xi,\gamma))^{[p]^2}=0$  but $(L_{5,9}^8(\delta))^{[p]^2}\neq 0$.\\

Next, We claim that $L_{5,9}^8(\delta)$ and $L_{5,9}^{11}(\xi,\delta)$ are not isomorphic. Suppose to the contrary that there exists an isomorphism $A: L_{5,9}^8(\delta)\to L_{5,9}^{11}(\xi,\delta)$.  Then 
\begin{align*}
A(\res x_2)=&\res {A(x_2)}\\
0=&\res{(a_{12}x_1+a_{22}x_2+a_{32}x_3+a_{42}x_4+a_{52}x_5)}\\
0=&a_{22}^p\xi x_4+a_{42}^p\delta x_5,
\end{align*}
which implies that $a_{22}^p\xi =0$. Since $\xi \neq 0$, we have $a_{22}=0.$
Thus, $u=-a_{12}a_{21}$. Also, we have
\begin{align*}
A(\res x_5)=&\res {A(x_5)}\\
0=&\res {(a_{12}ux_4+a_{22}ux_5)}\\
0=&a_{12}^pu^p\delta x_5,
\end{align*}
which implies that   $a_{12}^pu^p\delta=0$. Since $u\neq 0, \delta \neq 0$, we have $a_{12}=0$. Therefore, $u=0$, which is a contradiction.\\

Next, We claim that $L_{5,9}^8(\delta)$ and $L_{5,9}^{12}(\delta)$ are not isomorphic. Suppose to the contrary that there exists an isomorphism $A: L_{5,9}^{8}(\delta)\to L_{5,9}^{12}(\delta)$.  Then 
\begin{align*}
A(\res x_3)=&\res {A(x_3)}\\
0=&\res {(ux_3+(a_{11}a_{32}-a_{31}a_{12})x_4+(a_{21}a_{32}-a_{31}a_{22}x_5)}\\
0=&u^p x_4+(a_{11}a_{32}-a_{31}a_{12})^p\delta x_5.
\end{align*}
Therefore, $u=0$, which is a contradiction.\\
 
Next, We claim that $L_{5,9}^9(\alpha,\xi)$ and $L_{5,9}^{10}(\xi,\gamma)$ are not isomorphic. Suppose to the contrary that there exists an isomorphism $A: L_{5,9}^{9}(\alpha,\xi)\to L_{5,9}^{10}(\xi,\gamma)$.  Then 
\begin{align*}
A(\res x_3)=&\res {A(x_3)}\\
0=&\res {(ux_3+(a_{11}a_{32}-a_{31}a_{12})x_4+(a_{21}a_{32}-a_{31}a_{22}x_5)}\\
0=&u^p\gamma x_5.
\end{align*}
Therefore, $u=0$, which is a contradiction.\\
 
Note that $L_{5,9}^9(\alpha,\xi)$ is not isomorphic to any of $L_{5,9}^{11}(\xi,\delta)$, $L_{5,9}^{12}(\delta)$
 because 
$(L_{5,9}^9(\alpha,\xi))^{[p]^2}=0$  but $(L_{5,9}^{11}(\xi,\delta))^{[p]^2}\neq 0$, 
$(L_{5,9}^{12}(\delta))^{[p]^2}\neq 0$.\\

 Note that $L_{5,9}^{10}(\xi,\gamma)$ is not isomorphic to any of $L_{5,9}^{11}(\xi,\delta)$, $L_{5,9}^{12}(\delta)$
 because 
$(L_{5,9}^{10}(\xi,\gamma))^{[p]^2}=0$  but $(L_{5,9}^{11}(\xi,\delta))^{[p]^2}\neq 0$, 
$(L_{5,9}^{12}(\delta))^{[p]^2}\neq 0$.\\

Finally, We claim that $L_{5,9}^{11}(\xi,\delta)$ and $L_{5,9}^{12}(\delta)$ are not isomorphic. Suppose to the contrary that there exists an isomorphism $A: L_{5,9}^{11}(\xi,\delta)\to L_{5,9}^{12}(\delta)$.  Then 
\begin{align*}
A(\res x_3)=&\res {A(x_3)}\\
0=&\res {(ux_3+(a_{11}a_{32}-a_{31}a_{12})x_4+(a_{21}a_{32}-a_{31}a_{22}x_5)}\\
0=&u^p x_4+(a_{11}a_{32}-a_{31}a_{12})^p\delta x_5.
\end{align*}
Therefore, $u=0$, which is a contradiction.\\

\begin{lemma}\label{L6-59}
Let $\beta, \alpha, \xi\in \F_p^*$. Then  $L_{5,9}^6(\beta)\cong L_{5,9}^9(\alpha,\xi)$ over $\F_p$  if and 
only if $\beta=-1$ and $\xi/\alpha\in (\F^*)^2$.
\end{lemma}
\begin{proof}
Let $f=(a_{ij})\in \Aut(L_{5,9})$, where $a_{ij}\in \F_p$. Then  $f: L_{5,9}^6(\beta)\to  L_{5,9}^9(\alpha,\xi)$ is an isomorphism  if and only if

\begin{align}
& a_{11}u=\xi a_{21};\label{6-9-1}\\
 &\beta a_{12}u=\xi a_{22};\label{6-9-2}\\
 &a_{21}u=\alpha a_{11};\label{6-9-3}\\
 &\beta a_{22}u=\alpha a_{12}.\label{6-9-4}
 \end{align}
 First suppose that  $L_{5,9}^6(\beta)\cong L_{5,9}^9(\alpha,\xi)$. Then,
 Equations \eqref{6-9-1} and \eqref{6-9-2} imply that
 $a_{11}a_{22}=a_{12}a_{21} \beta$ whereas Equations \eqref{6-9-3} and \eqref{6-9-4} imply that
 $a_{11}a_{22}\beta=a_{12}a_{21}$.
Hence, $\beta^2=1$. Thus,  $(a_{11}a_{22})^2=(a_{12}a_{21})^2$ and since 
 $a_{11}a_{22}\neq a_{12}a_{21}$, we deduce that  $a_{11}a_{22}=- a_{12}a_{21}$.
 So, $\beta=-1$. Also, from Equations 
\eqref{6-9-1} and \eqref{6-9-3} we deduce that $\xi/\alpha=(\frac{a_{11}}{a_{21}})^2$.

To prove the converse, suppose that $\xi/\alpha=\epsilon^2$ and set 
$a_{11}=-\xi/(2\epsilon), a_{12}=\epsilon, a_{21}=\alpha/2, a_{22}=1$. Then, we can see that all the  Equations 
\eqref{6-9-1}-\eqref{6-9-4} are satisfied.

\end{proof}